\documentclass[a4paper,11pt, fleqn]{book}

\usepackage[T1]{fontenc}
\usepackage[utf8]{inputenc}
\usepackage[french,german,english]{babel}

\usepackage{fourier} 
\usepackage{utopia} 
\usepackage{setspace} 

\usepackage{textcomp}
\usepackage{amsfonts}
\usepackage{amssymb}
\usepackage{amsmath}
\usepackage{fancyhdr}
\usepackage{textcomp} 
\usepackage[T1]{fontenc}
\usepackage{amsthm}  
\usepackage{mathrsfs}
\usepackage{url}
\usepackage{graphicx}
\usepackage{graphics}
\usepackage{epsfig}
\usepackage{fancyvrb}
\usepackage{version}
\usepackage{pifont}
\usepackage{float}
\usepackage{enumitem}  
\usepackage{tikz}  
\usepackage{subfigure}
\usepackage{multirow}
\usetikzlibrary{positioning}  
\usepackage{longtable} 
\usepackage{tikz-cd}
\usetikzlibrary{matrix,arrows,decorations.pathmorphing}
\usepackage{xypic}
\usepackage{eufrak}
\usepackage{hyperref}
\usepackage{mathtools}
\usepackage{verbatim}
\usepackage{ytableau}
\usepackage{rotating}
\usepackage{color}

\theoremstyle{definition}
\newtheorem{definition}{Definition}[section]
\newtheorem{example}[definition]{Example}

\newtheorem{remark}[definition]{Remark}

\newtheorem{observation}[definition]{Observation}
\newtheorem{procedure}[definition]{Procedure}
\newtheorem{notation}[definition]{Notation}
\newtheorem{proofspecial}[definition]{Proof}

\theoremstyle{plain}
\newtheorem{theorem}[definition]{Theorem}
\newtheorem{lemma}[definition]{Lemma}
\newtheorem{proposition}[definition]{Proposition}
\newtheorem{corollary}[definition]{Corollary}

\newcommand{\NN}{\mathbb{N}}
\newcommand{\ZZ}{\mathbb{Z}}
\newcommand{\QQ}{\mathbb{Q}}
\newcommand{\CC}{\mathbb{C}}
\newcommand{\RR}{\mathbb{R}}
\newcommand{\AAA}{\mathbb{A}}

\newcommand{\TT}{\mathbb{T}}

\newcommand{\PP}{\mathbb{P}}

\newcommand{\mm}{\mathfrak{m}}

\newcommand{\Su}{\text{Suiv}}
\newcommand{\pos}{\text{pos}}
\newcommand{\ff}{f_{\alpha, \beta}}

\newcommand{\Hpt}{\text{Hilb}^{n,n+1,n+2}(0)}
\newcommand{\Hi}{\text{Hilb}}

\newcommand{\bigslant}[2]{{\raisebox{.05em}{$#1$}\!\left/\raisebox{-.1em}{$#2$}\right.}}

\def\quotient#1#2{%
 \raise1ex\hbox{$#1$}\Big/\lower1ex\hbox{$#2$}%
}

\begin{document}
\pagenumbering{gobble}
\frontmatter
\begin{titlepage}
\cleardoublepage
\centering
\title{Homology of the three flag Hilbert scheme}
\date{July 24, 2016}
\author{Daniele Boccalini}
\maketitle
\end{titlepage}

\markboth{}{}
\tableofcontents
\markboth{}{} 
\newpage

\cleardoublepage
\chapter*{R\'{e}sum\'{e}}
\addcontentsline{toc}{chapter}{Acknowledgements}\markboth{}{}

 On prouve l'existence d'un pavage affine pour le sch\'{e}ma de Hilbert 
\[
\Hpt := \left\{\CC[[x,y]]\supset I_n\supset I_{n+1}\supset I_{n+2}: I_i \,\,\text{ id\'{e}aux avec } \dim_{\CC} \quotient{\CC[x,y]}{I_i} = i  \right\}
\] 
des drapeaux de longueur trois des sous sch\'{e}mas 0-dimensionels qui sont support\'{e}s \`{a} l'origine de $\CC^2$. On atteint ce r\'{e}sultat en montrant que l'espace est stratifi\'{e} par des sous vari\'{e}t\'{e}s lisses, les strata de Hilbert-Samuel. On montre que chacun de ces strata a un pavage affine en cellules de dimension connue et index\'{e}es par des diagrammes de Young marqu\'{e}s.  Le pavage affine nous permet de montrer que les polyn\^{o}mes de Poincar\'{e} de $\Hi^{n,n+1, n+2}(0)$ sont tels que: 
\begin{equation}
\label{generating123intro1}
\sum_{n\geq 0} P_q\left(\Hi^{n,n+1, n+2}(0)\right) z^n =\frac{q+1}{(1-zq)(1-z^2q^2)}\,\,
\prod_{k\geq 1} \frac{1}{1-z^kq^{k-1}}.
\end{equation}
 Dans la preuve de (\ref{generating123intro1}) on construit une correspondance combinatoire entre l'homologie de nos espaces et l'homologie des certains sous espaces connus de  $\Hi^{n+1, n+3}(0)$. On obtient comme corollaire un pavage affine et une formule pour la s\'{e}rie g\'{e}n\'{e}ratrice des polyn\^{o}mes de Poincar\'{e} de $\Hi^{n, n+2}(0)$ pour tous les $n\in \NN$. \\
 
 \vspace{1em}

\textbf{\begin{large}{Mots-cl\'{e}s.}\end{large}} Sch\'{e}ma de Hilbert, homologie, drapeaux d'id\'{e}aux, strata de Hilbert-Samuel, pavage affine.

\chapter*{Abstract}
\addcontentsline{toc}{chapter}{Abstract}\markboth{}{}

We prove the existence of an affine paving for the three-step flag Hilbert scheme
\[
\Hpt := \left\{\CC[[x,y]]\supset I_n\supset I_{n+1}\supset I_{n+2}: I_i \,\,\text{ ideals with } \dim_{\CC} \quotient{\CC[x,y]}{I_i} = i  \right\}
\] 
of 0-dimensional subschemes that are supported at the origin of $\CC^2$. This is done by showing that the space stratifies in smooth subvarieties, the Hilbert-Samuel's strata, each of which has an affine paving with cells of known dimension, indexed by marked Young diagrams. The affine pavings of the Hilbert-Samuel's strata allow us to prove that the Poincar\'{e} polynomials for $\Hi^{n,n+1, n+2}(0)$ satisfy:\begin{equation}
\label{generating123intro}
\sum_{n\geq 0} P_q\left(\Hi^{n,n+1, n+2}(0)\right) z^n = \frac{q+1}{(1-zq)(1-z^2q^2)}\,\,
\prod_{k\geq 1} \frac{1}{1-z^kq^{k-1}}.
\end{equation}
In the process of proving (\ref{generating123intro}) we relate combinatorially the homology of our spaces with that of known subspaces of $\Hi^{n+1, n+3}(0)$. As a corollary we find an affine paving and a formula for the generating function of the Poincar\'{e} polynomials of $\Hi^{n, n+2}(0)$ for all $n\in \NN$. 

 \vspace{1em}

\textbf{\begin{large}{Keywords.}\end{large}} Hilbert scheme, homology, flags of ideals, Hilbert-Samuel's strata, affine paving. 
 \vspace{2em}

\newpage
\thispagestyle{empty}
\mbox{}
\newpage

\chapter*{Acknowledgements}
\addcontentsline{toc}{chapter}{Acknowledgements}\markboth{}{}

I would like to thank my advisor, Tam\'{a}s Hausel, for his fundamental help and encouragement during the preparation of this thesis, and more generally during these four long years of studying and exploring. I am very grateful to him also for the extraordinary group he created and guided here in Epfl. It was honestly a big honor and a lot fun to be part of it. Everyone contributed immensely in making this experience a really enjoyable and interesting adventure, and that is why I want to thank, and thank a lot, present and past members: Michael Groechenig, Michael Wong, Martin Mereb, Zongbin Chen, Mario Garcia Fernandez, Ben Davison, Szil\'{a}rd Szab\'{o}, Alexander Noll, Michael McBreen, and Yohan Brunebarbe. A special thank to Riccardo Grandi and Dimitri Wyss with whom not only I shared the experience of the PhD and a big messy desk, but also a deep and genuine friendship that I hope will last for many years to come. I also owe many thanks to the many visitors that I met during these years and from whom I learned so much. In particular I would like to thank Luca Migliorini for the truly invaluable teachings and support. Finally I would like to thank Pierrette Paulou that has always been the soul of our group: her help has been so fundamental and, more importantly, so kind and generous that it is impossible to overestimate it.

\hspace{3em} I would like to thank Kathryn Hess, Donna Testerman, Andr\'{a}s Szenes and Bal\'{a}zs Szendr\H{o}i for accepting to be part of the jury for my defense and devoting time to review my work.

\hspace{3em} The best part of all this experience has the name of Alexandre, the name of Marie and the name of Martina. Thanks a lot. 

\hspace{3em} Finally I would like to thank my family and my friends: after a lot of research they are still what I like the most. 

\hspace{3em} During these years, my research was supported by the Project Funding number 144119: \emph{Arithmetic harmonic analysis on Higgs, character and quiver varieties} issued by the Swiss National Science Foundation; Foundation that I thank a lot for this.

\chapter*{Introduction}
\addcontentsline{toc}{chapter}{Introduction}\markboth{INTRODUCTION}{}
\hspace{3em}The Hilbert scheme of points of a smooth surface is one of the most beautiful and most studied examples of a moduli space. It is an algebraic geometric object that relates to many branches of mathematics: symplectic geometry, representation theory, combinatorics and, recently, theoretical physics. This special position is ensured on one hand by the relative simplicity and naturalness of its definition, and, on the other hand, by the many interesting structures it is equipped with. Curiously these structures are of two natures: some are directly inherited from the base surface, others appear from the moduli problem.  \\

\hspace{3em} Given a smooth surface $X$, the Hilbert scheme of $n$ points of $X$, denoted as $\Hi^n(X)$, parametrizes 0-dimensional subschemes of $X$ of length $n$. The most generic example is a collection of distinct points of $X$: in this case the length is the number of points. However it is when the points start colliding together that the spectrum of possible scheme structures becomes more and more complicated and its geometry more and more interesting. For example when two points are \emph{infinitely close} the Hilbert scheme remembers the direction along which they came together, i.e. a tangent vector at the collision point. Subschemes that are entirely supported on a single point form a subvariety sometimes called the \emph{punctual Hilbert scheme} and denoted by $\Hi^n(0)$. This space is the same for every surface and every point. It is singular, not even normal, but projective, reduced and irreducible (Haiman \cite{haiman1998t} and Briancon \cite{brianccon1977description}). It precisely measures the difference between $\Hi^n(X)$ and $X^{(n)}$ the $n$-th symmetric power of $X$ that parametrizes $n$-tuples of points up to order. In this sense the punctual Hilbert scheme is of key importance for those structures of $\Hi^n(X)$ that are inherited from $X$: smoothness for example (Fogarty \cite{fogarty1968algebraic}), or an holomorphic symplectic form if $X$ has one (Beauville \cite{beauville1983varietes}). In fact $\Hi^n(0)$ is the most singular fiber of the natural forgetful map $\Hi^n(X) \to X^{(n)}$ that turns out to be a crepant resolution of singularities. The study of the geometry of $\Hi^n(0)$ is also an important step in the work of Haiman \cite{haiman2001hilbert} to prove the combinatorial conjecture of $n!$. In another sense, $\Hi^n(0)$ is interesting as it contains a lot of the topological information of $\Hi^n(X)$. One can prove that $\Hi^n(0)$ is a deformation retract of $\Hi^n(\CC^2)$, and, since $X$ is covered by open subvarieties diffeomorphic to $\CC^2$, $\Hi^n(X)$ is covered by open subvarieties diffeomorphic to $\Hi^n(\CC^2)$. \\

\hspace{3em} In 1987 Ellingsrud and Stromme \cite{ellingsrud1987homology} coronated the efforts of many presenting a neat description of the Borel-Moore homology of $\Hi^n(\CC^2)$ by exploiting a natural torus action on it. The exact form for the Poincar\'{e} polynomial of $\Hi^n(\CC^2)$ became more relevant when Goettsche \cite{gottsche1990betti} considered \emph{all of the Hilbert schemes $\Hi^n(\CC^2)$} for different $n$ \emph{at once} and proved the formula that bears his name
\[
\sum_{n=0}^{+\infty} \sum_{i\geq 0} \dim H_i\left(\Hi^n(\CC^2)\right)q^i z^n \quad = \quad \prod_{k=1}^{+\infty} \frac{1}{1-q^{2k-2}z^k}.
\]
Bundling all the $\Hi^n(\CC^2)$ together, not only produces prettier formulas, but it is the starting point for the study of those additional and somewhat mysterious structures that we mentioned above.  Motivated by Goettsche formula (that holds more generally \cite{gottsche1993perverse}) Witten and Vafa \cite{vafa1994strong} related the study of Hilbert schemes to string theory; Nakajima \cite{nakajima1997heisenberg} constructed a geometric representation of products of the Heisenberg and Clifford algebras on the homology of $\,\,\bigsqcup_{n} \Hi^n(\CC^2)$; Lehn \cite{lehn1999chern} used a Vertex Algebra structure to study the product in cohomology.  This just to cite some examples.\\

\hspace{3em} An important geometric player in studying all the Hilbert schemes together is a space that has also an intrinsic interest: the \emph{flag Hilbert scheme}. This parametrizes flags of 0-dimensional subschemes of specified lengths. Its global geometry deteriorates quickly: $\Hi^{n, n+1}(\CC^2)$ is the last one to be smooth, as Cheah \cite{cheah1998cellular} proves. Again, if we ask for all subschemes to be concentrated in only one point we get quite interesting varieties. For longer flags we get varieties with many irreducible components of different dimensions. The case we are most interested in in this thesis is $\Hi^{n, n+1, n+2}(0)$. The main goal is to prove the following result. 
\begin{theorem}
For every $n\in \NN$ the space $\Hi^{n, n+1, n+2}(0)$ has a cellular decomposition with cells that are isomorphic to affine spaces. These affine cells are indexed by Young diagrams of size $n+2$ with two marked boxes. The dimension of each affine cell is readable from its label and the homology classes of the closures of these cells give a graded basis for the homology of $\Hi^{n, n+1, n+2}(0)$. The Poincar\'{e} polynomials of $\Hi^{n, n+1, n+2}(0)$ for all $n$ fit into a generating function:
\[
\sum_{n=0}^{+\infty}\,\, \sum_{i\geq 0} \dim H_i\left(\Hi^{n, n+1, n+2}(0)\right)q^i z^n \quad = \quad \frac{q+1}{(1-zq^2)(1-z^2q^4)}\, \prod_{k=1}^{+\infty} \frac{1}{1-q^{2k-2}z^k}.
\]
\end{theorem}

The techniques we utilize are by now classical in this  area of studies. They were first used by Ellingsrud and Stromme \cite{ellingsrud1987homology, ellingsrud1988cell} , perfected by Goettsche \cite{gottsche1990betti, gottsche1994hilbert} and used by Cheah \cite{cheah1998cellular} to treat the case of $\Hi^{n, n+1}(0)$. To give some more details on the strategy of the proof we explain the structure of the different chapters.  \\

\hspace{3em}In Chapter I, after the general definitions, we quickly focus on flag Hilbert sche-\\mes of subvarieties concentrated at one point. Here we introduce a stratification due to Iarrobino \cite{iarrobino1972punctual, iarrobino1977punctual} that is key to understand the geometry of $\Hi^{n}(0)$ (and similar spaces). Every point of $\Hi^{n}(0)$ is an ideal $I \subset \CC[[x,y]]$ and as such has a Hilbert-Samuel's type $T(I) \in \NN^{n}$. The Hilbert-Samuel's strata $M_T$ are indexed by the possible Hilbert-Samuel's type $T \in \NN^{n}$ and contain all ideals $I$ such that $T(I)=T$. It turns out that the Hilbert-Samuel's strata are smooth, as Iarrobino \cite{iarrobino1972punctual} proves. In the last section we introduce the famous technique of Bialynicki-Birula \cite{bialynicki1973some} to prove that a \emph{smooth} space with a torus action has, under some conditions, an affine cell decomposition with cells labeled by torus fixed points. A result of Fulton \cite{fulton2013intersection} tells us that the closures of these cells give a graded basis for the Borel-Moore homology of the space. We show that $\Hi^{n}(\CC^2)$ and all related varieties carry a natural two dimensional torus action that comes from rescaling the coordinates of $\CC^2$. We finish the chapter by studying the torus fixed points of $\Hi^{n}(0)$, $\Hi^{n, n+1}(0)$ and $\Hi^{n, n+1, n+2}(0)$ and by relating these fixed points with marked Young diagrams. \\

\hspace{3em} In Chapter II we study the Zariski tangent spaces of $\Hi^{n}(\CC^2)$, $\Hi^{n, n+1}(\CC^2)$ and $\Hi^{n, n+1, n+2}(\CC^2)$ at their fixed points. In particular, we find a basis of eigenvectors for the two dimensional torus action at each fixed point and we interpret these eigenvectors as combinatorial gadgets of the marked Young diagram that labels the fixed point. The bases for the tangent spaces of $\Hi^{n, n+1, n+2}(\CC^2)$, that are our original contributions, are constructed extending the classical study of the similar bases for $\Hi^{n}(\CC^2)$ and $\Hi^{n, n+1}(\CC^2)$. The use of these bases is far reaching. We prove that $\Hi^{n}(\CC^2)$ is smooth and thus we describe a cell decomposition of $\Hi^{n, n+1}(0)$ and its homology (Fogarty \cite{fogarty1968algebraic}, Ellingsrud-Stromme \cite{ellingsrud1987homology}). We prove that all Hilbert-Samuel's strata $M_T\subset \Hi^{n}(\CC^2)$ are smooth and describe their cell decomposition and homology (Iarrobino\cite{iarrobino1977punctual, iarrobino2003family}, Goettsche \cite{gottsche1994hilbert}).  We prove that $\Hi^{n, n+1}(\CC^2)$ is smooth and thus describe a cell decomposition of $\Hi^{n, n+1}(0)$ and its homology (Cheah \cite{cheah1998cellular}). We prove that all Hilbert-Samuel's strata $M_{T_1,T_2}\subset \Hi^{n, n+1}(0)$ are smooth and describe their cell decomposition and homology (Cheah \cite{cheah1998cellular}). All of this crucially relies on smoothness and is possible thanks to the result of Bialynicki-Birula. Recall that unfortunately $\Hi^{n, n+1, n+2}(\CC^2)$ is not smooth. We end the section by giving an original description of the tangent spaces of the Hilbert-Samuel's strata $M_{T_1, T_2, T_3} $ of $ \Hi^{n, n+1, n+2}(0)$.  \\

\hspace{3em} In Chapter III we prove that the Hilbert-Samuel's strata $M_{T_1, T_2, T_3}$ are smooth. To do so we use results of Iarrobino \cite{iarrobino1977punctual} on special opens that cover $M_{T_1}$ and whose points have especially nice generators. We are then able to study the dimension of $M_{T_1, T_2, T_3}$ relating it with that of $M_{T_1}$. This, thanks to the knowledge acquired in Chapter 2 on the tangent spaces, is enough to prove smoothness. We can then apply Bialynicki-Birula decomposition to describe their cell decompositions and homologies, and ultimately the cell decomposition and the homology of $\Hi^{n, n+1, n+2}(0)$. In the last section we show that for longer flag cases, i.e. starting at $\Hi^{n, n+1, n+2, n+3}(0)$, smoothness of the Hilbert-Samuel's strata no longer holds. This means that the case of $\Hi^{n, n+1, n+2}(0)$ is the last case where the classical techniques we described yield interesting results. \\

\hspace{3em} In Chapter IV we prove the formula for the generating function. The fact that $ \Hi^{n, n+1, n+2}(\CC^2)$ is not smooth implies that the cell decomposition we obtained in Chapter 3 might depend on the single appropriate choice of a one dimensional subtorus action. The resulting combinatorics of the Poincar\'{e} polynomials is not well suited to prove the formula for the generating function. In the case of $\Hi^{n}(0)$ and $\Hi^{n, n+1}(0)$, smoothness of the ambient space guarantees that we are free to use any one dimensional subtorus. In fact we obtain many different cell decompositions that give rise to different combinatorial expressions of the same Poincar\'{e} polynomials $P_q\left(\Hi^{n}(0)\right)$ and $P_q\left(\Hi^{n, n+1}(0)\right)$. Understanding better the details of these cases is crucial to prove, combinatorially, that also in the case of $\Hi^{n, n+1, n+2}(0)$ we can rewrite  $P_q\left(\Hi^{n, n+1, n+2}(0)\right)$ more conveniently. Once this is done we actually do not need to sum the results. In fact it turns out that Nakajima and Yoshioka \cite{nakajima2008perverse} already considered the same generating function by studying a different family of smooth subspaces $\Hi^{n-1, n+1}(\CC^2)_{tr} \subset \Hi^{n-1, n+1}(\CC^2)$. Thus we only need to match the combinatorics. It remains unclear if there is also a geometrical connection between their spaces and ours. However we manage to deduce a last original result: an affine cell decomposition of the spaces $\Hi^{n, n+2}(0)$ and a generating function for their Poincar\'{e} polynomials.

\mainmatter
  
\chapter{Fundamental Facts}
\hspace{3em}In this chapter we introduce the geometrical spaces we are interested in and the techniques that will allow us to describe some of their geometrical properties. \\

\hspace{3em} The starting point is the definition of the Hilbert scheme of 0-dimensional subschemes of length $n$ on a smooth surface $X$. This variety parametrizes \emph{configurations} of $n$ points of $X$. We define the \emph{punctual Hilbert scheme} $\Hi^n(0)$ that measures the local difference between the Hilbert scheme $\Hi^n(X)$ and the $n$-th symmetric power of $X$ that parametrizes lists of $n$ points of $X$ up to order. \\

\hspace{3em} We then define the \emph{flag} version of the Hilbert scheme, that parametrizes flags of subschemes of specified length. Again we will be interested in the case where the support of all the subschemes is a single point of $X$ and thus we define the \emph{punctual flag Hilbert scheme}. Following Iarrobino \cite{iarrobino1977punctual}, we will stratify these spaces according to the Hilbert-Samuel's type of the ideals that compose the flags. \\

\hspace{3em} To study the topological properties of the spaces introduced we use the natural torus action induced by the rescaling action on the local coordinates of the plane $\CC^2$. An action with isolated fixed points on a smooth variety $Y$ has attracting sets that are affine cells thanks to the theorem of Bialynicki-Birula \cite{bialynicki1973some}. A result of Fulton \cite{fulton2013intersection} shows that in this situation the homology groups are freely generated by the homology classes of the affine cells. 

\section{Hilbert scheme of points}   

\hspace{3em} We start with some rather general definitions and then quickly specialize them to single out the spaces we are interested in. We give concrete presentations for them and work with these for the rest of the thesis. \\

\hspace{3em}Let $T$ be a locally noetherian  scheme, $X$ a  quasiprojective variety over $T$ and $\mathcal{L}$ a very ample invertible sheaf on $X$ over $T$.
\begin{definition}{ \cite{grothendieck1960techniques}}
 Let $\underline{\text{Hilb}}(\bigslant{X}{T})$ be the contravariant functor from the category of locally noetherian $T$-schemes to the category of sets, which, for locally noetherian $T$-schemes $U, V$ and a morphism $f: U\to V$, is given by:
\begin{align*}
\underline{\text{Hilb}}(\bigslant{X}{T})(U) = \left\{Z\subset X\times U \,\text{ closed subscheme, flat over }\, U \right\},\\
\underline{\text{Hilb}}(\bigslant{X}{T})(f) : \underline{\text{Hilb}}(\bigslant{X}{T})(V) \to \underline{\text{Hilb}}(\bigslant{X}{T})(U); \quad Z \mapsto Z \times_{U} V.
\end{align*}
For $U$ a locally noetherian $T$-scheme and $Z\subset X\times_T U \,\text{ closed subscheme, flat over }\, U$, let $p\colon Z\to X$ and $q\colon Z\to U $ be the two projections and $u\in U$. Define the Hilbert polynomial of $Z$ in $u$ as 
\[
P_u(Z)(m) := \chi (\mathcal{O}_{Z_u}(m))=\chi \left(\mathcal{O}_{Z_u}\otimes_{\mathcal{O_Z}}p^*(\mathcal{L}^m)\right), \qquad m\in \ZZ
\]
where $\chi$ is the Euler characteristic and $Z_u= q^{-1}(u).$ One can prove that $P_u(Z)(m)$ is a polynomial in $m$, independent of $u$, if $U$ is connected. Then we can fix the Hilbert polynomial to create a subfunctor. Let $P\in \QQ[x]$ and define $\Hi^{P}(X)$ to be the subfunctor given by
\[
\underline{\text{Hilb}}^P(\bigslant{X}{T})(U) = \left\{ \left.
\begin{matrix} Z\subset X\times U \\ \text{ closed subscheme} \end{matrix} \quad \right\vert
\begin{matrix} Z \text{ is flat over }\, U  \text{ and } \\ P_u(Z)=P \,\text{ for all } u\in U \end{matrix} 
 \right\}.
\] 
\end{definition}

\begin{theorem}{ \cite{grothendieck1960techniques}}
\label{hilbexists}
Let $X$ be projective over $T$. For every $P\in \QQ[x]$ the functor $\underline{\text{Hilb}}^P(\bigslant{X}{T})$ is representable by a projective $T$-scheme $\Hi^P(\bigslant{X}{T})$. For an open subscheme $Y \subset X$ the functor $\underline{\text{Hilb}}^P(\bigslant{Y}{T})$ is represented by an open subscheme 
\[
{\text{Hilb}}^P(\bigslant{Y}{T}) \quad \subset \quad {\text{Hilb}}^P(\bigslant{X}{T}).
\]
\end{theorem}

\begin{definition}[Hilbert scheme of points]
From now on we will be interested in the case where $T= \text{spec}(\CC)$, and we will write $\Hi^P(X)$ for  ${\text{Hilb}}^P(\bigslant{X}{T})$.  Moreover we will only be interested in the case where $P=n\in \NN$ is a constant polynomial. 
Then we will write either $X^{[n]}$ or $\Hi^n(X)$, and call it the \emph{Hilbert scheme of $n$ points over $X$}. In fact we can identify the closed $X^{[n]}(\CC)$ points with the closed zero-dimensional subschemes of length $n$ of $X$ which are defined over $\CC$. In the most simple case such a scheme is just the set of $n$ distinct points of $X$ with the reduced induced structure, hence the name.
\end{definition}

\begin{definition}
Let $S_n$ be the symmetric group in $n$ letters acting on $X^n$ by permuting the factors. The geometric quotient $\bigslant{X^n}{S_n}$ exists and is called the $n$-th symmetric power of $X$, and is denoted as $X^{(n)}$. We denote the quotient map as follow: 
\[
\Phi_n: X^n \to X^{(n)}.
\]
\end{definition}

The $n$-th symmetric power parametrizes effective zero-cycles of degree $n$ of $X$, i.e. formal linear combinations of points of $X$ with nonnegative integer coefficients that sum to $n$:
\[
\sum_i n_i [x_i] \in X^{(n)} \text{ with } x_i \in X, n_i \in \NN \text{ and } \sum_i n_i = n. 
\] 
We have a stratification into locally closed subsets given by prescribing how many of the points in the support of the zero-cycle actually coincide.
\begin{definition}
Let $\nu = \nu_0 \geq \dots \geq \nu_r $ be a partition of $n$. Denote the diagonals of $X^n$ by  
\[
\Delta_{\nu_i} := \left\{(x_1, \dots, x_{\nu_i}) \left\vert x_1=\dots=x_{\nu_i}\right. \right\} \subset X^{\nu_i}
\]
and define 
\[
X^{n}_{\nu} := \prod_{i} \Delta_{\nu_i} \subset \prod_{i} X^{\nu_i} \subset X^n.
\]
Then we set 
\[
\overline{X^{(n)}_{\nu}} := \Phi_n\left(X^{n}_{\nu}\right) \quad \text{ and }\qquad 
X^{(n)}_{\nu} := \overline{X^{(n)}_{\nu}}\,\, \setminus\, \bigcup_{\mu > \nu}\,\, \overline{X^{(n)}_{\mu}},
\]
where $\mu >\nu$ means that $\mu$ is a coarser partition than $n$. The geometric points of $X^{(n)}_{\nu}$ are 
\[
X^{(n)}_{\nu} = \left\{\sum_{i} n_i [x_i] \in X^{(n)} \,\,\,\left\vert\,\,\text{ the points } x_i \text{ are pairwise distinct } \right.\right\}. 
\]
\end{definition}

Given $Z$ a subscheme of $X$ of length $n$ its support is precisely an effective zero-cycle of degree $n$. This gives the following celebrated relation between $\Hi^{n}(X)$ and $X^{(n)}$. 
\begin{theorem}{ \cite[\S 5.4]{mumford1994geometric} } 
There is canonical morphism called the Hilbert-Chow morphism 
\[
\pi_n: \Hi^{n}(X)\to X^{(n)}
\] 
that at the level of points is given by 
\[
Z \mapsto \sum_{x\in X} \text{len}(Z_x) [x], 
\]
where $\text{len}(Z_x)$ is the multiplicity of $Z$ at the point $x$.
\end{theorem}

The above stratification of $X^{(n)}$ induces a stratification of $\Hi^n(X)$. Define its strata as $X^{[n]}_{\nu}:= \pi_n^{-1}(X^{(n)}_{\nu})$, for each $\nu\vdash n$.  Along a strata of $X^{(n)}$ the fibers of $\pi_n$ are constant and depend only on the dimension of $X$, if $X$ is smooth. For example, in the open smooth strata $X^{(n)}_{(1^n)}$ where all the points are distinct $\pi_n$ is an isomorphism. 
\begin{example}
Let $C$ be a smooth curve. Then the Hilbert-Chow morphism is actually an isomorphism that shows $C^{[n]} \cong C^{(n)}$. In fact there is only a single scheme structure on $n$ points. Suppose that $X=\CC$, then $X^{(n)} \cong X^n$ as Newton's theorem on symmetric functions
 \[
 \CC[x_1, \dots, x_n]^{S_n} \cong \CC[e_1, \dots, e_n]
 \]
 proves. Here $e_i$ is the $i$-th elementary symmetric function. One can also show that $(\PP^1)^{(n)} = \PP^n$. 
 \end{example}

\hspace{3em} We will only be interested in the case of a smooth surface $X$. In this case the symmetric power is singular: as soon as at least two points coincide the stabilizer is not trivial. Famously it turns out that in this case the Hilbert scheme is smooth, as we will see. \\
 
\hspace{3em}From now on we will focus on the case where $X=\CC^2$. Observe that the Theorem of existence \ref{hilbexists} proves that for every $X$, with a covering of opens isomorphic to $\CC^2$, there is an open cover of $\Hi^n(X)$ with opens that are isomorphic to $\Hi^n(\CC)$. This is true also in the category of complex analytic spaces thanks to the definition of the \emph{Duady space}.  We do not need this but we give as a reference \cite{de2000douady}.

\begin{example} Consider the case $X=\CC^2$. Denote $R=\CC[x,y]$ its ring of functions. Then we can identify 
\[
(\CC^2)^{[n]} = \left\{I \subset R\,\, \left\vert \,\,I \text{ is an ideal such that } \dim_{\CC} \left(\bigslant{R}{I}\right)=n \right. \right\}.
\]
We call the dimension of the vector space $\bigslant{R}{I}$ the \emph{length} of $I$. For example if $p_1=(x_1,y_1),\dots,  p_n=(x_n,y_n)$ are $n$ distinct points of $\CC^2$, then there is a unique ideal $I$ of length $n$ of functions that vanish exactly at $p_1, \dots, p_n$. These ideals represent the generic examples. At the other end of the spectrum of examples there are the powers of the maximal ideal $(x, y)^d$ that live in $\Hi^{\frac{d(d+1)}{2}}(\CC^2)$ for every $d \in \NN$. 
\end{example}

\begin{observation}{\cite[Chapter~1]{nakajima1999lectures}}
A possibly more explicit description of $(\CC^2)^{[n]}$ is the one of Nakajima in terms of commuting matrices. Consider $M_x, M_y \in \text{End}(\CC^n)$, and $v \in \text{Hom}_{\CC} (\CC, \CC^n)$. Then define 
\[
\widetilde{H} \,\,:=\,\, \left\{\left(M_x, M_y, v\right) \,\, \left\vert\,\, \begin{matrix} M_xM_y - M_yM_x = 0,\\
\left\langle M^k_xM^l_y(v(1))\,\,\,\left\vert\,\,\, k, l \geq 0\right. \right\rangle = \CC^n
\end{matrix}\right.\right\}\,.
\]
The first condition says that the actions of $M_x$ and $M_y$ commute. The second condition, that is a stability condition, says that there does not exist an $M_x$, $M_y$ invariant subspace of $\CC^n$ that contains the vector $v(1)\in \CC^n$, $1\in \CC$. There is an action of $\text{GL}_n(\CC)$ on such triples given by 
\[
g\cdot \left(M_x, M_y, v\right) :=  \left(gM_xg^{-1}, gM_yg^{-1}, gv\right)
\] 
for $g \in \text{GL}_n(\CC)$. The action turns out to have closed orbits and trivial stabilizers. Nakajima \cite[Theorem~1.9]{nakajima1999lectures} proves that we have the following isomorphism: 
\[
\Hi^n (\CC^2) \cong \bigslant{\tilde{H}}{\text{GL}_n(\CC)}.
\]
The above map is given by the following procedure on closed points. If $I \subset \CC[x,y]$ is an ideal of length $n$ we define $M_x$ as the multiplication action of $x$ on the $n$ dimensional vector space $\bigslant{\CC[x,y]}{I}$. Similarly for $M_y$. The homomorphism $v \in \text{Hom}_{\CC} (\CC, \CC^n)$ is given by $v(1) = 1$ $\text{mod} \,\, I$. It is clear that $M_x$ and $M_y$ commute since the multiplication in $\CC[x,y]$ is commutative. The stability condition follows from the fact that $1$ is a $\CC[x,y]$ generator of $\CC[x,y]$. Conversely given a triple $\left(M_x, M_y, v\right) \in \widetilde{H}$ we can define a map $\phi: \CC[x,y] \to \CC^n$ as $\phi(f)=  f(M_x, M_y)v(1)$. Stability of the triple proves that $\phi$ is surjective and that $I:=\text{Ker}(\phi)$ is an ideal of length $n$. \\

\hspace{3em} Since the two matrices commute we can always simultaneously conjugate both matrices to upper triangular matrices. Call  $(\lambda_1, \dots, \lambda_n)$ and $(\mu_1, \dots, \mu_n)$ the elements on the diagonal of $M_x$ and $M_y$ respectively. Then the Hilbert Chow map sends the triple $\left(M_x, M_y, v\right)$ to $\{(\lambda_1, \mu_1),\dots, (\lambda_n, \mu_n)\} \in (\CC^2)^{(n)}$ .
\end{observation}

\hspace{3em} For a more general surface $X$, not only smoothness, but most of the topological and geometrical properties of $X^{[n]}$ are a mixture of the corresponding properties of the base surface $X$ and of the most singular fiber of the Hilbert-Chow map, sometimes called the \emph{punctual Hilbert scheme}. \\

\hspace{3em} The punctual Hilbert scheme, and its flag versions, are the object of our study. 
\begin{definition}
Consider the Chow morphism $\pi : (\CC^2)^{[n]}\to (\CC^2)^{(n)}$, and denote $0=(0,\dots,0)$ $\in (\CC^2)^{(n)}$. Then the \emph{punctual Hilbert scheme} is 
\[
\Hi^{n}(0) := \pi^{-1} (0)
\]
with the induced scheme structure. 
\end{definition}

The closed points of $\Hi^{n}(0)$ are the schemes of length $n$ whose support is concentrated at the origin $0\in \CC^2$. An example is $(x,y^n)\in (\CC^2)^{[n]}$. If $\mm=(x,y)$ is the maximal ideal in $R=\CC[x,y]$ and $\hat{R}=\CC[[x,y]]$ is the completion of $R$ in $\mm$, then we will see that 
\[
\Hi^{n}(0) = \Hi^n(R)_{red} = \Hi^n(\bigslant{R}{\mm^n})_{red} \subset \text{Gras}(n, \bigslant{R}{\mm^n})\, .
\]
In terms of the description of commuting matrices:
\[
\Hi^{n}(0) = \quotient{\left\{\left(M_x, M_y, v\right) \in \widetilde{H}\,\, \left\vert\,\,
(M_x)^n= (M_y)^n=0\right.\right\} }{\text{GL}_n}\, .
\]
\begin{example}
Pose $n=2$. Then $\Hi^2(0)$ is isomorphic to $\PP^1$. In fact all ideals of length two of functions with zeros only in $(0,0)$ are of the form $(\omega_1 x + \omega_2 y ) +(x^2, xy, y^2)$ as ideals of $\CC[x,y]$, with $[\omega_1 :\omega_2] \in \PP^1$. The intuition behind this is that when two points collide at the origin $(0,0) \in \CC^2$ we remember the direction along which they collided i.e. the vector $[\omega_1 :\omega_2] \in \PP(T_{(0,0)}\CC^2)$, where $T_{(0,0)}\CC^2$ is the tangent space at the origin. \\

\hspace{3em}The global Hilbert scheme $\Hi^2(\CC^2)$ is then stratified as follow 
\[
\begin{tikzcd}
  (\CC^2)^{[2]}_{(1,1)} \bigsqcup  (\CC^2)^{[2]}_{(2)} \arrow[hook]{r} \arrow{d}{\pi_2} &  (\CC^2)^{[2]} \arrow{d}{\pi_2} \\
  (\CC^2)^{(2)}_{(1,1)} \bigsqcup (\CC^2)^{(2)}_{(2)} \arrow[hook]{r} & (\CC^2)^{(2)}
\end{tikzcd}
\]
where the vertical arrow is the Hilbert-Chow map and has fiber $\PP^1 = \Hi^2(0)$ over $(\CC^2)^{(2)}_{(2)}$, where two points coincide, and a single point $\{ \text{ pt }\}$ over $(\CC^2)^{(2)}_{(1,1)}$, where two points are distinct. In fact, in this case, it is easy to see that the Hilbert scheme is the blow up of the 2-symmetric product of $\CC^2$ at the diagonal, that is its singular locus. 
\end{example}

\begin{example}
Pose $n=3$. Then the punctual Hilbert scheme $\Hi^3(0)$ is a cone over $\PP^1$ with an isolated singular point. The singular point is the ideal $\frak{m}^2= (x^2, xy, y^2)$. It is different from all the others in the sense that, if we call $Z_{\infty}$ the corresponding subscheme, then $T_{(0,0)}Z_{\infty}$ is two dimensional, whereas all the other points of $\Hi^3(0)$ have tangent spaces that are one dimensional. Said in another way, the other points correspond to those subschemes that are contained in the germ of a smooth curve. They are called \emph{curvilinear}. They are of the form $I=(y^3, x+\omega_2 y +\alpha y^2)$ plus the ideals $I = (x^3, y+\alpha x^2)$. It is clear that these curvilinear ideals form an affine bundle over $\Hi^2(0)= (\omega_1 x + \omega_2 y ) +(x^2, xy, y^2)$, and that the bundle is compactified with the point $\frak{m}^2$ at infinity. One can write down explicitly a model for $\Hi^3(0)$: consider in $\PP^4 = \text{Proj}(\CC[a, b, c, d, e])$ the projective cone over a rational normal cubic given by equations $ac-b^2, ad-bc, bd-c^2$. Then the family of subschemes of $\CC^2$ parametrized by $\Hi^3(0)$ is the zero set of the ideal $(ax+by+ex^2, bx+cy+exy, cx+dy+ey^2)$. \\

\hspace{3em} The global Hilbert scheme $\Hi^3(\CC^2)$ is then stratified as follow:
\[
\begin{tikzcd}
  (\CC^2)^{[3]}_{(1,1,1)} \,\,\, \bigsqcup \,\,\,(\CC^2)^{[3]}_{(2,1)} \bigsqcup \,\,\,  (\CC^2)^{[3]}_{(3)} \arrow[hook]{r} \arrow{d}{\pi_3}&  (\CC^2)^{[3]} \arrow{d}{\pi_3} \\
  (\CC^2)^{(3)}_{(1,1,1)}\,\,\, \bigsqcup\,\,\, (\CC^2)^{(3)}_{(2,1)}\,\,\, \bigsqcup\,\,\,   (\CC^2)^{(3)}_{(3)} \arrow[hook]{r} &  (\CC^2)^{[2]}
\end{tikzcd}
\]
where the Hilbert Chow map $\pi_3$ has fibers that are respectively isomorphic to a point over $(\CC^2)^{(3)}_{(1,1,1)}$, i.e. where the three points are distinct, to $\Hi^2(0)$ over $(\CC^2)^{[3]}_{(2,1)}$, i.e. where two points coincide and the third is different, and to $\Hi^3(0)$ over $(\CC^2)^{[3]}_{(3)}$ i.e. where the three points coincide.  We will give a description of $\Hi^4(0)$ in the following section.
\end{example}

\hspace{3em}The induced scheme structure of $\pi_n^{-1}(0)$ is reduced as a result of Haiman proves. A celebrated result of Brian\c{c}on proves irreducibility and tells us the dimension of the punctual Hilbert scheme. Irreducibility is equivalent to say that the curvilinear ideals of the form $(y^n , x+\omega_2 y +\alpha_1 y^2 +\dots+\alpha_{n-2}y^{n_1})$ are dense in every $\Hi^n(0)$, $n\in \NN$. 
\begin{theorem}{ \cite[Theorem II.2.3]{brianccon1977description}, \cite[Poposition~2.10]{haiman1998t}}
\label{irreducible}
The punctual Hilbert scheme $\Hi^{n}(0)$ is projective and irreducible of dimension $n-1$. It is a locally complete intersection. 
\end{theorem}

The fact that $\Hi^2(\CC^2)$ is a resolution of $(\CC^2)^{(2)}$ is not accidental. In fact for every $n$ the Hilbert scheme is in some sense the best resolution of singularities of the symmetric product of $\CC^2$. 
\begin{theorem}{ \cite[Theorem~2.4]{fogarty1968algebraic}, \cite[Theorem~3]{beauville1983varietes}} The Hilbert-Chow morphism
$\pi_n : (\CC^2)^{[n]}\to (\CC^2)^{(n)}$ is a symplectic resolution of singularities. 
\end{theorem}

We are only interested in smoothness and not in the symplectic structure. Nowadays there are many different proofs of smoothness for  $(\CC^2)^{[n]}$. We will show it by describing the Zariski tangent space and giving its dimension at each point. \\

\hspace{3em}We now define the flag Hilbert scheme that parametrizes flags of ideals of points. 
\begin{definition}[Flag Hilbert scheme]
For $\mathfrak{n} = (n_1, \dots, n_k) \in \NN^k$, with $n_1< \dots < n_k$, define the \emph{flag Hilbert scheme} and the \emph{flag punctual Hilbert scheme} to be, respectively, 
\begin{align*}
\Hi^{\mathfrak{n}}(\CC^2) &:= \left\{I_{n_i}\in \Hi^{n_i}(\CC^2)\vert I_{n_1} \supset \dots \supset I_{n_k} \right\} \subset \bigtimes_{i=1}^{k}  \Hi^{n_i}(\CC^2), \\
\Hi^{\mathfrak{n}}(0) &:= \left\{I_{n_i}\in \Hi^{n_i}(0)\vert I_{n_1} \supset \dots \supset I_{n_k} \right\} \subset \bigtimes_{i=1}^{k}  \Hi^{n_i}(0)\, .
\end{align*}
\end{definition}
It is clear that we have an Hilbert-Chow map also for the flag Hilbert scheme and thus a corresponding stratification according to the multiplicities of the supports of the schemes in the flag. In particular the dimension of $\Hi^{\mathfrak{n}}(\CC^2)$ is $2n_k$. The geometry of $\Hi^{\mathfrak{n}}(0)$ becomes more complicated even for short flags. For example irreducibility holds only in the case where $\mathfrak{n} = (n, n+1)$, see \cite[Prop.~18, Thorem~19]{chaput2012equivariant}. The homology of $\Hi^{\mathfrak{n}}(0)$ is known for $\mathfrak{n} = n$ and $\mathfrak{n} = (n, n+1)$. We will give a basis of the homology for the case $\mathfrak{n} = (n, n+1, n+2)$. 


\section{Hilbert-Samuel's strata}
In the rest of the thesis we want to understand basic geometrical properties of $\Hi^n(0)$, of $\Hi^{n, n+1}(0)$ and of $\Hi^{n, n+1, n+2}(0)$. In order to do so we introduce a stratification of these spaces due to Iarrobino. 

\begin{definition}[Hilbert-Samuel type]
Let $\hat{R}= \CC[[x,y]]$, $\mm=(x,y)$ be its maximal ideal and $\hat{R}_i= \bigslant{\mm^i}{\mm^{i+1}}$ be the space of forms of degree $i$. Suppose $I \subset \hat{R}$ is an ideal of length $n$.  Then we define $T(I)=(t_i(I))_{i\geq 0} \in \NN^{\infty}$, the \emph{Hilbert-Samuel type}, of $I$ as:
\[
t_i(I):= \dim \left(\bigslant{\mm^i}{I\cap \mm^{i}+\mm^{i+1}}\right)= \dim \bigslant{\hat{R}_i}{I_i} \quad \text{where }\, I_i:= \bigslant{I\cap \mm^i}{I\cap \mm^{i+1}}.
\]
Denote $|T|=\sum_{i\geq 0 } t_i$. Call \emph{initial degree of $I$} the first index $d=d(I)$ such that $t_d< d+1$. 
\end{definition}
\begin{example}
Let $I=(x, y^n) \subset \hat{R}$. Then $I $ has length $n$. We have $t_i(I)=1$ for $0\leq i \leq n-1$ and $t_i(I)=0$ for $i\geq n$. Its initial degree is $1$. Let $\mm^d \subset \hat{R}$. Then the length of $\mm^d$ is $N = \frac{d(d+1)}{2}$. We have $t_i(\mm^d) = i+1$ for $0\leq i \leq d$, and $t_i(\mm^d) = 0$ for $d+1\leq i$. Its initial degree is $d$. 
\end{example}
\begin{lemma}
Let $I\subset \hat{R}$ be an ideal of length $n$ and $T=(t_i(I))_{i\geq 0}$, then
\begin{itemize}
\item[(1)] 
$\dim \bigslant{\mm^j}{I\cap \mm^j} = \sum_{i\geq j} t_i $
for all $j\geq 0$, in particular $|T|=n$.
\item[(2)] $I\supset \mm^n$.
\end{itemize}
\end{lemma}
\begin{proof}
Call $Z = \bigslant{\hat{R}}{I}$ and $Z_i$ the image of $Z$ under the projection map $\hat{R} \to \bigslant{\hat{R}}{I_i}$. Of course we have that $\bigcap_{i\geq 0 } Z_i = 0$. Since $Z$ is finite dimensional, there must exist $i_0$ such that $Z_{i_0}= 0 $, i.e. $\frak{m}^{i_0} \subset  I$.  There are isomorphisms of vector spaces: 
\[
Z_i = \bigslant{\mathfrak{m}^{i}}{\mathfrak{m}^{i} \cap I} \,\, \cong \,\, \bigoplus_{j = i }^{i_0 -1} \bigslant{\hat{R}_j}{I_j}\, .
\]
Then if we choose $i_0$ to be minimal we have $\bigslant{\hat{R}_i}{I_i}\neq 0$ for $i<i_0$. This proves (1). Since $t_j =0$ implies $\frak{m}^{i} \subset  I$, (2) is  a consequence of $|T|= n$.
\end{proof}
\begin{remark}
Since every length $n$ ideal of $\hat{R}$ contains $\mathfrak{m}^n$, we can see it as an ideal in $\bigslant{\hat{R}}{\mathfrak{m}^n}$. Thus the Hilbert scheme $\Hi^n\left(\bigslant{\hat{R}}{\mathfrak{m}^n}\right)$ also parametrizes the ideals of length $n$ in $R$. For the same reason all the reduced schemes $\left(\Hi^n\left(\bigslant{\hat{R}}{\mathfrak{m}^k}\right)\right)_{\text{red}}$ are naturally  isomorphic for $k \geq n$. We will denote one of these by $\left(\Hi^n(\hat{R})\right)_{\text{red}}$. Here $\left(\Hi^n(\hat{R})\right)_{\text{red}}$ is the closed subscheme with the reduced induced structure of the Grassmannian $\text{Grass}(n, \bigslant{\hat{R}}{\mathfrak{m}^n})$ of $n$ dimensional quotients of $\bigslant{\hat{R}}{\mathfrak{m}^n}$ whose geometric points are the ideal of length $n$ in $\bigslant{\hat{R}}{\mathfrak{m}^n}$. Of course the intuition is that, since a point in $\Hi^n(\CC^2)$ that is supported only at the origin is an ideal of $I \subset \CC[x, y]$ such that $\mathfrak{m} \subset {I}$ we can see it as a point of $\left(\Hi^n(\hat{R})\right)_{\text{red}}$ and vice-versa. To be more precise we state the following Lemma that can be found in Goettsche \cite[Chapter~2]{gottsche1994hilbert}.  
\end{remark}
\begin{lemma}{\cite[Lemmas~2.1.2, 2.1.4]{gottsche1994hilbert}}
The natural morphism that maps a subscheme of length $n$ supported at a point to this point  $\pi :(\CC^2)^{[n]}_{(n)}\to \CC^2$ factors through the Hilbert Chow map and is a locally trivial fiber bundle in the Zariski topology with fiber $\left(\Hi^n(\hat{R})\right)_{\text{red}} = \Hi^n(0)$. 
\end{lemma}
\hspace{3em} We can regroup ideals according to their Hilbert-Samuel function to get a stratification of $\Hi^n(\hat{R})_{\text{red}}$.

\begin{definition}[Hilbert-Samuel's strata]
Let $T=(t_i)_{i\geq 0}$ be a sequence of nonnegative integers, with $|T|=n$, we define the \emph{Hilbert-Samuel's stratum} $M_T$ and the \emph{homogenous Hilbert-Samuel's stratum} $G_T$ to be, respectively, 
\begin{align*}
M_T &:= \left\{I \in \Hi^n(0)\,\, \left\vert\,\, T(I)= T \right.\right\} \subset \Hi^n(0),\\
G_T &:= \left\{I \in \Hi^n(0) \,\, \left\vert\,\, T(I)= T, \quad I \text{ is homegenous } \right.\right\} \subset \Hi^n(0).
\end{align*}
Let $\rho_T : M_T \to G_T$ be the surjective morphism that associates to an ideal $I$ the associated homogenous ideal, i.e. the ideal generated by all the initial forms of $f\in I$. It is surjective and the natural embedding $G_T\subset M_T$ is a section. The fact that such a map is well defined is a classical observation, and can be proved, for example, by reasoning a way similar to the proof of Lemma \ref{lemmagoettschecells}. 
If $\mathbf{T}= (T_1, \dots, T_k)$ is a $k$-tuple of sequences of nonnegative integers $T_i=(t_{i,j})_{j\geq 0}$ satisfying $|T_i|=n_i$ then we define 
\begin{align*}
M_{\mathbf{T}} &:= \Hi^{\mathfrak{n}}(0) \cap \left(M_{T_1}\times \cdots\times M_{T_k}\right),\\
G_{\mathbf{T}} &:= \Hi^{\mathfrak{n}}(0) \cap \left(G_{T_1}\times \cdots\times G_{T_k}\right).
\end{align*}
Again we have a morphism $\rho_{\mathbf{T}}\colon M_{\mathbf{T}}\to G_{\mathbf{T}}$ with section $G_{\mathbf{T}}\subset M_{\mathbf{T}}$. It is clear that the strata $M_{\mathbf{T}}$ which are not empty stratify $\Hi^{\mathfrak{n}}(0)$. 
\end{definition}
\begin{lemma}{\cite[Lemma~1.3]{iarrobino1977punctual}}
\label{admissible}
Let $T=(t_i)_{i\geq 0}$ be a sequence of nonnegative integers with $|T|=n$, then $M_T$ and $G_T$ are not empty if and only if 
\[
T = (1, 2, \dots, d, t_d, t_{d+1}, \dots, t_{n-1}, 0, \dots) \quad \text{with } \, d+1> t_d\geq t_{d+1}\geq \dots\geq t_{n-1}\geq 0.
\]
Moreover, if $\mathfrak{n}= (n, n+1, \dots, n+k-1)$, given $\mathbf{T}= (T_1, \dots, T_k)$ with $|T_i|=n_i$, $G_{\mathbf{T}}$ and  $M_{\mathbf{T}}$ are not empty if and only if 
\begin{itemize}
\item[(1)] $T_i = (1, 2, \dots , d_i, t_{i, d}, t_{i, d+1}, \dots) \quad \text{with } d_i+1> t_{i,d}\geq t_{i,d+1}\geq \dots\geq t_{i,n_i-1}\geq 0$
\item[(2)] For all $j=2, \dots k$ there exists and index $m_j$ such that $t_{j,m_j}=t_{j-1, m_j}+1$. 
\end{itemize}
From now on, we call such a $k$-tuple of $T_i$ \emph{admissible}. 
\end{lemma}

\begin{example}
Let $T= (t_i)_{i\geq 0}$ be with $t_i=1$ for $0\leq i \leq n-1$ and $t_i=0$ for $i\geq n$, with $n\in \NN$. Then $G_T$ is isomorphic to $\PP^1$ and it is parametrized as follow: 
\[
G_T = (\omega_1 x +\omega_2 y ) + \mathfrak{m}^n, \quad [\omega_1, \omega_2] \in \PP^1.
\] 
On the other hand $M_T$ fibers on $G_T$ with affine fibers of dimension $n-2$ given, for example on the affine chart $\{( x +\omega_2 y ) + \mathfrak{m}^n \,\vert \omega_2 \in \CC \} \subset G_T$, by 
\[
\left\{(x +\omega_2 y +\alpha_1 y^2 +\dots +\alpha_{n-2} y^{n-1}, y^n) \left\vert\,\, \omega_2, \alpha_i \in \CC \,\, i=1,\dots n-2 \right.\right\}.
\]
One can work out the transition functions on the intersection with the affine chart $\{( \omega_1 x + y ) + \mathfrak{m}^n \,\vert \omega_1 \in \CC \} \subset G_T$ to check that $\rho_T$ is an affine bundle that is not a linear bundle. The ideals in $M_T$ are called \emph{curvilinear} as they arise when $n$ points collide following a trajectory that describes a smooth curve. This is equivalent of saying that the initial degree is one. The Theorem of Brian\c{c}on \ref{irreducible} proves that $M_T\subset \Hi^n(0)$ is a Zariski open and dense subset. In terms of commuting matrices these points can be parametrized as follow:
\[
\begin{pmatrix}
0&1&0&\dots &0\\
0&0&1&\dots &0\\
&\vdots&&\vdots& \\
0&0&\dots&0&1 \\
0&0&0&\dots&0
\end{pmatrix}\,\, ,
\quad \begin{pmatrix}
0&a_1&a_2&\dots &a_n\\
0&0&a_1&\dots &a_{n-1}\\
&\vdots&&\vdots& \\
0&0&\dots&0&a_1 \\
0&0&0&\vdots&0
\end{pmatrix}\,\, ,
\quad
\begin{pmatrix} 0 \\ 0 \\ \vdots \\0\\1
\end{pmatrix}\,. 
\]
\end{example}
\begin{example}
Pose $n=4$. The possible admissible types for an ideal $I \in \Hi^4(0)$ are $T=(1,1,1,1)$ and $T'=(1,2,1)$. As we have seen in the above example $M_T$ is an affine bundle over $\PP^1$ with fiber of dimension $2$. Instead $M_{T'} = G_{T'} \cong \PP^2$ as one can check with the following parametrization:
\[
M_{T'} = \left\{(\phi_1, \phi_2)+\mathfrak{m}^3\,\, \left\vert \,\,\dim_{\CC^2} \,\, \bigslant{\langle \phi_1, \phi_2 \rangle}{\mathfrak{m}^3} = 2,\quad \phi_i \in \mathfrak{m}^2   \right. \right\}. 
\]
\begin{remark}
All of the points of $M_{T'}$ are singular in $\Hi^4(0)$, even though not all points of $M_{T'}$ are analytically equivalent in $\Hi^4(0)$: in fact one can see that there are two possibilities that give rise to different geometrical behaviors. Interestingly enough the two different behaviors can be describe like this: one set of points is the set of points $I$ of  $M_{T'}$ such that there exists at least a point $J_I$ in $M_{(1,2,1,1)}$ with $(I, J_I) \in \Hi^{4,5}(0)$. The other set is the complementary in $M_{T'}$. Of all the examples we wrote up for small $n$ (say $n\leq 11$) similar criteria to identify the analytical type of points in $\Hi^{n}(0)$ always hold. The tangent spaces at points of $\Hi^{n}(0)$ is the subject of the recent paper \cite{bejleri2016tangent}.
\end{remark}
\hspace{3em} The admissible types for flags of two ideals $(I_1, I_2) \in  \Hi^{4,5}(0)$ are the following: 
$T_1, T_2 = (1,1,1,1, 0), (1,1,1,1,1, 0)$, $T_1, T'_2 = (1,1,1,1, 0), (1,2,1,1, 0)$, $T'_1, T'_2 = (1,2,1,0), (1,2,1,1,0)$ and $T'_1, T''_2 = (1,2,1,0), (1,2,2,0)$. We describe an open chart of the Hilbert-Samuel's strata, an obvious change of coordinates of the plane shows how to cover the corresponding stratum with such charts. 
\begin{align*}
M_{T_1, T_2} &= \text{Bundle with fiber } \AAA^3 \text{ over } \PP^1 \supset \left\{ \begin{matrix}(y^4, x+\omega_2 y +\alpha_1 y^2 +\alpha_2 y^3)\supset \\ (y^5, x+\omega_2 y +\alpha_1 y^2 +\alpha_2 y^3+\alpha_3 y^4)\end{matrix} \right\} \, , \\
M_{T_1, T'_2} & = \text{Bundle  with fiber } \AAA^2 \text{ over } \PP^1 \supset \left\{ \begin{matrix}(y^4, x+\omega_2 y +\alpha_1 y^2 +\alpha_2 y^3)\supset \\ \begin{pmatrix}y^4, x^2+\omega_2 xy +\alpha_1 xy^2 +\alpha_2 xy^3, \\ xy+\omega_2 y^2 +\alpha_1 y^3 +\alpha_2 y^4\end{pmatrix}\end{matrix} \right\} \, ,  
\end{align*}
\begin{align*}
M_{T'_1, T'_2} & = \text{Bundle with fiber } \AAA^2 \text{ over } \PP^1 \supset \left\{\begin{matrix}( y^3, x^2+\omega_2 xy, xy +\omega_2 y^2) \supset \\ \begin{pmatrix}y^4, x^2+\omega_2 xy +\alpha_1 y^3\\ xy +\omega_2 y^2+\alpha_2  y^3 \end{pmatrix} \end{matrix} \right\} \, , \\
M_{T'_1, T''_2} & = \text{Bundle with fiber } \PP^1 \text{ over } \PP^2 \supset \left\{\begin{matrix}( y^3, x^2+\alpha_1 y^2, xy +\alpha_2 y^2) \supset \\ (y^3, x^2+\theta xy+(\alpha_1+\theta \alpha_2) y^2 )\end{matrix} \right\} \, . 
\end{align*}

\hspace{3em} The admissible types for flags of three ideals $(I_1, I_2, I_3) \in  \Hi^{4,5, 6}(0)$ are 
\[
\begin{matrix}
T_1, T_2, T_3 = (1,1,1,1), (1,1,1,1,1), (1,1,1,1,1,1)\,,  & T_1, T_2, T'_3 = (1,1,1,1), (1,1,1,1,1), (1,2,1,1,1),\\
T_1, T'_2, T'_3 = (1,1,1,1), (1,2,1,1), (1,2,1,1,1)\,,  & T_1, T'_2, T''_3 = (1,1,1,1), (1,2,1,1), (1,2,2,1,0), \\
T'_1, T'_2, T'_3 = (1,2,1), (1,2,1,1), (1,2,1,1,1)\, , &T'_1, T'_2, T''_3 = (1,2,1), (1,2,1,1), (1,2,2,1)  \\
T'_1, T''_2, T''_3 = (1,2,1), (1,2,2), (1,2,2,1)\, , & T'_1, T''_2, T'''_3 = (1,2,1), (1,2,2), (1,2,3)\, . \\
\end{matrix}
\] 
As the list starts to be too long we give a description of only some of Hilbert-Samuel's strata. \begin{align*}
M_{T_1, T_2, T_3} &= \text{Bundle with fiber } \AAA^4 \text{ over } \PP^1,  \\
M_{T'_1, T'_2, T'_3} & = \text{Bundle  with fiber } \AAA^3 \text{ over } \PP^1,   \\
M_{T'_1, T''_2, T''_3} & = \text{Bundle with fiber } \PP^1 \text{ over  a bundle with fiber } \PP^1 \text{ over } \PP^2 \\
M_{T'_1, T''_2, T'''_3} & = \text{Bundle with fiber } \PP^0= \{\text{ pt }\} \text{ over }  M_{T'_1, T''_2}. 
\end{align*}
\end{example}

\begin{proposition}{ \cite[Theorem~3.13]{iarrobino1977punctual}}
\label{smoothnessIarrobino}
Let $T=(t_i)_{i\geq 0}$ be an admissible sequence of nonnegative integers. Then $M_T$ is smooth, and $G_T$ is smooth and projective. The map $\rho_T: M_T\to G_T$ is an affine fibration, Zariski locally trivial. The dimensions of $M_T$ and of $G_T$ are given in \ref{dimMT}.
\end{proposition}
We will see later a proof of this proposition. The techniques used are central to many of the discussions in this thesis.

\section{Torus action and Borel-Moore homology}

In this section we describe the techniques used to compute the basic topological properties of the spaces we introduced.  There are two main ingredients. The first is a result of Fulton that tells us that if we find an affine cell decomposition for a variety the homology classes of the closure of the cells form a graded basis for the Borel-Moore homology of the variety. The second is a result of Bialynicki-Birula that finds for us an affine cell decomposition for varieties with a nice enough torus action. In the rest of the section we describe a torus action on the relevant spaces and we study the fixed points. \\

\hspace{3em}Borel-Moore homology is historically the preferred homological theory to study the topology of Hilbert schemes. We will only need Proposition \ref{Fulton} below, and we refer at Fulton \cite[Chapter~19]{fulton2013intersection} and reference therein for more details. Borel-Moore homology, indicated with $H_{\ast}$, is singular homology with locally finite supports and integer coefficients. For a space $X$ that is imbedded as a closed subspace of $\RR^n$ one can see that  
\[
H_i (X) \cong H^{n-i} (\RR^n, \RR^n \setminus X) 
\]
where the group on the right is relative singular cohomology with integer coefficients. For a complex scheme $X$ there is a cycle map 
\[
cl: A_{\ast}(X) \to H_{\ast} (X) 
\]
where $A_{\ast}$ is the Chow group of $X$. 

\begin{definition}
Let $X$ be a scheme over $\CC$. A cell decomposition of $X$ is a filtration 
\[
X=X_n \supset X_{n-1} \supset \dots \supset X_{0} \supset X_{-1}=\emptyset
\] 
such that $X_i\setminus X_{i-1}$ is a disjoint union of schemes $U_{i, j}$ isomorphic to affine spaces $\AAA^{n_{i,j}}$ for all $i=0, \dots, n$. We call $U_{i, j}$ the cells of the decomposition. Often we stress the adjective \emph{affine} and say that $X$ has an affine cell decomposition.  
\end{definition}

\begin{proposition}{ \cite[Exercise~19.1.11]{fulton2013intersection}}
\label{Fulton}
Let $X$ be a scheme over $\CC$ with a cell decomposition. Then 
\begin{itemize}
\item[(1)] $H_{2i+1} \,(X) = 0$ for all $i$. 
\item[(2)] $H_{2i} \,(X)$ is the free abelian group generated by the homology classes of the closure of the $i$-dimensional cells, for all $i$.
\item[(3)] The cycle map $\text{cl} : A_{\ast}(X) \to H_{\ast}(X)$ is an isomorphism. 
\end{itemize}
\end{proposition}

\begin{definition}
We will only study the Borel-Moore homologies of varieties that have an affine cell decomposition. In particular all odd dimensional homology groups will be zero so that we will denote $b_i(X) \in \NN$, and call it the \emph{$i$-th Betti number} of the variety $X$, the dimension of $H_{2i}(X)$. Moreover we define $P_q(X) \in \NN[q]$, the \emph{Poincar\'{e}} polynomial of $X$, as: 
\begin{equation}
P_{q} (X) := \sum_{i} q^{b_i(X)} = \sum_{i} q^{\dim H^{2i}(X)}.
\end{equation}
\end{definition}

Let $X$ be a smooth projective variety over $\CC$ with an action of the multiplicative group $\TT = \CC^*$. Denote this action with a dot $"\cdot"$. If $x\in X$ is a fixed point for this action, the torus acts also on the tangent space at $x$, denote it $T_x X$. Let $T^+_xX \subset T_x X$ be the linear subspace on which all the weights of the induced action of $\TT$ are positive. 

\begin{theorem}{ \cite[Theorem 4.4]{bialynicki1973some}}
\label{Bialynicki-Birula}
Let $X$ be a smooth projective variety over $\CC$ with an action of $\TT$. Assume that the set of fixed points is the finite set $\left\{\,x_1, \dots, x_m\,\right\}$. For all $i= 1, \dots, m$ we define the \emph{attracting set} at the fixed point $x_i$ as: 
\[
X_{x_i} = X_{i} := \left\{\, x \in X \,\,\left\vert\,\,\lim_{t\to 0} t\cdot x = x_i\right.\right\}.
\]
Then we have: 
\begin{itemize}
\item[(1)] $X$ has an affine cell decomposition whose cells are the $X_i$. 
\item[(2)] $T_{x_i} X_i = T^+_{x_i} X $. 
\end{itemize}
\end{theorem}
\begin{remark}
The condition of projectivity is only needed to ensure that all the limits $\lim_{t\to 0} t\cdot x$, $x\in X$ actually exists. The theorem remains true if we replace the latter condition with the hypothesis of  projectivity of $X$. 
\end{remark}

\hspace{3em} We have a two dimensional torus action on $\Hi^{\mathfrak{n}}(0)$ that comes from the rescaling action on $\hat{R}=\CC[[x,y]]$. 
\begin{align*}
\text{If }\, \tau= (\tau_1, \tau_2)\in \TT^2  , f=f(x,y) &\in \CC[x,y], \quad \text{  then }  \tau\cdot f = f(\tau_1 x, \tau_2 y)\\
 I&\in \Hi^n(\CC^2),\quad\quad  \tau\cdot I = \left(\tau\cdot f \left\vert f \in I\right)\right.. 
\end{align*}
Of course this action is the restriction of the two torus action on $\Hi^{\mathfrak{n}}(\CC^2)$.\\

\hspace{3em}The following observation is well known and clear.
\begin{lemma}
The fixed points for the  $\TT^2 \circlearrowright \Hi^n(\CC^2)$ action are the monomial ideals in $\Hi^n(\CC^2)$. In particular we have a bijection of sets: 
\begin{align*}
\left\{\text{ fixed points of } \TT^2 \circlearrowright \Hi^n(\CC^2) \right\} & \longleftrightarrow \left\{ \text{ monomial ideals in } \Hi^n(\CC^2) \right\}
\end{align*}
Let $\mathfrak{n} \in \NN^{k}$ for some $k \in \NN$. The fixed points for the $\TT^2 \circlearrowright \Hi^{\mathfrak{n}}(\CC^2)$ action are the flags of monomial ideals in $\Hi^{\mathfrak{n}}(\CC^2)$.
\end{lemma}

Every one-parameter subgroup of $\TT^2$ will then act on $\Hi^{\mathfrak{n}}(0)$.
\begin{definition}
\label{tori}
Let $\mathfrak{n} = (n_1, \dots,n_k) \in \NN^k$. Let $\phi : G_m \to \TT^2$ be a one-parameter subgroup of the form $\phi(\tau) = (\tau^{w_1}, \tau^{w_2})$ with $w_1, w_2 \in \ZZ$. We say that it is \emph{generic} with respect to the action on $\Hi^{\mathfrak{n}}(\CC^2)$ if it has the same fixed points as $\TT^2$. This means avoiding a finite set of given hyperplanes in the lattice of one-parameter subgroups of $\TT^2$. We define two generic one-parameter subgroups that we will use, for different goals, in the rest of the thesis: 
\begin{equation}
\label{one-parameter}
\begin{matrix} 
\TT_{\infty} \text{ generic, with weights } w_1, w_2 \text{ such that } 0<w_1<w_2 \text{ and } 1 \ll \frac{w_2}{w_1}\, , \\
\TT_{1^+} \text{ generic, with weights } w_1, w_2 \text{ such that } 0<w_1<w_2 \text{ and } n_k\cdot w_1 > (n_k-1)\cdot w_2.\end{matrix}
\end{equation}  
Here $1 \ll \frac{w_2}{w_1}$ is relative to $n_k$ and in fact it is enough that $ n_k < \frac{w_2}{w_1}$.
\end{definition}

\hspace{3em} The action of $\TT_{1^+}$ behaves especially well with respect to the stratification in Hilbert-Samuel's strata.  
\begin{lemma}{\cite[Lemma~2.2.9]{gottsche1990betti}}
\ref{lemmagoettschecells} Let $T$ be an admissible sequence of nonnegative integers as in \ref{admissible}. Suppose that we are considering the $\TT_{1^+}$ torus action on $\Hi^{n}(0)$, i.e.  with weights $w_1, w_2$ such that $0<w_1<w_2$ and $nw_1 > (n-1)w_2$. Then we have that: 
\begin{itemize}
\item[(1)] $M_{{T}}$ is a union of attracting sets that are attracting sets of $\Hi^{n}(0)$. 
\item[(2)] $\rho_T : M_{{T}} \to G_{{T}}$ is equivariant with respect to the $\TT_{1^+}$ action.
\item[(3)] The $\TT_{1^+}$ action induces an attracting sets decomposition of $G_{{T}}$. The attracting sets are the intersection of the attracting sets of $M_{{T}} $ with $G_{{T}}$. 
\end{itemize} 

The same is true for $\mathfrak{n}=(n_1, \dots, n_k)$, $T= (T_1, \dots, T_k)$ admissible k-tuple of sequences of nonnegative integers, and the $\TT_{1^+}$ action with weights $w_1, w_2$ such that $0<w_1<w_2$ and $n_kw_1 > (n_k-1)w_2$. 
\end{lemma}
\begin{proof}
We prove it in the case of $\Hi^{n}(0)$ since the proof is the same and we do not want to complicate it with indexes. Let then $I \in \Hi^{n}(0)$ be with Hilbert function $T$. For $j \in \NN$ call $s_j = j +1 -t_j $ the dimension of $I_j$, the space of initial forms of $I$ of degree $j$. Call $J = \lim_{t\to 0} t \cdot I$. We need to prove that $J \in M_{{T}}$. Suppose that the Hilbert function of $J$ is $T' = (t'_i)_{i\geq 0}$. Choose $f_1, \dots, f_{s_j} \in I$ such that their initial forms $g_1, \dots, g_{s_j}$ are a basis for $I_j$, so that the $g_i$ are homogenous of degree $j$. We can assume, up to linear combinations, that the $g_i$ are of the form: 
\[
g_i = x^{l(i)}y^{j-l(i)} + \sum_{m>l(i)}g_{i, m} x^m y^{j-m} \qquad g_{i,m} \in \CC
\]
with $l(1)>l(2)>\dots > l(s_j)$. Then by the choice of weights we have that $\lim_{t\to0} t \cdot f_i = x^{l(i)}y^{j-l(i)}$, in fact other terms of $f_i$ either must have higher degree than $j$ and then go to $0$ faster, either have degree $j$ i.e. are in the support of $g_i$, and then have higher $y$ degree that forces them to go to zero faster. This proves that all the $x^{l(i)}y^{j-l(i)} \in J_j$, and then 
\[
t'_j = j+i - \dim J_j \geq t_j.
\]
Since $J \in \Hi^{n}(0)$, it is still true that $|T'| = n$, and thus $T'=T$. 
\end{proof}
\begin{observation}
In the case of $\Hi^{n}(\CC^2) \supset M_{{T}}$ all the attracting sets are, as we will see,  affine cells. This follows from smoothness of $\Hi^{n}(\CC^2) $, that implies smoothness of $\Hi^{n}(\PP^2)$ and the fact that its attracting sets are affine. It is then an easy observation that the attracting sets for $M_{{T}}$ are some of the attracting sets for $\Hi^{n}(\PP^2)$. This is the only case Goettsche was interested in. For the case $\Hi^{n, n+1}(\CC^2) \supset M_{{T_1, T_2}}$, always thanks to smoothness, the attracting sets are affine cells. The main geometric result of this thesis is that it is still true that the attracting cells are affine for $M_{{T_1, T_2, T_3}}$, since these spaces are smooth, even though the ambient space $\Hi^{n, n+1, n+2}(\CC^2)$ it is not. 
\end{observation}

\hspace{3em} To better deal with the fixed points of the torus action we need to introduce some notations. 
\begin{definition}[Partitions and Young diagrams]
Let $\nu=\nu_0\geq \nu_1\geq\dots $ be a partition  of $n$, i.e. the $\nu_i \in \NN$ are nonnegative integers weakly decreasing that sum to $n$. We will write $\nu \vdash n$ to say that $\nu$ is a partition of $n$. We will also write $|\nu |=n$. The \emph{length} of the partition $\nu$, denoted $\ell(\nu)$, is the minimum index $i$ for which $\nu_i=0$. We will also write a partition $\nu \vdash n$ as 
\[
\nu = (1^{\alpha_1}, 2^{\alpha_2}, \dots, n^{\alpha_n}) \quad \text{ where }\,\, \sum_i \alpha_i i = n \qquad \text{ and }\,\,\alpha_i \in \NN
\]
if the parts of $\nu$ are $\alpha_n$ times $n$, $\dots$, $\alpha_1$ times $1$. We will confuse the $n$-tuple of integers $(\alpha_1, \dots, \alpha_n)$ with $\nu$. \\

Consider the first quadrant of $\RR\times \RR$ as covered by square (boxes) with vertices the points of integer coordinates, side $1$, and  indexed by the coordinate of their left-lower vertex. We denote this set of boxes $\Delta$. A \emph{Young diagram} $\Gamma$ is a finite set of boxes of $\Delta$ such that if $(i, j) \in \Gamma$ then $(i-1, j)$ and $(i,j-1)$ are either in $\Gamma$ or have negative coordinates. Young diagrams are also called \emph{Ferrers} diagrams, and they are in bijection with partitions of integers.  The \emph{Young diagram} of $\nu$, $\Gamma(\nu)$, is the set of boxes labeled by $(0,0), (0,1),\dots (0,\nu_0-1)$, $(1,0), (1,1),\dots (1,\nu_1-1)$, $\dots$, i.e. is the Young diagram with $\nu_i$ boxes in the $i$-th column. \\

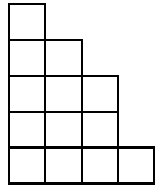
\begin{minipage}{0.45\textwidth}
\begin{center}
\ytableausetup{boxsize=1.2em, aligntableaux=bottom}
\begin{ytableau} 
\, \\
\,&\,  \\
\,&\, &\,\\
\,&\, &\,  \\
\,&\, &\,&\,  \\
\end{ytableau}
\end{center}\end{minipage}
\begin{minipage}{0.4\textwidth}
The Young diagram $\Gamma(\nu)$ associated to $\nu=(5,4,3,1)$. We think of $\Gamma$ as living in the lattice $\NN\times \NN$, that we called $\Delta$. 
\end{minipage}

\vspace{0.5em}
Conversely if we are given a Young diagram we write the partition associated to it as $\nu= \#\{\text{boxes in 0-th column}\} \geq \dots\geq \#\{\text{boxes in i-th column}\}\geq \dots\,$. This is a bijection between Young diagrams and partitions. From now on we will confuse the two and write $\Gamma \vdash n$. For example the \emph{length} of a Young diagram is the number of its columns. 

The \emph{diagonal sequence} of $\nu$ is $T(\nu)=(t_0(\nu), \dots, t_j(\nu), \dots)$, where
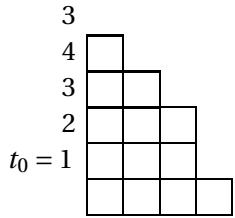
\[
t_i(\nu) :=  \#\, \left\{ \left. (l,j) \in \Gamma(\nu) \right\vert l+j=i \right\}.
\]
\begin{minipage}{0.45\textwidth}
\begin{center}
\ytableausetup{boxsize=1.2em, aligntableaux=bottom}
\begin{ytableau} 
\none[3]\\
\none[4]&\, \\
\none[3]&\,&\,  \\
\none[2]&\,&\, &\,\\
\none[t_0=1\qquad]&\,&\, &\,  \\
\none[]&\,&\, &\,&\,  \\
\end{ytableau}
\end{center}\end{minipage}
\begin{minipage}{0.35\textwidth}
The diagonal sequence of $\Gamma$ is the number of boxes on each antidiagonal. In this case is $T(\Gamma)=(1, 2, 3, 4, 3)$. Of course $|T|:= \sum_i t_i = |\Gamma|$. 
\end{minipage}

\vspace{0.7em}
The \emph{hook difference} of a box $(u,v)\in \Gamma(\nu)$, denoted as $h_{u,v}(\nu)$ or simply as $h_{u,v}$, is 
\[
h_{u,v} :=  \#\, \left\{\left.(l,v) \in \Gamma(\nu)\right\vert l>u \right\}- \#\, \left\{\left.(u,j) \in \Gamma(\nu)\right\vert j>v \right\}.
\]
It is the difference between the number of boxes in $\Gamma(\nu)$ in the same row and to the right of $(u,v)$ and the number of boxes in the same column and above $(u,v)$. \\

\begin{minipage}{0.45\textwidth}
\begin{center}
\ytableausetup{boxsize=1.1em, aligntableaux=bottom}
\begin{ytableau} 
\none[\,]&0 \\
\none[\,]&-1  \\
\none[\,]&0&1 &0\\
\none[\,]&-1&0 &-1  \\
\none[]&2&3 &2&3&2&1&0  \\
\end{ytableau}
\end{center}\end{minipage}
\begin{minipage}{0.35\textwidth}
The boxes of $\Gamma$ are each marked with their hook difference. 
\end{minipage}
\end{definition}

\begin{definition}[Bijection between Young diagrams and torus fixed points of $\Hi^n(0)$]
We can interpret each box in $\NN\times \NN$ as a monomial in $R=\CC[x,y]$: to the box labeled by $(i, j)$ we associate the monomial $x^iy^j$. Then we can associate to each partition a \emph{monomial ideal}. Explicitly if $\nu=\nu_0\geq  \nu_1\geq \dots, \nu_{n-1}\geq 0$ is a partition of $n$, we define $I_{\nu}$ as 
\[
I_{\nu} = \left(y^{\nu_0}, x^1y^{\nu_1}, \dots, x^iy^{\nu_i},\dots,  x^{\ell(n)}\right).
\]
Observe that the monomials that are \emph{not} in $I_{\nu}$ are exactly the monomials labeled by boxes in $\Gamma(\nu)$, so it is clear that $I_{\nu}\in \Hi^n(0)$. Moreover to each $I\in \Hi^n(0)$ we can associated a partition of $n$ by defining 
\[
\nu_i \,\,:=\,\, \min \left\{k \,\,\left\vert\,\, x^i y^k \in I \right.\right\}. 
\] 
This gives a bijection between monomial ideals and partitions. When we look at it as a bijection between Young diagrams and monomial ideals we write $I_{\Gamma}$ or $I(\Gamma)$. \\

\hspace{3em} For a monomial ideal in $I\in \Hi^n(0)$ we define the \emph{standard monomial generators} $(\alpha_0, \dots, \alpha_s)$ as the list of minimal monomial generators of $I$ ordered with decreasing $x$ power. If $\Gamma=\Gamma(I)$ is the corresponding Young diagram, its standard monomial generators are the boxes $(i,j)\in  \Delta\setminus \Gamma$ such that $(i-1, j)$ (and $(i, j-1)$) is either in $\Gamma$ or $i-1<0$ (or $j-1<0$). These are the \emph{external corners} of $\Gamma$.
\end{definition}
\begin{minipage}{0.45\textwidth}
\begin{center}
\ytableausetup{boxsize=1.3em, aligntableaux=bottom}
\begin{ytableau} 
\none[\alpha_4]\\
\,&\, &\none[\alpha_3]\\
\,&\,&\, &\none[\alpha_2] \\
\,&\,&\, &\,\\
\,&\,&\, &\, &\none[\alpha_1] \\
\,&\,&\, &\,&\, &\none[\alpha_0]  \\
\end{ytableau}
\end{center}\end{minipage}
\begin{minipage}{0.50\textwidth}
The standard monomial generators of the represented torus fixed point of $\Hi^n(\CC^2)$. The indexes are labelled so that to a lower index corresponds a monomial with higher $x$ degree. 
\end{minipage}

\hspace{0.4em}

\begin{lemma}
Let $\nu=\nu_0\geq \nu_1\geq\dots $ be a partition of $n$. Then the monomial ideal $I$ associated to the partition $\nu$, i.e. $I = \left(y^{\nu_1},\dots, x^iy^{\nu_i},\dots,  x^{\ell({\nu})}\right)$, satisfies $T(I)= T(\nu)$.
\end{lemma}
\begin{proof}
Let $T(I)= (t_i)_{i\geq 0}$. The monomials $x^iy^j$ with $i+j=l$ and $j\geq \nu_i$ are a basis of the space $I_l$ of homogenous polynomials of degree $l$ in $I$. Then we have 
\begin{align*}
t_l &= l+1- \# \left\{(i,j) \in \NN^2 \left\vert\,\,i+j=l, \,\,\, j\geq \nu_i\, \right.\right\} = \# \left\{(i,j) \in \Gamma(I) \left\vert\,\,i+j=l \right.\right\} = t_l(\Gamma(I)).
\end{align*} 
\end{proof}

\begin{definition}
Let $\Gamma_1$ and $\Gamma_2$ be two Young diagrams of size, respectively, $n$ and $n+1$, such that $\Gamma_1 \subset \Gamma_2$. Then the monomial ideals associated satisfy 
\[
I_{\Gamma_1} \supset I_{\Gamma_2}
\]
and the couple represents a fixed point of $\Hi^{n, n+1}(0)$. In this case we will use the following notation: 
\[
\Gamma = (\Gamma_1, \Gamma_2) \vdash \, [n, n+1].
\]
Clearly there is a bijection between fixed points of $\Hi^{n, n+1}(0)$ and couples $(\Gamma_1, \Gamma_2) \vdash [n, n+1]$.

\hspace{3em}Let $\Gamma_1$, $\Gamma_2$ and $\Gamma_3$ be three Young diagrams of size, respectively, $n$, $n+1$ and $n+2$, such that $\Gamma_1 \subset \Gamma_2 \subset \Gamma_3$. Then the monomial ideals associated satisfy 
\[
I_{\Gamma_1} \supset I_{\Gamma_2} \supset I_{\Gamma_3}
\]
and the triple represents a fixed point of $\Hi^{n, n+1, n+2}(0)$. In this case we will use the following notation: 
\[
\Gamma = (\Gamma_1, \Gamma_2, \Gamma_3) \vdash \, [n, n+1, n+2].
\]
Clearly there is a bijection between fixed points of $\Hi^{n, n+1, n+2}(0)$ and triples $(\Gamma_1, \Gamma_2, \Gamma_3) \vdash [n, n+1, n+2]$.
\end{definition}


\chapter{Tangent spaces at torus fixed points}

\hspace{3em}In this chapter we study the Zariski tangent space of the spaces $\Hi^{n}(\CC^2)$, \\ $\Hi^{n, n+1}(\CC^2)$ , and $\Hi^{n, n+1, n+2}(\CC^2)$ at a torus fixed point. 
 
 \hspace{3em}The stepping stone is an homological interpretation of the tangent spaces as spaces of $R$-homomorphisms, where $R=\CC[x, y]$, that dates back to Grothendieck. Then one can interpret these $R$-homomorphisms in terms of the combinatorial data that describe a torus fixed point. 
 
 \hspace{3em} The chapter is divided in three sections:  one for each of the spaces $\Hi^{n}(\CC^2)$, \\ $\Hi^{n, n+1}(\CC^2)$ , and $\Hi^{n, n+1, n+2}(\CC^2)$. The goal of each section is to define a pure weight basis for the tangent space at each fixed point and to study the weights of the elements of these bases. The study of these weights will also help to understand the tangent spaces of the Hilbert-Samuel's strata of our Hilbert schemes. 
 
 \hspace{3em}In the cases of $\Hi^{n}(\CC^2)$ and $\Hi^{n, n+1}(\CC^2)$ we will then be able to show that the spaces are smooth, and we will thus give graded bases for the homologies of $\Hi^{n}(0)$ and $\Hi^{n, n+1}(0)$. 
 
 \hspace{3em}As we will see the space $\Hi^{n, n+1, n+2}(\CC^2)$ is not smooth. However we will  prove in the next chapter that the Hilbert-Samuel's strata are still smooth, thus allowing us to use the same techniques to give graded bases for their homologies and for that of $\Hi^{n, n+1, n+2}(0)$. \\

 \hspace{3em}The rough strategy is the following: given a fixed point of $\Hi^{n}(\CC^2)$ labelled by $\Gamma_1$ we define a basis $B(\Gamma_1)$ for the tangent space at $I_{\Gamma_1}$. Then we will use $B(\Gamma_1)$ to build a basis $B(\Gamma_1, \Gamma_2)$ for the tangent space at the fixed point $(I_{\Gamma_1}, I_{\Gamma_2}) \in \Hi^{n,n+1}(\CC^2)$ with  $\Gamma_2 \vdash n+1$. The modifications we need to perform on $B(\Gamma_1)$ depend on the combinatorics of the couple of Young diagrams  $(\Gamma_1, \Gamma_2)$. Then we start by $B(\Gamma_1, \Gamma_2)$ to build a basis $B(\Gamma_1, \Gamma_2, \Gamma_3)$ for the tangent space of $\Hi^{n,n+1,n+2}(\CC^2)$  where  $\Gamma_3 \vdash n+2$. Most of the modifications needed will only depend on the couple $(\Gamma_2, \Gamma_3)$, exactly as in passing from $B(\Gamma_2)$ to $B(\Gamma_2, \Gamma_3)$.  In fact only few (in the general case only one) modifications will actually depend on the full triple $(\Gamma_1, \Gamma_2, \Gamma_3)$. This is really the key philosophical point of most of the arguments in this thesis: the geometrical or combinatorial properties of a flag of three ideals $(I_1, I_2, I_3)$ can be understood by looking at the  corresponding properties for two  flags of two ideals, i.e.  $(I_1, I_2)$ and $(I_2, I_3)$, independently, and then taking into account some, in our cases always manageable, properties that are truly intrinsic to the triple. \\

 \hspace{3em}We start the chapter with the interpretation of the tangent spaces of the various Hilbert schemes in terms of $R$-homomorphisms.

\begin{lemma}{\cite{grothendieck1960techniques}}
\label{Rmorphisms}
Let $I \in \Hi^n(\CC^2)$ be a fixed point and denote with $T_I\left(\Hi^n(\CC^2) \right)$ the Zariski tangent space of $\Hi^n(\CC^2)$ at $I$. Then we have a natural $\TT^2$-equivariant isomorphism  
\[
T_I \Hi^n(\CC^2) \quad \cong\quad \text{Hom}_R (I, \bigslant{R}{I}). 
\]
Let $\,\,\mathfrak{n}=(n, n+1, n+2, \dots, n+k)$. Let $\mathbf{I} = (I_1, I_{2}, \dots, I_{k}) \in \Hi^{\mathfrak{n}}(\CC^2)$ be a fixed point, and denote $T_{\mathbf{I}}\left(\Hi^{\mathfrak{n}}(\CC^2) \right)$ the Zariski tangent space of $\Hi^{\mathfrak{n}}(\CC^2) $ at the point $\mathbf{I}$. For $1\leq i<j\leq k$ we have $I_i \supset I_{j}$ so that we can define the obvious maps:
\begin{align*}
\phi_{ij} : \text{Hom}_R (I_{i}, \bigslant{R}{I_i}) \to \text{Hom}_R (I_{j}, \bigslant{R}{I_i}),\quad & (f:I_i\to \bigslant{R}{I_i})\mapsto (\left.f\right\vert_{I_{j}}:I_j\to \bigslant{R}{I_i}),   \\
\psi_{ij} : \text{Hom}_R (I_{j}, \bigslant{R}{I_j}) \to \text{Hom}_R (I_{j}, \bigslant{R}{I_i}), \quad & (f:I_j\to \bigslant{R}{I_j})\mapsto (p\circ f:I_j\to \bigslant{R}{I_j}\twoheadrightarrow \bigslant{R}{I_i})\, . 
\end{align*}
Define also the projection maps \[\pi_{ij}: \bigoplus_{l=1}^k  \text{Hom}_R (I_{i}, \bigslant{R}{I_i}) \to  \text{Hom}_R (I_{i}, \bigslant{R}{I_i})\oplus \text{Hom}_R (I_{j}, \bigslant{R}{I_j})\, .\]
Then we have a $\TT^2$-equivariant isomorphism  
\[
T_{\mathbf{I}}\left(\Hi^{\mathfrak{n}}(\CC^2) \right)\quad \cong \bigcap_{1\leq i<j\leq k} \left(\text{Ker}(\phi_{ij}-\psi_{ij})\circ \pi_{ij}\right). 
\]
\end{lemma}

 \hspace{3em}Now that we know what the tangent spaces are, we only need to find weight bases for them.


\section{A weight basis for $T_I \,\Hi^n(\CC^2)$}

 \hspace{3em}Suppose that $I \in \,\Hi^n(\CC^2)$ is a torus fixed point. The tangent space $T_I \,\Hi^n(\CC^2)$ at $I$ is then equipped with the torus action. The goal of this section is to understand the tangent space $T_I \,\Hi^n(\CC^2)$ in terms of $R$-homomorphisms $\text{Hom}_R (I, \bigslant{R}{I})$, and in particular of those $R$-homomorphisms that are of pure weights with respect to the torus action. The goal is to visualize them as arrows of boxes of $\Gamma(I)$.\\

 \hspace{3em}It is clear that, to describe an $f \in \text{Hom}_R (I, \bigslant{R}{I})$ we need only to prescribe the images of the generators of $I$, and we need only to describe these images in terms of linear combinations of monomials in $\Gamma(I)$. It is good to visualize the situation in terms of Young diagrams: on the left we have, in black, the set of standard monomial generators of $I$ and on the right we have, in gray,  the elements of $\Gamma(I)$ i.e. those monomials that are \emph{not} in $I$. 

\begin{minipage}{0.46\textwidth}
\begin{center}
\ytableausetup{boxsize=1.5em, aligntableaux=bottom}
\begin{ytableau} 
*(black) \\
\,&*(black)  \\
\,&\, &*(black) \\
\,&\, &\,  \\
\,&\, &\,&*(black)  \\
\,&\, &\, &\,&\, &\,&*(black)  \\
\,&\, &\, &\, &\, &\, &\, &*(black) \\
\end{ytableau}

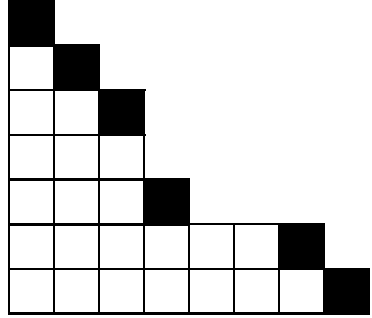
\small{In black the standard monomial generators of $I$. }
\end{center}
\end{minipage}
\begin{minipage}{0.46\textwidth}
\begin{center}
\begin{ytableau} 
*(gray)  \\
*(gray) &*(gray)  \\
*(gray) \,&*(gray) \, &*(gray) \, \\
*(gray) \, &*(gray) \, &*(gray) \, \\
*(gray) \,&*(gray) \, &*(gray) \, &*(gray) \,&*(gray) \, &*(gray) \, \\
*(gray) \,&*(gray) \, &*(gray) \, &*(gray) \, &*(gray) \, &*(gray) \, &*(gray) \, \\
\end{ytableau}

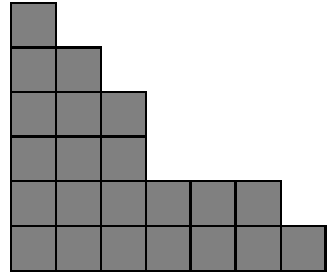
\small{In gray the elements of $\Gamma(I)$. }
\end{center}
\end{minipage}
\vspace{0.8em }

 \hspace{3em}If we only look for maps $f$ of pure weight, the image of a generator $\alpha$ must be itself a scalar multiple of a monomial $\beta$ in $\Gamma(I)$, so graphically $\alpha$ moves $p$ boxes to the left and $q$ boxes upward, where $p$ and $q$ can be negative, and reaches $\beta$ inside $\Gamma(I)$. In terms of monomials $f(\alpha)=c\beta= c x^py^q\alpha $ with $x^py^q \in \CC[x^{\pm 1}, y^{\pm 1}]$, $c\in \CC$. If the scalar $c$ is not zero, the fact that $f$ is an $R$-homomorphism forces every other $\alpha' \in I$ to be sent either to $c x^py^q\alpha'$ or to zero; graphically they must move by the same exact translation, or they must go to zero. For $\alpha$ a standard monomial generator of $I$, and $\beta\in \Gamma(I)$, call 
\[S_{\alpha, \beta} := \left\{f \in \text{Hom}_R (I, \bigslant{R}{I})\vert f(\alpha) = \beta \right\}.
\]
 \hspace{3em}It can happen that $S_{\alpha, \beta}= \emptyset$. For example, if $I=\mathfrak{m}^{2}= (x^2,xy,y^2) \in \Hi^3(0)$, then there does not exist $f \in \text{Hom}_R (I, \bigslant{R}{I})$ that sends $x^2\mapsto 1$.

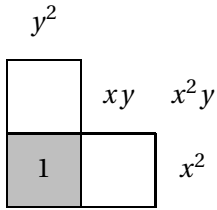
\begin{minipage}{0.45\textwidth}
\ytableausetup{boxsize=2.5em, aligntableaux=bottom}
\begin{ytableau}
\none[y^2] \\
\,& \none[xy]&\none[x^2y] \\
*(lightgray)1&\, &\none[x^2] \\
\end{ytableau}
\end{minipage}
\begin{minipage}{0.45\textwidth}
If $x^2\mapsto 1$ then $x^2 y \mapsto y$, 
but then $xy\mapsto \gamma$ cannot be defined. In fact its image $\gamma \in \bigslant{R}{I}$ should be such that $x\gamma= y$. This is impossible. 
\end{minipage}
\vspace{1em}

To understand for which $\alpha$ and $\beta$ we have that $S_{\alpha, \beta}\neq \emptyset$ we introduce the following notations. 

\begin{definition}
\label{defPalpha}
Let $\nu = \nu_0\geq \nu_1\dots\geq \nu_{n-1}$ be a partition of $n$, let $\Gamma$ be the associated Young diagram and let $I=I_{\Gamma}=(\alpha_0,\dots,  \alpha_s)$ be the monomial ideal associated to $\Gamma$, i.e. $I=\left(y^{\nu_1},\dots, x^iy^{\nu_i},\dots,  x^{\ell({\nu})}\right)$. We define
\begin{align*}
p_i &:= \deg_y \alpha_{i+1} - \deg_y \alpha_{i} = \text{ vertical distance between } \alpha_{i} \text{ and } \alpha_{i+1}, \qquad \,\,\,\,p_s := \infty, \\
q_i& := \deg_x \alpha_{i} - \deg_x \alpha_{i-1} = \text{ horizontal distance between } \alpha_{i} \text{ and } \alpha_{i-1}, \qquad q_0 := \infty.
\end{align*}
For each $\alpha = \alpha_i$ we  also define
\begin{align*}
P_{\alpha} = P_{\alpha_i} &:= \left\{\left.\beta \in \Gamma \right\vert \deg_x \beta < \deg_x \alpha \text{ and } \beta y^{p_i} \in I \right\}, \qquad P_{\alpha_s} := \emptyset, \\
Q_{\alpha} = Q_{\alpha_i} &:= \left\{\left.\beta \in \Gamma \right\vert \deg_y \beta < \deg_y \alpha \text{ and } \beta x^{q_i} \in I \right\}, \qquad Q_{\alpha_0} := \emptyset.
\end{align*} 
\end{definition}

\begin{example}
Let $I$ be the monomial ideal represented by the diagram below. Let $\alpha=\alpha_3$ be the generator of $I$ marked in the picture with its name. We have that $p_3=2$ and $q_3=3$. The elements of $P_{\alpha}$ are indicated with a $p$ while the elements of $Q_{\alpha}$ are indicated with a $q$. 
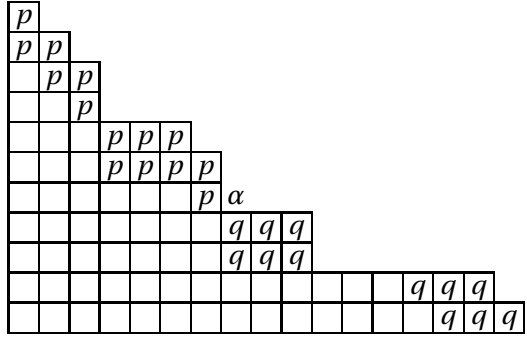
\begin{center}
\ytableausetup{boxsize=1.0em}
\begin{ytableau} 
p \\
p&p \\
\,&p &p \\
\,&\, &p \\
\,&\, &\, &p&p &p \\
\,&\, &\, &p&p &p &p \\
\,&\, &\, &\,&\, & \, & p &\none[\alpha] \\
\,&\, &\, &\,&\, &\, &\,&q &q &q \\
\,&\, &\, &\,&\, &\, &\,&q &q &q \\
\,&\, &\, &\,&\, &\, &\,&\, &\, &\,&\, &\, &\,&q &q &q \\
\,&\, &\, &\,&\, &\, &\,&\, &\, &\,&\, &\, &\,&\, &q &q &q
\end{ytableau}
\end{center}
\end{example}

The fact that $\beta \in P_{\alpha}\cup Q_{\beta}$ implies that there exists an $f \in S_{\alpha, \beta}$, as the next Lemma proves. 
\begin{lemma}[definition]
\label{definitionf}
Let $I=(\alpha_0, \dots, \alpha_s)$ be a monomial ideal of length $n$, and let  $\alpha= \alpha_i$ be one of its standard monomial generators. Let $\beta \in P_{\alpha} \cup Q_{\alpha}$. Then $S_{\alpha, \beta} \neq \emptyset $. In this case we define $f_{\alpha, \beta} \in S_{\alpha, \beta}$ as follow:\\

\hspace{3em}If $\beta \in P_{\alpha}$, then $\beta \in \Gamma(I)$ is $q$ boxes to the left and $p$ boxes above $\alpha$, where $p,q\geq 0$. We define $f_{\alpha, \beta}$ by prescribing the images of each of the generators of $I$ as:
\begin{align*}
\text{if } \beta \in P_{\alpha_i}, \beta= \alpha_i \left(\frac{y^p}{x^q}\right), \quad \text{ then }\,\,\,\,\,\, f_{\alpha_i, \beta}(\alpha_k) = 
\begin{cases}
\alpha_k \left(\frac{y^p}{x^q}\right) &\text{if } k  < i, \\
\beta &\text{if } k=i, \\
0 &\text{if } i< k.
\end{cases}  
\end{align*}
\hspace{3em}If $\beta \in Q_{\alpha}$, then $\beta \in  \Gamma(I)$ is $q$ boxes to the right and $p$ boxes below $\alpha$, where $p,q\geq 0$. We define $f_{\alpha, \beta}$ by prescribing the images of each of the generators of $I$ as:
\begin{align*}
\text{if } \beta \in Q_{\alpha_i}, \beta= \alpha_i \left(\frac{x^q}{y^p}\right), \quad \text{ then }\,\,\,\,\,\, f_{\alpha_i, \beta}(\alpha_k) = 
\begin{cases}
0 &\text{if } k < i, \\
\beta &\text{if } k=i, \\
\alpha_j \left(\frac{x^q}{y^p}\right) &\text{if } i< k.
\end{cases}  
\end{align*}
\end{lemma}
\begin{proof}
Suppose that $\alpha=\alpha_3$ is the one depicted below and that $\beta \in Q_{\alpha}$ is the box two boxes below it and one to the right, i.e. the box marked with the $3$. Then $f_{\alpha, \beta}$ is the homomorphism depicted in the picture, where with the index $k$  we denote the box inside $\Gamma(I)$  that is the image of $\alpha_k$ if and only if this image is not zero. 
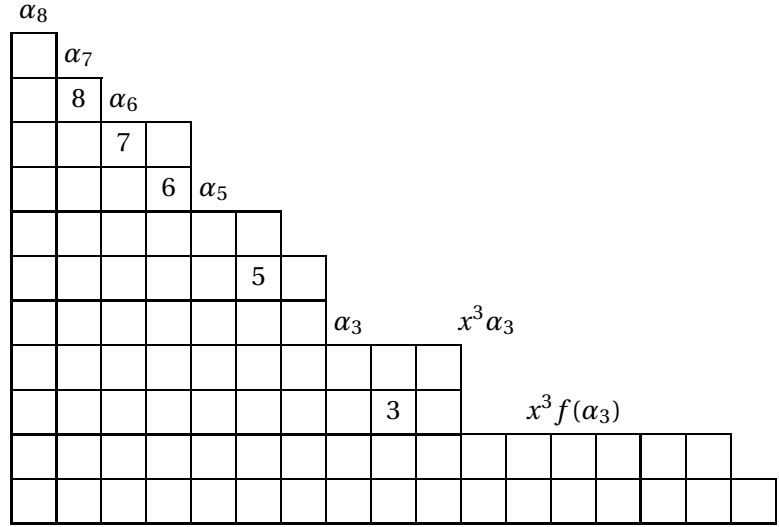
\begin{center}
\ytableausetup{boxsize=1.5em}
\begin{ytableau} 
\none[\alpha_8] \\
\, &\none[\alpha_7]\\
\,& 8& \none[\alpha_6] \\
\,&\, &7 &\, \\
\,&\, & \,& 6 &\none[\alpha_5] \\
\,&\, &\, &\,&\, &\, \\
\,&\, &\, &\,&\, & 5 &\, \\
\,&\, &\, &\,&\, & \, & \, &\none[\alpha_3]&\none&\none&\none[\,\,x^3\alpha_3] \\
\,&\, &\, &\,&\, &\, &\,&\, &\, &\, \\
\,&\, &\, &\,&\, &\, &\,&\, & 3 &\,&\none&\none&\none[x^3f(\alpha_3)] \\
\,&\, &\, &\,&\, &\, &\,&\, &\, &\,&\, &\, &\,&\, &\, &\, \\
\,&\, &\, &\,&\, &\, &\,&\, &\, &\,&\, &\, &\,&\, &\, &\, &\,
\end{ytableau}
\end{center}
Of course $f_{\alpha_i,\beta}(\alpha_k)$ is well defined if $k>i$ since it will be a monomial in $\CC[x,y]$ (graphically it will not fall out of the positive quadrant).
The condition $\beta \in Q_{\alpha_3}$ precisely makes sure that we can send all generators with lower index to zero: the element of $I$ with lowest degree that is both divisible by $\alpha_i$ and $\alpha_k$ with $k<i$ is $x^{q_i}\alpha_i = y^{p_{i-1}}\alpha_{i-1}$, but, by definition of $Q_{\alpha_i}$, $x^{q_i}\beta \in I$, and then $0$ mod $I$. If $\beta \in P_{\alpha}$ the argument is completely analogous. 
\end{proof}
\begin{observation}
This is slightly different from the construction of Cheah \cite{cheah1998cellular}, and it is the choice of Goettsche \cite{gottsche1990betti}. In the case of Cheah, $\ff$ is chosen to be the element in $S_{\alpha, \beta}$ that sends the biggest possible number of generators to zero. 
\end{observation}
\begin{definition}[$B(I)$]
\label{definitionB(I)}
For $I=(\alpha_0, \dots, \alpha_s)$ a monomial ideal with standard monomial generators $\alpha_i$'s, we define $B(I)$ to be the finite subset of $T_I \,\Hi^n(\CC^2))$ that contains all the $\ff$ as defined in Lemma \ref{definitionf} : 
\begin{equation}
B(I) \, := \,\left \{\ff \,\,\left\vert\,\, \alpha \text{ standard monomial generator and } \beta\in P_{\alpha}\cup Q_{\alpha} \right. \right\}.
\end{equation}
\end{definition}
\begin{lemma}
\label{cardinalityB(I)}
Let $I=(\alpha_0, \dots, \alpha_s)$ be a monomial ideal of length $n$ with prescribed standard monomial generators. The set $B(I)$ defined in \ref{definitionB(I)} has cardinality $2n$.
\end{lemma}
\begin{proof}
By definition it is clear that $\ff = f_{\alpha', \beta'}$ if and only if $(\alpha, \beta) = (\alpha', \beta')$. Moreover for a fixed $\alpha$, $P_{\alpha}$ and $Q_{\alpha}$ are disjoint. Then $\# B(I) = \sum_{i=0}^s( \#P_{\alpha_i} +\# Q_{\alpha_i}$). Suppose that $\nu_0 \geq \nu_1 \geq \dots \geq \nu_{s-1}> \nu_s=0$ is the associated partition of $n$, i.e. the number of columns of $\Gamma(I)$. Then for a generator $\alpha_i$ the distance $p_i = \nu_{s-i-1} -\nu_{s-i}$ (for $i=s$ we can put $p_s=\infty$, but it does not matter since $P_{\alpha_s}=0$). The number of elements in $P_{\alpha_i}$ is equal to $p_i$ times the number of columns in $\Gamma(I)$ that are on the left of $\alpha_i$. Thus we have that 
\begin{equation}
\label{cardinalityPalphas}
\# P_{\alpha_i} = (\nu_{s-i-1} -\nu_{s-i})(s-i) = \sum_{i \leq k< s} (\nu_{s-i-1} -\nu_{s-i}).
\end{equation}
Then we have 
\[
\sum_{0\leq i\leq s} \# P_{\alpha_i} = \sum_{0\leq i\leq s}  \sum_{0\leq k< s-i} (\nu_{s-i-1} -\nu_{s-i}) = \sum_{0 \leq i\leq j \leq s} (\nu_{s-i-1} -\nu_{s-i}) = \sum_{0< i \leq s} \nu_{s-i} = n, 
\]
since $\nu_s=0$. By transposing we prove the same for $\sum_i \#\,Q_{\alpha_i}$. 
\end{proof}
\begin{lemma}
\label{B(I)basis}
Let $I=(\alpha_0, \dots, \alpha_s)$ be a monomial ideal of length $n$ with prescribed standard monomial generators. The set $B(I)$ is a basis for $T_I \,\Hi^n(\CC^2)$.
\end{lemma}
\begin{proof}
We prove first that the $\ff$ are linearly independent. Let $\sum_{(\alpha, \beta)} c_{\alpha, \beta } \, \ff = 0 $ with $c_{\alpha, \beta} \in \CC$, and suppose by contradiction that at least one coefficient is not zero. We can suppose that the left hand side term is of pure weight, otherwise we can study it taking all the terms of given weight. If there is a couple $(\alpha, \beta)$ with $c_{\alpha, \beta} \neq 0$ and $\beta \in P_{\alpha}$ (resp. $\beta \in Q_{\alpha}$), then, for all the other couples $(\alpha', \beta')$ with $c_{\alpha', \beta'} \neq 0$, we have $\beta' \in P_{\alpha'}$ ( resp. $\beta' \in Q_{\alpha'}$). Suppose then $\beta \in P_{\alpha}$ for one couple. Then take $\overline{\alpha}$ the standard generator with $c_{\overline{\alpha}, \beta} \neq 0$ that is on the highest row, then $(\sum_{(\alpha, \beta)} c_{\alpha, \beta }\ff)(\overline{\alpha}) \neq 0 $ since every other generator of $I$ is sent to a lower row. This is absurd as the combination was zero. 

\hspace{3em}Now we prove that $B(I)$ generates $T_I \,\Hi^n(\CC^2)$. Since we know that $T_{I} \Hi^{n}(\CC^2)$ is generated by pure weight elements we only need to prove that every $f \in \text{Hom}_R (I, \bigslant{R}{I})$ of pure weight is in the span of $B(I)$. Let $n(f)=  \#\,\{\alpha_i\vert f(\alpha_i)\neq 0\}$ be the number of generators of $I$ that $f$ does not send to $0$. We use induction on $n(f)$. 

\hspace{3em}If $n(f)=0$, then $f=0$, and there is nothing to prove. 

\hspace{3em}Suppose now that all $\tilde{f}$ with $n(\tilde{f})<t$ are in the span of $B(I)$, and suppose $f$ has $n(f)=t$. Since $f$ is of pure weight we know that either:
\begin{itemize}
\item[1)] we have $\bar{\imath}:= \max \{i\vert f(\alpha_i) \neq 0\}$ is strictly smaller than $s$ and $f(\alpha_{\bar{\imath}})\in \langle \beta \rangle$ with $\beta \in P_{\alpha_{\bar{\imath}}}$, or 
\item[2)] we have $\bar{\imath}:= \min \{i\vert f(\alpha_i) \neq 0\}$ is strictly bigger than $0$ and $f(\alpha_{\bar{\imath}})\in \langle \beta \rangle$ with $\beta \in Q_{\alpha_{\bar{\imath}}}$. 
\end{itemize}
\hspace{3em}Suppose we are in the first case. Renormalize $f$ so that $f(\alpha_{\bar{\imath}})=\beta$.

Then thanks to the definition of $\bar{\imath}$, since $y^{p_{\bar{\imath}}}\beta \in I$, we have that  $
f-f_{\alpha_{\bar{\imath}}, \beta}$ is well defined and sends strictly more generators to $0$ mod $I$, i.e. $n(f-f_{\alpha_{\bar{\imath}}})<t$. By induction we have that $f \in \text{span}(B(I))$ as desired. \\

\hspace{3em}The case 2), is completely analogous. 

\end{proof}
\begin{proposition}{ \cite[Theorem~2.4]{fogarty1968algebraic}}
\label{hilbsmooth}
The Hilbert scheme $\Hi^n(\CC^2)$ is smooth. 
\end{proposition}
\begin{proof}
We just proved that at every torus fixed point $I$ the Zariski tangent space has dimension equal to the dimension of $\Hi^n(\CC^2)$, Lemma \ref{cardinalityB(I)} and Lemma \ref{B(I)basis}. However since every point in $\Hi^n(\CC^2)$ is attracted by a fixed point, and the dimension of the tangent space can only be greater at a fixed point, we have that the same statement holds for every point. Thus every point of $\Hi^n(\CC^2)$ is smooth. 
\end{proof}

\begin{theorem}{ \cite[Theorem~1.1]{ellingsrud1987homology}}
\label{ES}
The space $\Hi^{n}(0)$ has an affine cell decomposition given by the attracting sets at its torus fixed points for every generic one dimensional torus acting with weights $0<w_1<w_2$. If we chose the $\TT_{\infty}$ action of $\CC^*$ on $\CC[x,y]$, i.e. the action with weights $w_1, w_2$ such that $0<w_1<w_2$ and $1 \ll \frac{w_2}{w_1}$, then the affine cell attracted by a monomial ideal $I$  has dimension $n-d$ where $d= \min \{i\vert x^i \in I\}= \ell(\Gamma(I))$. In particular the Betti numbers for the Borel-Moore homology of $\Hi^{n}(0)$ satisfy: 
\[
b_i =  \#\,\left\{ \,\nu \vdash n\, \left\vert \,\,\,\ell(\nu) = n-i \right.\, \right\}.
\]

The Poincar\'{e} polynomial of $\Hi^{n}(0)$ is 
\[
P_q\,\left(\Hi^{n}(0)\right) = \sum_{\Gamma \vdash n} \,\, q^{n-\ell(\Gamma)}.
\]
\end{theorem}
\begin{proof}
Consider $\PP^2$ with homogenous coordinates $[Z:X:Y]$. Observe that $\Hi^{n}(\PP^2)$ is smooth since there is an open cover with opens that are isomorphic to $\Hi^{n}(\CC^2)$ that is smooth thanks to Proposition \ref{hilbsmooth}. Consider the $\TT$ action on it given by $t\cdot[Z:X:Y] = [t^{W_0}Z: t^{W_1}X: t^{W_2}]$, with $W_0+W_1+W_2=0$, and $W_0<W_1<W_2$. It has finitely many fixed points, so we can apply Theorem \ref{Bialynicki-Birula} to obtain a cell decomposition of $\Hi^{n}(\PP^2)$. Consider the subvariety  $ \Hi^n(P_0)\subset \Hi^{n}(\PP^2)$ parametrizing subschemes that have support in $P_0=[1:0:0]$. Since under the $\TT$ action every point of $\PP^2$ flows away from $P_0$ we have that 
\[
Z \in \Hi^n(P_0)\quad  \iff \quad \lim_{t\to 0} (t\cdot Z) \in \Hi^n(P_0).
\]
This shows that $\Hi^n(P_0)$ is a union of some of the cells of the Bialynicki-Birula  cell decomposition of $\Hi^{n}(\PP^2)$ and we need only to calculate their dimension to find a graded basis for the BM homology of $\Hi^n(0)$. Choose $W_1= \frac{2w_1-w_2}{3}$ and $W_2= \frac{2w_2-w_1}{3}$ and use the identifications
\[
x:= \frac{X}{Z}\quad y:= \frac{Y}{Z}\qquad R:= \CC[[x,y]],
\]
we have $\Hi^n(P_0) = \Hi^n(R)_{\text{red}}$ with the torus action described in the statement. \\

\hspace{3em}In particular $T^+_I \Hi^n(0)= T^+_I \Hi^n(\PP^2)=T^+_I \Hi^n(\CC^2)$. Let then $I=(\alpha_0, \dots, \alpha_s)$ be a monomial ideal of length $n$, with prescribed standard monomial generators. We want to study the positive part of the tangent space $T_I^+\Hi^{n}(\CC^2)$. Given the choice of torus, the weight of $\ff$ is positive if and only if $\beta \in P_{\alpha}$ and $\beta$ lies in a row strictly higher than the one of $\alpha$. We know that $\sum_i \#P_{\alpha_i}= n$ from (\ref{cardinalityPalphas}), and we know that the only $\beta \in P_{\alpha}$ that are \emph{not} on a row higher than $\alpha$ are those on the same row of $\alpha$. It is clear, by projecting down, that these are as many as the boxes in the first row, i.e. the number of columns. Then we have that  $\dim T_I^+\Hi^{n}(\CC^2) = n-d$ as wanted. \\

\hspace{3em}Then thanks to Theorem \ref{Bialynicki-Birula} we know that $\Hi^{n}(0)$ has an affine cell decomposition with cells labeled by monomial ideals and of dimensions $\dim T_I^+\Hi^{n}(0)$. Moreover, thanks to Proposition \ref{Fulton} we know that the cycles associated to these cells give us a graded basis for the homology of $\Hi^{n}(0)$. 
\end{proof}
\begin{theorem}{ \cite[Theorem~0.1]{gottsche1990betti}}
\label{generatingGoettsche}
The generating function for the Poincar\'{e} polynomials of $\Hi^{n}(0)$ as $n$ varies is 
\begin{equation}
\label{goettsheformula}
\sum_{n=0}^{+\infty } \, P_q\,\left(\Hi^{n}(0)\right) z^n = \prod_{k\geq 1} \, \frac{1}{1- z^kq^{k-1}}.
\end{equation}
\end{theorem}
\begin{proof}
Call $p(n,k)$ the number of partitions of $n$ with $l$ parts. Then we know $b_i (\Hi^n(0))= p(n, n-i)$, thanks to Theorem \ref{ES}. The formula in the statement then is simply the well known combinatorial identity 
\[
\sum_{n=0}^{+\infty} \sum_{i=0}^{+\infty} p(n, n-i) q^i z^n = \prod_{k=1}^{+\infty} \frac{1}{1-q^{k-1}z^k}
\]
that follows easily from the famous Euler identity for the generating function of the number  of partitions of an integer $n$. 
\end{proof}

\begin{remark} Goettsche proves a more general formula that holds for every smooth surface $X$. It is worth mentioning its ground breaking result even though we do not use it. We need to admit non zero odd homology groups to work in this generality, and the formula is:
\[
\sum_{n=0}^{+\infty } \,\sum_{j=0}^{+\infty} \dim H_j\,\left(\Hi^{n}(X)\right)q^j z^n = \prod_{k= 1}^{+\infty}\prod_{i=0}^4 \, \left(1- z^kq^{2k-2+i}\right)^{(-1)^{\dim H_i(X)}}.
\] 
\end{remark}
\subsection{The tangent space to the Hilbert-Samuel's strata, case $\mathfrak{n}=n$}

\hspace{3em}Now that we know a weight basis for the tangent space $T_I \Hi^{n}(\CC^2)$ at each fixed point $I$, we can study the tangent spaces $T_I M_{T}$ (resp. $T_I G_{T}$) for those fixed points $I \in  M_{T}$ i.e such that $T(I)= T$. Here we suppose that $T= (t_i)_{i\geq 0}$ is an admissible sequence of nonnegative integers as in \ref{admissible}. The key observation for this study is due to Goettsche and is the following . 

\begin{observation}{ \cite[Chapter~2]{gottsche1994hilbert}}
\label{goettschetangenttoM}
Let $I$ be a monomial ideal with $T(I)=T$ and standard monomial generators $(\alpha_0, \dots,\alpha_s)$.  With the choice of weights $\TT_{1^+}$, we have the natural identifications:
\begin{align*}
T_{I} M_{T} &= \left\langle f_{\alpha, \beta}\,\,\left\vert\,\, f_{\alpha,\beta} \text{ preserves or raises the degree}\right.\right\rangle=\left\langle f_{\alpha, \beta} \,\,\left\vert\,\, \deg{\beta}\geq \deg \alpha \right. \right\rangle\, ;\\
T_{I} G_{T} &= \left\langle f_{\alpha, \beta}\,\,\left\vert\,\, f_{\alpha,\beta} \text{ preserves the degree}\right. \right\rangle=\left\langle f_{\alpha, \beta}\,\,\left\vert\,\, \deg{\beta} = \deg \alpha\right. \right\rangle.
\end{align*}
Moreover we can determine the subspace of the tangent space where the weights of the torus action are positive simply as 
\begin{align*}
T_{I}^+ M_{T} &= \left\langle f_{\alpha, \beta}\,\,\left\vert\,\, f_{\alpha,\beta} \text{ raises the degree, or preserves it and strictly raises it in the } y\right.\right\rangle \\ &=\left\langle f_{\alpha, \beta}\,\,\left\vert\,\, \deg{\beta}> \deg \alpha, \text{ or }  \deg{\beta}= \deg \alpha\text{ and }\deg_y\beta > \deg_y\alpha\right. \right\rangle\,;\\
T_{I}^+ G_{T} &= \left\langle f_{\alpha, \beta}\,\,\left\vert\,\, f_{\alpha,\beta} \text{ preserves the degree and strictly raises it in the } y\right.\right\rangle \\ 
&=\left\langle f_{\alpha, \beta}\,\,\left\vert\,\, \deg{\beta} = \deg \alpha, \deg_y\beta  >\deg_y \alpha\right. \right\rangle.
\end{align*}
\end{observation}
\begin{lemma}{ \cite[Lemma~2.2.11]{gottsche1994hilbert}}
The dimension of the subspace $T_{I}^+ M_{T}$ of $T_{I} M_{T}$ where the weights of the action are positive is: 
\begin{equation}
\label{posskew1}
\dim(T_{I}^+ M_{T}) = n-  \#\, \left\{(u,v)\in \Gamma(\nu) \vert h_{u,v} = 0 \text{ or } h_{u,v} = 1 \right\}\, .
\end{equation}
\end{lemma}
\begin{proof}
We give the proof of Goettsche. Call $\nu=\nu_0\geq\dots > \nu_s=0$ the partition associated to $I$. Call $\lambda$ and $\mu$ two linearly independent characters of the two torus $\TT^2$ acting on $\hat{R}$ as $t\cdot x = \lambda(t) x$ and $t\cdot y = \mu(t) y$. Then the existence of the basis $B(I)$ shows that as representation of $\TT^2$ we have the identity
\begin{equation}
\label{characterformula}
T_{I}\Hi^n(\CC^2) \quad =\quad \sum_{0\leq i\leq j <s } \,\, \sum_{k= \nu_{j+1}}^{\nu_j -1} \, (\lambda^{i-j-1}\mu^{ \nu_i-k-s} +\lambda^{i-j}\mu^{ k-\nu_i})
\end{equation}  
as each addendum is the weight of one of the $\ff \in B(I)$. \\

\hspace{3em}The action of $\TT_{1^+}$ has positive weight on $\lambda^{a}\mu^{b}$
 if and only if $a+b>0$ or a+b=0 and $b>0$. Then the term $(\lambda^{i-j-1}\mu^{ \nu_i-k-s})$ has positive weight if and only if $i+\nu_i> j+k+1$ and the term $(\lambda^{i-j}\mu^{ k-\nu_i})$  has positive weight if and only if $i+\nu_{i}< j+ k$. Denote with $\hat{\nu}$ the \emph{transpose} partition of $\nu$, i.e. $\hat{\nu}_i = \# \left\{(m,i) \vert (m,i) \in \Gamma(\nu) \right\}$. Then it is clear that $\hat{\nu}_k$ is the smallest $j$ satisfying $k \geq \nu_{j}$, so that $\hat{\nu}_{k}-1$ is the smallest $j$ satisfying $k \geq \nu_{j+1}$. Notice also that we can reformulate the definition of the hook difference in terms of the transpose partition as
 \[
 h_{u,v}(\nu) = \nu_u-u -\hat{\nu}_v +v. 
 \]
 Then we have that 
 \begin{align*}
 \dim T^+_{I} M_{T}  = \,\, & \sum_{0\leq i\leq j <s } \left( \nu_j -\nu_{j+1} - \#\, \left\{k \in \ZZ \left\vert \begin{matrix} \nu_{j+1}\leq k < \nu_{j};\\ 0 \leq j+k-i-\nu_{i}+1\leq 1\end{matrix}\right.\right\}\right) \\
 =\,\, &  \sum_{0\leq i <s } \left( \nu_i  - \#\, \left\{k \in \ZZ \left\vert 0\leq k <\nu_i, \,\,\,   0 \leq\nu_i+i-\hat{\nu}_k-k \leq 1\right.\right\}\right) \\
 =\,\, & n- \#\left\{(u,v) \in \Gamma(\nu) \left\vert \,\, 0 \leq h_{u,v}(\nu)\leq 1\, \right.\right\}.
 \end{align*}
 
  \end{proof}
For the homogenous Hilbert-Samuel's stratum $G_T$ Goettsche shows, with similar arguments, the following Lemma of which we omit the proof. 
\begin{lemma}{\cite[Lemma~2.2.12]{gottsche1994hilbert}}
The dimension of the subspace $T_{I}^+ G_{T}$ of $T_{I} G_{T}$ where the weights of the action are positive is: 
\[
\dim(T_{I}^+ G_{T}) = n-  \#\, \left\{(u,v)\in \Gamma(\nu) \vert h_{u,v} = 1 \right\}\, .
\]
\end{lemma}

Thanks to smoothness, Proposition \ref{smoothnessIarrobino}, we can apply Theorem \ref{Bialynicki-Birula} and Proposition \ref{Fulton} to immediately get the following. 
\begin{theorem}{\cite[Theorem~2.2.7]{gottsche1994hilbert}}
Let $T= (t_i)_{i\geq 0 }$ be a sequence of nonnegative integers, with $|T|= n$ as in \ref{admissible}. Then we have that:
\begin{itemize}
\item[(1)] The Hilbert-Samuel's strata $M_T$ and $G_T$ have an affine cell decomposition. 
\item[(2)] The Betti numbers of $M_T$ satisfy
\[
b_{i}(M_T) =  \#\,\left\{  \nu \vdash n \left\vert \begin{matrix}T(\nu)=T\\  \#\,\left\{(u,v)\in \Gamma(\nu)\,\,\vert\,\, h_{u,v}\in \{0,1\} \right\}=n-i \end{matrix} \right. \right\}\,.
\]
\item[(3)] The Betti numbers of $G_T$ satisfy
\[
b_{i}(G_T) =  \#\,\left\{  \nu \vdash n \left\vert \begin{matrix}T(\nu)=T\\  \#\,\left\{(u,v)\in \Gamma(\nu)\,\,\vert\,\, h_{u,v}=1\right\}=i \end{matrix} \right. \right\}\, .
\]
\end{itemize}
\end{theorem}

We can then utilize smoothness of the Hilbert-Samuel's strata and their homogenous counterparts to give their dimensions. 
\begin{corollary}{\cite[Theorem~2.12]{iarrobino1977punctual}}
\label{dimMT}
Let $T= (t_i)_{i\geq 0 }$ be a sequence of nonnegative integers, with $|T|= n$ as in \ref{admissible}. The dimension of $M_T$ and $G_T$ are: 
\begin{align*}
\dim M_{T} &= n- \sum_j (t_{j-1}-t_j)\frac{(t_{j-1}-t_j+1)}{2},\\
\dim G_T &= \sum_{j\geq d} (t_{j-1}-t_j+1)(t_{j}-t_{j+1}).
\end{align*}
\end{corollary}

\begin{remark}
As an immediate consequence we obtain that the Poincar\'{e} polynomial for $\Hi^{n}(0)$ can be written as 
\[
P_q\left(\Hi^{n}(0)\right)\quad =\quad \sum_{\Gamma \vdash n} q^{\pos_{\TT_{1^+}}(\Gamma)}
\]
where $\pos_{\TT_{1^+}}(\Gamma)$ is the quantity described in formula (\ref{posskew1}). Just looking at it combinatorially, it is not immediately clear that this formula coincides with the one we wrote in Theorem \ref{ES}, nor that these polynomials can be summed to give the same generating function as in Proposition \ref{generatingGoettsche}. In the last chapter we devote a section to explain in some details how this is working.
\end{remark}


\section{A weight basis for $T_{I_1, I_2} \Hi^{n, n+1}(\CC^2)$}

\hspace{3em}Now that we defined a basis for the tangent space of  $\Hi^{n}(\CC^2)$ at each of its fixed points, it is time to do the same for the tangent spaces of $\Hi^{n, n+1}(\CC^2)$ at its fixed points. If $I_1, I_2$ is such a fixed point, the vector space we want to understand, thanks to Lemma \ref{Rmorphisms},  is  
\[
T_{I_1, I_2} \Hi^{n, n+1}(\CC^2) = \text{Ker}(\phi_{12}-\psi_{12}) \subset \text{Hom}_R(I_1, \bigslant{R}{I_1}) \oplus \text{Hom}_R(I_2, \bigslant{R}{I_2}).
\]

\hspace{3em}In the previous section we saw how to define a weight basis for $ \text{Hom}_R(I_1, \bigslant{R}{I_1})$ and for  $\text{Hom}_R(I_2, \bigslant{R}{I_2})$. The idea is that the two bases share a lot in common, since $I_1$ and $I_2$ differ by a small amount. There is  only one monomial in $I_1$ but not in $I_2$, call it $\alpha_j$.  

\hspace{3em}In the weight basis $B(I_1, I_2)$  there will be two types of vectors: The first type underlines the similarities between $ \text{Hom}_R(I_1, \bigslant{R}{I_1})$ and $\text{Hom}_R(I_2, \bigslant{R}{I_2})$: more precisely these are vectors of the form $(\ff, \,\,\star)$ with $\ff \in B(I_1)$ and $\star$ some appropriate vector in $T_{I_2} \Hi^{n+1}(\CC^2)$ that looks like $\ff$ so that the difference is zero for $\phi_{12}-\psi_{12}$. The second type underlines the differences between $I_1$ and $I_2$: more precisely these are vectors of the form $(0, (\alpha \mapsto \alpha_j))$, where we use the fact that $\alpha_j$ is $0$ mod $I_1$ and not $0$ mod $I_2$. \\

\hspace{3em}Lets dive into the details.  \\
 
\begin{notation}
Let $(I_1, I_2)$ be a fixed point of the $\TT^2$ action on $\Hi^{n, n+1}(\CC^2)$, i.e. $I_1, I_2$ are monomial ideals of lengths $n$ and $n+1$ respectively, with $I_1\supset I_2$. Call $(\alpha_1, \dots, \alpha_s)$ the standard monomial generators of $I_1$ and $(\alpha'_1, \dots, \alpha'_{s'})$ the standard monomial generators of $I_2$. Call $\Gamma_i=\Gamma(I_i)$, then call $j$ the index for which
\[
\Gamma_2 = \Gamma_1 \cup \{\alpha_j\}.
\]
We will denote this configuration also as $\Gamma= (\Gamma_1, \Gamma_2) \vdash [n, n+1]$. 

\hspace{3em}Call $p_i$ and $q_i$, as before, the distances between generators of $I_1$, and $p'_i$ and $q'_i$ the distances between generators of $I_2$: 
\begin{align*}
&p_i := \deg_y \alpha_{i+1} - \deg_y \alpha_{i} \quad p_s := \infty,\quad  q_i := \deg_x \alpha_{i} - \deg_x \alpha_{i-1} \quad q_0 := \infty,
\\
&p'_i := \deg_y \alpha'_{i+1} - \deg_y \alpha'_{i} \quad p'_{s'} := \infty, \quad q'_i := \deg_x \alpha'_{i} - \deg_x \alpha'_{i-1} \quad q'_0 := \infty.
\end{align*}

Analogously we call $P_{\alpha}=P_{\alpha_i}$ and  $Q_{\alpha}=Q_{\alpha_i}$ the relevant sets for  $\alpha=\alpha_i$ a generator of $I_1$, and $P_{\alpha'}=P_{\alpha'_i}$ and  $Q_{\alpha'}=Q_{\alpha'_i}$ the relevant sets for $\alpha'=\alpha'_i$ a generator of $I_2$. Whenever confusion is possible we will clarify if $\ff$ is seen as an element in the basis $B(I_1)$ of $T_{I_1} \Hi^{n}(\CC^2)$ or as an element in the basis $B(I_2)$ of $T_{I_2} \Hi^{n+1}(\CC^2)$, and so on. 

\end{notation}

\begin{definition}[Cases]
\label{cases12}
Let $I_1=(\alpha_1, \dots,\alpha_s), I_2=(\alpha'_1, \dots,\alpha'_{s'})$ be  ideals with prescribed standard monomial generators, such that $(I_1, I_2)$ is a fixed point for $ \Hi^{n, n+1}(\CC^2)$, and such that $\Gamma(I_2)=\Gamma(I_1)\cup\{\alpha_j\}$. There are four possible different cases we need to distinguish. In the following pictures we will indicate with a bullet $\bullet$ those standard monomial generators of $I_2$ that are not already standard monomial generators of $I_1$. 

\textbf{\large{Cases 1a), 1b)} ($s=s'$)}

The number of generators $s+1$ and $s'+1$ is the same for $I_1$ and $I_2$. This can only happen if $p_j=1$ or $q_j=1$ but not both. Then we distinguish the two possibilities, and we look at the generators of $I_2$ in terms of those of $I_1$:\\

\begin{minipage}{0.45\textwidth}
\[
\begin{matrix}\text{Case 1a) },\\ q_j=1\end{matrix} \quad
\begin{cases}
\alpha'_1&= \alpha_1 \\
&\dots\\
\alpha'_j &= y\alpha_j\\
& \dots\\
\alpha'_{s'}&=\alpha'_{s}=\alpha_s.
\end{cases}
\]
\end{minipage}
\begin{minipage}{0.45\textwidth}
\[
\begin{matrix}\text{Case 1b) },\\ p_j=1 \end{matrix}\quad
\begin{cases}
\alpha'_1&= \alpha_1 \\
&\dots\\
\alpha'_j &= x\alpha_j\\
&\dots\\
\alpha'_{s'}&=\alpha'_{s}=\alpha_s.
\end{cases}
\]
\end{minipage}

\begin{minipage}{0.45\textwidth}
\begin{center}
\ytableausetup{boxsize=1.2em, aligntableaux=bottom}
\begin{ytableau} 
\, \\
\,&\,  \\
\,&\, &\,&\none[\bullet] \\
\,&\, &\,&*(green) \alpha_j  \\
\,&\, &\,&\,  \\
\,&\, &\, &\,&\, &\,
\end{ytableau}
\end{center}
\end{minipage}
\begin{minipage}{0.45\textwidth}
\begin{center}
\ytableausetup{boxsize=1.2em, aligntableaux=bottom}
\begin{ytableau} 
\, \\
\,&\,  \\
\,&\, &\, \\
\,&\, &\,  \\
\,&\, &\,&\,&*(green) \alpha_j &\none[\bullet] \\
\,&\, &\, &\,&\, &\,
\end{ytableau}
\end{center}
\end{minipage} 

\vspace{1em}

\textbf{\large{Case 2)} ($s'=s+1$)}

This happens if and only if $p_j>1, q_j>1$, or $j=0$ and $p_0>1$, or $j=s$ and $q_s>1$. Then we have 

\begin{minipage}{0.45\textwidth}
\[
\begin{matrix}\text{Case 2) },\\
 p_j,q_j>1
 \end{matrix} \quad
\begin{cases}
\alpha'_1&= \alpha_1 \\
&\dots\\
\alpha'_j &= x\alpha_j\\
\alpha'_{j+1}&= y\alpha_j \\
\alpha'_{j+2}&= \alpha_{j+1}\\
&\dots\\
\alpha'_{s'}&=\alpha'_{s+1}=\alpha_s.
\end{cases}
\]
\end{minipage}
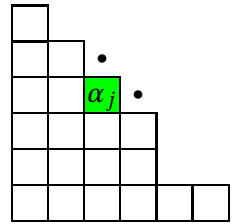
\begin{minipage}{0.45\textwidth}
\begin{center}
\ytableausetup{boxsize=1.2em, aligntableaux=bottom}
\begin{ytableau} 
\, \\
\,&\,&\none[\bullet]  \\
\,&\, &*(green) \alpha_j&\none[\bullet] \\
\,&\, &\,&\,  \\
\,&\, &\,&\,  \\
\,&\, &\, &\,&\, &\,
\end{ytableau}
\end{center}
\end{minipage}

\vspace{1em}
\textbf{\large{Case 3} ($s'=s-1$)}

This happens if and only if $p_j=1, q_j=1$. Then we have

\begin{minipage}{0.45\textwidth}
\[
\begin{matrix}\text{Case 3) },\\ p_j,q_j=1\end{matrix} \quad
\begin{cases}
\alpha'_1&= \alpha_1 \\
&\dots\\
\alpha'_{j-1} &= \alpha_{j-1}\\
\alpha'_{j}&= \alpha_{j+1} \\
\alpha'_{j+1}&= \alpha_{j+2}\\
&\dots\\
\alpha'_{s'}&=\alpha'_{s-1}=\alpha_s.
\end{cases}
\]
\end{minipage}
\begin{minipage}{0.45\textwidth}
\begin{center}
\ytableausetup{boxsize=1.2em, aligntableaux=bottom}
\begin{ytableau} 
\, \\
\,&\,&*(green) \alpha_j  \\
\,&\, &\, \\
\,&\, &\,&\,  \\
\,&\, &\,&\,  \\
\,&\, &\, &\,&\, &\,
\end{ytableau}
\end{center}
\end{minipage}
\end{definition}

\hspace{3em}We now focus on the task of understanding for which $\ff \in B(I)$ there exists $h\in$ $ \text{Hom}_R(I_2, \bigslant{R}{I_2})$ such that $(\ff,h)\in T_{I_1, I_2} \Hi^{n, n+1}(\CC^2)$, i.e.  $\phi_{12}(\ff)-\psi_{12}(h)=0 \,\,\text{mod}\,\,I_1$. Such an $h$ does not always exist, as the following example shows. 
\begin{example}
Let the following diagram represents the fixed point of $\Hi^{3, 4}(\CC^2)$ with $I_1=(x^2, xy, y^2)$, $I_2=(x^2, xy, y^3)$, and $\alpha_j= y^2$. Let $\ff$ be the map $f_{x^2, y}\in B(I)$. Then there does not exist a map $h\in \text{Hom}_R(I_2, \bigslant{R}{I_2})$ such that $\phi_{12}(\ff)-\psi_{12}(h)=0$, because such a map would send $x^2 \mapsto y$ and $x^2y \mapsto y^2\neq 0 \,\text{ mod } I_2$ which is absurd since it should also be $xy\mapsto 0 \implies x^2y \mapsto x\cdot0=0$. 

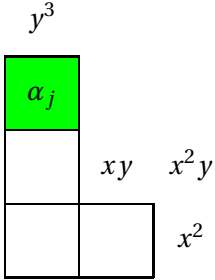
\begin{minipage}{0.45\textwidth}
\ytableausetup{boxsize=2.5em, aligntableaux=bottom}
\begin{ytableau}
\none[y^3] \\
*(green)\alpha_j\\
\,& \none[xy]&\none[x^2y] \\
\,&\, &\none[x^2] \\
\end{ytableau}
\end{minipage}
\begin{minipage}{0.45\textwidth}
If $f(x^2)=y$ then $h(x^2)=y$, \\
but then $h(xy)$ cannot be defined as we discussed before. 
\end{minipage}
\vspace{1em}
\end{example}

\begin{definition}
\label{obs}
Let $I_1=(\alpha_1, \dots,\alpha_s), I_2=(\alpha'_1, \dots,\alpha'_{s'})$ be a fixed point for $ \Hi^{n, n+1}(\CC^2)$, such that $\Gamma(I_2)=\Gamma(I_1)\cup\{\alpha_j\}$, with prescribed standard monomial generators. We define a subset of indexes $(\alpha, \beta)$ of the basis $ B(I_1)= \left \{f_{\alpha, \beta}\vert (\alpha, \beta) \right \}$ and we denote it $\text{Obs}(I_1,I_2)$.  We will see that these correspond to those $\ff \in B(I_1)$ that we do not want to try to extend to elements $(\ff, h) \in T_{I_1, I_2} \Hi^{n, n+1}(\CC^2)$. The definition of $\text{Obs}(I_1,I_2)$ varies according to the possible cases of Definition \ref{cases12}. 
\begin{equation}
\label{IObs12}
\begin{matrix}
\begin{matrix}
\text{Case 1a)} \\
\text{Obs}(I_1,I_2):=
\left\{
\begin{matrix}
(\alpha_i, \frac{\alpha_j}{x^{q_i}}) \text{ for } i>j,\\
(\alpha_i, \frac{\alpha_j}{y^{p_i}}) \text{ for } i<j-1.
\end{matrix}
\right\}  
\end{matrix}\qquad
\vspace*{0.5em}
&
\begin{matrix}
\text{Case 1b)} \\
\text{Obs}(I_1,I_2):=
\left\{
\begin{matrix}
(\alpha_i, \frac{\alpha_j}{x^{q_i}}) \text{ for } i>j+1,\\
(\alpha_i, \frac{\alpha_j}{y^{p_i}}) \text{ for } i<j.
\end{matrix}
\right\}
\end{matrix}  
\vspace*{0.5em} \\
\begin{matrix}
\text{Case 2)} \\
\text{Obs}(I_1,I_2):=
\left\{
\begin{matrix}
(\alpha_i, \frac{\alpha_j}{x^{q_i}}) \text{ for } i>j,\\
(\alpha_i, \frac{\alpha_j}{y^{p_i}}) \text{ for } i<j.
\end{matrix}
\right\}  
\end{matrix}\qquad
\vspace*{0.5em}
&
\begin{matrix}
\text{Case 3)} \\
\text{Obs}(I_1,I_2):=
\left\{
\begin{matrix}
(\alpha_i, \frac{\alpha_j}{x^{q_i}}) \text{ for } i>j+1,\\
(\alpha_i, \frac{\alpha_j}{y^{p_i}}) \text{ for } i<j-1.
\end{matrix}
\right\}
\end{matrix}
\end{matrix}
\end{equation}
\end{definition}

\begin{example}The pictures below represent some of elements of $\text{Obs}$ that we defined.  For each picture there are two $\ff$ represented: the one represented with stars sends the generator of $I_1$ marked with a star to the element of $\Gamma_1$ marked with a star. The other is represented in a similar way with two bullets. 
\begin{equation*}
\begin{matrix}
\begin{matrix}
\text{Case 1a)}\\

\ytableausetup{boxsize=1.2em, aligntableaux=bottom}
\begin{ytableau} 
\none[] \\
\, \\
\,&\,&\none[\bullet]  \\
\,&\, &\, \\
\,&\, & \bullet &*(green) \alpha_j  \\
\,&\, &\,& \star  \\
\,&\, &\, &\,&\, &\,&\none[\star]
\end{ytableau}\end{matrix} \qquad\qquad \qquad\qquad\qquad \qquad
&
\begin{matrix}
\text{Case 1b)}\\

\vspace*{0.5em}
\ytableausetup{boxsize=1.2em, aligntableaux=bottom}
\begin{ytableau} 
\none[] \\
\, \\
\,&\,&\none[\bullet]  \\
\,&\, &\,  \\
\,&\, &\, \\
\,&\, &\,&\bullet&*(green) \alpha_j  \\
\,&\, &\, &\,&\star &\,&\none[\star]
\end{ytableau}\end{matrix} \\
\begin{matrix}
\text{Case 2)}\\

\ytableausetup{boxsize=1.2em, aligntableaux=bottom}
\begin{ytableau} 
\none[] \\
\, &\none[\bullet]\\
\,&\,  \\
\,&\bullet &*(green) \alpha_j \\
\,&\, &\,&\,  \\
\,&\, &\star&\, &\none[\star] \\
\,&\, &\, &\,&\, &\,
\end{ytableau}\end{matrix}\qquad\qquad\qquad\qquad\qquad \qquad
& \begin{matrix}
\text{Case 3)}\\

\ytableausetup{boxsize=1.2em, aligntableaux=bottom}
\begin{ytableau} 
\none[\bullet]\\
\, \\
\,&\bullet&*(green) \alpha_j  \\
\,&\, &\, \\
\,&\, &\star&\,  \\
\,&\, &\,&\,&\none[\star]  \\
\,&\, &\, &\,&\, &\,
\end{ytableau}\end{matrix}
\end{matrix}
\end{equation*}
\end{example}
\begin{observation}
\label{numberobs}
Observe that in terms of $s'=s_2$, the highest index of a standard generator of $I_2$, the set $\text{Obs}(I_1,I_2) $ has the same number of elements for all cases.
\[
\begin{matrix}
\begin{matrix}
\text{Case 1a)}\\
 \#\,\text{Obs}(I_1,I_2) =  s_1+1-2 = s_2+1-2.
\end{matrix} \qquad
&
\begin{matrix}
\text{Case 1b)}\\
 \#\,\text{Obs}(I_1,I_2) =  s_1+1-2 = s_2+1-2.
\end{matrix} 
\\
\begin{matrix}
\text{Case 2)}\\
 \#\,\text{Obs}(I_1,I_2) =  s_1+1-1 = s_2+1-2.
\end{matrix} 
&
\begin{matrix}
\text{Case 3)}\\
 \#\,\text{Obs}(I_1,I_2) =  s_1+1-3 = s_2+1-2.
\end{matrix} 
\end{matrix}
\]
\end{observation}
\begin{lemma}[definition]
\label{suiv}
Let $I_1=(\alpha_1, \dots,\alpha_s), I_2=(\alpha'_1, \dots,\alpha'_{s'})$ be  ideals with prescribed standard monomial generators, such that $(I_1, I_2)$ is a fixed point for $ \Hi^{n, n+1}(\CC^2)$, and such that $\Gamma(I_2)=\Gamma(I_1)\cup\{\alpha_j\}$. Let $f_{\alpha, \beta}$ be one of the elements of the basis $B(I_1)$ of $T_{I_1}\Hi^{n}(\CC^2)$ such that $(\alpha, \beta) \notin \text{Obs}(I_1, I_2)$. Then we define
\begin{equation}
\label{definitionSuiv}
\Su(f_{\alpha, \beta})(\alpha'_k) = \ff(\alpha'_k), \text{ for all } k=0, \dots, s'.
\end{equation}
Then $\Su(\ff)$ is well defined as an element of $ \text{Hom}_R(I_2, \bigslant{R}{I_2})$. Moreover it is such that
 \[\left(\ff, \Su(\ff)\right) \in T_{I_1, I_2}\Hi^{n, n+1}(\CC^2)\,.\] 
\end{lemma}
\begin{proof}
Suppose $\alpha\neq \alpha_j$, then $\alpha$ is also a generator of $I_2$. Moreover the hypothesis that $(\alpha, \beta)$ is not an element of  $\text{Obs}(I_1, I_2)$ ensures that $\beta \in P'_{\alpha} \cup Q'_{\alpha}$ so that we can define $\ff' \in B(I_2)$ and we have that $\ff'=\Su(\ff)$. Of course 
\[
\phi_{12}(\ff)- \psi_{12}(\Su(\ff)) = 0\
\] 
since it is zero on each generator $\alpha'_k$ of $I_2$. Similarly if $\alpha=\alpha_j$ but $x\alpha$ (or $y\alpha$) is a generators of $I_2$, we have that $\Su(\ff) = f'_{x\alpha, x\beta} \in B(I_2)$ (or $\Su(\ff) = f'_{y\alpha, y\beta} \in B(I_2)$) so that it is well defined and it is, again obviously, in the kernel of $\phi_{12}-\psi_{12}$. If, finally, $\alpha=\alpha_j$ and we are in case 3) of \ref{cases12} then, if $\beta \in Q_{\alpha}$,  $y\frac{\beta}{x^{q_{j+1}}}\in Q_{\alpha'}$ and, if  $\beta \in P_{\alpha}$, $x\frac{\beta}{x^{p_{j-1}}}\in P_{\alpha'}$ and we have that
\begin{align*}
\text{If } \beta \in Q_{\alpha}, &\text{ then } \Su(\ff)= f'_{\alpha'_{j}, y\frac{\beta}{x^{q_{j+1}}}}\, ,\\
\text{If } \beta \in P_{\alpha}, &\text{ then } \Su(\ff)= f'_{\alpha'_{j-1}, x\frac{\beta}{x^{p_{j-1}}}}\, .
\end{align*}
This shows that also in this case $\Su(\ff)$ is well defined. Again it is clear that the vector $(\ff, \Su(\ff))$ is in $ \text{Ker}(\phi_{12}-\psi_{12})$.
\end{proof}
\begin{definition}
\label{definitionh}
For every generator $\alpha'_i$ of $I_2$ define:
\[
h_{\alpha'_i, \alpha_j}\in \text{Hom}_R(I_2, \bigslant{R}{I_2}) \qquad\text{ as } \quad h_{\alpha'_i, \alpha_j}(\alpha'_k) = \begin{cases}0 &\text{if } k\neq i, \\ \alpha_j &\text{if } k=i \, .\end{cases} 
\] 
\end{definition}
\begin{observation}
The $h_{\alpha'_i, \alpha_j}$ as above are well defined. Moreover
\[
(0, h_{\alpha_i, \alpha_j}) \in \text{Ker}(\phi_{12}-\psi_{12})  
\] 
are linearly independent vectors, as it follows recalling that $\alpha_j \cong 0 \text { mod }\, I_1$ but $\alpha_j \ncong 0 \text { mod }\, I_1.$
\end{observation}
\begin{definition}
\label{definitionB(I1I2)}
Let $I_1=(\alpha_1, \dots,\alpha_s), I_2=(\alpha'_1, \dots,\alpha'_{s'})$ be a fixed point for $ \Hi^{n, n+1}(\CC^2)$, such that $\Gamma(I_2)=\Gamma(I_1)\cup\{\alpha_j\}$, with prescribed standard monomial generators. Then we define the set: 
\begin{align}
B(I_1,I_2) := &\left\{\left(0, h_{\alpha'_i, \alpha_j}\right)\,\,\left\vert\,\,  i=0,\dots,s' \right.\right\}\,\, \cup\\
 &\cup\,\left\{\left(f_{\alpha, \beta}, \Su(f_{\alpha, \beta})\right)\,\,\left\vert\,\, \ff\in B(I_1), \, (\alpha, \beta)\notin \text{Obs}(I_1, I_2))\right.\right\} \nonumber
 \\
 & \qquad \subset\,\, \text{Hom}_R(I_1, \bigslant{R}{I_1})\oplus\text{Hom}_R(I_2, \bigslant{R}{I_2}). \nonumber
\end{align}
\end{definition}

\begin{observation}
\label{cardinalityB(I1I2)}
Observe that $ \#\, B(I_1,I_2)= 2n+2$. In fact $\# B(I_1) = 2n$, $\# \text{Obs}(I_1, I_2) = s'+1-2$ thanks to Observation \ref{numberobs}, and $\#\left\{(0, h_{\alpha'_i, \alpha_j})\vert i=0,\dots,s' \right\} = s'+1$. 
\end{observation}
\begin{lemma}
\label{B(I1I2)isbasis}
With the notations as in Definition \ref{definitionB(I1I2)}, we have that $B(I_1,I_2)$ is a basis for $T_{I_1, I_2} \Hi^{n, n+1}(\CC^2)$. 
\end{lemma}
\begin{proof}
First of all we observe that by definition all elements of $B(I_1, I_2)$ are in the $\text{Ker}(\phi_{12}-\psi_{12})$, and are of pure weights. Abbreviate $\text{Obs} (I_1, I_2)$ to $\text{Obs}$. \\

\hspace{3em}Let us now prove that they are linearly independent. Suppose
\[
\sum_{i=0}^{s'} c_i\,(0,h_{\alpha'_i, \alpha_j}) + \sum_{(\alpha,\beta)\notin \text{Obs}}c_{\alpha, \beta}\,(f_{\alpha, \beta}, \Su(f_{\alpha, \beta})) = (0,0).
\]
Then in particular $\sum_{(\alpha,\beta)\notin \text{Obs}}c_{\alpha, \beta}\,f_{\alpha, \beta}=0$ as an element of  $\text{Hom}_R(I_1, \bigslant{R}{I_1})$ where we know that the vectors $f_{\alpha, \beta}$ are linearly independent, so that $c_{\alpha, \beta}= 0$ for each $(\alpha, \beta)$. Then, also $\sum_{i=0}^{s'}\, c_i\,h_{\alpha'_i, \alpha_j}\in \text{Hom}_R(I_2, \bigslant{R}{I_2})$, which again implies that $c_i=0$ for all $i=0, \dots, s'$. 

\hspace{3em}Let us then prove that $B(I_1,I_2)$ generates $T_{I_1, I_2} \Hi^{n, n+1}(\CC^2)$. Since we know that $T_{I_1, I_2} \Hi^{n, n+1}(\CC^2)$ is generated by pure weight elements we only need to prove that every $g=(f,h) \in \text{Ker}(\phi_{12}-\psi_{12})$ of pure weight is in the span of $B(I_1,I_2)$. Let $n(f)$ be the number of generators of $I_1$ that the first coordinate of $g$ does not send to $0$ i.e $n(f)=  \#\,\{\alpha_i\vert f(\alpha_i)\neq 0\,\}$. Then we use induction on $n(f)$. \\

\hspace{3em}If $n(f)=0$, then $f=0$, and then $g=(0,h)$ is such that $h(\alpha'_i)= 0 \,\,\,\text{mod }\, I_1$ so that $h(\alpha'_i)\in  \langle \alpha_j\rangle $. Then, clearly $h \in \langle h_{\alpha'_i, \alpha_j} \vert i=0, \dots s' \rangle $. Suppose now that all $\tilde{g}=(\tilde{f}, \tilde{h})$ with $n(\tilde{f})<t$ are in the span of $B(I_1,I_2)$, and suppose $g=(f,h)$ with $n(f)=t$. Since $f$ is of pure weight we know that either:
\begin{itemize}
\item[1)] we have $\bar{\imath}:= \max \{i\vert f(\alpha_i) \neq 0\}$ is strictly smaller than $s$ and $f(\alpha_{\bar{\imath}})\in \langle \beta \rangle$ with $\beta \in P_{\alpha_{\bar{\imath}}}$, or 
\item[2)] we have $\bar{\imath}:= \min \{i\vert f(\alpha_i) \neq 0\}$ is strictly bigger than $0$ and $f(\alpha_{\bar{\imath}})\in \langle \beta \rangle$ with $\beta \in Q_{\alpha_{\bar{\imath}}}$. 
\end{itemize}
Suppose we are in the first case. Renormalize $g=(f,h)$ so that $f(\alpha_{\bar{\imath}})=\beta$. First of all notice that $(\alpha_{\bar{\imath}}, \beta)\notin \text{Obs}$, otherwise we would be in the situation where
\[
\bar{\imath}\leq j\qquad \text{ and }\quad \beta =\frac{\alpha_j}{y^{p_{\bar{\imath}}}}
\]
with $f(\alpha_i)=0$ for all $i>\bar{\imath}$. But all $\theta \in \text{Hom}_R(I_2, \bigslant{R}{I_2})$ such that $\theta(\alpha_{\bar{\imath}})= \frac{\alpha_j}{y^{p_{\bar{\imath}}}}$ have $\theta(\alpha_{\bar{\imath}+1})\neq 0 $, if not 
\[
y^{p_{\bar{\imath}}}\alpha_{\bar{\imath}} =x^{q_{\bar{\imath}+1}} \alpha_{\bar{\imath}+1} \quad\implies \qquad\alpha_j=y^{p_{\bar{\imath}}}\frac{\alpha_j}{y^{p_{\bar{\imath}}}}=\theta(y^{p_{\bar{\imath}}}\alpha_{\bar{\imath}})=\theta(x^{q_{\bar{\imath}+1}}\alpha_{\bar{\imath}+1})=0.
\]
This equality mod $I_2$ is absurd since $\alpha_j \notin I_2$. Then, always thanks to the definition of $\bar{\imath}$, we have that $y^{p_{\bar{\imath}}}\beta \in I_1$, and thanks to the fact that $(\alpha_{\bar{\imath}}, \beta)\notin \text{Obs}$, we have that 
\[
(f-f_{\alpha_{\bar{\imath}}, \beta}, h-\Su(f_{\alpha_{\bar{\imath}}, \beta})) \in \text{Ker}(\phi_{12}-\psi_{12})
\]
is well defined and has first coordinate that sends strictly more generators to $0$ mod $I_1$, i.e. $n(f-f_{\alpha_{\bar{\imath}}})<t$. By induction we have that $g \in \text{span}(B(I_1,I_2))$ as desired. \\

\hspace{3em} The case 2) is completely analogous. 
\end{proof}

\begin{theorem}{ \cite[Theorem~3.2.2, 3.3.3]{cheah1998cellular}}
The space $\Hi^{n, n+1}(\CC^2)$ is smooth. The space $\Hi^{n, n+1}(0)$ has an affine cell decomposition given by the attracting sets at torus fixed points. Moreover if the $\TT^1$ action on $R$ has weights $0<w_1< w_2$ and $1\ll \frac{w_2}{w_1}$ then the affine attracting cell in $\Hi^{n, n+1}(0)$ of the point $(I_1, I_2)$ has dimension $n+1-d$ where $d=\min\{i\vert x^i \in I_2\}$ i.e. $d=\ell(\Gamma(I_2))$. In particular the Betti numbers of $\Hi^{n, n+1}(0)$ satisfy: 
\[
b_i =  \#\,\left\{(\Gamma_1, \Gamma_2)  \left\vert  \Gamma_1\vdash n, \,\, \Gamma_2\vdash n+1,\,\, \Gamma_1\subset \Gamma_2 \text{ and }\,\ell(\Gamma_2) = n+1-i\right. \right\}\, .
\]
The Poincar\'{e} polynomial of $\Hi^{n, n+1}(0)$ is 
\[
P_q\,\left(\Hi^{n, n+1}(0)\right) = \sum_{(\Gamma_1, \Gamma_2)\, \vdash\, [n, n+1]} \,\, q^{n+1-\ell(\Gamma_2)}.
\]
The generating function for these polynomials is 
\begin{equation}
\label{cheahgeneratingformula}
\sum_{n=0}^{+\infty } \, P_q\,\left(\Hi^{n, n+1}(0)\right) z^n = \left(\frac{1}{1- zq}\right)\,\,\prod_{k\geq 1} \, \frac{1}{1- z^kq^{k-1}}.
\end{equation}
\end{theorem}
\begin{proof}
The proof of smoothness follows immediately from Observation \ref{cardinalityB(I1I2)} and Lemma \ref{B(I1I2)isbasis} that tell us that the dimension of the tangent space at each torus fixed point is the same as the dimension of the variety. \\

\hspace{3em} The proof on the Betti numbers goes exactly as the proof of Theorem \ref{ES}. First of all $\Hi^{n,n+1}(\PP^2)$ is smooth since $\Hi^{n,n+1}(\CC^2)$ is. Then Theorem \ref{Bialynicki-Birula} gives us a cell decomposition with the dimension of the cells specified by the positive part of the tangent spaces at fixed points. Call $P_0\in \PP^2$ the fixed points with zero dimensional attracting cells for the $\TT$ action on $\PP^2$; since the limiting process preserves the support for subschemes with support concentrated in $P_0$ we have that $\Hi^{n, n+1}(0)= \Hi^{n,n+1}(P_0)$ is union of cells of $\Hi^{n,n+1}(\PP^2)$. \\
 
 \hspace{3em} We do not include the proof of the statement on the generating function. Cheah \cite[pp~69-70]{cheah1998cellular} proves it directly in one page. In \cite[Chapter~5]{nakajima2008perverse} the authors give a slightly more indirect and more geometrical proof whose ingredients are more similar to the discussion we will give in the last section of Chapter 4.
\end{proof}

\subsection{The tangent space to the Hilbert-Samuel's strata, case $\mathfrak{n}=(n, n+1)$}

\hspace{3em}For $T_1, T_2$ two admissible sequences of nonnegative integers as in \ref{admissible}, we want to study the tangent space of $M_{T_1, T_2}$ ( resp. of $G_{T_1, T_2}$) at a flag of monomial ideals $(I_1, I_2)$ with $T(I_i)= T_i$.

\hspace{3em}Thanks to Observation \ref{goettschetangenttoM} we can reduce this to the study of the weights of the elements of the basis $B(I_1, I_2)$. Moreover to understand the cellular decomposition of $M_{T_1, T_2}$ induced by the action of the torus $\TT_{1^+}$ we are interested in studying the positive parts of these tangent spaces.\\

\hspace{3em}We divide the study of elements of $B(I_1, I_2)$ into two observations, according to their type. 
\begin{observation}
\label{weight(0,h)}
To study the weight of the first kind of elements
\[
(0,\, h_{\alpha'_i,\alpha_j}), \quad \text{ with } i=0, \dots, s'
\] 
we need only to compare the degree of a generator of $I_2$ with that of $\alpha_j$ the only  monomial in $I_1$ but not in $I_2$. More precisely, we have that the $(0,h_{\alpha'_i, \alpha_j}) \in B(I_1, I_2)$ is tangent to $M_{T(I_1), T(I_2)}$ if and only if $\deg \alpha'_i\leq\deg \alpha_j$; tangent to $G_{T(I_1), T(I_2)}$ if and only if $\deg \alpha'_i= \deg \alpha_j$; tangent to  $M_{T(I_1), T(I_2)}$ and with positive weight if and only if $\deg \alpha'_i\leq\deg \alpha_j$ and $\deg_y \alpha'_i<\deg_y \alpha_j$; tangent to  $G_{T(I_1), T(I_2)}$ and with positive weight if and only if $\deg \alpha'_i=\deg \alpha_j$ and $\deg_y \alpha'_i<\deg_y \alpha_j$. 

\begin{example}
$\,$

\begin{minipage}{0.6\textwidth}\begin{small}
\ytableausetup{boxsize=2.1em, aligntableaux=bottom}
\,\,
\begin{ytableau} 
\none[m,g]\\
\,&\none[m,g] \\
\,&\,&*(green) \alpha_j  \\
\,&\, &\,&\none[\quad m^+,g^+] \\
\,&\, &\,&\,  \\
\,&\, &\,&\,&\none[m^+]  \\
\,&\, &\, &\,&\, &\,&\,&\,
\end{ytableau}\end{small}
\end{minipage}
\begin{minipage}{0.35\textwidth}
With $m$ (resp. $g$) we indicate a generator 
of $I_2$ that produces a tangent vector 
to \[ M_{T(I_1), T(I_2)} (\text{ resp. } G_{T(I_1), T(I_2)})\] 
The plus is added if the vector produced has positive weight. 
\end{minipage}
\vspace{1em}
\end{example}
\end{observation}
\begin{observation}
\label{weight(f,f)}
We want now to analyze the second type of elements of the basis $B(I_1, I_2) $:
\[
\left\{\left(f_{\alpha, \beta}, \Su(f_{\alpha, \beta})\right) \,\,\left\vert\,\, \ff\in B(I_1), \, (\alpha, \beta)\notin \text{Obs}(I_1, I_2))\right.\right\} .
\]
From the definition of $\Su(\ff)$ we have that $(\ff,\Su(\ff))$ has the same weight as $\ff$. So  $(\ff,\Su(\ff)) \in B(I_1, I_2)$ is tangent to $M_{T(I_1), T(I_2)}$ if and only if $\ff$ is tangent to $M_{T(I_1)}$ i.e. if and only if $\deg \alpha \leq \deg \beta$. Similarly for the positive part, and for the tangent to $G_{T(I_1), T(I_2)}$ and its positive part. In fact it is enough to understand when $\ff$ with $(\alpha, \beta)\in \text{Obs}(I_1,I_2)$ \emph{is} raising or preserving degree, and so on, and then understand the dimension of the appropriate vector spaces by \emph{difference}: if $(\alpha, \beta)\in \text{Obs}(I_1,I_2)$, $\ff$ does \emph{not} extend to an element of $B(I_1, I_2) $. 

Looking at definition \ref{obs} we have that $f=\ff$ with $(\alpha, \beta)\in \text{Obs}(I_1,I_2)$ is such that either 
\begin{itemize} 
\item[] case1) $\qquad f(\alpha_i) = \frac{\alpha_j}{y^{p_i}}  \text{ with } i<j\, \text{ or } $
\item[] case2)   $\qquad f(\alpha_i) = \frac{\alpha_j}{x^{q_i}}  \text{ with } i>j\,.$
\end{itemize}
In case1) we have the weight of $ \ff $ is equal to $\deg \frac{\alpha_j}{y^{p_i}} -\deg \alpha_i = \deg {\alpha_j} -y^{p_i}\deg \alpha_i$. The same formula holds for the $y$-degree. In case 2), similarly, we have: $\text{wt}(\ff) = \deg \frac{\alpha_j}{x^{q_i}} -\deg \alpha_i = \deg {\alpha_j} -x^{q_i}\deg \alpha_i$. 

Notice also that for each $\alpha_i$ generator of $I_1$ there is at most one $\beta$ such that $(\alpha_i, \beta)\in \text{Obs}(I_1,I_2)$. Then to know for which $\alpha_i, i=0,\dots,s$ there is a $\beta$ such that $\ff \in T_{I_1} \Hi^{n}(\CC^2)$ does not produce an element of the basis $B(I_1, I_2)$ of $T_{I_1, I_2} \Hi^{n, n+1}(\CC^2)$ but that is, say, raising degree (i.e. $\ff$ is tangent to $M_{T(I_1)}$),  the relevant check on degrees is between $\alpha_j$ and $y^{p_i}\alpha_i$ if $i\leq j$ and between $\alpha_j$ and $x^{q_i}\alpha_i$ if $i\geq j$. Notice that we do not need to distinguish cases as in the definition of $\text{Obs}(I_1, I_2)$ \ref{obs}, since for the indexes $i=j-1,j,j+1$, where there is actually a difference in the definition of $\text{Obs}(I_1, I_2)$, we have $\deg \alpha_j<\deg y^{p_i}\alpha_i$ or $\deg \alpha_j<\deg x^{q_i}\alpha_i$ so that we would not consider them anyway, since the corresponding $\ff$ is of negative weight. We can then summarize and say
\begin{equation}
f_{\alpha_i, \beta} \in T_{I_1} M_{T(I_1)} \text{ but } (\alpha_i, \beta)\in \text{Obs}(I_1,I_2) \iff 
\begin{cases}
\deg y^{p_i}\alpha_i\leq  \deg \alpha_j & \text{ if } i\leq j, \\
 \deg x^{q_i}\alpha_i \leq\deg \alpha_j & \text{ if } i\geq j\, . 
\end{cases}
\end{equation}
Remark once more that these are the vectors we will need to exclude. Similarly for all the other vector spaces we are interested in. Let us visualize it in an example: 

\begin{minipage}{0.6\textwidth}
\ytableausetup{boxsize=2em, aligntableaux=bottom}
\,\,
\begin{ytableau} 
\,&\, &\, &\,&\, &\,&\,&\, &\,\\
*(lightgray)\,&*(green)\alpha_j\, &\, &\,&\, &\,&\,&\, &\,\\
*(lightgray)\,&*(lightgray)\, &\, &\,&\, &\,&\,&\, &\,\\
*(lightgray)\,&*(lightgray)\, &*(lightgray)\, &\,&\, &\,&\,&\, &\,\\
*(lightgray)\,&*(lightgray)\, &*(lightgray)\, &\,&m^+ g^+ &\,&\,&\, &\,\\
*(lightgray)\,&*(lightgray)\, &*(lightgray)\, &*(lightgray)\,&\, & m^+ &\,&\, &\,\\
*(lightgray)\,&*(lightgray)\, &*(lightgray)\, &*(lightgray)\,&*(lightgray)\, &\,&\,&\, &\,\\
*(lightgray)\,&*(lightgray)\, &*(lightgray)\, &*(lightgray)\,&*(lightgray)\, & \, &\,&\, &\,\\
\end{ytableau}
\end{minipage}
\begin{minipage}{0.35\textwidth}
With $m$ (resp. $g$)  we indicate $y^{p}$ times a generator  
of $I_1$ that produce a tangent vector to $M_{T_1}$ (resp. $G_{T_1}$) that cannot be extended to a tangent vector of $M_{T_1, T_2}$ (resp. $G_{T_1, T_2}$). The plus as exponent is there if the tangent to $M_{T_1}$ (resp. $G_{T_1}$)  is also attracting. 
\end{minipage}
\vspace{1em}
\end{observation}

\begin{definition}
\label{starsanddots}
We put together what we remarked in the last two observations to give a combinatorial rule of how the dimension of the tangent space at the two step flag case grows with respect to the one step flag case.\\ 

\hspace{3em}Let $\Gamma_2= \Gamma_1\cup\{\alpha_j\}$ be the Young diagrams associated to monomial ideals $(I_1, I_2)$ in $\Hi^{n, n+1}(\CC^2)$. Then put a $\bullet$ at each box occupied by a standard monomial generators of $I_2$, and a $\star$ at each vertex of $\Gamma_2$, i.e. wherever on the same column below and on the same row on the left there is a $\bullet$, and no others marked boxes in between.

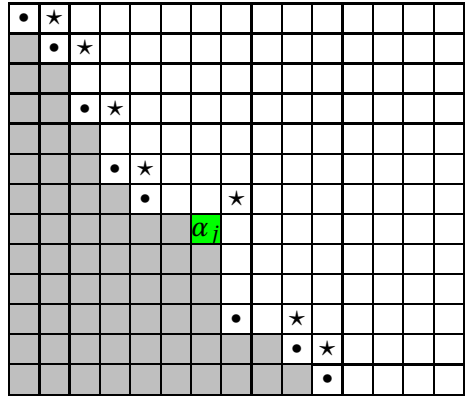
\begin{center}
\ytableausetup{boxsize=1em, aligntableaux=bottom}
\begin{ytableau} 
\bullet &\star &\, &\,&\, &\,&\,&\, &\,&\,&\, &\,&\,&\, &\,\\
*(lightgray)\,&\bullet &\star &\,&\, &\,&\,&\, &\,&\,&\, &\,&\,&\, &\,\\
*(lightgray)\,&*(lightgray)\, &\, &\,&\, &\,&\,&\, &\,&\,&\, &\,&\,&\, &\,\\
*(lightgray)\,&*(lightgray)\, &\bullet &\star&\, &\,&\,&\, &\,&\,&\, &\,&\,&\, &\,\\
*(lightgray)\,&*(lightgray)\, &*(lightgray)\, &\,&\, &\,&\,&\, &\,&\,&\, &\,&\,&\, &\,\\
*(lightgray)\,&*(lightgray)\, &*(lightgray)\, &\bullet&\star &\,&\,&\, &\,&\,&\, &\,&\,&\, &\,\\
*(lightgray)\,&*(lightgray)\, &*(lightgray)\, &*(lightgray)\,&\bullet &\,&\,&\star &\,&\,&\, &\,&\,&\, &\,\\
*(lightgray)\,&*(lightgray)\, &*(lightgray)\, &*(lightgray)\,&*(lightgray)\, &*(lightgray)\,&*(green)\alpha_j&\, &\,&\,&\, &\,&\,&\, &\,\\
*(lightgray)\,&*(lightgray)\, &*(lightgray)\, &*(lightgray)\,&*(lightgray)\, &*(lightgray)\,&*(lightgray)\,&\, &\,&\,&\, &\,&\,&\, &\,\\
*(lightgray)\,&*(lightgray)\, &*(lightgray)\, &*(lightgray)\,&*(lightgray)\, &*(lightgray)\,&*(lightgray)\,&\, &\,&\,&\, &\,&\,&\, &\,\\
*(lightgray)\,&*(lightgray)\, &*(lightgray)\, &*(lightgray)\,&*(lightgray)\, &*(lightgray)\,&*(lightgray)\,&\bullet &\,&\star&\, &\,&\,&\, &\,\\
*(lightgray)\,&*(lightgray)\, &*(lightgray)\, &*(lightgray)\,&*(lightgray)\, &*(lightgray)\,&*(lightgray)\,&*(lightgray)\, &*(lightgray)\,&\bullet&\star &\,&\,&\, &\,\\
*(lightgray)\,&*(lightgray)\, &*(lightgray)\, &*(lightgray)\,&*(lightgray)\, &*(lightgray)\,&*(lightgray)\,&*(lightgray)\, &*(lightgray)\,&*(lightgray)\,&\bullet &\,&\,&\, &\,\\
\end{ytableau}
\end{center}

Define $\Lambda_i$ the $i$-th antidiagonal of the positive quadrant, i.e. all boxes $(k,m)$ with $k+m=i$ for $i\geq 0$. Say that $\Lambda_i< \Lambda_k$ if $i< k$, i.e. if $\Lambda_i$ is a lower antidiagonal than $\Lambda_k$. Call $\Lambda_{\alpha_j}$ the antidiagonal that passes through $\alpha_j$. Then define the following numbers: 
\begin{align*}
&\bullet \, M(\Gamma_1, \Gamma_2)\,\, :=\,\,  \#\, \{\bullet \in \Lambda_i \left\vert \Lambda_i < \Lambda_{\alpha_j} \right. \} +   \#\, \{\bullet \in \Lambda_{\alpha_j} \} \\
&\bullet \, G(\Gamma_1, \Gamma_2) \,\,:=\,\,   \#\, \{\bullet \in \Lambda_{\alpha_j} \} \\
&\bullet M^+(\Gamma_1, \Gamma_2) \,\,:=\,\,   \#\, \{\bullet \in \Lambda_i \left\vert \Lambda_i < \Lambda_{\alpha_j} \right. \} +   \#\, \{\bullet \in \Lambda_{\alpha_j}\left\vert \bullet \text{ to the right of } \alpha_j \right. \} \\
&\bullet G^+(\Gamma_1, \Gamma_2) \,\,:=\,\,    \#\, \{\bullet \in \Lambda_{\alpha_j}\left\vert \bullet \text{ to the right of } \alpha_j \right. \} \\
&\star \, M(\Gamma_1, \Gamma_2)\,\, :=\,\,  \#\, \{\star \in \Lambda_i \left\vert \Lambda_i < \Lambda_{\alpha_j} \right. \} +   \#\, \{\star \in \Lambda_{\alpha_j} \} \\
&\star \, G(\Gamma_1, \Gamma_2) \,\,:=\,\,   \#\, \{\star \in \Lambda_{\alpha_j} \} \\
&\star M^+(\Gamma_1, \Gamma_2) \,\,:=\,\,   \#\, \{\star \in \Lambda_i \left\vert \Lambda_i < \Lambda_{\alpha_j} \right. \} +   \#\, \{\star \in \Lambda_{\alpha_j}\left\vert \star \text{ to the right of } \alpha_j \right. \} \\
&\star G^+(\Gamma_1, \Gamma_2) \,\,:=\,\,    \#\, \{\star \in \Lambda_{\alpha_j}\left\vert \star \text{ to the right of } \alpha_j \right. \} .
\end{align*}
\end{definition}

\begin{lemma}{\cite[Lemma~3.4.9]{cheah1998cellular}}
\label{dimensiontangent12}
Suppose that $(I_1, I_2) \in G_{T_1, T_2} \subset M_{T_1, T_2}$ is a fixed point of the $\TT_{1^+}$ action on $\Hi^{n, n+1}(0)$ with generic weights $w_1<w_2$ and $(n+1)w_1>n w_2$.  Call $(\Gamma_1, \Gamma_2)$ the couple of Young diagrams associated to $(I_1, I_2)$. Then: 
\begin{itemize}
\item[1)] The dimension of the tangent space to $M_{T_1, T_2}$ at $(I_1, I_2)$ is equal to 
\[
\dim M_{T_1} + \bullet \, M(\Gamma_1, \Gamma_2) -\star \, M(\Gamma_1, \Gamma_2).
\]
\item[2)] The dimension of the tangent space to $G_{T_1, T_2}$ at $(I_1, I_2)$ is equal to 
\[
\dim G_{T_1} + \bullet \, G(\Gamma_1, \Gamma_2) -\star \, G(\Gamma_1, \Gamma_2).
\]
\item[3)] The dimension of the positive part of tangent space to $M_{T_1, T_2}$ at $(I_1, I_2)$ is equal to 
\[
\dim T^+_{I_1}M_{T_1}  + \bullet M^+(\Gamma_1, \Gamma_2) -\star M^+(\Gamma_1, \Gamma_2).
\]
\item[4)] The dimension of the positive part of the tangent space to $G_{T_1, T_2}$ at $(I_1, I_2)$ is equal to 
\[
\dim T^+_{I_1}G_{T_1} + \bullet G^+(\Gamma_1, \Gamma_2)-\star G^+(\Gamma_1, \Gamma_2).
\]
\end{itemize}
\end{lemma}
\begin{proof}
The proof is immediate thanks to the definition of $B(I_1, I_2)$ \ref{definitionB(I1I2)} and Observations \ref{weight(0,h)}  and \ref{weight(f,f)}.
\end{proof}

In the next chapter we will see that these strata are smooth, and we will talk about their dimension and about their homology.


\section{A weight basis for $T_{I_1,I_2,I_3} \Hi^{n,n+1,n+2}(\CC^2)$}

\hspace{3em} If $(I_1,I_2,I_3)$ is a fixed point of $\Hi^{n,n+1,n+2}(\CC^2)$ we denote by $\alpha_j$ the only monomial in $I_1$ but not in $I_2$, and by $\alpha'_l$ the only monomial in $I_2$ but not in $I_3$. We need to find a weight basis for 
\[
T_{I_1,I_2,I_3} \Hi^{n,n+1,n+2}(\CC^2)\cong 
\bigcap_{1\leq i<j\leq 3} \left(\text{Ker}(\phi_{ij}-\psi_{ij})\circ \pi_{ij}\right).
\]
 Trying to mimic what we did for the two step flag case, we want to find elements in \\$\text{Hom}_R(I_1, \bigslant{R}{I_1})\oplus\text{Hom}_R(I_2, \bigslant{R}{I_2})\oplus\text{Hom}_R(I_3, \bigslant{R}{I_3})$ of the form 
\begin{align*}
(0,0, (\alpha \mapsto \alpha'_l)), \\
(0,(\alpha \mapsto \alpha_j), \Su((\alpha \mapsto \alpha_j)),\\ 
(\ff, \Su(\ff), \Su(\Su(\ff))).
\end{align*} 

\hspace{3em}As before, then, most of the work is to understand for what elements we can actually define the Suiv. For example, we know already that for the last kind of vectors we need to take $\ff \in B(I_1)$ with $(\alpha, \beta) \notin \text{Obs}(I_1, I_2)$. Presumably we need to take care about the problem of extending vectors from $B(I_2)$ to $B(I_2, I_3)$. For this reason we will define \ref{definitionObs23}. In fact the real twist of the three flag case is Definition \ref{definitionNotP}: this tackles the problem of making sure that we pick out those vectors of $B(I_2)$ that do not extend to vectors in $B(I_2, I_3)$ but that actually matter for the construction of $B(I_1, I_2, I_3)$. Philosophically we are trying to define $B(I_1, I_2, I_3)$ by using twice, and independently, the two flag case and then slightly correcting with some information that is intrinsic to the three flag case. 

\begin{notation}
Let $(I_1,I_2,I_3)$ be a fixed point of $\Hi^{n,n+1,n+2}(\CC^2)$, i.e. each $I_i$ is a monomial ideal with $\dim\bigslant{R}{I_i}= n+i-1$ and $I_1\supset I_2\supset I_3$. We assume that: 
\begin{align*}
I_1=(\alpha_0, \dots, \alpha_s),\quad I_2=(\alpha'_0, \dots, \alpha'_{s'}),\quad I_3=(\alpha''_0, \dots, \alpha''_{s''}), \\
\Gamma(I_3) = \Gamma(I_2) \cup \{\alpha'_l\} = \Gamma(I_1)\cup \{\alpha_j\} \cup \{\alpha'_l\},
\end{align*}
where the $\alpha_i$s are standard monomial generators of $I_1$, the $\alpha'_i$s are standard monomial generators of $I_2$, the $\alpha''_i$s are standard monomial generators of $I_3$ and there are $s''+1=s_3+1$ of them. In the first step we add the box $\alpha_j$ to $\Gamma(I_1)$ that was a generator of $I_1$. In the second step we add $\alpha'_l$ to $\Gamma(I_2)$ that was a generator of $I_2$. \end{notation}

\begin{definition}
\label{casesjl}
We distinguish two cases based on the relative position of $\alpha_j$ and $\alpha'_l$. The corresponding fixed points will have different geometrical properties. 
\vspace{0.5 em}

\begin{minipage}{0.6\textwidth} 
\begin{center}\textbf{case a)}\end{center}The second box we add, $\alpha'_l$, is a minimal generator also of $I_1$. This is always the case unless $\alpha'_l$ is immediately above or on the right of $\alpha_j$. In this case we often drop the prime and call the second box simply $\alpha_l$. 

\end{minipage}
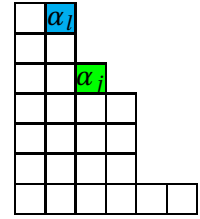
\begin{minipage}{0.35\textwidth}
\begin{center}
\ytableausetup{boxsize=1em, aligntableaux=bottom}
\begin{ytableau} 
\,&*(cyan)\alpha_l\\
\,&\, \\
\,&\,&*(green) \alpha_j  \\
\,&\, &\,&\, \\
\,&\, &\,&\,  \\
\,&\, &\,&\,  \\
\,&\, &\, &\,&\, &
\end{ytableau}
\end{center} 
\end{minipage}

\vspace*{2em}

\begin{minipage}{0.6\textwidth} \begin{center}\textbf{case b)}\end{center}The second box we add, $\alpha'_l$, is \emph{not} a minimal generator of $I_1$. This can happen only if  $\alpha'_l$ is immediately above or to the right of $\alpha_j$, i.e. $\alpha'_l=y\alpha_j$ or $\alpha'_l=x\alpha_j$. We will keep the prime in $\alpha'_l$ in this case.  
\end{minipage}
\begin{minipage}{0.35\textwidth}
\begin{center}
\ytableausetup{boxsize=1em, aligntableaux=bottom}
\begin{ytableau} 
\,\\
\,&\,&*(cyan)\alpha_l \\
\,&\,&*(green) \alpha_j  \\
\,&\, &\,&\, \\
\,&\, &\,&\,  \\
\,&\, &\,&\,  \\
\,&\, &\, &\,&\, &
\end{ytableau}
\end{center} 
\end{minipage}
\end{definition}

\begin{notation}
Call $p_i$ and $q_i$, as before, the distances between generators of $I_1$, and $p'_i$ and $q'_i$ the distances between generators of $I_2$, and $p''_i$ and $q''_i$ the distances between generators of $I_3$: 
\begin{align*}
&p_i := \deg_y \alpha_{i+1} - \deg_y \alpha_{i} \quad p_s := \infty,\quad  q_i := \deg_x \alpha_{i} - \deg_x \alpha_{i-1} \quad q_0 := \infty,
\\
&p'_i := \deg_y \alpha'_{i+1} - \deg_y \alpha'_{i} \quad p'_{s'} := \infty, \quad q'_i := \deg_x \alpha'_{i} - \deg_x \alpha'_{i-1} \quad q'_0 := \infty,\\
&p''_i := \deg_y \alpha''_{i+1} - \deg_y \alpha''_{i} \quad p''_{s''} := \infty,\quad  q''_i := \deg_x \alpha''_{i} - \deg_x \alpha''_{i-1} \quad q''_0 := \infty.
\end{align*}
Analogously we call $P_{\alpha}=P_{\alpha_i}$ and  $Q_{\alpha}=Q_{\alpha_i}$ the relevant sets for $\alpha=\alpha_i$ generator of $I_1$, $P_{\alpha'}=P_{\alpha'_i}$ and  $Q_{\alpha'}=Q_{\alpha'_i}$ those for the generators of $I_2$ and $P_{\alpha''}=P_{\alpha''_i}$ and  $Q_{\alpha''}=Q_{\alpha''_i}$ those for the generators of $I_3$. Whenever confusion is possible we will clarify if $\ff$ is seen as an element in the basis $B(I_1)$ of $T_{I_1} \Hi^{n}(\CC^2)$ or as an element in the basis $B(I_2)$ of $T_{I_2} \Hi^{n+1}(\CC^2)$, and so on. 
\end{notation}
\begin{definition}
\label{cases123}
We need to distinguish between the possible cases of Definition \ref{cases12} on $s, s'$ and $s''$. To give a name to all possible cases we will combine the name of the case as in \ref{cases12}  and the name of the boxes it refers to. For example the case $l 1a)$ means that, we do not care about the position of $\alpha_j$, and when we add $\alpha'_l$ we add it as a minimal generator of $I_2$ that has $p'>1$ and $q'=1$. Another example:  $j 1a), l 3)$ means that  $p_j\geq 2, q_j=1$ and $p'_l=1=q'_l$. \\

\hspace{3em}It is clear that, between these cases and the cases for the relative position of $\alpha_j$ and $\alpha_l$ of \ref{casesjl}, we have quite a few possible cases. 
\end{definition}

\begin{example}
Before introducing $\text{Obs}(I_2,I_3)$, we show with an example why we need to care about extending vectors from $T_{I_2} \Hi^{n+1}(\CC^2)$ to $T_{I_2, I_3} \Hi^{n+1, n+2}(\CC^2)$ in order to define vectors of the third type $\left(\ff, \Su(\ff)), \Su(\Su(\ff)\right)$. 

\vspace{0.6em}
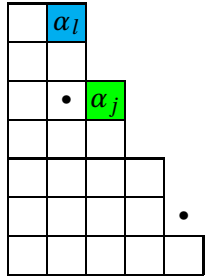
\begin{minipage}{0.35\textwidth}
\begin{center}
\ytableausetup{boxsize=1.3em, aligntableaux=bottom}
\begin{ytableau} 
\,&*(cyan)\alpha_l \\
\,&\,\\
\,&\bullet&*(green) \alpha_j  \\
\,&\, &\,\\
\,&\, &\,&\,  \\
\,&\, &\,&\,&\none[\bullet]  \\
\,&\, &\, &\,&\,
\end{ytableau}
\end{center} 
\end{minipage}
\begin{minipage}{0.6\textwidth} \small{Let $\ff \in B(I_1)$ be the map depicted by the bullets, i.e. that sends the bullet outside $\Gamma_{1}$ to the bullet inside and is zero on the other generators of $I_1$. Then $(\ff, \Su(\ff)) \in B(I_1, I_2)$, but there does not exist $g \in T_{I_3} \Hi^{n+2}(\CC^2)$ such that $(\ff, \Su(\ff), g)$ is tangent to $T_{I_1, I_2,I_3} \Hi^{n, n+1,n+2}(\CC^2)$, simply because there does not exist a map $g \in T_{I_3} \Hi^{n+2}(\CC^2)$ that sends $\bullet$ to $\bullet$.}
\end{minipage}
\vspace{1em}
\end{example}

\begin{definition}
\label{definitionObs23}
We need to define $\text{Obs}(I_1,I_2)$ and $\text{Obs}(I_2,I_3)$. We define them only according to the respective cases of the last box added, and not on the relative position of $\alpha_j$ and $\alpha_l$. Thus, for example, $\text{Obs}(I_2,I_3)$ depends only on $p'_l$ and $q'_l$ and not on $p_j$ and $q_j$. 

\hspace{3em}For  $\text{Obs}(I_1,I_2)$, then, the definition is exactly as in \ref{IObs12}. The definition of $\text{Obs}(I_2,I_3)$ is completely analogous but we add it for convenience. 
\[
\begin{matrix}
\begin{matrix}
\text{Case $l$ 1a)} \\
\text{Obs}(I_2,I_3):=
\left\{
\begin{matrix}
(\alpha'_i, \frac{\alpha_l}{x^{q_i}}) \text{ for } i>l\\
(\alpha'_i, \frac{\alpha_l}{y^{p_i}}) \text{ for } i<l-1.
\end{matrix}
\right\}  
\end{matrix}\qquad
&
\begin{matrix}
\text{Case $l$ 1b)} \\
\text{Obs}(I_2,I_3):=
\left\{
\begin{matrix}
(\alpha'_i, \frac{\alpha_l}{x^{q_i}}) \text{ for } i>l+1\\
(\alpha'_i, \frac{\alpha_l}{y^{p_i}}) \text{ for } i<l.
\end{matrix}
\right\}
\end{matrix}  \\
\vspace{0.7em}
\begin{matrix}
\text{Case $l$ 2)} \\
\text{Obs}(I_2,I_3):=
\left\{
\begin{matrix}
(\alpha'_i, \frac{\alpha_l}{x^{q_i}}) \text{ for } i>l\\
(\alpha'_i, \frac{\alpha_l}{y^{p_i}}) \text{ for } i<l.
\end{matrix}
\right\}  
\end{matrix}\qquad
&
\begin{matrix}
\text{Case $l$ 3)} \\
\text{Obs}(I_2,I_3):=
\left\{
\begin{matrix}
(\alpha'_i, \frac{\alpha_l}{x^{q_i}}) \text{ for } i>l+1\\
(\alpha'_i, \frac{\alpha_l}{y^{p_i}}) \text{ for } i<l-1.
\end{matrix}
\right\}
\end{matrix}  \\
\end{matrix}
\]
\end{definition}
\begin{observation}
\label{numberobs3}
Again observe that in terms of $s''=s_3$, where $s''+1$ is the number of generators of $I_3$, the set $\text{Obs}(I_2,I_3) $ has the same number of elements along all cases.
\[
\begin{matrix}
\begin{matrix}
\text{Case $l$ 1a)}\\
 \#\,\text{Obs}(I_1,I_2) =  s_2-2 = s_3+1-2
\end{matrix} \qquad
&
\begin{matrix}
\text{Case l 1b)}\\
 \#\,\text{Obs}(I_2,I_3) =  s_2-2 = s_3+1-2
\end{matrix} 
\\
\begin{matrix}
\text{Case $l$ 2)}\\
 \#\,\text{Obs}(I_2,I_3) =  s_2-1 = s_3+1-2
\end{matrix} 
&
\begin{matrix}
\text{Case $l$ 3)}\\
 \#\,\text{Obs}(I_2,I_3) =  s_2-3 = s_3+1-2
\end{matrix} 
\end{matrix}
\]
\end{observation}

However, as suggested at the beggining of the section, not all elements of $\text{Obs}(I_2, I_3)$ are relevant. The next example shows why this can happen. 
\begin{example}
$\,$\\
$\,$\\
\begin{minipage}{0.35\textwidth}
\begin{center}
\ytableausetup{boxsize=1.3em, aligntableaux=bottom}
\begin{ytableau} 
\,&*(cyan)\alpha_l \\
\,&\,\\
\,&\bullet&*(green) \alpha_j  \\
\,&\, &\,&\none[\bullet]\\
\,&\, &\,&\,  \\
\,&\, &\,&\,  \\
\,&\, &\, &\,&\, &
\end{ytableau}
\end{center} 

\end{minipage}
\begin{minipage}{0.6\textwidth} Let $f_{\bullet,\bullet} \in B(I_2)$ be the map depicted by the bullets, i.e. that sends the bullet outside $\Gamma_{2}$ to the bullet inside and is zero on other generators of $I_2$ that are above the bullet outside (but is not zero on the remaining generators of $I_2$, $\alpha'_0$ and $\alpha'_1$). Notice that $p'_{\bullet}= 2$. Then $(\alpha, \beta) \in \text{Obs}(I_2, I_3)$, but there does not exist $f \in T_{I_1}\Hi^n(\CC^2)$ such that $\ff= \Su(f)$ since $p_{\bullet}=1$. Then the couple $(\alpha, \beta)$ in $B(I_2, I_3)$ is actually irrelevant to our purposes. 
\end{minipage}
\vspace{1em}
\end{example}

To avoid counting unnecessarily obstructions at the second step we introduce this definition.
\begin{definition}
\label{definitionNotP}
We define a set of indexes $\text{NotP}=\text{NotP}(\Gamma_1, \Gamma_2,\Gamma_3) :=\{(\alpha'_i, \beta)\}$ whose corresponding elements $\left\{f_{\alpha'_i, \beta}\right\}$ are in $ \text{Hom}_R(I_2, \bigslant{R}{I_2})$. In the general case there will be only one of these, and it will be such that $f_{\alpha'_i, \beta} \in \text{Obs}(I_2,I_3)$ but there is not $f$ an $R$-homomorphism in $\text{Hom}_R(I_1, \bigslant{R}{I_1})$ with $\Su(f)= f_{\alpha'_i, \beta}$. We need to distinguish between the two cases $a)$ and $b)$ of \ref{casesjl}.

\begin{center}
\textbf{case a)} \\
$\alpha_l$ not immediately above or to the right of $\alpha_j$
\end{center}
The definition is according to the cases \ref{cases12} for $\alpha_j$.  
\[
\begin{matrix}
\begin{matrix}
\text{Case $j$ 1a)}\\
\text{If } (j<l) \text{ then }\, \text{NotP} :=  \left\{\left(\alpha_{j-1}, \frac{\alpha_l}{y^{p'_{j-1}}}\right)\right\}\\
\text{If } (j>l) \text{ then }\, \text{NotP} :=  \left\{\left(y\alpha_{j}, \frac{\alpha_l}{x}\right)\right\}.
\end{matrix} \qquad
&
\begin{matrix}
\text{Case $j$ 1b)}\\
\text{If } (j<l) \text{ then }\, \text{NotP} :=  \left\{\left(x\alpha_{j}, \frac{\alpha_l}{y}\right)\right\}\\
\text{If } (j>l) \text{ then }\, \text{NotP} :=  \left\{\left(y\alpha_{j+1}, \frac{\alpha_l}{x^{q'_{j+1}}}\right)\right\}.
\end{matrix}
\\
\begin{matrix}
\text{Case $j$ 2)}\\
\text{If } (j<l) \text{ then }\, \text{NotP} :=  \left\{\left(x\alpha_{j}, \frac{\alpha_l}{y}\right)\right\}\\
\text{If } (j>l) \text{ then }\, \text{NotP} :=  \left\{\left(y\alpha_{j}, \frac{\alpha_l}{x}\right)\right\}.
\end{matrix} \qquad
&
\begin{matrix}
\text{Case $j$ 3)}\\
\text{If } (j<l) \text{ then }\, \text{NotP} :=  \left\{\left(\alpha_{j-1}, \frac{\alpha_l}{y^{p'_{j-1}}}\right)\right\}\\
\text{If } (j>l) \text{ then }\, \text{NotP} :=  \left\{\left(\alpha_{j+1}, \frac{\alpha_l}{x^{q'_{j+1}}}\right)\right\}.
\end{matrix}
\end{matrix}
\]

\begin{center}
\textbf{case b)} \\
$\alpha'_l$ is immediately above or to the right of $\alpha_j$. 
\end{center}
\begin{minipage}{0.45\textwidth}
\begin{center}
If $\alpha'_l= y\alpha_j$ we define 
\end{center}
\[
\text{NotP}:= \left\{(\alpha'_i, \alpha_j)\left\vert i< j , \text{ and } p'_i=1\right.\right\}. 
\]
\end{minipage}
\begin{minipage}{0.45\textwidth}
\begin{center}
If $\alpha'_l= x\alpha_j$ we define 
\end{center}
\[
\text{NotP}:= \left\{(\alpha'_i, \alpha_j)\left\vert i> j , \text{ and } q'_i=1 \right.\right\}. 
\]
\end{minipage}
\end{definition}
\begin{observation}
In case a), i.e. $\alpha_l$ not immediately above or on the right of $\alpha_j$, there is always only a single element in $\text{NotP}$. 
\end{observation}
\begin{example}
In the example below we depict four possible cases of nested triples of monomial ideals, and the corresponding $\text{NotP}\in  \text{Hom}_R(I_2, \bigslant{R}{I_2})$, which is just a singleton. As usual we represent an $R$-homomorphism with a couple of corresponding symbols, the one outside $\Gamma_2$ represents the generator of $I_2$, the one inside represents its image. 
\begin{center}
\ytableausetup{boxsize=1.5em, aligntableaux=bottom}
\begin{ytableau}  
   *(lightgray)& \none[\bullet] &  \none& \none &  \none & \none & \none& \none & \none& \none& \none & \none& \none& \none & \none& \none& \none & \none\\
   *(lightgray)& *(green)\alpha_j &  \none& \none &  \none & \none & \none& \none & \none& \none& \none & \none& \none& \none & \none& \none& \none & \none\\
 \none &  *(lightgray)&  \none[\star]& \none&  \none & \none & \none& \none & \none& \none& \none & \none& \none& \none & \none& \none& \none & \none\\
  \none&  \none& *(lightgray)&  *(green)\alpha_j&  \none & \none & \none& \none & \none& \none& \none & \none& \none& \none & \none& \none& \none & \none\\
  \none&  \none&  \none&  *(lightgray)&  *(lightgray)&  \none & \none& \none & \none& \none& \none & \none& \none& \none & \none& \none& \none & \none\\
  \none&  \none&  \none&  \none& *(lightgray) &  *(lightgray) & \none& \none & \none& \none& \none & \none& \none& \none & \none& \none& \none & \none\\
  \none&  \none &  \none &  \none& *(lightgray) \star &  *(lightgray) \bullet &*(cyan)\alpha_l& \none & \none& \none& \none & \none& \none& \none & \none &   \none&  \none&  \none\\
  \none&  \none&  \none&  \none&  \none&  \none &  *(lightgray) \diamond &  \none&  \none & \none& \none& \none & \none& \none& \none & \none& \none& \none & \none \\
  \none&  \none &  \none &  \none&\none&\none& *(lightgray) \bigtriangleup &   *(lightgray) &  \none& \none & \none& \none& \none & \none& \none& \none & \none\\ 
  \none&  \none &  \none &  \none&\none&\none&  \none &\none &   *(lightgray) &  *(lightgray)& *(green)\alpha_j  & \none& \none& \none & \none& \none& \none & \none\\ 
  \none&  \none &  \none &  \none&\none&\none&  \none &\none &   \none &  \none& *(lightgray) & \none[\bigtriangleup]& \none& \none & \none& \none& \none & \none\\ 
 \none&  \none &  \none &  \none&\none&\none&  \none &\none &   \none &  \none& *(lightgray) &*(lightgray) & \none& \none & \none& \none& \none & \none\\ 
  \none&  \none &  \none &  \none&\none&\none&  \none &\none &   \none &  \none&\none& *(lightgray) & \none& \none& \none & \none& \none& \none & \none\\ 
  \none&  \none &  \none &  \none&\none&\none&  \none &\none &   \none &  \none&\none& *(lightgray) & *(green)\alpha_j & \none[\diamond]& \none & \none& \none& \none & \none\\ 
 \none&  \none &  \none &  \none&\none&\none&  \none &\none &   \none &  \none& \none&\none& *(lightgray) & *(lightgray)& \none& \none & \none& \none& \none & \none\\ 
\end{ytableau}
\end{center}
\end{example}

\begin{lemma}[definition]
Let $(\alpha', \beta') \in \text{Obs}(I_2, I_3)$ be such that $(\alpha', \beta')\notin \text{NotP}$. Then there exists $\ff\in B(I_1)$ such that $(\ff, f_{\alpha', \beta'})\in \text{Ker}(\phi_{12}-\psi_{12})$ and $(\alpha', \beta') = (x^ay^b\alpha, x^ay^b \beta)$ with $(a, b) = (0,0), (0,1)$ or $(1, 0)$. We define $\ff := \text{Prec}(f_{\alpha', \beta'}) $ and 
\[
\text{PObs}(I_1, I_2, I_3)=\text{PObs}=\left\{(\alpha, \beta) \left\vert\, \ff = \text{Prec}(f_{\alpha', \beta'}) \text{ for } (\alpha', \beta') \in \text{Obs}(I_2, I_3)\setminus \text{NotP} \right.\right\}
\]
\end{lemma}
\begin{proof}
It is enough to notice that, under the hypothesis on $(\alpha', \beta')$, either $\alpha$ is a minimal generator of $I_1$ (or one of $\frac{\alpha}{x^ay^b}$ with $(a, b) =(0,1), (1, 0)$ is), and $\beta \in P_{\alpha} \cup Q_{\alpha}$ (or $\frac{\beta}{x^ay^b} $ is in $ P_{\frac{\alpha}{x^ay^b}} \cup Q_{\frac{\alpha}{x^ay^b}}$). 
\end{proof}

\begin{lemma}
\label{inKer}
Let $(f,h,g) \in \text{Hom}_R(I_1, \bigslant{R}{I_1})\oplus\text{Hom}_R(I_2, \bigslant{R}{I_2})\oplus\text{Hom}_R(I_3, \bigslant{R}{I_3})$ be such that $(f,h) \in \text{Ker}(\phi_{12}-\psi_{12})$ and $(h,g) \in \text{Ker}(\phi_{23}-\psi_{23})$, then $(f,g) \in \text{Ker}(\phi_{13}-\psi_{13})$, so that $(f,h,g) \in \bigcap_{1\leq i<j\leq 3} \left(\text{Ker}(\phi_{ij}-\psi_{ij})\circ \pi_{ij}\right)$. 
\end{lemma}
\begin{proof}
It is enough to check that 
\[
(\left. f \right\vert_{I_3} - \left. h\right\vert_{I_3}) + (\left. h\right\vert_{I_3} - g) = 0 \quad \text{ mod } I_1
\]
since the two terms in parenthesis are $0$: the first directly by hypothesis, the second because by hypothesis is $0$ mod $I_2$ and $I_1\supset I_2$.
\end{proof}

\begin{lemma}[definition]
\label{suivsuiv}
Let $(\alpha, \beta)$ with $\ff \in B(I_1)$ and $(\alpha, \beta) \notin \text{Obs}(I_1, I_2)\cup \text{PObs} $. Then we can define 
\[
\left(\ff, \Su(\ff), \Su(\Su(\ff))\right) \in \bigcap_{1\leq i<j\leq 3} \left(\text{Ker}(\phi_{ij}-\psi_{ij})\circ \pi_{ij}\right)\, .
\]
\end{lemma}
\begin{proof}
The proof is completely analogous to the proof of Lemma \ref{suiv}. 
\end{proof}

Now that we have dealt with vectors of type $\left(\ff, \Su(\ff), \Su(\Su(\ff))\right)$, we will deal with vectors of type $\left(0, (\alpha \mapsto \alpha_j), \Su(\alpha \mapsto \alpha_j)\right)$. The situation is easier: if we are in case a) of \ref{casesjl} $\Su(\alpha \mapsto \alpha_j)$ always exists, and the definition is as below. In case b) we still need some care as it might happen that $(\alpha'_i, \alpha_j) \in \text{Obs}(I_2, I_3)$. 

\begin{definition} Let $(\alpha'_0, \dots, \alpha'_{s'})$ be the list of standard monomial generators of $I_2$. 

\begin{center}
\textbf{case a)} \\
\begin{small}$\alpha_l$ is not immediately above or to the right of $\alpha_j$. \end{small}
\end{center}
Let $h_{\alpha'_i, \alpha_j} \in B(I_2)$ be the vector already defined in (\ref{definitionh}), i.e. the map that takes $\alpha'_i \in I_2$ to $\alpha_j \in \Gamma_2$ and sends all other generators of $I_2$ to zero. Then we define $\Su(h_{\alpha'_i, \alpha_j})\in \text{Hom}_R(I_3, \bigslant{R}{I_3})$ by specifying the image of each generator of $I_3$ as:
\[
\Su(h_{\alpha'_i, \alpha_j})(\alpha''_k) = \begin{cases}h_{\alpha'_i, \alpha_j}(\alpha''_k) &\text{ if } \alpha''_k= x^ay^b\alpha'_i, \text{ with } a,b\geq 0, \\
0 &\text{ otherwise}. 
 \end{cases}
\]
This works for all $\alpha'_i$ exactly because $(\alpha'_i, \alpha_j) \notin \text{Obs}(I_2, I_3)$. The $h_{\alpha'_i, \alpha_j}$  we just defined are clearly linearly independent. 
\begin{center}
\textbf{case b)} \\
\begin{small}$\alpha'_l$ is immediately above or to the right of $\alpha_j$. \end{small}
\end{center}
In this case we define $h_{\alpha'_i, \alpha_j} :=f_{\alpha'_i, \alpha_j} \in B(I_2)$, i.e. through the usual Definition \ref{definitionf}. Then, if $(\alpha'_i, \alpha_j) \notin \text{Obs}(I_2, I_3)$,  we also define $\Su(h_{\alpha'_i, \alpha_j})$ with the usual definition, i.e. through Lemma \ref{suiv}. 
\end{definition}

\begin{definition}
 Let $(\alpha''_0, \dots, \alpha''_{s''})$ be the list of standard monomial generators of $I_3$. 
 We define $h_{\alpha''_i, \alpha'_l}\in \text{Hom}_R(I_3, \bigslant{R}{I_3})$ as 
\[
h_{\alpha''_i, \alpha'_l}(\alpha''_k) = \begin{cases}\alpha'_l &\text{ if } \alpha''_k = \alpha''_i, \\
0 &\text{ otherwise. }\end{cases}
\]
We have $s''+1=s_3+1$ of these and they are linearly independents. 
\end{definition}

\begin{definition}($B(I_1, I_2, I_3)$)
\label{definitionB(I1I2I3)}
Let $(I_1,I_2,I_3)$ be a fixed point of $\Hi^{n, n+1, n+2}(\CC^2)$ with $\alpha_j$ the only monomial in $I_1$ but not in $I_2$ and $\alpha'_l$ the only monomial in $I_2$ but not in $I_3$. Then we define the subset
\[B(I_1,I_2,I_3) \subset \text{Hom}_R(I_1, \bigslant{R}{I_1})\oplus\text{Hom}_R(I_2, \bigslant{R}{I_2})\oplus\text{Hom}_R(I_3, \bigslant{R}{I_3})\]
as
\begin{align}
\label{base3}
B(I_1,I_2,I_3) := &\left\{ (0,0,h_{\alpha''_i, \alpha'_l})\,\vert \,\,i=0, \dots, s'' \right\}\,\cup  \nonumber\\
&\cup\, \left\{\left(0,h_{\alpha'_i, \alpha_j}, \Su(h_{\alpha'_i, \alpha_j})\right)\,\,\vert\,\, (\alpha'_i, \alpha_j)\notin \text{Obs}(I_2, I_3) \right\} \,\cup  \\
&\cup \left\{\left(\ff, \Su(\ff), \Su(\Su(\ff))\right)\,\,\left\vert \,\, (\alpha,\beta) \notin \text{Obs}(I_1,I_2)\cup\text{PObs} \right.\right\}\nonumber
\end{align}
\end{definition}
\begin{observation}
\label{cardinalityB(123)}
In case a), i.e. $\alpha_l$ is not immediately above or to the left of $\alpha_j$, we have that:
\[
 \#\, B(I_1, I_2, I_3) = 2n+5.
\]
To see this, start observing that we have $s''+1$ elements in the first set. Then $(\alpha'_i, \alpha_j) \notin \text{Obs}(I_2, I_3)$ for all $i=0,\dots,s'$, so that we have $s'+1$ elements in the second set.
\begin{align*}
 \#\,\left\{ (0,0,h_{\alpha''_i \alpha_l})\,\vert \,\,i=0, \dots, s'' \right\} \,=\, s''+1 \\
 \#\, \left\{\left(0,h_{\alpha'_i, \alpha_j}, \Su(h_{\alpha'_i, \alpha_j})\right)\,\,\vert\,\, i=0,\dots, s' \right\}\,=\, s'+1
\end{align*}

Moreover $\text{Obs}(I_1, I_2)$ and $\text{PObs}$ are distinct sets of cardinality $s'+1-2$ and $s''+1-3$ respectively. So
\[
 \#\,\left\{\left(\ff, \Su(\ff), \Su(\Su(\ff))\right)\,\,\left\vert \,\, (\alpha\beta) \notin \text{Obs}(I_1,I_2)\cup\text{PObs} \right.\right\}=  \#\, B(I_1) - (s'+1-2) - (s''+1-3).
\]
Finally remembering $ \#\, B(I_1)=2n$ and putting together the results we have $ \#\, B(I_1,I_2,I_3) = 2n+5$.
\end{observation}
\begin{observation}
In case b), i.e. $\alpha'_l$ is immediately above or to the left of $\alpha_j$, we have that:
\[
 \#\, B(I_1, I_2, I_3) = 2n+4.
\]
Suppose, for example that $\alpha'_l=y\alpha_j$. Then we have $s''+1$ elements in the first set. Moreover $(\alpha'_i, \alpha_j) \notin \text{Obs}(I_2, I_3)$ if and only if $i<j$ and $p'_i=1$ i.e. if and only if $(\alpha'_i, \alpha_j) \in \text{NotP}$, so that we have $s'+1-  \#\, \text{NotP}$ elements in the second set. 
\begin{align*}
 \#\,\left\{ (0,0,h_{\alpha''_i \alpha_l})\,\vert \,\,i=0, \dots, s'' \right\} \,=\, s''+1 \\
 \#\, \left\{\left(0,h_{\alpha'_i, \alpha_j}, \Su(h_{\alpha'_i, \alpha_j})\right)\,\,\vert\,\,(\alpha'_i\alpha_j) \notin \text{Obs}(I_2,I_3)  \right\}\,=\, s'+1- \#\, \text{NotP}
\end{align*}
Finally $\text{Obs}(I_1, I_2)$ and $\text{PObs}$ are distinct sets of cardinality $s'+1-2$ and $s''+1-2- \#\,\text{NotP}$ respectively. So that 
\[
 \#\,\left\{\left(\ff, \Su(\ff), \Su(\Su(\ff))\right)\,\,\left\vert \,\, (\alpha\beta) \notin \text{Obs}(I_1,I_2)\cup\text{PObs} \right.\right\}=  \#\, B(I_1) - (s'+1-2) - (s''+1-3).
\]
Again $ \#\, B(I_1)=2n$, and putting together the results we have $ \#\, B(I_1,I_2,I_3) = 2n+4$.

\end{observation}

\begin{lemma}
\label{B(123)base}
Let $(I_1,I_2,I_3)$ be a fixed point of $\Hi^{n, n+1, n+2}(\CC^2)$. The set $B(I_1,I_2,I_3)$ as defined in \ref{definitionB(I1I2I3)} is a weight basis for $T_{I_1,I_2,I_3} \Hi^{n,n+1,n+2}(\CC^2)$
\end{lemma}
\begin{proof}

We prove first that they are all linearly independent. Let
\begin{align*}
\sum_{i=0}^{s''} c''_i\,\,\left(0,0,h_{\alpha''_i, \alpha'_l}\right) \,
+\, &\sum_{i=0}^{s'}c'_i\,\,\left(0,h_{\alpha'_i, \alpha_j}, \Su(h_{\alpha'_i, \alpha_j})\right) \,+\\+\, &\sum_{(\alpha, \beta)\notin \text{Obs}(I_1, I_2)\cup \text{Pobs}} c_{\alpha, \beta}\,\, \left(\ff, \Su(\ff), \Su(\Su(\ff))\right) = (0,0,0).
\end{align*}
Then we have $\sum_{\alpha, \beta} c_{\alpha, \beta}\,\, \left(\ff \right)=0 \in \text{Hom}_R(I_1, \bigslant{R}{I_1})
$ that implies $c_{\alpha, \beta} = 0$ for all $(\alpha, \beta)$. Then also $\sum_{i=0}^{s'} c'_i\,\,\left(h_{\alpha'_i, \alpha_j}\right)\in\text{Hom}_R(I_2, \bigslant{R}{I_2})$, that in turns implies $c'_i=0$ for all $i=0,\dots,s'$; and finally $\sum_{i=0}^{s''} c''_i\,\,\left(h_{\alpha''_i, \alpha'_l}\right)=0$ implies $c''_i=0$ for all $i=0,\dots,s''$. \\

\hspace{3em}We prove that they actually generate $T_{I_1,I_2,I_3} \Hi^{n,n+1,n+2}(\CC^2)$. As we know, we can generate the tangent space with elements of pure weight, thus it is enough to show that all such elements are in the span of $B(I_1,I_2,I_3)$. Let $\tau=(f,h,g)$ be a tangent vector of pure weight. We will prove that it belongs to $\text{span}(B(I_1,I_2,I_3))$ by induction on $n(f)$ the number of generators of $I_1$ that $f$ does not send to $0$. \\

\hspace{3em}If $n(f)=0$, then $f=0$, and $\tau= (0,h,g)$. Since in particular $(0,h) \in \text{Ker}(\phi_{12}-\psi_{12})$ we have that $h(\alpha'_i)\in \langle \alpha_j \rangle$ for all $i=0,\dots,s'$. Then we argue by induction on $ n(h)$, the number of generators of $I_2$ that $h$ does not send to $0$. 

\hspace{3em}Either $n(h)=0$, i.e. $h=0$ and, reasoning in the same way, $g \in \langle h_{\alpha''_i \alpha_l}\vert i=0,\dots, s''\rangle \subset T_{I_3} \Hi^{n+2}(\CC^2)$, or $n(h)=a$. 

\hspace{3em}If we are in case a) of cases \ref{casesjl}, we have immediately that $h \in \langle h_{\alpha'_i, \alpha_j}\vert i=0,\dots, s'\rangle \subset T_{I_2} \Hi^{n+1}(\CC^2)$ and we can subtract 
\[
(h-h_{\alpha'_k \alpha_j}, g-\Su(h_{\alpha'_k \alpha_j})) \text{ with } h(\alpha'_k) \neq 0
\] 
to be able to use the induction step. 

\hspace{3em}If we are in case b), we need to be slightly more careful and use the same reasoning as the induction step in Lemma \ref{B(I)basis}. In particular if we suppose, say, $\alpha_j \in P'_{\alpha'_i}$ and $\bar{\imath}:= max_k\{k \vert h(\alpha'_k)=0\}$, we see that $(\alpha'_{\bar{\imath}}, \alpha_j) \notin \text{Obs}(I_2, I_3)$ and we subtract 
\[
(h-h_{\alpha'_{\bar{\imath}} \alpha_j}, g-\Su(h_{\alpha'_k \alpha_j}))
\]
to be able to use the induction step. If $\alpha_j \in Q'_{\alpha'_i}$ the reasoning is similar. \\

\hspace{3em}Suppose now the statement true for all $\tau'= (f',h',g')$ with $n(f')<t$, and suppose that $\tau=(f,h,g)$ has $n(f)=t$. Since $f$ is of pure weight we know that either:
\begin{itemize}
\item[1)] we have $\bar{\imath}:= \max \{i\vert f(\alpha_i) \neq 0\}$ is strictly smaller than $s$ and $f(\alpha_{\bar{\imath}})\in \langle \beta \rangle$ with $\beta \in P_{\alpha_{\bar{\imath}}}$, or 
\item[2)] we have $\bar{\imath}:= \min \{i\vert f(\alpha_i) \neq 0\}$ is strictly bigger than $0$ and $f(\alpha_{\bar{\imath}})\in \langle \beta \rangle$ with $\beta \in Q_{\alpha_{\bar{\imath}}}$. 
\end{itemize}
Suppose we are in the first case. Renormalize $\tau=(f,h,g)$ so that $f(\alpha_{\bar{\imath}})=\beta$.

First of all notice that $(\alpha_{\bar{\imath}}, \beta)\notin \text{Obs}(I_1,I_2)\cup \text{PObs}$. In fact if $(\alpha_{\bar{\imath}}, \beta)\in \text{Obs}(I_1,I_2)$  we would be in the situation where
\[
\bar{\imath}\leq j\qquad \text{ and } \beta =\frac{\alpha_j}{y^{p_{\bar{\imath}}}}
\]
with $f(\alpha_i)=0$ for all $i>\bar{\imath}$. But all $\theta \in \text{Hom}_R(I_2, \bigslant{R}{I_2})$ such that $\theta(\alpha_{\bar{\imath}})= \frac{\alpha_j}{y^{p_{\bar{\imath}}}}$ has $\theta(\alpha_{\bar{\imath}+1})\neq 0 $, if not 
\[
y^{p_{\bar{\imath}}}\alpha_{\bar{\imath}} =x^{q_{\bar{\imath}+1}} \alpha_{\bar{\imath}+1} \implies \alpha_j=y^{p_{\bar{\imath}}}\frac{\alpha_j}{y^{p_{\bar{\imath}}}}=\theta(y^{p_{\bar{\imath}}}\alpha_{\bar{\imath}})=\theta(x^{q_{\bar{\imath}+1}}\alpha_{\bar{\imath}+1})=0 
\]
that mod $I_2$ is absurd, since $\alpha_j \notin I_2$.  Then, indeed, $(\alpha_{\bar{\imath}}, \beta)\notin \text{Obs}(I_1, I_2)$. Suppose, now, $(\alpha_{\bar{\imath}}, \beta)\in \text{PObs}$. This can happen only in the case a). Moreover, by definition, it means that 
\[
\bar{\imath}\leq l\qquad \text{ and }\quad  \beta =\frac{\alpha_j}{y^{p'_{\bar{\imath}}}}.
\]
Then, since $(f,h) \in \text{Ker}(\phi_{12}-\psi_{12})$
\[
h(\alpha_{\bar{\imath}})=\beta, \text{ and } h(\alpha'_{i})\in\langle \alpha_j \rangle \text{ for all } i\geq \bar{\imath}.
\]
So that we can suppose that $h(\alpha'_{i})=0 \text{ for all } i\geq \bar{\imath}$ by changing $\tau=(f,h,g)$ with appropriate combinations of $(0,h_{\alpha'_i, \alpha_j}, \Su(h_{\alpha'_i, \alpha_j})) \in B(I_1,I_2,I_3)$. But then, exactly as before, it is impossible that there exists a $g\in \text{Hom}_R(I_3, \bigslant{R}{I_3})$ such that $(h,g)\in \text{Ker}(\phi_{23}-\psi_{23})$. So, in fact, $(\alpha_{\bar{\imath}}, \beta)\notin \text{Obs}(I_1,I_2)\cup \text{PObs}$.

Then, always thanks to the definition of $\bar{\imath}$, we have that $y^{p_{\bar{\imath}}}\beta \in I_1$, and thanks to the fact that $(\alpha_{\bar{\imath}}, \beta)\notin \text{Obs}(I_1,I_2)\cup \text{PObs}$, we have that 
\[
\left(f-f_{\alpha_{\bar{\imath}}, \beta}, h-\Su(f_{\alpha_{\bar{\imath}}, \beta}),g - \Su(\Su(f_{\alpha_{\bar{\imath}}, \beta}))\right) \in \bigcap_{1\leq i<j\leq 3} \left(\text{Ker}(\phi_{ij}-\psi_{ij})\circ \pi_{ij}\right)
\]
is well defined and has first coordinate that sends strictly more generators to $0$ mod $I_1$, i.e. $n(f-f_{\alpha_{\bar{\imath}}})<t$. So that by induction we have that $\tau \in \text{span}(B(I_1,I_2,I_3))$ as desired. \\

\hspace{3em}The case 2) is completely analogous. 
\end{proof}
\begin{lemma}[Cheah]
The spaces $\Hi^{n,n+1,n+2}(\CC^2)$ are singular for all $n\geq 1$.
\end{lemma}
\begin{proof}
We always have a fixed point of type described in case $a)$ of Definition \ref{casesjl}. For that fixed point Observation \ref{cardinalityB(123)}  and Lemma \ref{B(123)base} tell us that the tangent space is $2n+5$ dimensional, whereas we know that the dimension of the space is $2n+4$. 
\end{proof}

\subsection{The tangent to the Hilbert-Samuel's strata, case $\mathfrak{n}= (n, n+1, n+2)$}

\hspace{3em}We can now study the weights of the elements of the basis $B(I_1, I_2, I_3)$ in order to understand the tangent space at $(I_1, I_2, I_3)$ of $M_{T(I_1), T(I_2), T(I_3)}$ and of  $G_{T(I_1), T(I_2), T(I_3)}$, and their respective positive parts. We divide this task in three observations, one for each type of vectors in the basis $B(I_1, I_2, I_3)$. 
\begin{observation}
\label{weight(0,0,h)}
Let us look first at the first kind of elements in $B(I_1,I_2,I_3)$ i.e. 
\[
\left\{ (0,0,h_{\alpha''_i, \alpha'_l})\,\vert \,\,i=0, \dots, s'' \right\}.
\]
For these we only need to compare the degree of a generator of $I_3$ with that of $\alpha'_l$ the only  monomial in $I_2$ but not in $I_3$. More precisely, we have that $(0,0,h_{\alpha''_i, \alpha'_l}) \in B(I_1, I_2, I_3)$ is tangent to $M_{T(I_1), T(I_2), T(I_3)}$ if and only if $\deg \alpha''_i\leq\deg \alpha'_l$; tangent to $G_{T(I_1), T(I_2), T(_3)}$ if and only if $\deg \alpha''_i= \deg \alpha'_l$; tangent to  $M_{T(I_1), T(I_2), T(I_3)}$ and with positive weights if and only if $\deg \alpha''_i\leq\deg \alpha'_l$ and $\deg_y \alpha''_i<\deg_y \alpha'_l$; tangent to  $G_{T(I_1), T(I_2), T(I_3)}$ and with positive weight if and only if $\deg \alpha''_i=\deg \alpha'_l$ and $\deg_y \alpha''_i<\deg_y \alpha'_l$. So everything is perfectly analogous to the case described in Observation \ref{weight(0,h)} and its graphic interpretation. 

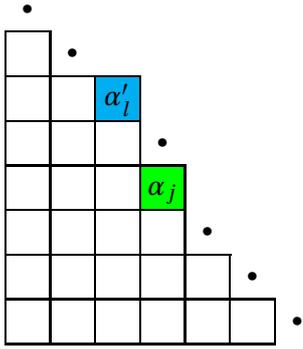
\begin{minipage}{0.4\textwidth}
\begin{center}
\ytableausetup{boxsize=1.5em, aligntableaux=bottom}
\,\,
\begin{ytableau} 
\none[\bullet]\\
\,&\none[\bullet] \\
\,&\,&*(cyan) \alpha'_l  \\
\,&\, &\,&\none[\bullet] \\
\,&\, &\,&*(green)\alpha_j  \\
\,&\, &\,&\,&\none[\bullet] \\
\,&\, &\,&\,&\,&\none[\bullet] \\
\,&\, &\, &\,&\, &\,&\none[\bullet]
\end{ytableau}

\begin{small}The bullets are generators of $I_3$. \end{small}
\end{center}
\end{minipage}
\begin{minipage}{0.58\textwidth}
The first set of elements of $B(I_1, I_2, I_3)$ contains as many vectors tangent to $M_{T(I_1), T(I_2), T(I_3)}$ as bullets on the same anti-diagonal of $\alpha'_l$ or on lower anti-diagonals. They are positive if they are on a strictly lower anti-diagonal or are on the right of $\alpha'_l$. Similarly for  $G_{T(I_1), T(I_2), T(I_3)}$. The relative position of $\alpha'_l$ and $\alpha_j$ does not matter here.  
\end{minipage}

\end{observation}
\begin{observation}
\label{weight(0,h,h)}
Let us look now at the second kind of elements in $B(I_1,I_2,I_3)$ i.e. 
\[
\left\{\left(0,h_{\alpha'_i, \alpha_j}, \Su(h_{\alpha'_i, \alpha_j})\right)\,\,\vert\,\, i=0,\dots, s',\text{ and } (\alpha'_i,  \alpha_j) \notin \text{Obs}(I_2, I_3) \right\}.
\]
We need to distinguish cases of Definition \ref{casesjl}. In case a), i.e. $\alpha_l$ is not immediately to the left or above $\alpha_j$ we have that $(\alpha'_i, \alpha_j)\notin \text{Obs}(I_2,I_3)$, for all $i=0,\dots,s'$. Again  $(0, h, \Su(h))$ has the same weight as $h$, so that  we can study the weight of the triple only by comparing the degree of $\alpha'_i$, a generator of $I_2$, with that of $\alpha_j$ the only monomial in $I_1$ but not in $I_2$. More precisely, we have that $\left(0,h_{\alpha'_i, \alpha_j}, \Su(h_{\alpha'_i, \alpha_j})\right) \in B(I_1, I_2, I_3)$ is tangent to $M_{T(I_1), T(I_2), T(I_3)}$ if and only if $\deg \alpha'_i\leq\deg \alpha_j$; tangent to $G_{T(I_1), T(I_2), T(_3)}$ if and only if $\deg \alpha'_i= \deg \alpha_j$; tangent to  $M_{T(I_1), T(I_2), T(I_3)}$ and with positive weights if and only if $\deg \alpha'_i\leq\deg \alpha_j$ and $\deg_y \alpha'_i<\deg_y \alpha_j$; tangent to  $G_{T(I_1), T(I_2), T(I_3)}$ and with positive weights if and only if $\deg \alpha'_i=\deg \alpha_j$ and $\deg_y \alpha'_i<\deg_y \alpha_j$. Graphically: 

\begin{center}\textbf{case a)}\end{center}
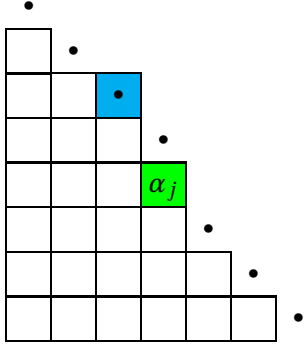
\begin{minipage}{0.4\textwidth}
\begin{center}
\ytableausetup{boxsize=1.5em, aligntableaux=bottom}
\,\,
\begin{ytableau} 
\none[\bullet]\\
\,&\none[\bullet] \\
\,&\,&*(cyan)\bullet \\
\,&\, &\,&\none[\bullet] \\
\,&\, &\,&*(green)\alpha_j  \\
\,&\, &\,&\,&\none[\bullet] \\
\,&\, &\,&\,&\,&\none[\bullet] \\
\,&\, &\, &\,&\, &\,&\none[\bullet]
\end{ytableau}

\begin{small}The bullets are the generators of $I_2$.\end{small}
\end{center}
\end{minipage}
\begin{minipage}{0.58\textwidth}
Here the position of $\alpha_l$ does not matter, as long as it is not immediately above or to the left of $\alpha_j$. The second set of elements of $B(I_1, I_2, I_3)$ contains as many vectors tangent to $M_{T(I_1), T(I_2), T(I_3)}$ as bullets on the same anti-diagonal of $\alpha_j$ or on lower anti-diagonals. They are positive if they are on a strictly lower anti-diagonal or are on the right of $\alpha_j$. Similarly for  $G_{T(I_1), T(I_2), T(I_3)}$. \end{minipage}

\vspace{1em}
In case b) i.e. if $\alpha'_l= y\alpha_j$ or $\alpha'_l= x\alpha_j$ we have, instead, that some of $(\alpha'_i, \alpha_j)$ are actually in $\text{Obs}(I_1,I_2)$. Precisely these are $\left\{(\alpha_i, \alpha_j)\vert i<j \text{ and } p'_i=1\right\}$ if   $\alpha'_l= y\alpha_j$, (resp. $\{(\alpha_i, \alpha_j)\vert i>j \text{ and } q'_i=1\}$ if   $\alpha'_l= x\alpha_j$). 

\hspace{3em}However instead of not counting them we do like this: 
we count them, and we consider their weight as usual by comparing $\deg \alpha'_i$ with $\deg \alpha_j$ if $i<j$ (resp. $\deg \alpha'_i$ with $\deg \alpha_j$ if $i>j$), but then we \emph{subtract their contribution} considering their weight by comparing $\deg y\alpha'_i$ with $\deg y\alpha_j=\alpha'_l$ if $i<j$ (resp. $\deg x\alpha'_i$ with $\deg x\alpha_j=\alpha'_l$ if $i>j$). Graphically:

\begin{center}\textbf{case b)}\end{center}
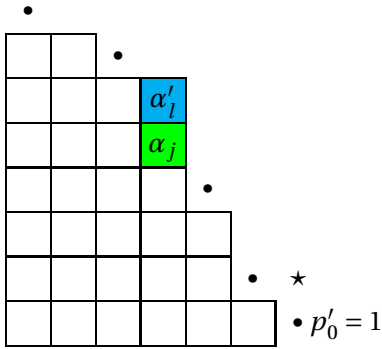
\begin{minipage}{0.4\textwidth}
\begin{center}
\ytableausetup{boxsize=1.5em, aligntableaux=bottom}
\,\,
\begin{ytableau} 
\none[\bullet]\\
\,&\,&\none[\bullet] \\
\,&\, &\,&*(cyan)\alpha'_l \\
\,&\, &\,&*(green)\alpha_j  \\
\,&\, &\,&\,&\none[\bullet] \\
\,&\, &\,&\,&\, \\
\,&\, &\,&\,&\,&\none[\bullet] &\none[\star]\\
\,&\, &\, &\,&\, &\,&\none[\bullet] &\none[\,\,p'_0=1]
\end{ytableau}

\begin{small}The bullets are some generators of $I_2$.\end{small}
\end{center}
\end{minipage}
\begin{minipage}{0.58\textwidth}
Here $\alpha'_l=y\alpha_j$. The second set of elements of $B(I_1, I_2, I_3)$ contains as many vectors tangent to $M_{T(I_1), T(I_2), T(I_3)}$ as bullets on the same anti-diagonal of $\alpha_j$ or on lower anti-diagonals. They are positive if they are on a strictly lower anti-diagonal or are on the right of $\alpha_j$. But we should have not counted those on the right of $\alpha_j$ with $p'_i=1$ (in this case only $\alpha'_0$). Then we subtract all the $\star$ for these, according to their weight that we calculate by comparing the degree of the star with the degree of $\alpha'_l=y\alpha_j$. Similarly for $G_{T(I_1), T(I_2), T(I_3)}$.
 \end{minipage}

\end{observation}
\vspace{0.8em}
\begin{observation}
\label{weight(f,f,f)}
Finally let us look at the third type of elements of $B(I_1,I_2,I_3)$, those of the form 
\[
\left(\ff, \Su(\ff), \Su(\Su(\ff))\right) \,\text{ with }\, (\alpha\beta) \notin \text{Obs}(I_1,I_2)\cup\text{PObs}.\]
\hspace{3em}Again we need only to consider the weight of $\ff$. In particular $(\ff, \Su(\ff),$ $ \Su(\Su(\ff))) \in B(I_1, I_2, I_3)$ is tangent to $M_{T(I_1), T(I_2), T(I_3)}$ if and only if $\ff$ is tangent to $M_{T(I_1)}$ i.e. if and only if $\deg \alpha \leq \deg \beta$. Similarly for the positive part, and for the tangent to $G_{T(I_1), T(I_2), T(I_3)}$ and its positive part. 

\hspace{3em}In fact we need to record when $\ff$ with $(\alpha, \beta)\in \text{Obs}(I_1,I_2) \cup \text{PObs}$ \emph{is} raising or preserving degree, and so on, and then understand the dimension of the appropriate vector spaces by \emph{difference}. 

Suppose first $(\alpha, \beta)\in \text{Obs}(I_1,I_2)$, then everything goes as in Observation \ref{weight(f,f)}, and for each $\alpha_i$ generator of $I_1$ there is at most one $\beta$ such that $(\alpha_i, \beta)\in \text{Obs}(I_1,I_2)$ and the relevant checks on degree are between either $y^{p_i}\alpha_i$ and $\alpha_j$ or between $x^{q_i}\alpha_i$ and $\alpha_j$, so that 
\begin{equation}
f_{\alpha_i, \beta} \in T_{I_1} M_{T(I_1)} \text{ but } (\alpha_i, \beta)\in \text{Obs}(I_1,I_2) \iff 
\begin{cases}
 \deg y^{p_i}\alpha_i \leq\deg \alpha_j& \text{ if } i\leq j, \\
 \deg x^{q_i}\alpha_i \leq\deg \alpha_j& \text{ if } i\geq j 
\end{cases}
\end{equation}
and similarly for all the other spaces we are interested in. Graphically: 

\begin{minipage}{0.5\textwidth}
\begin{center}
\ytableausetup{boxsize=1.5em, aligntableaux=bottom}
\,\,
\begin{ytableau} 
\,\\
\,\\
\,&\, &\,\\
\,&\, &\,&*(green)\alpha_j  \\
\,&\, &\,&\, \\
\,&\, &\,&\,&\none[]&\none[\star] \\
\,&\, &\,&\,&\,&\none[] &\none[\star]\\
\,&\, &\, &\,&\, &\,&\none[]
\end{ytableau}

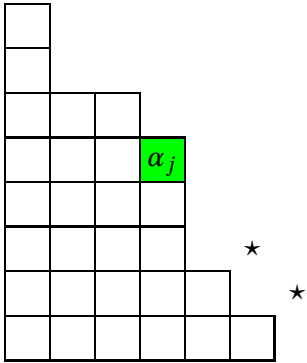
\begin{small}The stars are in correspondence with $\text{Obs}(I_1, I_2)$. \end{small}
\end{center}
\end{minipage}
\begin{minipage}{0.48\textwidth}
Here the position $\alpha_l$ does not matter. We want to understand who are those $(\alpha, \beta)$ such that $\ff \in B(I_1)$ but $(\alpha, \beta) \in \text{Obs}(I_1, I_2)$, and study their weights. Then exactly as in the Observation \ref{weight(f,f)}, we need to subtract a star according to the anti-diagonal it is in relative to the anti-diagonal of $\alpha_j$.   \end{minipage}
 \vspace*{1em}

Suppose now that  $(\alpha, \beta)\in \text{PObs}$, then necessarily $(x^ay^b\alpha,\, x^ay^b\beta)\in \text{Obs}(I_2, I_3)$ with $(a,b)=$ $(0,0), (0,1)$ or $(1,0)$ and $x^ay^b\alpha=\alpha'_i$ a minimal generator of $I_2$. Then we can focus on $\text{Obs}(I_2, I_3)$. For these we can reproduce, completely analogously, the same reasoning we did in the case of $\text{Obs}(I_1, I_2)$, replacing, though, generators of $I_1$ with generators of $I_2$, and $\alpha_j$ with $\alpha_l$. 

\textbf{Case a)}. Remember that there is only one element in $\text{NotP}$ and it corresponds to the index of $\text{Obs}(I_2, I_3)$ that is not relevant, because there is not $\ff \in B(I_1)$ that realizes it as $\Su(\ff)$.  Recall that $\text{NotP}$ is either
\begin{align*}
\text{If } (j<l)\, \text{ then }\, \text{NotP} = 
\begin{cases} 
\left(\alpha_{j-1}, \frac{\alpha_l}{y^{p'_{j-1}}}\right) &\text{ or }\\
\left(x\alpha_{j}, \frac{\alpha_l}{y}\right). &
\end{cases}\\
\text{If } (j>l)\, \text{ then }\,   \text{NotP} = 
\begin{cases} 
\left(y\alpha_{j}, \frac{\alpha_l}{x}\right) &\text{ or }\\
\left(\alpha_{j+1}, \frac{\alpha_l}{x^{q'_{j-1}}}\right). &
\end{cases} 
\end{align*}
Now look first at one case, say $(j<l)$. Then interestingly enough, the relevant test on the degrees is between $y^{p'_{j-1}}\alpha_{j-1}$ and $\alpha_l$ or between $yx\alpha_{j}$ and $\alpha_l$ but in any case $y^{p'_{j-1}}\alpha_{j-1}$ or $yx\alpha_{j}$ is the box one to the left and one above of $\alpha_j$, so that in the graphical visualization of $\ref{starsanddots}$ the relevant check is always the same. The same is true if $(j>l)$. Thus, in all cases, the $Obs(I_2, I_3)$ that we need not to consider, is the one represented by the box one to the right and one above $\alpha_j$. This starts to be relevant only if it is on the same anti-diagonal as $\alpha'_l$ or on a lower one, i.e. \emph{only if $\alpha_j$ is on an anti-diagonal at least two steps lower than the one of $\alpha_l$}. 
Graphically we can summarize as follow:  

\begin{center}\textbf{Case a}) \end{center}
\vspace*{1em}
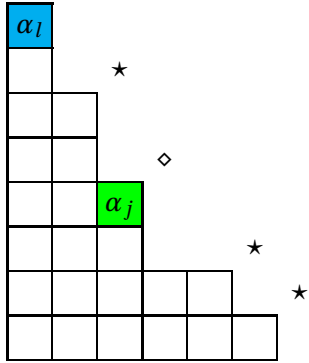
\begin{minipage}{0.42\textwidth}
\begin{center}
\ytableausetup{boxsize=1.5em, aligntableaux=bottom}
\begin{ytableau} 
*(cyan) \alpha_l\\
\,&\none[]&\none[\star]\\
\, &\,\\
\, &\,& \none[\,] &\none[\diamond]\\
\, &\,&*(green)\alpha_j   \\
\, &\,&\,&\none[] & \none[\,] &\none[\star] \\
\,&\, &\,&\,&\,&\none[\,] &\none[\star]\\
\,&\, &\, &\,&\, &\,&\none[]
\end{ytableau}

\begin{small}The stars are in correspondence with the $\text{Obs}(I_2, I_3)$ that matter. \end{small}
\end{center}
\end{minipage}
\begin{minipage}{0.57\textwidth}
We need to subtract those $(\alpha, \beta)$ with $\ff \in B(I_1)$ and the correct weight, such that $(\alpha, \beta) \in \text{PObs}$. For each such, there is an $(\alpha'_i, \beta') \in \text{Obs}(I_2, I_3)$. We can then proceed by studying this $(\alpha'_i, \beta')\in \text{Obs}(I_2, I_3)$ i.e. comparing the associated star with $\alpha_l$. However we need not consider $\text{NotP}$, the corner of $\alpha_j$, that is the box one to the left and one above of $\alpha_j$, is represented as a $\diamond\,$ in the picture.
\end{minipage}
\vspace{2em}

\textbf{Case b)}
Now either $\alpha'_l=y\alpha_j$ or $\alpha'_l=x\alpha_j$. Suppose to be in the first case. When we list the elements $\ff\in B(I_1)$ according to their weight, we have still to exclude all those for which $(\alpha, \beta) \in \text{PObs}$. The $\text{PObs}$ are in one to one correspondence with the $\text{Obs}(I_2, I_2)$ except for $\text{NotP} = \left\{(\alpha'_i, \alpha_j) \left\vert i<j \text{ and } p'_i=1 \right.\right\}.$

\begin{center}\textbf{case b)}\end{center}
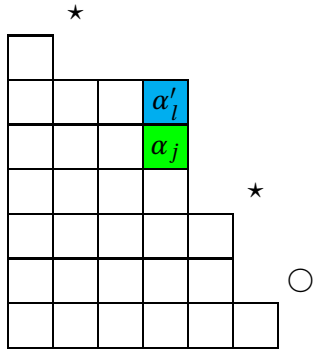
\begin{minipage}{0.4\textwidth}
\begin{center}
\ytableausetup{boxsize=1.5em, aligntableaux=bottom}
\,\,
\begin{ytableau} 
\none[]&\none[\star]\\
\, \\
\,&\, &\,&*(cyan)\alpha'_l \\
\,&\, &\,&*(green)\alpha_j  \\
\,&\, &\,&\,&\none[]&\none[\star] \\
\,&\, &\,&\,&\, \\
\,&\, &\,&\,&\,&\none[] &\none[\bigcirc]\\
\,&\, &\, &\,&\, &\,&\none[]
\end{ytableau}

\begin{small}The stars represent $\text{Obs}(I_2, I_3)\setminus \text{NotP}$\\
The $\bigcirc$ represents $\text{NotP}$. 
\end{small}
\end{center}
\end{minipage}
\begin{minipage}{0.58\textwidth}
Here $\alpha'_l=y\alpha_j$. We need to subtract $\text{Obs}(I_2, I_3)\setminus \text{NotP}$ according to their weights, i.e. according to the relative position of the anti-diagonal they are in and $\alpha'_l$. The elements in $ \text{NotP}$ are the $\bigcirc$ above the $\alpha'_i$ with $p'_i=1$. 
 \end{minipage}

\begin{observation}
\label{countsarethesame}
The elements in $\text{NotP}$, the $\bigcirc$ of the picture above, are exactly those stars we subtracted in the Observation \ref{weight(0,h,h)} \textbf{case b)}. This insures that the final formula is the same along all cases. 
\end{observation}
\end{observation}

We are now ready to write combinatorial formulas for the dimensions of the tangent spaces at $M_{T_1, T_2, T_3}$ $G_{T_1, T_2, T_3}$ and their positive parts.  To state them we need to recall Definition \ref{starsanddots}.  We will apply it to the two couples of Young diagrams $(\Gamma_1, \Gamma_2)$ and $(\Gamma_2, \Gamma_3)$. Each couple is considered independently from the other. The only extra bit of information that remembers that we are actually dealing with a triple of nested Young diagrams is given by the following.
\begin{definition}
$\,$\\
$\,$\\
\begin{minipage}{0.7\textwidth}
Let 
\[
\Gamma_3 = \Gamma_2 \cup \{\alpha_1\}= \Gamma_1\cup\{\alpha_2\}\cup\{\alpha_1\}  \]
be a triple of nested Young diagrams. Then if $\alpha_2$ is \emph{not} immediately above or on the right of $\alpha_1$ put a diamond $\diamond$ in the box immediately above and to the right of $\alpha_1$ i.e. in position $xy\alpha_1$.  
\end{minipage}
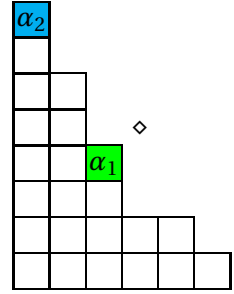
\begin{minipage}{0.3\textwidth}
\begin{center}
\ytableausetup{boxsize=1.2em, aligntableaux=bottom}
\begin{ytableau} 
*(cyan) \alpha_2\\
\,&\none[]\\
\, &\,\\
\, &\,& \none[\,] &\none[\diamond]\\
\, &\,&*(green)\alpha_1   \\
\, &\,&\,&\none[] & \none[\,]  \\
\,&\, &\,&\,&\,&\none[\,] \\
\,&\, &\, &\,&\, &\,&\none[]
\end{ytableau}
\end{center}
\end{minipage}

\vspace{1em}
Then define, in all cases, 
\begin{align*}
&\diamond \, M(\Gamma_1, \Gamma_2, \Gamma_3)\,\, :=\,\,  \#\, \{\diamond \in \Lambda_i \left\vert \Lambda_i < \Lambda_{\alpha_2} \right. \} +   \#\, \{\diamond \in \Lambda_{\alpha_2} \} \\
&\diamond \, G(\Gamma_1, \Gamma_2, \Gamma_3) \,\,:=\,\,   \#\, \{\diamond \in \Lambda_{\alpha_2} \} \\
&\diamond M^+(\Gamma_1, \Gamma_2, \Gamma_3) \,\,:=\,\,   \#\, \{\diamond \in \Lambda_i \left\vert \Lambda_i < \Lambda_{\alpha_2} \right. \} +   \#\, \{\diamond \in \Lambda_{\alpha_2}\left\vert\,\, \diamond \text{ to the right of } \alpha_2 \right. \} \\
&\diamond G^+(\Gamma_1, \Gamma_2, \Gamma_3) \,\,:=\,\,    \#\, \{\diamond \in \Lambda_{\alpha_2}\left\vert\,\, \diamond \text{ to the right of } \alpha_2 \right. \}.
\end{align*}

\end{definition}


\begin{lemma}
\label{dimensiontangent123}
Suppose that $(I_1, I_2, I_3) \in G_{T_1, T_2, T_3} \subset M_{T_1, T_2, T_3}$ is a fixed point of the $\TT_{1^+}$ action on $\Hi^{n, n+1, n+2}(0)$ with generic weights $w_1, w_2$ such that $w_1<w_2$ and $(n+2)w_1 >(n+1)w_2$. Call $(\Gamma_1, \Gamma_2, \Gamma_3)$ the triple of Young diagrams associated to $(I_1, I_2, I_3)$. 
\begin{itemize}
\item[1)] The dimension of the tangent space to $M_{T_1, T_2, T_3}$ at $(I_1, I_2, I_3)$ is equal to 
\[
\dim M_{T_1} + \bullet \, M(\Gamma_1, \Gamma_2)+ \bullet \, M(\Gamma_2, \Gamma_3) -\star \, M(\Gamma_1, \Gamma_2)-\star \, M(\Gamma_2, \Gamma_3) + \diamond \, M(\Gamma_1, \Gamma_2, \Gamma_3).
\]
\item[2)] The dimension of the tangent space to $G_{T_1, T_2, T_3}$ at $(I_1, I_2, I_3)$ is equal to 
\[
\dim G_{T_1} + \bullet \, G(\Gamma_1, \Gamma_2)+ \bullet \, G(\Gamma_2, \Gamma_3) -\star \, G(\Gamma_1, \Gamma_2)-\star \, G(\Gamma_2, \Gamma_3) + \diamond \, G(\Gamma_1, \Gamma_2, \Gamma_3).
\]
\item[3)] The dimension of the positive part of tangent space to $M_{T_1, T_2, T_3}$ at $(I_1, I_2, I_3)$ is equal to 
\[
\dim T^+_{I_1} M_{T_1} + \bullet M^+(\Gamma_1, \Gamma_2)+ \bullet M^+(\Gamma_2, \Gamma_3) -\star M^+(\Gamma_1, \Gamma_2)-\star M^+(\Gamma_2, \Gamma_3) + \diamond M^+(\Gamma_1, \Gamma_2, \Gamma_3).
\]
\item[4)] The dimension of the positive part of the tangent space to $G_{T_1, T_2, T_3}$ at $(I_1, I_2, I_3)$ is equal to 
\[
\dim T^+_{I_1} G_{T_1} + \bullet G^+(\Gamma_1, \Gamma_2)+ \bullet G^+(\Gamma_2, \Gamma_3) -\star G^+(\Gamma_1, \Gamma_2)-\star G^+(\Gamma_2, \Gamma_3) + \diamond G^+(\Gamma_1, \Gamma_2, \Gamma_3).
\]
\end{itemize}
\end{lemma}
\begin{proof}
The statement follows immediately from the definition of $B(I_1, I_2, I_3)$ in \ref{definitionB(I1I2I3)} and   Observations \ref{weight(0,0,h)},  \ref{weight(0,h,h)}, \ref{weight(f,f,f)} and \ref{countsarethesame}.
\end{proof}

In the next chapter we will prove that the Hilbert-Samuel's strata are smooth using the first two points of Lemma \ref{dimensiontangent123} and studying the dimensions of the strata. Once smoothness is proven the last two points of Lemma \ref{dimensiontangent123} will give us the homological degrees of a basis for the homology of $\Hi^{n, n+1, n+2}(0)$. 
 
\chapter{Smoothness of the Hilbert-Samuel's strata}

\hspace{3em}In this chapter we prove the main geometric result we need, namely that the Hilbert-Samuel's strata for the Hilbert Scheme $\Hi^{n, n+1, n+2}(0)$ are smooth. This will allow us to use the theorem of Bialynicki-Birula to study their cell decompositions, and, ultimately, to study the homology of $\Hi^{n, n+1, n+2}(0)$ itself. 

\hspace{3em}The strategy is the following. Now that we know the dimensions of the tangent spaces at the fixed points we can prove that the spaces $M_{T_1, T_2, T_3}$ (and $G_{T_1, T_2, T_3}$) are smooth by studying their dimensions. In particular we want to prove the following. 
\begin{proposition}
\label{geometricdimensionvstangentspace}
Let $T_1, T_2, T_3$ be three admissible  sequences of nonnegative integers as in \ref{admissible}. Then 
\[
\begin{matrix}
\dim M_{T_1, T_2, T_3 } \geq \dim T_{I_1, I_2, I_3} M_{T_1, T_2, T_3 } &\quad & \dim G_{T_1, T_2, T_3 } \geq \dim T_{I_1, I_2, I_3} G_{T_1, T_2, T_3 }
\end{matrix}
\]
for all $I_1, I_2, I_3$ fixed points of the $\TT^2$ action on $M_{T_1, T_2, T_3 }$. 
\end{proposition}

\hspace{3em} As a consequence of smoothness, we know that the attracting sets are affine cells, and their dimensions are equal to the positive part of the tangent spaces that we already calculated in the previous chapter. 

\hspace{3em}To prove Proposition \ref{geometricdimensionvstangentspace} we use the results of Iarrobino \cite{iarrobino1977punctual}. For every $T$, admissible sequence of integers, he defines a special Young diagram $\Gamma_T$, whose attracting set $A_T$ he proves being affine by giving \emph{explicitly} a set of special generators for each ideal in the cell. Moreover he proves that $M_T$ is covered by a finite union of spaces isomorphic to this cell. 

\hspace{3em}Then we see how this can be extended to $M_{T_1, T_2}$. More precisely we introduce an affine space $\AAA^{\bullet \, M(\Gamma_{T_1}, \Gamma_{T_2})}$ and $\star \, M(\Gamma_{T_1}, \Gamma_{T_2})$ equations that cut out of $\AAA_T \times \AAA^{\bullet \, M(\Gamma_{T_1}, \Gamma_{T_2})}$ exactly the attracting cell labeled by $(\Gamma_{T_1}, \Gamma_{T_2})$. This suffices to prove that the dimension of $M_{T_1, T_2}$ is greater than or equal to the dimension of its tangent space at each point. 

\hspace{3em}The previous step can be extended without problem to $M_{T_1, T_2, T_3}$ by iterating it twice. However to get to the dimension of the tangent space in this case, we need to gain an extra dimension whenever $\diamond\, (\Gamma_{T_1}, \Gamma_{T_2}, \Gamma_{T_3}) =\,1$, recalling point (1) of Lemma \ref{dimensiontangent123}. To do so we show that one of the $\star\,M(\Gamma_{T_2}, \Gamma_{T_3}) $ equations we mentioned above is actually trivially satisfied. \\

\hspace{3em} We repeat all the arguments to prove similar statements for $G_{T_1, T_2, T_3}$.\\

\hspace{3em} Finally, in the last section, we present some direct computations for attracting sets of Hilbert schemes of longer flags. As a result we give sharp bounds for when Hilbert-Samuel's strata are no longer all smooth. We give the Poincar\'{e} polynomials for those few remaining Hilbert schemes whose attracting sets are all affine.   

\section{Iarrobino's standard generators}

\hspace{3em}We need to recall the results of Iarrobino on the punctual Hilbert scheme. The crucial step is that of the definition of \emph{normal pattern} associated to a type $T$ and a \emph{system of parameters} $(u, v)=(ax+by, cx+dy)$. For an ideal with normal pattern we find especially nice generators, the so called \emph{standard generators}. \\

A system of parameters is simply a linear change of coordinates of the plane
\[
(u,v) = \begin{pmatrix} a & b\\ c&d
\end{pmatrix} \begin{pmatrix}x\\y\end{pmatrix} \quad \text{with }  \begin{pmatrix} a & b\\ c&d
\end{pmatrix} \in \text{GL}_2(\CC),
\]
as $\text{GL}_2(\CC)$ acts by automorphisms on $R=\CC[[x,y]]$. Fixing coordinates $(u,v)$, we consider monomials in $u$ and $v$.
Let us begin to define a \emph{pattern} for an ideal $I \in \Hi^n(0)$. This is a set of monomials, in some sense, as disjoint as possible from $I$. 

\begin{definition}
Let $P$ be a set of monomials in $(u,v)$. We denote by $P_j$ the monomials in $P$ having degree $j$. By \emph{type}  $T(P)$ we mean the sequence\[
T(P) \,=(t_0, \dots, t_j,\dots), \qquad \text{where } t_j = \#\,P_j\, .
\]
Let $I \in \Hi^n(0)$. We say that $I$ has pattern $P$ if one of the following equivalent conditions is satisfied. 
\begin{itemize}
\item[(i)] For all $j$,  $\langle P\cap m^j\rangle \oplus I\cap m^j = m^j$,
\item[(ii)] For all $j$,  $\langle P_j\rangle \oplus I_j = R_j$, 
\item[(iii)] $\langle P \rangle \cap I = 0, \quad \text{and }\,T(P) = T(I).$
\end{itemize} 
\end{definition}

\begin{definition}
Let $T$ be an admissible type  as in \ref{admissible} and $(u, v)=(ax+by, cx+dy)$ a system of parameters. Define $\Gamma_T$ to be the only Young diagram that has diagonal sequence $T(\Gamma_T)$ equal to $T$ and is such that all its boxes have the lowest possible $x$ coordinate.

\hspace{3em}The \emph{normal pattern} $P=P\left((u,v),T\right)$ is the set of monomials in $u,v$ with exponents in $\Gamma_T$. Explicitly, $P=\bigcup_j P_j$, where
\begin{align*}
P_j &:= \left\{\,\,\text{the }t_j \text{ monomials of degree } j \text{ with highest $v$ degree}\,\, \right\} \\ &\qquad= \left\{\,\,u^{j-t_j}v^{t_j+1},\,\, u^{j-t_j-1}v^{t_j+2},\,\,\dots,  \,\,uv^{j-1},\,\, v^j  \,\,\right\}.
\end{align*}

\hspace{3em}Once a system of parameters is chosen we define $I_T=I_{(u,v), T}$ to be the monomial ideal associated to $\Gamma_T$ in the parameters $(u,v)$. Most of the time there will be no possible confusion on the choice of parameters, so we simply write $I_T$. 
\end{definition}

\begin{example}
Graphically we can see an example as follow. Let $T=(1, 2, \dots, 6, 5, 3, 2, 0)$. Then $\Gamma_T$ is:

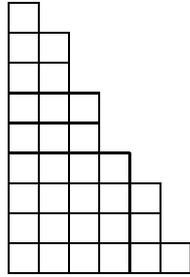
\begin{minipage}{0.5\textwidth}
\begin{center}
\ytableausetup{boxsize=1 em, aligntableaux=bottom}
\begin{ytableau} 
\,\\
\,&\,\\
\,&\,\\
\,&\,&\, \\
\, &\,&\,\\
\, &\,&\,&\,   \\
\, &\,&\,&\,&\,  \\
\,&\, &\,&\,&\,\\
\,&\, &\, &\,&\, &\,
\end{ytableau}
\end{center}
\end{minipage}
\begin{minipage}{0.45\textwidth}
A Young diagram with normal pattern $T=(t_i)_{i\geq 0}$ : all the anti-diagonals $\Lambda_i$ are filled with $t_i$ elements starting from the left and without interruptions.  
\end{minipage}
\end{example}

\begin{definition}
Let $P$ be a normal pattern for $T$ and $(u,v)$. Then we define 
\[
\begin{matrix}
M_P\,=\,M_{P\left((u,v), T\right)}\,:=\,\left\{I\in \Hi(0)\left\vert\, I \text{ has pattern }\, P  \right. \right\} \subset M_T \\
\text{and}\\
G_P\,=\,G_{P\left((u,v), T\right)}\,:=\,\left\{I\in \Hi(0)\left\vert\, I \text{ is homogeneous and  has pattern }\, P  \right. \right\} \subset G_T.
\end{matrix}
\]
\end{definition}
\begin{remark}
\label{normalpararecell}
In the original choice of parameters $(u,v)=(x,y)$ we have that $M_P$ is exactly the attracting set of $M_T$ to $I_T$, and $G_P$ is the attracting set of $G_T$ to $I_T$. Fixing a type $T$ all $M_P$ (resp. $G_P$) for different systems of parameters are clearly isomorphic.
\end{remark} 
\begin{proposition}{\cite[Proposition~3.2]{iarrobino1977punctual}}
\label{normalcover}
Fix an admissible type $T$. Let $N=\sum_j t_j(j+1-t_j)$. Then whenever $P$ runs through the normal patterns of type $T$ in any sets of distinct systems of parameters $(x, y-a_0x), \dots, (x, y-a_N x)$ with $a_i \in \CC$ we have
\[
\bigcup_{j=0}^N \,\,M_{P(x, y-a_jx)} \,= \, M_T, \quad \text{ and } \quad \bigcup_{j=0}^N\,\, G_{P(x, y-a_jx)} \,=\, G_T.
\]
In other words, each ideal $I\in M_T$ (resp. any $I\in G_T$) fails to have normal pattern in at most $N$ systems of parameters of the specified sort.  Observe that the $M_{P(x, y-a_jx)}$ (resp. the $G_{P(x, y-a_jx)}$) share the fixed point $I_T$ so that the above unions are connected. 
\end{proposition}
\hspace{3em}Iarrobino then, in order to prove that $M_T$ is smooth, proves by brute force that $M_P$ is isomorphic to an affine space by giving an explicit isomorphism. This isomorphism specifies the coefficients of a fixed number of some polynomials that are a special set of generators of ideals in $M_P$. Even if remark \ref{normalpararecell} and the fact that the attracting sets are affine cells (thanks to Byialinucky-Birula) already prove smoothness, the explicitness of the work of Iarrobino is essential to have more informations on $M_T$ and $G_T$. \\

\hspace{3em}To go on we need to introduce more notations. We work for the rest of the chapter only with the starting systems of parameter $(u,v)=(x,y)$. \\

\hspace{3em}Suppose $T=(t_i)_{i\geq 0}$ is chosen. Observe that for $I_T$, the only monomial ideal of normal pattern, the standard monomial generators are especially easy. We have $I=(\alpha_0, \alpha_1,\dots,  \alpha_d)$ where $d$ is the initial degree of $I$ and there exist $k_i \in \NN$ such that $k_0=0< k_1<\dots< k_d$ and $\alpha_i=x^{d-i}y^{k_i}.$

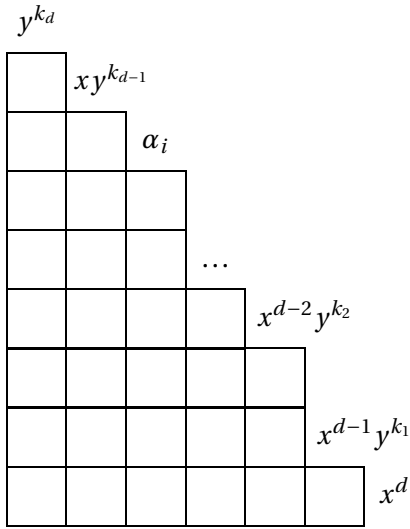
\begin{minipage}{0.5\textwidth}
\begin{center}
\ytableausetup{boxsize=2em, aligntableaux=bottom}
\begin{ytableau} 
\none[y^{k_d}]\\
\,&\none[\quad xy^{k_{d-1}}]\\
\,&\,&\none[\alpha_i]\\
\,&\,&\, \\
\, &\,&\, &\none[\dots]\\
\, &\,&\,&\,&\none[\qquad x^{d-2}y^{k_2}]   \\
\, &\,&\,&\,&\,  \\
\,&\, &\,&\,&\,&\none[\qquad x^{d-1}y^{k_1}]\\
\,&\, &\, &\,&\, &\,&\none[x^d]
\end{ytableau}
\end{center}
\end{minipage}
\begin{minipage}{0.5\textwidth}
The standard monomial generators of the monomial ideal $I_T$ with normal pattern are especially easy: 
\begin{itemize} 
\item[$\bullet$] There are exactly $d+1$ monomial generators where $d$ is that initial degree of $I_{T}$ i.e. $s=d$ following previous notation. 
\item[$\bullet$] For all $i=1, \dots, d$ we have $q_i = 1$. Recall that $q_i$ is the $x$ distance between $\alpha_i$ and $\alpha_{i+1}$. 
\item[$\bullet$] The $k_i $ are computable from the $p_t$ for $t \leq i$, where $p_t$ is the $y$ distance between $\alpha_t$ and $\alpha_{t+1}$. 
\end{itemize}
\end{minipage}
\vspace{0.7em}

\begin{definition} Let $I_T$ be the monomial ideal with normal pattern $P(T)$, and let $(\alpha_0, \dots, \alpha_d)$ be the list of its standard monomial generators. If  $\alpha=\alpha_i$ is one of them, then we define, using the notation of \ref{defPalpha}, the following: 
\begin{align}
&\qquad \qquad \qquad\qquad  P^{\geq}_{\alpha}\,\, :=\,\, \left\{\beta \in P_{\alpha} \vert \deg \beta \geq \deg \alpha \right\},\nonumber\\
&\qquad \qquad\qquad \qquad  P^{=}_{\alpha} \,\, :=\,\, \left\{\beta \in P_{\alpha} \vert \deg \beta = \deg \alpha \right\},\nonumber\\
&\qquad \qquad \qquad \qquad S^{\geq}_{\alpha} \,\, :=\,\, \left\{\gamma \in \Gamma_T \setminus P_{\alpha} \vert \deg \gamma \geq \deg \alpha \right\},\\
&\qquad \qquad \qquad \qquad S^{=}_{\alpha} \,\, :=\,\, \left\{\gamma \in \Gamma_T \setminus P_{\alpha} \vert \deg \gamma = \deg \alpha \right\}.\nonumber
\end{align}
We also set:
\begin{align*}
PM:= &\,\, \bigcup_{i=0}^d P^{\geq}_{\alpha_i} \,\, =\,\, \bigcup_{i=0}^d\left\{(\alpha_i, \beta)\left\vert \beta \in P_{\alpha_i} \deg \beta \geq \deg \alpha_i \right.\right\}, \\
PG:= &\bigcup_{i=0}^d P^{=}_{\alpha_i}= \,\, \bigcup_{i=0}^d\left\{(\alpha_i, \beta)\left\vert \beta \in P_{\alpha_i} \deg \beta  = \alpha \right.\right\}, \\
SM:=&\bigcup_{i=0}^d \,\, S^{\geq}_{\alpha_i}= \bigcup_{i=0}^d\left\{(\alpha_i, \gamma)\left\vert\,\gamma \in \Gamma_T \setminus P_{\alpha} \text{ and } \deg \gamma \geq \deg \alpha_i \right.\right\}, \\
SG:=&\bigcup_{i=0}^d S^{=}_{\alpha_i} \,\, = \bigcup_{i=0}^d\,\, \left\{(\alpha_i, \gamma)\left\vert \,\gamma \in \Gamma_T \setminus P_{\alpha} \text{ and } \deg \gamma = \deg \alpha_i \right.\right\}.
\end{align*}
\end{definition}
\begin{theorem}{\cite[Lemma 2.4, Prop. 2.5, Lemma 2.7, Prop. 2.8]{iarrobino1977punctual}}

\label{standard}
Let $T$ be an admissible sequence of nonnegative integers.  
\begin{itemize}
\item[(1)] We have that $\dim M_T= \#\, PM$ and there is an isomorphism $\AAA^{\dim M_T} \cong M_P$. The isomorphism is explicitly constructed by determining the coefficients of certain polynomials that will be generators of a point in $M_P$. More precisely:
\item[(2)] For each $(a_{\alpha, \beta})_{(\alpha, \beta)\in PM}$ point in $\AAA^{\dim M_T}$ there exist unique coefficients $d_{\alpha, \gamma}\in \CC$ for all $(\alpha, \gamma)\in SM$ such that the ideal $I=I(a_{\alpha, \beta})$ generated by the polynomials
\begin{equation}
\label{standardgenerators}
f_{i} = \alpha_i + \sum_{\beta \in P^{\geq}_{\alpha_i}} a_{\alpha_i, \beta}\, \beta + \sum_{\gamma \in S^+_{\alpha_i}} d_{\alpha_i, \gamma}\, \gamma
\end{equation}  
is in $M_P$. These polynomials are called the \emph{standard generators} for $I$. The coefficients $d_{\alpha_i, \gamma}$ are polynomial expressions in the $a_{\alpha_k, \beta}$'s, for $k\leq i$.
\item[(3)] We have that $\dim G_T= \#\, PG$ and there is an isomorphism $\AAA^{\dim G_T} \cong G_P$. The isomorphism is explicitly constructed by determining the coefficients of certain polynomials that will be generators of a point in $G_P$. More precisely:
\item[(4)] For each $(a_{\alpha, \beta})_{(\alpha \beta)\in PG}$ point in $\AAA^{\dim G_T}$ there exist unique coefficients $d_{\alpha, \gamma}\in \CC$ for all $(\alpha, \gamma)\in SG$ such that the ideal $I=I(a_{\alpha, \beta})$ generated by the polynomials
\begin{equation}
\label{standardgeneratorsG}
f_{i} = \alpha_i + \sum_{\beta \in P^=_{\alpha_i}} a_{\alpha_i, \beta}\, \beta + \sum_{\gamma \in S^=_{\alpha_i}} d_{\alpha_i, \gamma}\, \gamma
\end{equation}  
is in $G_P$. These polynomials are called the \emph{standard generators} for $I=I(a_{\alpha, \beta})$. The coefficients $d_{\alpha_i, \gamma}$ are polynomial expressions in the $a_{\alpha_k, \beta}$'s, for $k\leq i$. 
\end{itemize}
\end{theorem}
\begin{remark}
We will not prove the theorem. However the arguments we use in the next few lemmas can be easily adapted to prove it. As a matter of fact the arguments that follow are adapted from the proof of Iarrobino. 
\end{remark}
\begin{example}
\label{examplestandardgenerators}
We would like to visualize more clearly the standard generators we introduced in (\ref{standardgenerators}). Let then $I(a_{\alpha, \beta})$ be the ideal in $M_P$ defined as in (\ref{standardgenerators}). We focus for example on $f_1$: 
\begin{equation*}
\qquad \qquad f_{1} = \alpha_1 + \sum_{\beta \in P^{\geq}_{\alpha_1}} a_{\alpha_1 \beta}\, \beta + \sum_{\gamma \in S^{\geq}_{\alpha_1}} d_{\alpha_1\gamma}\, \gamma\, .
\end{equation*} 

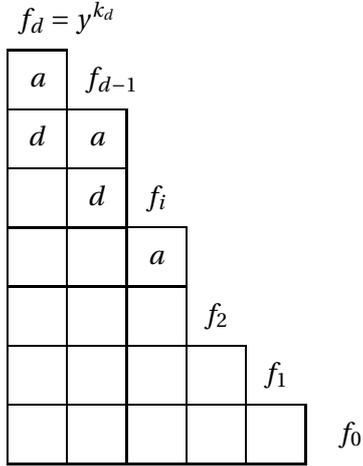
\begin{minipage}{0.45\textwidth}
\begin{center}
\ytableausetup{boxsize=2em, aligntableaux=bottom}
\begin{ytableau} 
\none[\qquad f_d=y^{k_d}]\\
a&\none[\quad f_{d-1}] \\
d&a\\
\,&d&\none[f_i]\\
\,&\,&a \\
\, &\,& \, &\none[f_2]\\
\, &\,&\,&\, &\none[f_1]   \\
\, &\,&\,& \, & \,&\none[\quad f_0 ]  
\end{ytableau}
\end{center}
\end{minipage}
\begin{minipage}{0.55\textwidth}
The standard generators of an ideal with normal pattern: 
Here we concentrate on $f_1$. The boxes marked with an $a$ correspond to the monomial in $P^{\geq}_{\alpha_1}$, while the boxes marked with a $d$ correspond to monomials in $S^{\geq}_{\alpha_1}$. The $a$'s are the free coefficients to be chosen in $\AAA^{\# P^{\geq}_{\alpha_1}}$ while the $d$'s are the unique coefficients that depend on the free choices made for the free coefficients of $f_1$ and $f_0$. Notice in particular that there are no monomials in the support of $f_1$ with degree lower than $\alpha_1$. Also, all other monomials have strictly bigger $y$ degree than $\alpha_1$. 
\end{minipage}
\end{example}

\vspace{1 em}
Let now $T_1=(t_{1,i})_{i\geq 0}, T_2=(t_{2,i})_{i\geq 0}, T_3=(t_{3,i})_{i\geq 0}$ be three admissible sequences of nonnegative integers, nested as in (\ref{admissible}), with $|T_i|=n+i-1$. We call $m,m'$ the indexes of where we have the jump, i.e. the indexes such that 
\begin{equation}
\label{admissiblejump}
t_{2, i} = \begin{cases} t_{1, i} &\text{ if } i\neq m, \\ t_{1, m}+1 &\text{ if } i= m \end{cases} \qquad\text{ and }\quad t_{3, i} = \begin{cases} t_{2, i} &\text{ if } i\neq m', \\ t_{2, m}+1 &\text{ if } i= m' \end{cases}
\end{equation}
Let $\Gamma_i= \Gamma_{T_i}$ be the standard Young diagram associated to $T_i$, and let $I_{\Gamma_i}$ be the monomial ideal associated to $\Gamma_i$. 
\begin{observation}  One of the gifts of working with the normal patterns is that most of the cases of $\ref{cases12}$, among which we had to distinguish in the previous sections, will simply not happen now. More precisely only case 1a) and case 2) are possible. This is true at both steps $j$ and $l$. Moreover case 2) is possible if and only if $t_{1,d}=d$ and $\alpha_j=\alpha_0$ or $t_{2,d}=d$ and $\alpha_l=\alpha'_0$.

\begin{minipage}{0.45\textwidth}
\begin{center}\text{Case 1a):  possible.} \end{center}\quad
\begin{center}
\ytableausetup{boxsize=1.2em, aligntableaux=bottom}
\begin{ytableau} 
\, \\
\,&\,  \\
\,&\, &\,\\
\,&\, &\,&*(green) \alpha  \\
\,&\, &\,&\,  \\
\,&\, &\, &\,&\,
\end{ytableau}
\end{center}
\end{minipage} 
\begin{minipage}{0.45\textwidth}
\begin{center}\text{Case 1b): not possible.} \end{center}\quad
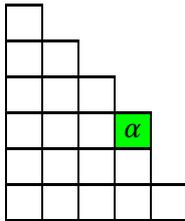
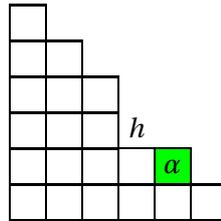
\begin{center}
\ytableausetup{boxsize=1.2em, aligntableaux=bottom}
\begin{ytableau} 
\, \\
\,&\,  \\
\,&\, &\, \\
\,&\, &\, &\none[h] \\
\,&\, &\,&\,&*(green) \alpha \\
\,&\, &\, &\,&\, &\,
\end{ytableau}
\end{center}
\end{minipage} 

The boxes marked with the $h$ for \emph{hole} show why some of the cases of \ref{cases12} are not standard. 

\begin{minipage}{0.45\textwidth}
\begin{center}\text{Case 2):  possible only with $\alpha=\alpha_0$.} \end{center}\quad
\begin{center}
\ytableausetup{boxsize=1.2em, aligntableaux=bottom}
\begin{ytableau} 
\, \\
\,&\,  \\
\,&\, &\,\\
\,&\, &\, \\
\,&\, &\,&\,  \\
\,&\, &\, &\,&*(green) \alpha
\end{ytableau}
\end{center}
\end{minipage} 
\begin{minipage}{0.45\textwidth}
\begin{center}\text{Case 3): not possible.} \end{center}\quad
\begin{center}
\ytableausetup{boxsize=1.2em, aligntableaux=bottom}
\begin{ytableau} 
\, &\none[h]\\
\,&\,  &*(green) \alpha\\
\,&\, &\, \\
\,&\, &\,  \\
\,&\, &\,&\, \\
\,&\, &\, &\,&\, 
\end{ytableau}
\end{center}
\end{minipage} 

\end{observation}
\begin{observation}
Observe that for $I_{\Gamma_1}$ and $I_{\Gamma_2}$, nested monomial ideals with normal patterns, the numbers we defined in \ref{starsanddots} are also more explicit. In fact elements on the same diagonal or on lower diagonals must all lie on their right. Explicitly:
\begin{align}
\label{bulletsforstandard}
\bullet \, M(\Gamma_1, \Gamma_2) &= \# \left\{i=0, \dots, d\,\left\vert \, \deg \alpha_i \leq \deg \alpha_j \right.\right\} = j,\\
\star \, G(\Gamma_1, \Gamma_2) &= \# \left\{i=0, \dots, d\,\left\vert \, \deg \alpha_i = \deg \alpha_j \right.\right\}= \begin{cases} t_{1,m-1}-t_{1, m}+1 & \text{ if } m=d, \\ t_{1,m-1}-t_{1, m} & \text{ if } m>d,
\end{cases} \nonumber\\
\, \nonumber\\
\label{r}
\star \, M(\Gamma_1, \Gamma_2) &  = \# \left\{i=0, \dots, d\,\left\vert \, \deg \alpha_i \leq \deg \alpha_j -1\right.\right\}= \sum_{k=d}^{m-1} t_{1,k-1}-t_{1, k},\\
\star \, G(\Gamma_1, \Gamma_2) &  = \# \left\{i=0, \dots, d\,\left\vert \, \deg \alpha_i = \deg \alpha_j -1\right.\right\}= t_{1,k-2}-t_{1, k-1}.\nonumber
\end{align}
(See next subsection for more details on the case for $G(\Gamma_1, \Gamma_2)$). If $I_{\Gamma_1}$,  $I_{\Gamma_2}$ and $I_{\Gamma_3}$ are nested monomial ideals with normal patterns, the interesting observation is that: 
\begin{equation}
\begin{matrix}
\diamond \, M(\Gamma_1, \Gamma_2, \Gamma_3) = \begin{cases}1 &\text{ if } \deg \alpha_j\leq \deg \alpha_l -2, \\ 0 &\text{ otherwise. }\end{cases} &
\diamond \, G(\Gamma_1, \Gamma_2, \Gamma_3) = \begin{cases}1 &\text{ if } \deg \alpha_j = \deg \alpha_l -2, \\ 0 &\text{ otherwise. }\end{cases}
\end{matrix}
\end{equation}
\end{observation}

\section{$M_{T_1, T_2, T_3}$ is smooth.} 

 \hspace{3em}We start now to apply the results of Iarrobino. Thanks to Proposition \ref{normalcover} we have that:
\begin{align*}
M_{P((x,y), T_1, T_2, T_3)} &= \left\{(I_1, I_2, I_3)\in\Hi^{n,n+1,n+2}(0)\left\vert \,\,I_i \cap \langle P((x,y), T_i) \rangle = \{0\}, \, i=1, 2, 3\right.\right\} \\
&= \text{ attracting set of } M_{ T_1, T_2, T_3} \text{ to the fixed point } \left(I_{\Gamma_1},I_{\Gamma_2},I_{\Gamma_3}\right).
\end{align*}
is open in $M_{ T_1, T_2, T_3}$ (and in fact $M_{ T_1, T_2, T_3}$ is covered with a finite number of opens of the form $M_{P((u,v), T_1, T_2, T_3)}$). In particular $\dim M_{ T_1, T_2, T_3} = M_{P((x,y), T_1, T_2, T_3)}$. Now  $M_{P((x,y), T_1, T_2, T_3)}$ has a unique fixed point, namely $(I_{\Gamma_1},I_{\Gamma_2},I_{\Gamma_3})$, and we already know the dimension of the tangent space at this point thanks to Lemma \ref{dimensiontangent123}. Thus, in order to prove Proposition \ref{geometricdimensionvstangentspace}, we need only to prove the following. 
\begin{proposition}For $T_1, T_2, T_3$ three admissible sequences of nonnegative integers as in \ref{admissible}, we have
\[
\dim M_{P((x,y), T_1, T_2, T_3)} \geq \dim T_{I_{\Gamma_1},I_{\Gamma_2},I_{\Gamma_3}} M_{ T_1, T_2, T_3}\, .
\]
\end{proposition}

\begin{remark}
As a consequence we have that for all $I_1, I_2, I_3$ fixed points of the $\TT^2$ action on $M_{T_1, T_2, T_3 }$ the dimension of $T_{I_1, I_2, I_3} M_{T_1, T_2, T_3 }$ is the same. This can also be seen as a combinatorial result. 
\end{remark}

Even if the possibilities for nested Young diagrams with normal patterns are fewer, we still need to treat differently two possible cases. 
\begin{definition}
\label{casesstandard}
\textbf{Case i) } For the three monomial ideals $I_{\Gamma_1}$,  $I_{\Gamma_2}$ and $I_{\Gamma_3}$ we have that the initial degrees coincide $d=d'=d''$. This is the generic case. It happens if we are in case $j 1)$ and in case $l 1)$ of the cases discussed in Definition \ref{cases123}.  

\textbf{Case ii) } All the other cases, i.e. either $d+1=d'=d''$ or $d+1=d'+1=d''$.  This happens if and only if we are, respectively, in case $j 2)$ or in case $l 2)$: 

\begin{minipage}{0.45\textwidth}
\begin{center}\textbf{Case j 2)}\\ \small{$t_{1,d}=d$ and $m=d$} \end{center}\quad
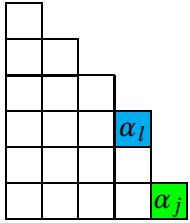
\begin{center}
\ytableausetup{boxsize=1.2em, aligntableaux=bottom}
\begin{ytableau} 
\,\\
\,&\,  \\
\,&\, &\,\\
\,&\, &\,&*(cyan) \alpha_l  \\
\,&\, &\,&\,  \\
\,&\, &\, &\,&*(green) \alpha_j
\end{ytableau}
\end{center}
\end{minipage} 
\begin{minipage}{0.45\textwidth}
\begin{center}\textbf{Case l 2)}\\ \small{$t_{2,d}=d$ and $m'=d$} \end{center}\quad
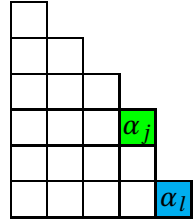
\begin{center}
\ytableausetup{boxsize=1.2em, aligntableaux=bottom}
\begin{ytableau} 
\, \\
\,&\,  \\
\,&\, &\,\\
\,&\, &\,&*(green) \alpha_j \\
\,&\, &\,&\,  \\
\,&\, &\, &\,&*(cyan) \alpha_l
\end{ytableau}
\end{center} 
\end{minipage} 
\end{definition}

We start by considering the case where $d=d'=d''$ as it is more interesting. In fact, having understood that, we will be able to deal with the other cases with some simple observations. Consider the parametrization of $M_{P(T_1)}\cong \AAA^{ \#\,PM}$ given in \ref{standard}, and call $I(a_{\alpha, \beta})$, $(\alpha, \beta) \in PM$, the ideal $I((a_{\alpha, \beta}))=(f_0,\dots, f_d)$ generated by the standard generators associated to $(a_{\alpha, \beta}))_{(\alpha, \beta) \in PM}$ like in (\ref{standardgenerators}). \\

\hspace{3em}We want to see what kind of ideals $J$ are such that $(I,J) \in M_{P((x,y), T_1, T_2)}$. For each $\theta_i \in \CC$ with $i=0, \dots, j-1$ we define the ideal $J(a_{\alpha, \beta}, \theta_i)$ giving its generators in terms of those of $I(a_{\alpha, \beta})$: 
\begin{equation}
\label{definitionJ}
J(a_{\alpha, \beta}, \theta_i) :=
\begin{pmatrix} 
f'_0 = &f_0 + \theta_0 f_j \\
\dots&\dots\\
f'_{j-1} =&f_{j-1}+\theta_{j-1} f_j\\
f'_j =&yf_j\\
f'_{j+1} =&f_{j+1} \\
\dots&\dots\\
f'_{d}=&f_d
\end{pmatrix}\,\,,
\qquad \begin{matrix}
\text{where the $f_i$, recall, are }\\
f_{i} = \alpha_i + \sum_{\beta \in P^{\geq}_{\alpha_i}} a_{\alpha_i, \beta}\, \beta + \sum_{\gamma \in S^{\geq}_{\alpha_i}} d_{\alpha_i, \gamma}\, \gamma \\
\text{ and }\\
I((a_{\alpha, \beta}))=(f_0, \dots, f_d),\,\, a_{\alpha, \beta},\theta_i \in \CC.
\end{matrix}
\end{equation}
\begin{lemma}
\label{dim2}
Call $r= *M(\Gamma_1, \Gamma_2)$. Then there exist $r$ equations $g_0, \dots, g_{r-1}$ in the\\ $((a_{\alpha, \beta}),(\theta_i))_{(\alpha, \beta) \in PM, i=0, \dots, j-1}$, defined as 
\begin{equation}
\label{equationsg}
y^{p_i}f_i - xf_{i+1} = g_i (a_{\alpha, \beta},\theta_s) f_j \quad \text{mod } \, (f'_{i+1}, \dots, f'_d), 
\end{equation}
such that 
\[
(I(a_{\alpha, \beta}), J(a_{\alpha, \beta}, \theta_i)) \in M_{P((x,y), T_1, T_2)} \iff g_i (a_{\alpha, \beta},\theta_s) = 0 \text{ for all } i=0,\dots,  r-1.
\]
Equivalently if we define $Y$ to be the variety cut out by the $g_i$'s then we have 
\begin{equation}
\label{starequationsinclusion}
\AAA^{ \#\,P
M} \times \AAA^{j} \supset Y := \left\{(a_{\alpha, \beta},\theta_s)_{(\alpha, \beta), s}\,\, \left\vert\,\,\, g_i (a_{\alpha, \beta},\theta_s) = 0\right.\right\} \hookrightarrow  M_{P((x,y), T_1, T_2)}.
\end{equation}
\end{lemma}
\begin{proof}

We divide the proof in two parts: in the first part a division procedure will define for us the equations $g_i$. In the second part we will prove that satisfying these equations is a sufficient condition to be in $M_{P((x,y), T_1, T_2)}$. The first part, or rather the procedure that proves it, will be used also later, so we add a name for future reference.

\begin{procedure}
\label{divisionprocedure}

We want to prove that condition \ref{equationsg} actually defines $r$ equations $g_i$. We do a reduction (division), reasoning by degree. For every element $h \in R=\CC[[x,y]]$, write $\text{in}_y(h) \in \NN\cup \{+\infty\}$ to be the lowest degree of a $y$ power appearing in the expansion of $h$. Write $y^{p_i}f_i - xf_{i+1}$ as a linear combination of monomials: some will be in $\Gamma_1$, some will be outside $\Gamma_2$ and one will be $\alpha_j$:
\begin{equation}
\label{a}
y^{p_i}f_i - xf_{i+1} = \sum_{\beta \in \Gamma_1} c_{\beta}\, \beta +\sum_{\gamma \notin \Gamma_2} c_{\gamma}\, \gamma + c_{\alpha_j} \alpha_j \qquad c_{\beta},c_{\gamma}, c_{\alpha_j} \in \CC.
\end{equation}
Now we will use the $f'_s, s>i $ to eliminate the $\gamma \notin \Gamma_2$ in order, starting from the lowest $\text{in}_y(\gamma)$. Precisely: let $t= \text{in}_y(y^{p_i}f_i)$. For $p\geq t$ order all monomials $\gamma \notin \Gamma_2$ with $\text{in}_y(\gamma)=p$ by their ascending $x$ degree. To eliminate them in order means to perform many successive steps.
  \\

\textbf{[step $p=t$]} Looking at the definition of $f_i$ and $f_{i+1}$ there are no $\gamma \notin \Gamma_2$ with $c_{\gamma} \neq 0$  and $p= \text{in}_y(\gamma)$ in (\ref{a}). 

\textbf{[step $p=t+1(\text{ if }< t+p_{i+1})$]} Looking at the definition of $f_i$, $f_{i+1}$ and $f'_{i+1}$, there is only at most one $\gamma \notin \Gamma_2$ with $c_{\gamma} \neq 0$ and $\text{in}_y(\gamma)=p$ in the left hand side, and this is $y\alpha_{i+1}$. Then eliminate it subtracting $c_{y\alpha_{i+1}}yf'_{i+1}$. Now we do not have any $\gamma \notin \Gamma_2$  with $c_{\gamma} \neq 0$ and $\text{in}_y \gamma \leq p $. More specifically we are left with:
\begin{align*}
y^{p_i}f_i - xf_{i+1} -c_{y\alpha_{i+1}}yf'_{i+1} &= \sum_{\beta \in \Gamma_1} G_{i, \beta, p}(a_{\alpha, \beta}, \theta_s)\, \beta \\ & +\sum_{\gamma \notin \Gamma_2, \text{in}(\gamma)>p} c_{\gamma,p}\, \gamma +  g_{i,p}(a_{\alpha, \beta}, \theta_s)\alpha_j.
\end{align*}
Here $G_{i,\beta,p}(a_{\alpha, \beta}, \theta_s), g_{i,p}(a_{\alpha, \beta}, \theta_s)$ and $c_{\gamma,p}$ are polynomials in the $a_{\alpha, \beta}, \theta_s$ and depend on the current step  $p$. We repeat this step similarly until we reach $p=t+p_{i+1}$.

\textbf{[...]}

\textbf{[step $p=t+p_{i+1}$]} Let $(\gamma_1, \dots,\gamma_m)$ be the ordered (by ascending $x$ degree) list of monomials $\gamma \notin \Gamma_2$ with $c_{\gamma} \neq 0$ and $\text{in}_y(\gamma)=p$. Multiply $y^{p_{i+1}}f'_{i+1}$ by appropriate coefficients and monomials in $x$ to eliminate the $\gamma_e$ starting from the last one. If needed eliminate $\alpha_{j+2}$ by a \emph{linear} multiple of $f_{j+2}$. The result is: 
\begin{align*}
y^{p_i}f_i - xf_{i+1} + h(x)yf'_{i+1} &+ c(a_{\alpha, \beta}, \theta_s) f'_{j+2} = \sum_{\beta \in \Gamma_1} G_{i,\beta, p-1}(a_{\alpha, \beta}, \theta_s)\, \beta \, + \\ & +\sum_{\gamma \notin \Gamma_2, \text{in}(\gamma)>p} c_{\gamma,p-1}\, \gamma +  g_{i,p-1}(a_{\alpha, \beta}, \theta_s)\alpha_j.
\end{align*}
Here again we have polynomial expressions in the $(a_{\alpha, \beta}, \theta_s)$, like $c(a_{\alpha, \beta}, \theta_s)\in\CC$ or the coefficients of $h(x)$. The sum on the $\gamma \notin \Gamma_2$ is now only on those $\gamma$ with $\text{in}_y(\gamma)>p$. We repeat this step until we reach $p=t+p_{i+1}+p_{i+2}$, where we repeat this step but use also $f'_{i+3}$ to eliminate, if necessary $\alpha_{i+3}$.

\textbf{[...]}

\textbf{[step $p= \text{in}(\alpha_j)=k_j$]} As all the previous steps but with the difference that {we eliminate} multiple of $\alpha_j$ with $f_j$ and not with $f'_j$.

\textbf{[...]}

\textbf{[step p=$k_d$ or $p=k_d+1$] }Now we are left only with multiple of $y^{k_d}$, or $y^{k_d+1}$, depending on whether $j=d$ i.e. on whether $f'_d=$ $y^{k_d}$, or $y^{k_d+1}$. Then we can eliminate all the remaining $\gamma \notin \Gamma_2$ with appropriate multiples of $f_d$ (and not $f'_d$: cfr. step $p= \text{in}(\alpha_j)=k_j$, if $f'_d = yfd$ then $j=d$ and we use $f_d$, otherwise $f'_d=f_d$ so there is no difference) and be left with an expression:
\begin{equation}
\label{equazione23giugno}
y^{p_i}f_i - xf_{i+1} + \sum_{t>i} h_i f'_i + g_i(a_{\alpha, \beta}, \theta_s) f_j= \sum_{\beta \in \Gamma_1} G_{i,\beta}(a_{\alpha, \beta}, \theta_s)\, \beta .
\end{equation}
Here, again, the $G_{i,\beta}(a_{\alpha, \beta}, \theta_s)$ and $g_i(a_{\alpha, \beta}, \theta_s)$ are coefficients that depend polynomially on $a_{\alpha, \beta}, \theta_s \in \AAA^{ \#\,PG} \times \AAA^{j}$, and, since this is the last step, we do not write the dependence on the step $p$. 

Now since the left hand side of (\ref{equazione23giugno}) is in $I(a_{\alpha, \beta})$ and the right hand side is in $\langle \Gamma_1 \rangle$, and by hypothesis $I(a_{\alpha, \beta})$ has $T$ normal pattern, i.e. 
\[
I(a_{\alpha, \beta}) \cap \langle \Gamma_1 \rangle= \{0\},
\] 
we know that the $G_{i, \beta}(a_{\alpha, \beta}, \theta_s)$ must all vanish. Thus we are left with $g_i$, the equation we wanted. 

\begin{example}[Procedure, picture below] We want to see how the equation $g_0$, relating $f_0$ and $f_1$, arises. The boxes marked with $a_1$ (resp. $a_{1,2}$) are the free coordinates for $f_1$ (resp. $f_1$ and $f_2$, of course these coefficients can be different one from the others), and those marked with $d_1$ are the possibly non-zero but constrained coefficients of $f_1$. All others are zero. When we multiply $f_0$ by $y^2$ the boxes of its potentially nonzero coefficients move two steps up, and those of $f_1$ multiplied by $x$ move one to the right. First step $p=t=2$: there are no boxes that get out of $\Gamma_1$ on the 3-rd row (that of $y$ degree 2). At the $4$-th row (step $p=3$) we eliminate what got out with $yf_1'$. Step $p=t+p_1=4$, row=$5$, we eliminate those that we can (i.e. above and to the right, in this specific case none) with $f'_1$, the other, only that marked with $f_2$,  with $f'_2$. When we get to step=$7=k_j$ we use $f_j$ instead of $f'_j$. We finish with $f_d$.

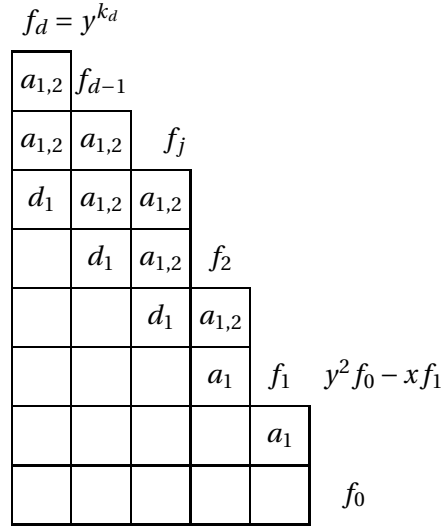
\begin{center}
\ytableausetup{boxsize=2em, aligntableaux=bottom}
\begin{ytableau} 
\none[\qquad f_{d}=y^{k_d}]\\
a_{1,2}&\none[ f_{d-1}]\\
a_{1,2}&a_{1,2}&\none[\quad f_{j}] \\
d_1&a_{1,2}&a_{1,2}\\
\,&d_1&a_{1,2}&\none[f_2]\\
\,&\,&d_1&a_{1,2} \\
\,&\, &\,& a_1 & \none[f_1] & \none[\qquad\quad y^2f_0 -xf_1]\\
\,&\, &\,&\,& a_1    \\
\,&\, &\,&\,& \, &\none[\quad f_0 ]  
\end{ytableau}
\end{center}
\end{example}
\textbf{End procedure}
\end{procedure}

\hspace{3em}We have defined equations $g_i$ for all $i=0, \dots, r-1$. Recall now that $r-1 = \max \{i \vert \deg{\alpha_i} < \deg \alpha_j\}$. Then if we repeat the previous argument for $i\geq r$ we get that the corresponding $g_i$ are identically zero on $\AAA^{ \#\,PM} \times \AAA^{j}$, for degree reasons. (This is a big part of Iarrobino's proof of Proposition \ref{standard}).
By the definition of $J$ we have that 
\[
(J(a_{\alpha, \beta}, \theta_i), f_j) = I(a_{\alpha, \beta}).
\]
We also claim that 
\begin{equation}
\label{xfj}
xf_j \in J(a_{\alpha, \beta}, \theta_i).
\end{equation}
This can be seen by applying the above Procedure \ref{divisionprocedure} to $xf_j$ and see that we do not have a remainder outside $J(a_{\alpha, \beta}, \theta_i)$.
Then we have that either $J(a_{\alpha, \beta}, \theta_i) = I$ or $(I(a_{\alpha, \beta}),$ $ J(a_{\alpha, \beta}, \theta_i))$ $ \in M_{P((x,y), T_1, T_2)}$ as we want. 

\hspace{3em}To prove the lemma we then need to check that we are in the second case. Thus it is sufficient to show that, whenever the coefficients $((a_{\alpha, \beta})_{(\alpha, \beta)}, (\theta_i)_{i}) \in \AAA^{\dim M_P}\times \AAA^j $ satisfy $g_i(a_{\alpha, \beta}, \theta_s)=0$, we have that $J(a_{\alpha, \beta}, \theta_i)$ has normal pattern $P(T_2)$. This amounts to show that 
\[
J(a_{\alpha, \beta}, \theta_i) \cap \langle \Gamma_2 \rangle = 0.
\]
Suppose then that we have chosen $(a_{\alpha, \beta}, \theta_s) \in \AAA^{ \#\,PM} \times \AAA^{j}$, such that $g_i (a_{\alpha, \beta},\theta_s) = 0$ for all $i=0, \dots, r-1$. 
Since by hypothesis $I(a_{\alpha, \beta}) \cap \langle \Gamma_1 \rangle = 0$, we need to prove that 
\[
\text{there does not exist } f \in J(a_{\alpha, \beta}, \theta_i) \text{ such that } f = \alpha_j + \sum_{\beta \in \Gamma_1} c_{\beta} \beta, c_{\beta} \in \CC.
\] 
Suppose by absurd that there exists such a $f$, and find a contradiction. Write $f=\sum_{i\geq 0} h_if'_i$ for some polynomial $h_i \in R$, and call $\text{in}_J(f)$ the minimum index for which $h_i\neq 0$. Then we will prove the statement by reverse induction on $\text{in}_J(f)$. 

We have a contradiction if $\text{in}_J(f)=d$, because then 
\[
f= h_df'_d = \begin{cases} h_df_d =h_d y^{k_d} & \text{ if } j<d \\
h_dyf_d= h_d y^{k_d+1} & \text{ if } j=d,
\end{cases}
\] 
does not contain  in its expansion $\alpha_j$ as a monomial, as the degree is strictly bigger. This is indeed  true for all $\text{in}_J(f)\geq j$. 

Suppose now the statement true for all $\text{in}_J(f)>t$, $j>t$ and lets prove it for $\text{in}_J(f)=t$. We write $h_t$ as a series in $y$ 
\begin{align}
\label{1}
&h_t = h_{t,0}(x)+ h_{t,1}(x)y^1+ \dots + h_{t, k_d-1}(x)y^{k_d -1}+ h_{t, k_d}(x)y^{k_d} +\dots, \text{ where } h_{t,s} \in \CC[[x]].\nonumber\\
&f = h_tf'_t + \sum_{i>t} h_i f'_i =  \alpha_j + \sum_{\beta \in \Gamma_1} c_{\beta} \beta. 
\end{align}
Observe that if $h_{t,0}(x)= h_{t,1}(x)= \dots = h_{t, k_d-1}(x)= h_{t, k_d}(x)=0$ then we do arrive at a contradiction since $y^{k_d+1}=f'_d$ (or $y^{k_d+1}=yf'_d$) and then we can rewrite $f$ as \[f=\sum_{t<i<d}h_i f'_i + (h_d+h_t(x, y/y^{k_d+1})f'_t)f'_d\] i.e. as an element in $(f'_{t+1}, \dots, f'_d)$ with $\text{in}_J(f)>t$ for which, by induction hypothesis, we already know there is a contradiction. 

\hspace{3em}We will then proceed by reverse induction on $\text{in}_y(h_t)$, the minimal index for which $h_{t,s}(x) \neq 0$. We just finished giving the base induction case when $\text{in}_y(h_t)=k_d+1$, and we proceed with the induction step. We observe that $\text{in}_y(h_t)$ cannot be $0$, otherwise the monomial $x^a\alpha'_t$, for the appropriate $a\in \NN$,  would appear only once in the lefthand side of (\ref{1}) and never in the righthand side. In fact, for the same reason, it must be that $h_{t,m} = 0$ for all $m\leq p'_i$. Observe that $p'_i\geq p_i$ for $i<j$. Then we are finally able to use the equations $g_i$. In particular for $i=t$ we have: 
\begin{equation}
\label{2}
y^{p_t} f_t - x f_{t+1} = g_t(a_{\alpha, \beta}, \theta_m)f_j + \sum_{i\geq t} h'_i f'_i
\end{equation}
and since we picked $(a_{\alpha, \beta}, \theta_s) \in \AAA^{ \#\,PM} \times \AAA^{j}$, such that $g_i (a_{\alpha, \beta},\theta_s) = 0$ for all $i$ we can substitute (\ref{2}) into (\ref{1}) to obtain an $f$ with $\text{in}(h_t)$ strictly larger. Then we are done. 
\end{proof}
\begin{corollary} 
\label{MT1T2dimension}
Suppose that $T_1, T_2$ is an admissible nested couple of sequences of nonnegative integers. Then
\begin{equation}
\label{dimensionMp12}
\dim M_{P((x,y), T_1, T_2)} \geq \dim M_{P((x,y), T_1)} +\bullet \, M(\Gamma_1, \Gamma_2) -\star \, M(\Gamma_1, \Gamma_2).
\end{equation}
\end{corollary}
\begin{proof}
We have just proven the statement in case i), i.e. whenever $d=d'$ in Lemma \ref{dim2}. We only need to deal with the other case. 
Suppose then $d+1=d'$. This implies $\alpha_j=\alpha_0$. Then there is actually only one ideal $J(a_{\alpha, \beta})$ such that 
\[
\left(I(a_{\alpha, \beta}), J(a_{\alpha, \beta})\right) \in M_{P((x,y), T_1, T_2)}
\]
and it is $J(a_{\alpha, \beta}) = (xf_0,yf_0, f_1, \dots, f_d)$ where the $f_i$ are the standard generators of $I(a_{\alpha, \beta})$. Then it is still true that 
\[
\dim M_{P((x,y), T_1, T_2)} \geq \dim M_{P((x,y), T_1)} +\bullet \, M(\Gamma_1, \Gamma_2) -\star \, M(\Gamma_1, \Gamma_2).
\]
In this case: $\bullet \, M(\Gamma_1, \Gamma_2)\,= \,\star \, M(\Gamma_1, \Gamma_2)=0$ and $\dim M_{P((x,y), T_1, T_2)} \geq \dim M_{P((x,y), T_1)}$.
\end{proof}

\begin{corollary}{\cite[Proposition 3.4.11.]{cheah1998cellular}}
Let $T_1$ and $T_2$ be two sequences of nonnegative integers as in \ref{admissible}. Call $m$ the index such that $t_{2,m}=t_{1,m}+1$, and $d$ the initial degree of $T_1$. Then $M_{T_1, T_2}$ is smooth of dimension
\[ \dim M_{T_1, T_2} = n + \sum_{j\geq d} \frac{(t_{j-1}-t_{j})(t_{j-1}-t_{j}+1)}{2}\,\, + \,\, (t_{m-1}-t_{m}+1).\]
\end{corollary}
\begin{proof}
Lemma \ref{dim2} and at Lemma \ref{dimensiontangent12} prove that  $\dim M_{P((x,y), T_1, T_2)} \geq \dim T_{I_{\Gamma_1},I_{\Gamma_2}} M_{ T_1, T_2}$. This proves that $M_{P((x,y), T_1, T_2)}$ is smooth: in fact all other points have Zariski tangent space of dimension smaller or equal to that of the fixed point. The Proposition of Iarrobino \ref{normalcover} proves that $M_{T_1, T_2}$ is covered by opens that are isomorphic to $M_{P((x,y), T_1, T_2)}$.  The formula for the dimension in terms of the $t_{j,i}$ is clear by looking at Lemma \ref{dimensiontangent12} and Proposition \ref{dimMT} on the dimension of $M_{T_1}$.\\

\hspace{3em}Remember that in this case that we do not need this proof: in fact as we know $\Hi^{n, n+1}(\PP^2)$ is smooth and has an affine paving. Moreover $M_{T_1, T_2}$ is the union of some of those cells, and $M_{P((x,y), T_1, T_2)}$ is one of those cells. However, as already mentioned, $[n, n+1]$ is the last case where smoothness of the ambient space  $\Hi^{n, n+1}(\CC^2)$ holds. 
\end{proof}

\begin{observation}Observe that as a consequence we get that the inclusion in (\ref{starequationsinclusion}) is actually an isomorphism. \\
 
\hspace{3em} When we write the generators of $J(a_{\alpha, \beta}, \theta_i)$ we actually do not write its \emph{standard generators} since the element $f'_j=yf_j$ could have expansion with non zero coefficient for a monomial outside $\Gamma_2$. We could easily remedy to this but it is not needed. In fact to prove the key Lemma \ref{dim2} we used the fact that the expansions of the standard generators $f_i$ has only terms of higher or equal degree than  $\alpha_i$ and the only one with smallest possible $y$ degree is exactly $\alpha_i$. This remains true for the set of generators we gave for $J(a_{\alpha, \beta}, \theta_i)$. This means that we are able to iterate Lemma $\ref{dim2}$ to find enough ideals $K(a_{\alpha, \beta}, \theta_i, \eta_s)$ such that $(a_{\alpha, \beta}, \theta_i), K(a_{\alpha, \beta}, \theta_i, \eta_s) \in M_{P((x,y), T_2, T_3)}$. That is exactly what we are about to do.
\end{observation}

Now that we understand how and why the dimension of $M_{P((x,y), T_1, T_2)}$ changes with respect to the dimension of $M_{P((x,y), T_1)}$, we can also understand how it changes adding a further step and considering $M_{P\left((x,y), T_1, T_2, T_3\right)}$. Lemma \ref{dimensiontangent123} shows that the dimension of $T_{I_{\Gamma_1},I_{\Gamma_2},I_{\Gamma_3}} M_{P((x,y), T_1, T_2, T_3}$ grows, with respect to $\dim T_{I_{\Gamma_1}} M_{P}$, exactly following the rule (\ref{dimensionMp12}) twice for the inclusions $\Gamma_1\supset \Gamma_2$ and $\Gamma_2\supset\Gamma_3$ except when $\deg \alpha_j \leq \deg \alpha_l-2$. In this case \emph{it grows by one more}. \\

\hspace{3em}Let us iterate twice the results in Lemma \ref{dim2}. For this let $I(a_{\alpha, \beta}) \in M_{P(T_1)}$ with $(a_{\alpha, \beta})_{\alpha, \beta} \in \AAA^{ \#\,PM}$ be the ideal with standard generators $(f_0, \dots, f_d)$ as in (\ref{standardgenerators}). Let $(\theta_i)_{0\leq i< j} \in \AAA^j $, and $(\eta_i)_{0\leq i< l} \in \AAA^l$. We define the ideals $J(a_{\alpha, \beta}, \theta_i)$ and $K(a_{\alpha, \beta}, \theta_i, \eta_s)$ mimicking the construction in (\ref{definitionJ})
\[
J(a_{\alpha, \beta}, \theta_i) :=
\begin{pmatrix} 
f'_0 = &f_0 + \theta_0 f_j \\
\dots&\dots\\
f'_{j-1} =&f_{j-1}+\theta_{j-1} f_j\\
f'_j =&yf_j\\
f'_{j+1} =&f_{j+1} \\
\dots&\dots\\
f'_{d}=&f_d
\end{pmatrix},\quad
K(a_{\alpha, \beta}, \theta_i, \eta_s) :=
\begin{pmatrix} 
f''_0 = &f'_0 + \eta_0 f'_l \\
\dots&\dots\\
f''_{l-1} =&f'_{l-1}+\eta_{l-1} f'_l\\
f''_l =&yf'_l\\
f''_{l+1} =&f'_{l+1} \\
\dots&\dots\\
f''_{d}=&f'_d
\end{pmatrix}\, .
\]
The results of Lemma (\ref{dim2}) tell us that there exists $r = *M(\Gamma_1, \Gamma_2)$ equations $g_i(a_{\alpha, \beta}, \theta_s)$ in $\AAA^{ \#\,PM}\times \AAA^j$ and $r' = *M(\Gamma_2, \Gamma_3)$ equations $g'_i(a_{\alpha, \beta}, \theta_t, \eta_s)$ in $\AAA^{ \#\,PM}\times \AAA^j \times \AAA^l$ such that:
\begin{align*}
(I(a_{\alpha, \beta}), J(a_{\alpha, \beta}, \theta_i)) \in M_{P((x,y), T_1, T_2)} &\iff g_i (a_{\alpha, \beta},\theta_s) = 0 \text{ for all } i=0,\dots,  r-1,
\end{align*} and 
\begin{equation}
\label{equationfor3}
\begin{matrix}
(I(a_{\alpha, \beta}), J(a_{\alpha, \beta}, \theta_i), K(a_{\alpha, \beta}, \theta_i, \eta_s)) \in M_{P((x,y), T_1, T_2,T_3)} \text{ with }  g_i (a_{\alpha, \beta},\theta_s) = 0\,\, \forall \,i=0,\dots,  r-1 \\
\begin{sideways}$\iff\,\,$ \end{sideways} \\ 
g'_i (a_{\alpha, \beta},\theta_t,\eta_s) = 0 \text{ for all } i=0,\dots,  r'-1. 
\end{matrix}
\end{equation}
Then, since $j=\bullet \, M(\Gamma_1, \Gamma_2)$ and $l= \bullet \, M(\Gamma_2, \Gamma_3)$ we have that
\begin{equation}
\label{MT1T2T3dimension}
\dim M_{P((x,y), T_1, T_2, T_3)}\geq \dim M_{P((x,y), T_1} + \bullet \, M(\Gamma_1, \Gamma_2) +\bullet \, M(\Gamma_2, \Gamma_3)-\star \, M(\Gamma_1, \Gamma_2)-\star \, M(\Gamma_2, \Gamma_3).
\end{equation}
This is good enough to prove $\dim M_{P((x,y), T_1, T_2, T_3)} \geq \dim T_{I_{\Gamma_1},I_{\Gamma_2},I_{\Gamma_3}} M_{ T_1, T_2, T_3}$ in all cases where $\,\,\diamond \, M(\Gamma_1, \Gamma_2, \Gamma_3)=0$. \\

\hspace{3em}Thus the last step is to show that when  $\diamond \, M(\Gamma_1, \Gamma_2, \Gamma_3)=1$, one of the equations we found is not necessary.  More precisely.

\begin{lemma}
\label{diamondequationissatisfied}
Utilize all notations as above. Suppose $\deg \alpha_j \leq \deg \alpha_{l} -2$. Then $j-1< r'$ and  
\begin{align*}
g'_{j-1}(a_{\alpha, \beta},\theta_t,\eta_s)=0 \text{ for all } &\left((a_{\alpha, \beta})_{\alpha, \beta},(\theta_t)_{0\leq t< j},(\eta_s)_{0\leq t< l}\right) \in \AAA^{ \#\,PM}\times \AAA^j \times \AAA^l \\ 
\text{ such that }\quad g_i (a_{\alpha, \beta},\theta_s)& = 0\,\,\, \forall \,i=0,\dots, r-1 \quad \text{ and } \\
g'_i (a_{\alpha, \beta},\theta_t,\eta_s)& = 0 \,\, \forall \,i=0,\dots, r'-1,\,\,\text{ for }\qquad  i\neq j-1.
\end{align*}
In other words the equation $g'_{j-i}$ is not necessary because trivially satisfied whenever $\deg \alpha_j \leq \deg \alpha_{l} -2$, even though it appears in the list (\ref{equationfor3}), being $j-1< r'$. 
\end{lemma}
\begin{proof}
Since  $\deg \alpha_j \leq \deg \alpha_{l} -2$, necessarily $j<l$ and $j-1< r'$. Then in particular we have: 
\begin{equation}
\label{f2j}
\begin{matrix}
f''_{j-1} = f_{j-1} +\theta_{j-1} f_j + \eta_{j-1} f_l,\\
f''_{j} = yf_j+ \eta_{j} f_l, \\
f''_{l} = yf_l. 
\end{matrix}
\end{equation}
The equations $g'_i(a_{\alpha, \beta},\theta_t,\eta_s)$ are defined by 
\begin{equation}
\label{equationsg'}
y^{p'_i}f'_i - xf'_{i+1} = g'_i (a_{\alpha, \beta},\theta_t, \eta_s) f'_l \quad \text{mod } \, (f''_{i+1}, \dots, f''_d).
\end{equation}
In fact we will see that for $i=j-1$ we have
\begin{equation}
\label{questogoal}
y^{p'_{j-1}}f'_{j-1} - xf'_{j} = 0 \quad \text{mod } \, (f''_{j}, \dots, f''_d).
\end{equation}
This implies that $g'_{j-1}(a_{\alpha, \beta},\theta_t, \eta_s)$ is always equal to zero, as wanted. \\

\hspace{3em}Call $\frac{\alpha_l}{y}=\beta \in \Gamma_1$. Consider an elimination procedure exactly  as in \ref{divisionprocedure} in the proof of Lemma \ref{dim2} to write, relative to $I_1$, 
\[
y^{p_{j-1}}f_{j-1} - xf_{j} = G_{j-1}(a_{\alpha, \beta})\beta \quad \text{mod } \, (f_{j}, \dots, f_d).
\]
Here $G_{j-1}(a_{\alpha, \beta})$ is a polynomial expression in the coefficients $(a_{\alpha, \beta})_{\alpha, \beta} \in \AAA^{\dim M_P}$. 
 
Since $I_1 \cap \langle \Gamma_1 \rangle=0$, $G_{j-1}(a_{\alpha, \beta})=0$ identically on $\AAA^{\dim M_P}$, i.e. $y^{p_{j-1}}f_{j-1} - xf_{j} = \sum_{i\geq j} h_i f_{i}$. Observe that $\text{in}_y(h_j) > 0$. Multiply now both sides of the last equation by $y$, and add and subtract few terms to get: 
\begin{equation}
\label{questa}
y^{p_{j-1}}\left(yf_{j-1}+\theta_{j-1}yf_j -\theta_{j-1}yf_j\right) - xyf_{j} =  h_j yf_{j}+ \dots+h_l yf_l +\dots+ h_dy f_d.
\end{equation}
Observe that $p'_{j-1}=p_{j-1}+1$. Then rearranging we have: 
\[
y^{p'_{j-1}}f'_{j-1} - xf'_{j} \,=\, \left(\theta_{j-1}y^{p_{j-1}}+  h_j\right) yf_{j}+ \dots+h_l yf_l +\dots+ h_dy f_d.
\]
Since as observed $\text{in}_y(h_j) > 0$, and $p_{j-1}\geq 1$ we can factor a further $y$ from the $R$-multiple of $yf_j$, add and subtract $\eta_j f_l$, to write: 
\[
y^{p'_{j-1}}f'_{j-1} - xf'_{j} \,=\, \left(\theta_{j-1}y^{p_{j-1}-1}+  \tilde{h}_j\right) y(yf_{j}+\eta_j f_l-\eta_j f_l)+ \dots+h_l yf_l +\dots+ h_dy f_d.
\]
Now by substituting (\ref{f2j}) in (\ref{questa}), we have (\ref{questogoal}), as desired.
\end{proof}
\begin{example}
$\,$

\begin{minipage}{0.45\textwidth}
\ytableausetup{boxsize=2em, aligntableaux=bottom}
\begin{ytableau} 
\none[\qquad f_{d}=y^{k_d}]\\
a_{1,2}&\none[ f_{d-1}]\\
a_{1,2}&a_{1,2}&\none[\quad f_{l}] \\
d_1&a_{1,2}&*(lightgray)\quad \, a_{1,2}\longleftarrow & \none[\,\,\,\,\,\beta] \\
\,&d_1&a_{1,2}&\none[f_2]\\
\,&\,&d_1&a_{1,2}&\none[yf_1] & \none[\qquad\quad y^3f'_0 -xyf_1]\\
\,&\, &\,& a_1 & *(green)f_1 \\
\,&\, &\,&\,& a_1    \\
\,&\, &\,&\,& \, &\none[\quad f_0 ]  
\end{ytableau}
\end{minipage}
\begin{minipage}{0.54\textwidth}\small{
In the situation depicted on the left $\alpha_j$ is the green box, marked $f_1$, while $\alpha_l$ is the box marked with $f_l$. Since $\deg \alpha_l = \deg \alpha_j +2$  we need to prove that $g'_0$ is trivially satisfied. The idea is that this equation was already satisfied since there must be a relation for $y^2f_0-xf_1$. More explicitly in the proof we write an equation $G_0$ that represent the coefficient of $\beta$ (the gray box) once we apply the Procedure \ref{divisionprocedure} to $y^2f_0-xf_1$: since $\beta$ is in $\Gamma_1$ we find that $G_0$ must be identically zero. Now the key point is that when we do this we have an extra power of $y$ in all the interesting terms. This allows us to relate $G_0$ to $g_0$ with some manipulations. }
\end{minipage}
\end{example}

\begin{observation}Now we deal with case ii) of \ref{casesstandard}. If $d+1=d'+1=d''$ then there is nothing new to prove, as $\diamond\,M (\Gamma_1, \Gamma_2, \Gamma_3)=0$. Moreover, as we already observed there exists only one $K$ that sits inside $J(a_{\alpha, \beta}, \theta_i)$, so the dimension is the expected one. If $d+1=d'=d''$ then we still need to prove that: if $\deg \alpha_j \leq \deg \alpha_l-2$ then $g'_0$, the first of the equations $g'_i$ $i=0, \dots, \star \, M(\Gamma_2, \Gamma_3)$ is trivially satisfied. But this is obvious since in this case $g_0$ is the equation relating $yxf_0$ and $xyf_0$: 
\[
y^{p_0}f'_0+xf'_1 = yxf_0-xyf_0 = 0. 
\]
\end{observation}

\begin{proposition}
\label{mainproposition}
Let $T_1, T_2, T_3$ be three admissible sequences of nonnegative integers as in \ref{admissible} that differ in the indexes $m$ and $m'$ as in (\ref{admissiblejump}). Then the Hilbert-Samuel's strata $M_{{T_1, T_2, T_3}}$ is a smooth variety of dimension
\begin{align*}
\dim\,\, M_{{T_1, T_2, T_3}} \,\, &=\,\, (n-1) + \sum_{j\geq d }(t_{j-1}-t_j)\frac{(t_{j-1}-t_j+1)}{2} + (t_{m-1}-t_m) + (t_{m'-1}-t_{m'}) +\\&+ \begin{cases}1 &\text{ if } m'\geq m+2, \\0   &\text{ otherwise.}\end{cases}
\end{align*}
It has an affine cell decomposition with cells parametrized by nested Young diagrams $\Gamma_1, \Gamma_2, \Gamma_3$ that differ in only one box: $\Gamma_3=\Gamma_2 \sqcup \{\alpha_l\}=  \Gamma_1 \sqcup \{\alpha_l\}\sqcup \{\alpha_j\}$ with $\Gamma_1\vdash n$ and such that $T(\Gamma_i) = T_i$. Call $\alpha_j = (u_j, v_j)$ and $\alpha_l = (u_l, v_l)$, then the affine cell indexed by $ \left(\Gamma_1, \Gamma_2, \Gamma_3\right)$ has dimension given by the following formula: 
\begin{align}
\label{posskewmale}
\pos \left(\Gamma_1, \Gamma_2, \Gamma_3\right) &= \#\,\left\{(u,v) \in \Gamma_1 \left\vert \,\, h_{u,v} \neq 0,1 \right.\right\}\,+\nonumber \\ &+ \#\,\left\{(u,v) \in \Gamma_1 \left\vert \,\, (h_{u,v} = 0 \text{ and } u=u_j )\text{ or } (h_{u,v} = 1 \text{ and } v=v_j) \right.\right\}\, + \nonumber\\ & +  \#\,\left\{(u,v) \in \Gamma_2 \left\vert \,\, (h'_{u,v} = 0 \text{ and } u=u_l )\text{ or } (h'_{u,v} = 1 \text{ and } v=v_l) \right.\right\}\, +\\ & + \begin{cases} 1 & \text{ if } u_l+v_l > u_j +v_j +2, \\
1 & \text{ if } u_l+v_l = u_j +v_j +2 \text{and } v_l >v_j, \\
0 & \text{ otherwise.} \end{cases}\nonumber
\end{align}
Here $h_{u,v}$ is the hook difference of the box $(u,v)$ with respect to the Young diagram $\Gamma_1$, and $h'_{u,v}$ is the hook difference of the box $(u,v)$ with respect to the Young diagram $\Gamma_2$. The Poincar\'{e} polynomial of $M_{{T_1, T_2, T_3}}$ is 
\[
P_q\, \left( M_{{T_1, T_2, T_3}} \right) \quad = \quad \sum_{(\Gamma_1, \Gamma_2, \Gamma_3)\vdash [n, n+1, n+2],\,\,\, T(\Gamma_i)= T_i} \,\,\, q^{\pos \left(\Gamma_1, \Gamma_2, \Gamma_3\right)  }\,\,.
\]
\end{proposition}
\begin{proof}
Thanks to equation (\ref{MT1T2T3dimension}) and Lemma \ref{diamondequationissatisfied} we know that the Hilbert-Samuel stratum $M_{{T_1, T_2, T_3}}$ is covered by opens isomorphic to $M_{P((x,y), T_1, T_2,T_3)}$. One of these opens is smooth since at his only torus fixed point it has dimension equal to the dimension of the Zariski tangent space,  thanks to equation (\ref{MT1T2T3dimension}), Lemma \ref{diamondequationissatisfied} and Lemma \ref{dimensiontangent123}. Thanks to Theorem \ref{Bialynicki-Birula} we know that it has a cell decomposition with cells whose dimensions are given by Lemma \ref{dimensiontangent123}. Thanks to Theorem \ref{Fulton} the decomposition in affine cells gives us a basis for the homology, and we can then calculate the Poincar\'{e} polynomial. 
\end{proof}

As a consequence we immediately get the following. 
\begin{proposition}
\label{affinepavingfor123}
The space $\Hi^{n, n+1, n+2}(0)$ has an affine paving given by the attracting sets at the fixed points for the $\TT_{1^+}$ action. The Poincar\'{e} polynomial of $\Hi^{n, n+1, n+2}(0)$ is given by 
\begin{equation}
P_q\, \left( \Hi^{n, n+1, n+2}(0) \right) \quad = \quad \sum_{(\Gamma_1, \Gamma_2, \Gamma_3)\vdash [n, n+1, n+2]} \,\,\, q^{\pos \left(\Gamma_1, \Gamma_2, \Gamma_3\right)  }\,\,
\end{equation}
where $\pos \left(\Gamma_1, \Gamma_2, \Gamma_3\right)$ is as in (\ref{posskewmale}). 
\end{proposition}

\begin{remark}
The natural goal would now be to find a generating series for the Poincar\'{e} polynomials of $\Hi^{n, n+1, n+2}(0)$ as $n$ varies. However the above expression appears hard to sum. We will deal with this problem in the next chapter. 
\end{remark}

\subsection{$G_{T_1, T_2, T_3}$ is smooth}

\hspace{3em}We prove, with the same exact arguments, that the homogenous Hilbert-Samuel's strata $G_{T_1, T_2, T_3}$ are smooth. One of the main reasons why the proof is basically the same as for the case of $M_{T_1, T_2, T_3}$ is the following lemma, that dates back to Iarrobino. 

\begin{lemma}{\cite[Lemma~2.6.]{iarrobino1977punctual}}
Let $I\in M_{P(T_1)}$ be an ideal with normal pattern. Then $I\in G_{P(T_1)}$ if and only if each standard generator is homogenous. If the standard generators of $I$ are $(f_0, \dots, f_d)$ then the associated graded ideal $\rho_T(I)$ has standard generators $(F_0, \dots, F_d)$ where $F_i$ is the initial form of $f_i$, i.e. $\rho(f_i) = F_i$. 
\end{lemma} 

For completeness we expand on notation and content. The complication on the ideals is inversely proportional to the complications on the indexes: even though everything is homogeneous so contains less terms, we need to distinguish those terms with more complicated indexes. In fact basically only the definition of $J(a_{\alpha, \beta}, \theta_i)$ of \ref{definitionJ}, and the form of the equations $g_i(a_{\alpha, \beta}, \theta_s)$ of \ref{equationsg} are different but all the arguments are the same.\\ 

\hspace{3em}Let $T_1, T_2, T_3$ be admissible sequences of nonnegative integers. Let $I(a_{\alpha, \beta})\in G_{P(T_1)}$ be with standard generators $(f_0, \dots, f_d)$ given by (\ref{standardgeneratorsG}) for the point $(a_{\alpha, \beta})_{\alpha, \beta}\in \AAA^{\dim G_{T_1}}$. Then we can read from the elements of $T_1$ how many generators we have of each degree: we have $t_{1, d-1}-t_{1, d}+1$ generators of degree $d$, and $t_{1, k-1}-t_{1, k}$ generators of degree $k$ with $k>d$, where $d$ is the initial degree of $I(a_{\alpha, \beta})$. \\

\hspace{3em}Suppose now that $\alpha_j$, the box in $\Gamma_{T_2}\setminus \Gamma_{T_1}$, is of degree $d+k$ i.e. 
\begin{align*}
T_1= (0,1, 2, \dots, d, t_{1, d}, t_{1, d+1}, \dots, t_{1, d+k}, t_{1, d+k+1}, \dots ), \\
T_2= (0,1, 2, \dots, d, t_{1, d}, t_{1, d+1}, \dots, t_{1, d+k}+1, t_{1, d+k+1}, \dots ).
\end{align*}

Suppose first $k>0$ so that $\alpha_j$ does not have degree $d$ and $d'=d$. This is the general case, we will treat the others later. In particular it must be $j = t_{1, d+k+1}-1$, since we always want Young diagrams of standard normal form $T_2$. Then for $\theta_i$ with $i= t_{1, d+k},\dots, t_{1, d+k+1}-2 $ we can define the family of ideals $J(a_{\alpha, \beta}, \theta_i)$ as 
\begin{equation}
\label{definitionJ}
J(a_{\alpha, \beta}, \theta_i) :=
\begin{pmatrix} 
f'_0 = &f_0  \\
\dots&\dots\\
f'_{t_{1, d+k}-1} =&f_{t_{1, d+k}-1}\\
f'_{t_{1, d+k}} =&f_{t_{1, d+k}} + \theta_{t_{1, d+k}} f_{t_{1, d+k+1}-1} \\
\dots \\
f'_{t_{1, d+k+1}-2} =&f_{t_{1, d+k+1}-2} + \theta_{t_{1, d+k+1}-2} f_{t_{1, d+k+1}-1} \\
f'_{t_{1, d+k+1}-1} =&yf_{t_{1, d+k+1}-1}\\
f'_{t_{1,d+k+1}} =&f_{t_{1, d+k+1}} \\
\dots&\dots\\
f'_{d}=&f_d
\end{pmatrix},
\text{ with }
\begin{matrix}
(a_{\alpha, \beta})_{\alpha, \beta} \in \AAA^{\dim G_{T_1}}\\
\, \\
(\theta_i)_{i} \in \AAA^{t_{1, d+k+1}-t_{1, d+k}-1}.
\end{matrix}
\end{equation}
We define this family because we are looking for enough $J$ such that $(I(a_{\alpha, \beta}), J) \in G_P(T_1, T_2)$. 

\begin{example}
$\,$

\begin{minipage}{0.45\textwidth}
\begin{center}
\ytableausetup{boxsize=2em, aligntableaux=bottom}
\begin{ytableau} 
\none[\qquad f_4=y^{k_4}]\\
\,&\none[\quad f_{3}] \\
\,&\,\\
\,&\,&\none[\quad f_{t_{k+d+1}-1}]\\
\,&\,&\,& \none[f_2]\\
\, &\,& \, &\, &\none[f_1]\\
\, &\,&\,&\, &\,    \\
\, &\,&\,& \, & \,&\none[\quad f_0 ]  
\end{ytableau}
\end{center}
\end{minipage}
\begin{minipage}{0.45\textwidth}
In this example: $T_1=(1, 2, 3, 4, 5, 5, 2)$, and $T_2=(1, 2, 3, 4, 5, 5, 3)$. Here $\alpha_j$ is marked by $f_j = f_{t_{k+d+1}-1}$, its degree is $d+k=d+2=7$, $j$ is $3=t_{k+d+1}-1$, and finally $t_{1, d+k+1}-t_{1, d+k}-1$ is $5-2-1=2$. Every polynomial here is homogeneous, i.e. it contains monomials that are only on its antidiagonal. Then, when we add the box $\alpha_j$ we can modify only $f_1$ and $f_2$. 
\end{minipage}
\end{example}

\begin{observation}
As it was happening in the case $M_P$ the possible free new coordinates we are adding are as many as $\bullet \, G(\Gamma_1, \Gamma_2)$. Now we look at the equations we have to impose. 
\end{observation}

\begin{lemma}
\label{dim2G}
Call $r= *G(\Gamma_1, \Gamma_2)$. Then there exist $r$ equations $g_{t_{1, d+k-2}}, \dots, g_{t_{1, d+k-1}}$ in the $((a_{\alpha, \beta}),(\theta_i))_{(\alpha, \beta) \in PG, i=t_{1, d+k},\dots, t_{1, d+k+1}-2}$ defined as 
\begin{equation}
\label{equationsg}
y^{p_{i-1}}f_{i-1} - xf_{i} = g_i (a_{\alpha, \beta},\theta_s) f_j \quad \text{mod } \, (f'_{i+1}, \dots, f_{j-1}), 
\end{equation}
such that 
\[
(I(a_{\alpha, \beta}), J(a_{\alpha, \beta}, \theta_i)) \in G_{P((x,y), T_1, T_2)} \iff g_i (a_{\alpha, \beta},\theta_s) = 0 \text{ for all } i={t_{1, d+k-2}},\dots,  {t_{1, d+k-1}}.
\]
Equivalently if we define $Y$ to be the variety cut out by the $g_i$ then we have 
\[
\AAA^{ \#\,PG} \times \AAA^{t_{1, d+k+1}-t_{1, d+k}-1} \supset Y := \{(a_{\alpha, \beta},\theta_s) \vert g_i (a_{\alpha, \beta},\theta_s) = 0\} \xrightarrow{\text{isom}}  M_{P((x,y), T_1, T_2)}.
\]
\end{lemma}
\begin{proof}
The proof is exactly the same as for Lemma \ref{dim2}. The equations are in correspondence, by multiplying by $x$, with the free boxes on the antidiagonal immediately below $\alpha_j$ since everything is homogenous and has to remain homogenous. 
\end{proof}

\begin{corollary}
\label{othercasesond}
With the notations as above we have that: 
\[
\dim G_{P((x,y), T_1, T_2)} \geq \dim G_{P((x,y), T_1} +\bullet \, G(\Gamma_1, \Gamma_2) -\star \, G(\Gamma_1, \Gamma_2).
\]
\end{corollary}
\begin{proof}
We need to consider the other cases i.e. when $d\neq d'$ and when $\deg \alpha_j = d$. Observe that in the homogenous case the above expression actually sais that the dimension of $G_{P((x,y), T_1, T_2)}$ can be lower than that of $G_{P((x,y), T_1}$. \\

\hspace{3em}Suppose we are in the case $d+1=d'$. Then $\alpha_j=\alpha_0$ $\deg \alpha_j = d$. Then there is actually only one ideal $J(a_{\alpha, \beta})$ such that 
\[
\left(I(a_{\alpha, \beta}), J(a_{\alpha, \beta})\right) \in G_{P((x,y), T_1, T_2)}
\]
and it is $J(a_{\alpha, \beta}) = (xf_0,yf_0, f_1, \dots, f_d)$ where the $f_i$ are the standard generators of $I(a_{\alpha, \beta})$. It is still true that 
\[
\dim G_{P((x,y), T_1, T_2)} \geq \dim G_{P((x,y), T_1)} +\bullet \, G(\Gamma_1, \Gamma_2) -\star \, G(\Gamma_1, \Gamma_2)
\]
because $\bullet \, G(\Gamma_1, \Gamma_2)\,= \,\star \, G(\Gamma_1, \Gamma_2)=0$ and $\dim G_{P((x,y), T_1, T_2)} \geq \dim G_{P((x,y), T_1)}$. \\

\hspace{3em}Suppose finally $d=d'$ and $\deg \alpha_j = d$. Then the reasoning for the case $d=d'$ and $\deg \alpha_j > d$ still works, the only difference being that there is one more standard generator of $I(a_{\alpha, \beta})$ with degree $d$, i.e. $f_0$. So up to a change of indexes everything is the same. 
\end{proof}

\begin{corollary}{\cite[Proposition 3.4.12]{cheah1998cellular}}
Let $T_1$ and $T_2$ be two admissible nested sequences of nonnegative integers as in \ref{admissible}. Call $m$ the index for which $t_{2,m}= t_{1,m}+1$ and $d$ the initial degree of $T_1$. Then  $G_{T_1, T_2}$ is smooth, and of dimension
\[ \dim G_{T_1, T_2} = \sum_{j\geq d}\, (t_{j-1}-t_j+1)(t_{j}-t_{j+1}) + (t_{m-1}-t_{m}) -(t_{m-2}-t_{m-1}).
\]
It has the affine cell decomposition given by the Bialynicki-Birula decompositions whose cells have dimensions as specified in Lemma \ref{dimensiontangent12}. 
\end{corollary}
\begin{proof}
The proof is clear by looking at Lemma \ref{dim2} and at Lemma \ref{dimensiontangent12} on the dimension of the tangent space and at Proposition \ref{dimMT} on the dimension of $G_T$. Observe that in this case, contrary to the case $M_{ T_1, T_2}$ we actually need this proof, since it is not true that $G_{ T_1, T_2}$ is union of affine cells for $\Hi^{n, n+1}(\PP^2)$. 
\end{proof}

Having understood how to pass from $G_{P(T_1)}$ to $G_{P(T_1, T_2)}$ we can understand how to go from $G_{P(T_1, T_2)}$ to $G_{P(T_1, T_2, T_3)}$.\\

\hspace{3em}We start by supposing that $d=d'=d''$ and $\deg \alpha_j, \deg \alpha'_l > d$, in fact say $\deg \alpha_j = d+ k$ and $\deg \alpha'_l= d+k'$, $k, k' >0$. Let again $I(a_{\alpha, \beta}) \in G_{P(T_1)}$ with $(a_{\alpha, \beta})_{\alpha, \beta} \in \AAA^{ \#\,PG}$ be the ideal with standard generators $(f_0, \dots, f_d)$ as in (\ref{standardgeneratorsG}). Let $(\theta_i)_{t_{1, d+k} \leq i \leq t_{1, d+k+1}-2} \in \AAA^{t_{1, d+k+1}-t_{1, d+k}-1} $, and $(\eta_i)_{t_{2, d+k'} \leq i \leq t_{2, d+k'+1}-2} \in \AAA^{t_{2, d+k'+1}-t_{2, d+k'}-1} $. Then we define the ideals $J(a_{\alpha, \beta}, \theta_i)$ as in (\ref{definitionJ}) and the ideals $K(a_{\alpha, \beta}, \theta_i, \eta_s)$ mimicking the construction in definition (\ref{definitionJ}).
\begin{equation}
\label{definitionK}
K(a_{\alpha, \beta}, \theta_i, \eta_s) :=
\begin{pmatrix} 
f''_0 = &f'_0  \\
\dots&\dots\\
f''_{t_{2, d+k'}-1} =&f'_{t_{2, d+k'}-1}\\
f''_{t_{2, d+k'}} =&f'_{t_{2, d+k'}} + \eta_{t_{2, d+k'}} f'_{t_{2, d+k'+1}-1} \\
\dots \\
f''_{t_{2, d+k'+1}-2} =&f'_{t_{2, d+k'+1}-2} + \eta_{t_{2, d+k'+1}-2} f_{t_{2, d+k'+1}-1} \\
f''_{t_{2, d+k'+1}-1} =&yf'_{t_{2, d+k'+1}-1}\\
f''_{t_{2, d+k'+1}} =&f_{t_{2, d+k'+1}} \\
\dots&\dots\\
f''_{d}=&f'_d
\end{pmatrix}.
\end{equation}
The results of Lemma (\ref{dim2G}) tell us that there exists $r = *G(\Gamma_1, \Gamma_2)$ equations $g_i(a_{\alpha, \beta}, \theta_s)$ in $\AAA^{ \#\,PG}\times \AAA^{t_{1, d+k+1}-t_{1, d+k}-1}$ and $r' = *G(\Gamma_2, \Gamma_3)$ equations $g'_i(a_{\alpha, \beta}, \theta_t, \eta_s)$ in $\AAA^{ \#\,PG}\times \AAA^{t_{1, d+k+1}-t_{1, d+k}-1} \times \AAA^{t_{2, d+k'+1}-t_{2, d+k'}-1}$ such that:
\begin{align*}
(I(a_{\alpha, \beta}), J(a_{\alpha, \beta}, \theta_i)) \in G_{P((x,y), T_1, T_2)} &\iff g_i (a_{\alpha, \beta},\theta_s) = 0 \text{ for all } i,
\end{align*} and 
\begin{equation}
\label{equationfor3G}
\begin{matrix}
(I(a_{\alpha, \beta}), J(a_{\alpha, \beta}, \theta_i), K(a_{\alpha, \beta}, \theta_i, \eta_s)) \in G_{P((x,y), T_1, T_2,T_3)} \text{ with }  g_i (a_{\alpha, \beta},\theta_s) = 0\,\, \forall \,i \\
\begin{sideways}$\iff\,\,$ \end{sideways} \\ 
g'_i (a_{\alpha, \beta},\theta_t,\eta_s) = 0 \text{ for all } i. 
\end{matrix}
\end{equation}
Then, since $t_{1, d+k+1}-t_{1, d+k}-1=\bullet \, G(\Gamma_1, \Gamma_2)$ and $t_{2, d+k'+1}-t_{2, d+k'}-1= \bullet \, G(\Gamma_2, \Gamma_3)$ we have that
\begin{equation}
\label{equationdimensionGT1T2T3}
\dim G_{P((x,y), T_1, T_2, T_3)}\geq \dim G_{P((x,y), T_1} + \bullet \, G(\Gamma_1, \Gamma_2) +\bullet \, G(\Gamma_2, \Gamma_3)-\star \, G(\Gamma_1, \Gamma_2)-\star \, G(\Gamma_2, \Gamma_3)
\end{equation}
which is good enough to prove $\dim G_{P((x,y), T_1, T_2, T_3)} \geq \dim T_{I_{\Gamma_1},I_{\Gamma_2},I_{\Gamma_3}} G_{ T_1, T_2, T_3}$ in all cases where $\diamond \, M(\Gamma_1, \Gamma_2, \Gamma_3)=0$. \\

\hspace{3em}The last step needed is then to show that when  $\diamond \, M(\Gamma_1, \Gamma_2, \Gamma_3)=1$, one of the equations we found is not necessary. This is the content of the next Lemma.

\begin{lemma}
\label{starnotneeded}
Utilize all notations as above, and suppose $\deg \alpha_j = \deg \alpha'_{l} -2$. Then $t_{2, d+k'} \leq j \leq t_{1, d+k+1}-2$ and  
\begin{align*}
g'_{j}(a_{\alpha, \beta},\theta_t,\eta_s)=0 \text{ for all } &\left((a_{\alpha, \beta})_{\alpha, \beta},(\theta_t)_{t},(\eta_s)_{t}\right) \in \AAA^{ \#\,PG}\times \AAA^{\bullet \, G(\Gamma_1, \Gamma_2)} \times \AAA^{\bullet \, G(\Gamma_2, \Gamma_3)} \\ 
\text{ such that }\quad g_i (a_{\alpha, \beta},\theta_s)& = 0\,\,\, \forall \,i=t_{1, d+k},\dots, t_{1, d+k+1}-2  \text{ and } \\
g'_i (a_{\alpha, \beta},\theta_t,\eta_s)& = 0 \,\, \forall \,i=t_{2, d+k'},\dots, t_{2, d+k'+1}-2,\,\, i\neq j.
\end{align*}
In other words the equation $g'_{j}$ is not necessary because trivially satisfied whenever $\deg \alpha_j = \deg \alpha_{l} -2$, even though it appears in the list (\ref{equationfor3G}), being $t_{2, d+k'}\leq j \leq t_{1, d+k+1}-2$. 
\end{lemma}
\begin{proof}
The proof is exactly as in the case of $M_{P((x,y), T_1, T_2, T_3)} $.
\end{proof}

\begin{observation}

Now we will deal with the cases left aside on $d, d'$ and $d''$. If $d+1=d'+1=d''$ then there is nothing new to prove, as observed in the proof of \ref{othercasesond}. If $d+1=d'=d''$ then we still need to prove that: if $\deg \alpha_j \leq \deg \alpha_l-2$ then $g'_0$ the first of the equations $g'_i$ $i=0, \dots, \star \, M(\Gamma_2, \Gamma_3)$ is trivially satisfied. But this is obvious since, in this case: 
\[
y^{p_0}f'_0+xf'_1 = yxf_0-xyf_0 = 0. 
\]
\end{observation}

\begin{proposition}
Let $T_1, T_2, T_3$ be three admissible sequences of nonnegative integers as in \ref{admissible}. Then the homogenous Hilbert-Samuel's stratum $G_{{T_1, T_2, T_3}}$ is a smooth variety of dimension
\begin{align*}
\dim\,\, G_{{T_1, T_2, T_3}} \,\, &=\,\,  \sum_{j\geq d }(t_{j-1}-t_j+1)(t_{j}-t_{j+1}) \,\,+\\ &+ \,\,(t_{m-1}-t_m-1) -(t_{m-2}-t_{m-1}) + (t_{m'-1}-t_{m'}) -(t_{m'-2}-t_{m'-1})\,\,+\\&+\,\, \begin{cases}1 &\text{ if } m'\geq m+2, \\0   &\text{ otherwise.}\end{cases}
\end{align*}
It has an affine cell decomposition with cells parametrized by nested Young diagrams $\Gamma_1, \Gamma_2, \Gamma_3$ that differ in only one box: $\Gamma_3=\Gamma_2 \sqcup \{\alpha_l\}=  \Gamma_1 \sqcup \{\alpha_l\}\sqcup \{\alpha_j\}$ with $\Gamma_1\vdash n$ and such that $T(\Gamma_i) = T_i$. The affine cell indexed by $ \left(\Gamma_1, \Gamma_2, \Gamma_3\right)$  has dimension given by the following formula: 
\begin{align}
\pos \left(\Gamma_1, \Gamma_2, \Gamma_3\right) &= \#\,\left\{(u,v) \in \Gamma_1 \left\vert \,\, h_{u,v} =-1 \right.\right\}\,\,+\,\, \bullet\,G^+(\Gamma_1, \Gamma_2)\,\, -\,\,\star\,G^+(\Gamma_1, \Gamma_2) \,\,+ \\ &\qquad+\,\,\bullet\,G^+(\Gamma_2, \Gamma_3)\,\, -\,\,\star\,G^+(\Gamma_2, \Gamma_3)\,\, +\,\,\diamond\, G^+(\Gamma_1, \Gamma_2, \Gamma_3).\nonumber
\end{align}
 The Poincar\'{e} polynomial of $G_{{T_1, T_2, T_3}}$ is 
\[
P_q\, \left( G_{{T_1, T_2, T_3}} \right) \quad = \quad \sum_{(\Gamma_1, \Gamma_2, \Gamma_3)\vdash [n, n+1, n+2], \,\,T(\Gamma_i)= T_i} \,\,\, q^{\pos \left(\Gamma_1, \Gamma_2, \Gamma_3\right)  }\,\,.
\]
\end{proposition}
\begin{proof}
Thanks to equation \ref{equationdimensionGT1T2T3} and Lemma \ref{starnotneeded} we know that the Hilbert-Samuel stratum $G_{{T_1, T_2, T_3}}$ has everywhere dimension equal to the dimension of its Zariski tangent space, then it is smooth. Thanks to Theorem \ref{Bialynicki-Birula} we know that it has a cell decomposition with cells whose dimensions are given by Lemma \ref{dimensiontangent123}. Thanks to Theorem \ref{Fulton} the decomposition in affine cells gives us a basis for the homology, and we can then calculate the Poincar\'{e} polynomial. 
\end{proof}
\section{Strata for longer flags are singular}
\hspace{3em}In this section we deal with flags of more than three ideals, i.e. starting at $\Hi^{n, n+1, n+2, n+3}(0)$. The goal is to prove that the Hilbert-Samuel's strata in these cases are not all smooth, so that the techniques we used until now will not yield results towards the understanding of their homology. In fact we will show a bit more, namely that there are attracting sets that are not smooth, in particular not isomorphic to affine cells. We will specify at what $n$, the length of the biggest ideal, these singular attracting sets start appearing, while for lower $n$ we still have an affine cell decomposition and a basis for the homology. Everything in this section is obtained by direct computation. 

\subsection{Four flag case}
Consider the punctual Hilbert scheme of flags of four nested ideals
\[
\Hi^{n,n+1, n+2, n+3} (0).
\] 
We will see that starting at $n=6$ we can find an Hilbert-Samuel's stratum
\[
M_{T_1, T_2, T_3, T_4} \subset \Hi^{6,7,8, 9} (0)
\]
that is not smooth. We will see that it is not smooth by exhibiting an attracting cell that is not smooth. We can also show, by direct computation that $6$ is the smallest integer for which this happens. 
\begin{lemma}
For $n= 1, 2, 3, 4, 5$ all the Hilbert-Samuel's strata $M_{T_1, T_2, T_3, T_4}$ of the Hilbert scheme $\Hi^{n,n+1, n+2, n+3} (0)$ are smooth. They have an affine paving that implies that $\Hi^{n,n+1, n+2, n+3} (0)$ has an affine paving. The Poincar\'{e} polynomials of the total spaces are given by: 
\begin{align*}
P_q \left(\Hi^{1,2, 3, 4} (0)\right) &= 1+3q+4q^2+2q^3,\\
P_q \left(\Hi^{2,3, 4, 5} (0)\right) &= 1+4q+8q^2+9q^3+4q^4,\\
P_q \left(\Hi^{3,4, 5, 6} (0)\right) &= 1+4q+10q^2+14q^3+13q^4+6q^5,\\
P_q \left(\Hi^{4,5, 6, 7} (0)\right) &= 1+4q+11q^2+22q^3+30q^4+25q^5+9q^6,\\
P_q \left(\Hi^{5,6, 7, 8} (0)\right) &= 1+4q+11q^2+24q^3+42q^4+51q^5+36q^6+11q^7.\\
\end{align*} 

For all $n\geq 6$ there is at least one Hilbert-Samuel's stratum that is not smooth and with an attracting set that is not isomorphic to an affine space.
\end{lemma} 

For the proof of the first part we checked by hand that the all the attracting sets for $n\leq 5$ are indeed affine, the computation is a bit long but straightforward. \\

\hspace{3em} We now show that for $n=6$ we have an attracting cell that is not isomorphic to an affine space, and thus proving that the corresponding Hilbert-Samuel's strata is not smooth. The fixed point is the following: \\

\begin{minipage}{0.3\textwidth}
\begin{center}
\ytableausetup{boxsize=1.5em, aligntableaux=bottom}
\begin{ytableau} 
3 \\
\,  \\
\,&2\\
\,&\, \\
\,&\, &1
\end{ytableau}
\end{center}
\end{minipage} 
\begin{minipage}{0.68\textwidth}
The fixed point represented on the left is the flag
\begin{align*}
(&x^2 , xy^2, y^4)\supset (x^3, x^2 y, xy^2, y^4)\supset \\
&\supset (x^3, x^2y , xy^3, y^4)\supset (x^3,x^2y , xy^3, y^5).
\end{align*} 
It is the flag of monomial ideals with normal patterns \\
$\quad
(1,2,2,1), (1,2,3,1), (1,2,3,2), (1,2,3,2,1).
$ 
\end{minipage} 

\vspace{1em}

The attracting cell is described as follow. Let $I_1=(f_0,f_1, f_2)$ be in $M_{P(T_1)}$ with standard generators $f_0, f_1, f_2$ given by 

\begin{minipage}{0.50\textwidth}
\begin{align*}
f_0 &= x^2+a_1xy+a_2y^2+a_3y^3, \\
f_1&= xy^2+b_1y^3,\\
f_2&=y^4. 
\end{align*}
\end{minipage}
\begin{minipage}{0.38\textwidth}
where $a_1, a_2, a_3$ and $b_1$ in $\CC$ are the free affine coordinates. 
\end{minipage}

The attracting cell is 
\begin{align}
\label{attractingcell4}
(f_0,f_1, f_2)\supset (xf_0, yf_0, f_1, &f_2)\supset (xf_0+\theta_0 f_1, yf_0+\theta_1 f_1, yf_1, f_2)\supset\\  &\supset(xf_0+\theta_0 f_1+\eta_0 y^4, yf_0+\theta_1 f_1+\eta_1 y^4, yf_1+\eta_2y^4, y^5)\nonumber
\end{align} 
where in the last step we need to impose the equation 
\[
G\left((a, b), (\theta), (\eta)\right)=0 \qquad\text{ for }\quad (a_1, a_2, a_3, b_1, \theta_0, \theta_1,\eta_0, \eta_1,\eta_2) \in \AAA^9.
\]
Here $G$ is defined with a procedure similar to the Procedure \ref{divisionprocedure} 
as 
\[
yg_0 - xg_1 \equiv G\left((a, b), (\theta), (\eta)\right) y^4 \quad \text{mod } \quad (g_1, g_2, g_3),
\]
where we called $g_0, g_1, g_2, g_3$ the generators (that need not be standard) of the last ideal in (\ref{attractingcell4}) in the order we have written them. Doing explicitly the computations we find 
\[
G\left((a, b), (\theta), (\eta)\right)= \eta_2(\theta_1(b_1-a_1)-\theta_0-\theta^2_1)+b_1\theta_1(b_1-a_1)+a_2\theta_1.
\]
Taking partial derivatives we see that the points 
\[
a_2= b_1(a_1-b_2), \theta_0=\theta_1=\eta_2=0, a_1, b_1, \eta_0, \eta_1 \in  \CC
\]
are all singular points for the equation $G$, so that the attracting cell is not isomorphic to an affine space. \\

\hspace{3em}Observe that we just proved that also the projective strata $G_{T_1, T_2, T_3, T_4}$ is not smooth, as the extra dimension of $M_{T_1, T_2, T_3, T_4}$, represented by the coordinate $a_3$, does not play a role in what we said.  If $n\geq 7$, the situation is slightly different as what we describe is a phenomenon that belongs to the affine fiber of the projective map $\rho_{T_1, T_2, T_3, T_4}: M_{T_1, T_2, T_3, T_4} \to  G_{T_1, T_2, T_3, T_4}$. The reasoning is however completely similar. We give the details for completness.  

\vspace{0.7em}
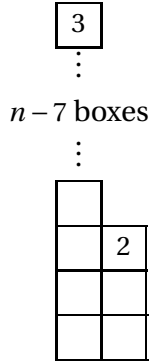
\begin{minipage}{0.3\textwidth}
\begin{center}
\ytableausetup{boxsize=1.5em, aligntableaux=bottom}
\begin{ytableau} 
3 \\
\none[\vdots] \\
\none[n-7 \text{ boxes}]\\
\none[\vdots] \\
\,  \\
\,&2\\
\,&\, \\
\,&\, &1
\end{ytableau}
\end{center}
\end{minipage} 
\begin{minipage}{0.68\textwidth}
The fixed point represented on the left is the flag
\begin{align*}
(&x^2 , xy^2, y^{n-2})\supset (x^3, x^2 y, xy^2, y^{n-2})\supset \\
&\supset (x^3, x^2y , xy^3, y^{n-2})\supset (x^3,x^2y , xy^3, y^{n-2}).
\end{align*} 
It is the flag of monomial ideals with normal patterns 
\begin{align*}
(1,2,2,1, 1, \dots, 1), (1,2,3&,1,1,  \dots, 1), (1,2,3,2,1 \dots, 1),\\& (1,2,3,2,1, \dots, 1, 1).
\end{align*}
\end{minipage} 

\vspace{1em}

The attracting cell is described as follow. Let $I_1=(f_0,f_1, f_2)$ be in $M_{P(T_1)}$ with standard generators $f_0, f_1, f_2$ given by 
\begin{align*}
f_0 &= x^2+a_1xy+a_2y^2+a_3y^3+\dots+a_{n-3}y^{n-3}, \\
f_1&= xy^2+b_1y^3+b_2y^5+\dots+b_{n-5}y^{n-3},\\
f_2&=y^{n-2}. 
\end{align*}
Now $a_1, a_{n-4}, a_{n-3}$ and $b_1, b_2, \dots, b_{n-5}$ in $\CC$ are the free affine coordinates, 
while $a_{2}, \dots,  a_{n-5}$ depend polynomially on the previous coordinates according to equations that arise imposing
\begin{equation}
\label{attracting5interactive}
y^2f_0-xf_1 \equiv 0 \text{ mod } (f_1, f_2).
\end{equation}
The attracting cell is 
\begin{align}
\label{attractingcell4bis}
(f_0,f_1, f_2)\supset (&xf_0, yf_0, f_1, f_2)\supset (xf_0+\theta_0 f_1, yf_0+\theta_1 f_1, yf_1, f_2)\supset\\  &\supset(xf_0+\theta_0 f_1+\eta_0 y^{n-2}, yf_0+\theta_1 f_1+\eta_1 y^{n-2}, yf_1+\eta_2y^{n-2}, y^{n-1})\nonumber
\end{align} 
where in the last step we need to impose the equation 
\[
G\left((a, b), (\theta), (\eta)\right)=0 \qquad\text{ for }\quad ((a_i)_{1\leq i\leq n-3}, (b_i)_{1\leq i\leq n-5}, \theta_0, \theta_1,\eta_0, \eta_1,\eta_2) \in \AAA^{2n-3}.
\]
Here $G$ is defined with a Procedure similar to \ref{divisionprocedure} as 
\[
yg_0 - xg_1 \equiv G\left((a, b), (\theta), (\eta)\right) y^{n-2} \quad \text{mod } \quad (g_1, g_2, g_3).
\]
We called $g_0, g_1, g_2, g_3$ the generators (that need not be standard) of the last ideal in (\ref{attractingcell4bis}) in the order we have written them. Doing explicitly the computations as in Procedure \ref{divisionprocedure}, and using the equations coming from (\ref{attracting5interactive}),  we find 
\[
G((a, b), (\theta), (\eta))= \eta_2(\theta_1(b_1-a_1)-\theta_0-\theta^2_1)+b_{n-5}(2\theta_1b_1+\theta_1a_1)+a_{n-3}\theta_1.
\]
Taking partial derivatives we see that the points 
\[
a_{n-3}= b_{n-5}(a_1-2b_1), \theta_0=\theta_1=\eta_2=0 
\]
are all singular points for the equation $G$, so that the attracting cell is not isomorphic to an affine space.

\subsection{Five, and longer, flag case}
Consider the punctual Hilbert scheme of flags of five nested ideals
\[
\Hi^{n,n+1, n+2, n+3, n+4} (0).
\] 
We will see that starting at $n=3$ we can find an Hilbert-Samuel's stratum 
\[
M_{T_1, T_2, T_3, T_4, T_5} \subset \Hi^{3,4,5, 6,7} (0)
\]
that is not smooth. Again we will see that it is not smooth by exhibiting an attracting cell that is not smooth. We can also show, by direct computation, that $3$ is the smallest integer for which this happens. 
\begin{lemma}
For $n= 1, 2$ all the Hilbert-Samuel's strata $M_{T_1, T_2, T_3, T_4, T_5}$ of the Hilbert scheme $\Hi^{n,n+1, n+2, n+3,n+4} (0)$ are smooth. The spaces $\Hi^{n,n+1, n+2, n+3,n+4} (0)$ for $n=1,2$ have an affine paving and their Poincar\'{e} polynomials are given by: 
\begin{align*}
P_q \left(\Hi^{1,2, 3, 4,5} (0)\right) &= 1+4q+8q^2+9q^3+4q^4,\\
P_q \left(\Hi^{2, 3, 4,5,6} (0)\right) &= 1+5q+13q^2+22q^3+23q^4+11q^5+q^6.\\
\end{align*} 
For $n\geq 3$ there is at least one Hilbert-Samuel's stratum that is not smooth, and with an attracting cell that is not isomorphic to an affine space. Let $n\geq 2$, then at least one of the Hilbert-Samuel's stratum of $\Hi^{n, \dots, n+5}(\CC^2)$ is not smooth, and has an attracting cell that is not isomorphic to an affine space. For all $k\geq 6$ and all $n\geq 1$ at least one of the Hilbert-Samuel's stratum of $\Hi^{n, \dots, n+k}(\CC^2)$ is not smooth, and has an attracting cell that is not isomorphic to an affine space.
\end{lemma} 

The only case left aside is $\Hi^{n, \dots, n+5}(\CC^2)$ for $n=1$. This is naturally isomorphic to $\Hi^{2,3,4,5,6}(\CC^2)$, so we know an affine paving for it. Again, to prove the first part of the Lemma  we wrote down explicitly the attracting sets and checked that they are all affine, in the cases where we claim that they are so. \\

\hspace{3em} We now show that for $n=3$ we have an attracting cell that is not isomorphic to an affine space, thus proving also that the corresponding Hilbert-Samuel's stratum is not smooth. 

\begin{example}
\label{exampleflag5}
The fixed point is the following: \\

\begin{minipage}{0.3\textwidth}
\begin{center}
\ytableausetup{boxsize=1.5em, aligntableaux=bottom}
\begin{ytableau} 
4 \\
3  \\
\,\\
\,&2 \\
\,&1
\end{ytableau}
\end{center}
\end{minipage} 
\begin{minipage}{0.68\textwidth}
The fixed point represented on the left is the flag
\begin{align*}
(x,y^3)\supset (x^2, &xy, y^3)\supset (x^2, xy^2, y^3)\supset \\
& \supset(x^2, xy^2, y^4)\supset (x^2, xy^2, y^5).
\end{align*} 
This is the flag of monomial ideals with normal patterns \\
$\quad (1,1,1), (1,2,1), (1,2,2), (1,2,2,1), (1,2,2,1,1)$
\end{minipage} 

\vspace{0.5em}
Let us call $(\omega, a, \theta, \alpha, \eta, \beta_1, \beta_2)$ the coordinates of  $\AAA^7$. The attracting set is isomorphic to the subspace $Y
\subset \AAA^7$ cut out by the equation $
(\eta-\omega)(\eta-\theta)=0$
i.e. it is the intersection of two linear spaces of dimension $6$ in a linear space of dimension $5$. The explicit isomorphism between $Y$ and $M_P$ is given by parametrizing $M_P$ as 
\begin{align*}
\begin{pmatrix}x+\omega y+a y^2\\y^3\end{pmatrix}&\supset \begin{pmatrix}x^2+\omega xy\\ xy+\omega y^2\\ y^3\end{pmatrix}\supset \begin{pmatrix}x^2+(\omega+\theta)xy+\omega\theta y^2\\ xy^2\\ y^3\end{pmatrix}\supset \\
& \supset\begin{pmatrix}x^2(\omega+\theta)xy+\omega\theta y^2+\alpha y^3\\ xy^2+\eta y^3\\ y^4\end{pmatrix}\supset \begin{pmatrix}x^2(\omega+\theta)xy+\omega\theta y^2+\alpha y^3+\beta_1 y^4\\ xy^2+\eta y^3+\beta_2 y^4\\ y^5\end{pmatrix}
\end{align*} 
and noticing that in the last inclusion we have to impose the equation 
\[
0=\det \begin{pmatrix}1 & \omega+\theta & \omega\theta & 0 & 0\\ 
0 & 1 & \omega+\theta & \omega\theta & 0\\
0 & 0&1 & \omega+\theta & \omega\theta \\
0 & 0 & 1 & \eta  &0\\
0 & 0 & 0 & 1 & \eta \\
\end{pmatrix} \,= \, (\eta-\omega)(\eta-\theta)=0
\]
because otherwise $y^4 \in I_5$. 
\end{example}
\begin{remark}
The fixed point we just described is also a fixed point of $\Hi^{1,2, 3, 4,5,6,7} (0)$,  in the sense that the flag of five ideals we described can be uniquely extended to a flag of seven ideals (combinatorially there is only one way to extend the above skew standard diagram to a Young standard diagram). Observe now that the projection map $p_2: \Hi^{1,2, 3, 4,5,6,7} (0) \to \Hi^{2}(0)$ is a Zariski locally trivial fibration, as $\text{GL}_2$ acts by automorphisms on $\Hi^{2}(0) \cong \PP^1$ by changing the coordinates of the plane. \\

\hspace{3em}One can then check that the attracting set described above in Example \ref{exampleflag5} is \emph{the only one} that is not an affine cell. The checking is done by writing explicitly all the remaining towhundredthirtyone cells. In particular all the sets attracted by fixed points of the form
\ytableausetup{boxsize=0.8em, aligntableaux=bottom}
\begin{ytableau} 
\none[\dots] \\
1&2 &\none[\dots]
\end{ytableau} 
 are affine cells. This mean that the fiber $p_2^{-1}((y, x^2)) \subset \Hi^{1,2, 3, 4,5,6,7} (0)$ has an affine paving, and thus, since $p_2$ is locally trivial on $\PP^1$, that the entire  $\Hi^{1,2, 3, 4,5,6,7} (0)$ has an affine paving, even though it is not the one coming from the attracting sets. Then we can compute by hand the Poincar\'{e} polynomial of $\Hi^{1,2, 3, 4,5,6,7} (0)$ and the result is the following: 
 \[
 P_q \left(\Hi^{1,2, 3, 4,5,6,7} (0)\right) = 1+6q+19q^2+41q^3+63q^4+64q^5+32q^6+5q^7.
 \]
 For longer flags, or for flags of five ideals but of higher lengths, unfortunately this no longer holds and I do not know if there is an affine cell decomposition.
\end{remark}

\begin{remark}
\label{positivepartlongercases}
One can still write a weight basis for the tangent spaces at the fixed points of $\Hi^{1,2, 3, 4,5,6,7} (\CC^2)$ and see if the weights with respect to the torus $\TT_{1^+}$ action have some geometrical meaning. The positive part, however, fails to give the right Poincar\'{e} polynomial, as direct computations show. This is why we didn't investigate it further for longer flag cases. 
\end{remark}

\hspace{3em}Now we show that for flags of ideals of higher lengths the analogous stratum is not smooth, slightly differently, but in a similar way. We will also show that all longer flags have the analogous stratum that is not smooth. 

\vspace{0.7em}
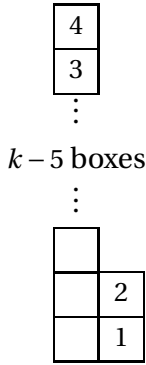
\begin{minipage}{0.3\textwidth}
\begin{center}
\ytableausetup{boxsize=1.5em, aligntableaux=bottom}
\begin{ytableau} 
4 \\
3  \\
\none[\vdots] \\
\none[k-5 \text{ boxes}]\\
\none[\vdots] \\
\,\\
\,&2 \\
\,&1
\end{ytableau} 
\end{center}
\end{minipage} 
\begin{minipage}{0.68\textwidth}
The fixed point represented on the left is the flag
\begin{align*}
(x,y^k)\supset (x^2, &xy, y^k)\supset (x^2, xy^2, y^k)\supset \\
& \supset(x^2, xy^2, y^{k+1})\supset (x^2, xy^2, y^{k+1})
\end{align*} 
that is the flag of monomial ideals with normal patterns
\begin{align*}\quad (1,1,\dots, 1), (1,2,1, \dots, 1), (1,2,2,1, \dots, 1),  \\
(1,2,2,1, \dots, 1, 1), (1,2,2,1, \dots, 1, 1, 1)
\end{align*}
\end{minipage} 

\vspace{0.5em}
Here we call $(\omega, a_1,\dots, a_{k-2}, \theta, \beta_1, \beta_2, \gamma_1, \gamma_2)$ the coordinates of  $\AAA^{k+4}$. Then the attracting cell is isomorphic to the subspace $Y
\subset \AAA^{k+4}$ cut out by the single equation \[
(a_{k_3}-\beta_2)(\omega-\theta)=0.\]
Thus for every $n\geq 3$ we have an Hilbert-Samuel's stratum in $\Hi^{n, n+1, n+2, n+3, n+4}(0)$ not smooth. 

\begin{observation} We now show that for all longer flags the analogous stratum is not smooth. Consider the above picture. There is only one way to complete it to a skew diagram with $k$ boxes, i.e. we need to put the number $1, 2, \dots, k-5$ in the first column. For the same reason there is only one attracting cell associated to this picture in $\Hi^{N-k, N-k+1,\dots, N}$ and its description is exactly as before. Then this proves that for every $k$ there is an Hilbert-Samuel's stratum in $\Hi^{N-k, N-k+1,\dots, N}$ with $k\leq N$ that is not smooth, thus showing the claim that the three step case is the last one where all the Hilbert-Samuel's strata are smooth. 
\end{observation}

\chapter{Generating function and $\Hi^{n, n+2}(\CC^2)$.}

\hspace{3em}The goal of this chapter is to prove the formula for the generating function of the Poincar\'{e} polynomials of $\Hi^{n, n+1, n+2}(0)$:
\begin{equation}
\label{generating123}
\sum_{n\geq 0} P_q\left(\Hi^{n,n+1, n+2}(0)\right) z^n = \frac{q+1}{(1-zq)(1-z^2q^2)}\,\,
\prod_{k\geq 1} \frac{1}{1-z^kq^{k-1}}.
\end{equation}
To do so we will need three ingredients.\\

\hspace{3em}\textbf{(1)} Study the positive part of the tangent spaces at fixed points of $\Hi^{n, n+1, n+2}(\CC^2)$ with respect to the torus $\TT_{\infty}$ i.e. the generic one dimensional torus that has weights $w_1, w_2$ such that $w_1<w_2$ and $1\ll \frac{w_1}{w_2}$. As it happens for previous cases the positive part with respect to this torus is much easier to understand combinatorially. Observe, however, that since $\Hi^{n, n+1, n+2}(\CC^2)$ is not smooth we do not know that the attracting sets for this torus character are affine varieties, nor do we have an immediate geometric interpretation of the positive part. This is why we need the following step. \\

\hspace{3em}\textbf{(2)} Prove that we can actually change the one dimensional torus and take weights as described in the previous point but still obtain the same result for the Poincar\'{e} polynomial. In the process we show that the attracting sets for this torus are still affine cells.\\

\hspace{3em}\textbf{(3)} Study the space $\Hi^{n, n+2}(\CC^2)$ and two associated spaces as described below. One of them is smooth, the generating function for its Betti numbers is known, and we claim it has the same Betti numbers as $\Hi^{n, n+1, n+2}(0)$. To prove the claim, we will only need to match the combinatorial data of the fixed points and the respective positive part of the tangent spaces to obtain the wanted result for $\Hi^{n, n+1, n+2}(0)$. This relation is clear only at the combinatorial level, and not at the geometric one. As a byproduct we will obtain the last result of the thesis i.e. a combinatorial formula for the Poincar\'{e} polynomials of $\Hi^{n, n+2}(0)$ and a closed expression for their generating function:
\begin{equation}
\sum_{n\geq 0 } P_q\, \left(\Hi^{n, n+2}(0) \right) z^n = \frac{1+q-qz}{(1-zq)(1-z^2q^2)}\,\, \prod_{m\geq 1} \frac{1}{1-z^mq^{m-1}}.
\end{equation} 

\section{A tale of two tori}
\hspace{3em}We start motivating the work of this section by showing that, indeed, the positive part of the tangent spaces at fixed points of $\Hi^{n, n+1, n+2}(\CC^2)$ for the torus with weights $1\ll \frac{w_1}{w_2}$ is combinatorially easier. Then we will prove that we can use the positive part with respect to this torus to write down the Poincar\'{e} polynomial. 

\begin{lemma}
Consider the action of $\TT_{\infty}$ on $R=\CC[x,y]$. Let $(\Gamma_1, \Gamma_2, \Gamma_3)$ be three nested Young diagrams representing a fixed point for this torus action on $\Hi^{n, n+1, n+2}(\CC^2)$. Then the positive part of the tangent space at $(\Gamma_1, \Gamma_2, \Gamma_3)$ is 
\begin{equation}
\label{posskew}
\text{pos}_{\infty}\left(\Gamma_1, \Gamma_2, \Gamma_3 \right) = \begin{cases} n+2-\ell(\Gamma_3)+1 & \text{ if } \deg_y \alpha'_l \geq \deg_y \alpha_j+2, \\
n+2-\ell(\Gamma_3) & \text{ otherwise. }
\end{cases}
\end{equation}
Here $\ell(\Gamma_3)$ is the number of columns of $\Gamma_3$ and with the notations of the last chapter we call $\alpha'_l$ and $\alpha_j$ the boxes such that 
\[
\Gamma_3=\Gamma_2\cup \{\alpha'_l\} = \Gamma_1 \cup \{\alpha_j\}  \cup \{\alpha'_l\}.
\]
\end{lemma}

\begin{proof}
Observe that an eigenvector for $\TT^2$ with eigenvalues $\lambda^r\mu^s$, where $\lambda$ and $\mu$ are independent torus characters, is positive with respect to the chosen weights if and only if $s$ is positive.  \\

\hspace{3em}Recall the definition of $B=B(I_{\Gamma_1}, I_{\Gamma_2}, I_{\Gamma_3})$ in \ref{definitionB(I1I2I3)}.\begin{align}
\label{eqbi1i2i3}
B\left(I_{\Gamma_1}, I_{\Gamma_2} , I_{\Gamma_3}\right) \quad &=\nonumber\\
& B\left(I_{\Gamma_1}\right) \setminus \left\{\ff \lvert (\alpha, \beta) \in \text{Obs}(I_{\Gamma_1}, I_{\Gamma_2})\right\} \setminus \left\{f_{\alpha, \beta} \lvert (\alpha, \beta) \in \text{PObs}(I_{\Gamma_1}, I_{\Gamma_2}, I_{\Gamma_3})\right\}  \nonumber\\
\cup\,\,& \left\{\left(0, h_{\alpha'_i,\alpha_j }, \Su(h_{\alpha'_i,\alpha_j })\right)\left\vert\,\, (\alpha'_i,\alpha_j) \notin \text{Obs}(I_{\Gamma_2}, I_{\Gamma_3})    \right.\right\} \cup \\\cup \,\,&\left\{\left(0, 0, h_{\alpha''_i,\alpha'_l }\right) \left\vert\,\,i=0,\dots s'' \right.\right\}\, . \nonumber
\end{align}

\hspace{3em}A vector $v\in B$ is positive with respect to the chosen weights, if and only if it represents a map from a box to a box that is in a strictly higher row. In terms of monomials, if $v$ is label by $\alpha\mapsto \beta$ then $v$ is positive if and only if $\deg_y \beta > \deg_y \alpha$. Let us analyze all the terms in (\ref{eqbi1i2i3}). 

\hspace{3em}It is clear, and we already know, that $\pos(B(\Gamma_1))= n - \ell(\Gamma_1)$. \\

\hspace{3em}Suppose for the moment that $\alpha'_l$ is not on the same row or column of $\alpha_j$. \\

\hspace{3em}Consider now what we exclude and what we add by removing the set of vectors $\{\ff \lvert (\alpha, \beta) \in$ $ \text{Obs}(I_{\Gamma_1}, I_{\Gamma_2})\} $ and by adjoining $\left\{\left(0, h_{\alpha'_i,\alpha_j }, \Su(h_{\alpha'_i,\alpha_j })\right)\left\vert\,\, (\alpha'_i,\alpha_j) \notin \text{Obs}(I_{\Gamma_2}, I_{\Gamma_3})    \right.\right\}$. In the latter set we do not need to exclude any term worrying about $(\alpha'_i,\alpha_j) \in \text{Obs}(I_{\Gamma_2}, I_{\Gamma_3})$,  thanks to the current hypothesis on the relative position of $\alpha'_l$ and $\alpha_j$. Looking at definition \ref{definitionObs23}, we have that $\ff$ is such that $(\alpha, \beta) \in  \text{Obs}(I_{\Gamma_1}, I_{\Gamma_2})$ and positive if and only if $\alpha = \alpha_i$, $\beta=\frac{\alpha_j}{y^{p_i}}$ with $i\leq j-1$ since we have that
\begin{equation}
\label{degalphaj}
\deg_y \alpha_j = \deg_y \alpha_i + \sum_{k=i}^{j-1}p_k
\end{equation}
for all $i<j$. Still looking at equation  (\ref{degalphaj}), we have that $h_{\alpha'_i,\alpha_j}$ is positive if and only if $i\leq j$. Thus the total contribution of these two sets to the positive part of $B(\Gamma_1, \Gamma_2, \Gamma_3)$ is
\begin{align*}
\pos \left( \left\{\left(0, h_{\alpha'_i,\alpha_j }, \Su(h_{\alpha'_i,\alpha_j })\right)\left\vert\,\, (\alpha'_i,\alpha_j) \notin \text{Obs}(I_{\Gamma_2}, I_{\Gamma_3})    \right.\right\} \right) +&\\
-\pos \left( \left\{\ff \lvert (\alpha, \beta) \in \text{Obs}(I_{\Gamma_1}, I_{\Gamma_2})\right\} \right)
 &\quad= \quad \begin{cases}1& \text{ if } 0<j,\\
0 &\text{ if } 0=j.
\end{cases}
\end{align*}
Completely analogously for the contribution of $\left\{f_{\alpha', \beta} \lvert (\alpha', \beta) \in \text{Obs}(I_{\Gamma_2}, I_{\Gamma_3})\right\} $ when we remove it, and of $\left\{\left(0, 0, h_{\alpha''_i,\alpha'_l }\right) \left\vert\,\,i=0,\dots s'' \right.\right\}$ when we add it. \\

\hspace{3em}The only term left to be considered is $\left\{f_{\alpha', \beta} \lvert (\alpha', \beta) \in \text{NotP}(I_{\Gamma_1}, I_{\Gamma_2}, I_{\Gamma_3})\right\}$. This, under the current hypothesis that  $\alpha'_l$ is not on the same row or column as $\alpha_j$, contains only an element given by Definition \ref{definitionNotP}, and this element is positive with respect to the chosen torus action if and only if $\deg_y \alpha_l > \deg_y\alpha_j + 2$. 

Then noticing that  
\[
\ell(\Gamma_3) = \ell(\Gamma_1) \quad +\quad  \begin{cases} 1 & \text{ if } 0<j,\\  0 & \text{ if } 0=j\end{cases} \quad +\quad  \begin{cases} 1 & \text{ if } 0<l,\\  0 & \text{ if } 0=l\end{cases} \quad 
\]
we obtain exactly what we wanted. \\

\hspace{3em}If we now suppose that $\alpha'_l=y\alpha_j$ we can say, analogously, that
\begin{align}
\label{equationsome1}
\left( \left\{\left(0, h_{\alpha'_i,\alpha_j }, \Su(h_{\alpha'_i,\alpha_j })\right)\left\vert\,\, (\alpha'_i,\alpha_j) \notin \text{Obs}(I_{\Gamma_2}, I_{\Gamma_3})    \right.\right\} \right) +&\nonumber\\
+ \pos\left(\left\{\left(0, 0, h_{\alpha''_i,\alpha'_l }\right) \left\vert\,\,i=0,\dots, s'' \right.\right\} \right) +\nonumber\\
-\pos \left( \left\{\ff \lvert (\alpha, \beta) \in \text{Obs}(I_{\Gamma_1}, I_{\Gamma_2})\right\} \right)+\\
- \pos \left( \left\{f_{\alpha', \beta} \lvert (\alpha', \beta) \in \text{Obs}(I_{\Gamma_2}, I_{\Gamma_3})\right\}\right)+ \nonumber\\
+ \pos \left( \left\{f_{\alpha', \beta} \lvert (\alpha', \beta) \in \text{NotP}\right\}\right)
 &\quad= \quad \begin{cases}2 & \text{ if }\,\,\,\, 0<j,\\
1 &\text{ if } \,\,\,\,0=j \nonumber
\end{cases}
\end{align}
since in this case $(\alpha'_i,\alpha_j) \notin \text{Obs}(I_{\Gamma_2}, I_{\Gamma_3}) $ if and only if  $(\alpha'_i, \frac{\alpha_j}{y}) \in \text{NotP}(I_{\Gamma_1}, I_{\Gamma_2}, I_{\Gamma_3})$  (notice also that, since $\alpha'_l=y\alpha_j$, we always have at least a positive vector as $h_{\alpha''_0, \alpha'_l}$). Then the statement is proved also in this case. \\

\hspace{3em}Finally if $\alpha'_l=x\alpha_j$ then we do not have any positive vectors in either \\ $\left\{h_{\alpha'_i,\alpha_j }\left\vert\,\, (\alpha'_i,\alpha_j) \notin \text{Obs}(I_{\Gamma_2}, I_{\Gamma_3})\right. \right\}$ or in $\left\{f_{\alpha', \beta} \lvert (\alpha', \beta) \in \text{NotP}(I_{\Gamma_1},I_{\Gamma_2}, I_{\Gamma_3})\right\}$. So that the left hand side of (\ref{equationsome1}) is either $1$ if $0<j$ or $0$ if $j=0$, that, noticing $j=0$ iff $\ell(\Gamma_3)= \ell(\Gamma_1)+2$, is what we wanted. 
\end{proof}

\hspace{3em}As promised formula (\ref{posskew}) is much easier than the formula (\ref{posskewmale}). 
Now we deal with the more cumbersome task of proving that we get the same polynomial by using one or the other. We will see that the main part of the work is to understand in more details why we can use freely one formula instead of the other in the case of $\Hi^{n}(0)$ and in the case of $\Hi^{n, n+1}(0)$. Then we will be able to tackle the  case $\Hi^{n, n+1, n+2}(0)$. We start by fixing some notation to better frame the problem.

\begin{definition}
We want to give names to the one dimensional subtori of $\TT^2$ that are most relevant for us. We will always assume that the weights $w_1, w_2$ are such that $w_1< w_2$. We will denote $\TT_W \subset \TT^2$ the one dimensional subtorus given by putting $\frac{w_2}{w_1}=W$. Most of these tori are \emph{generic} in the sense that the fixed points sets for their action on $\Hi^{n}(\CC^2)$, $\Hi^{n+1}(\CC^2)$ and $\Hi^{n+2}(\CC^2)$ are isolated. However we care also about those special values of $W$ that are not \emph{generic}. So we say that $W \in \QQ$ is a \emph{wall} for $n$ if the fixed point set of $\Hi^{n}(\CC^2)$ is not discrete. Similarly we say that $W \in \QQ$ is a \emph{wall} for $[n, n+1]$ if the fixed point set of $\Hi^{n, n+1}(\CC^2)$ is not discrete. Finally we say that $W \in \QQ$ is a \emph{wall} for $[n, n+1, n+2]$ if the fixed point set of $\Hi^{n, n+1, n+2}(\CC^2)$ is not discrete. This typically happens for values like: $\frac{m}{k}$ with $1\leq m,k \leq n+1$. \\

\hspace{3em}If $W$ is a wall for $n$ (resp. for $[n, n+1]$, resp. for $[n,n+1,n+2]$) we will denote $W^+$ and $W^-$ two values in $\QQ$ that are respectively bigger and smaller than $W$ by a small amount, small enough that between $W^+$ and $W$ and between $W$ and $W^-$ there are no other walls for $n$ (resp. for $[n, n+1]$, resp. for $[n,n+1,n+2]$). \\

\hspace{3em} Recall the special names we gave the two extreme cases we are interested in with respect to $n\in \NN$:
\[
\TT_{\infty} := \begin{cases} w_1< w_2 & \text{\small{ and}} \\
n+2 < \frac{w_2}{w_1} &\,
\end{cases} \qquad \qquad \TT_{1^+} := \begin{cases} w_1< w_2 & \text{ \small{and}} \\
\frac{n+2}{n+1} > \frac{w_2}{w_1}. &\,
\end{cases}
\] 
These two tori are generic. The torus $\TT_{\infty}$ is the one we would like to use, as it gives formula (\ref{posskew}). However only the second one i.e. $\TT_{1^+}$ is such that the Hilbert-Samuel strata $M_{T}$ are union of attracting sets of $\Hi^{\mathfrak{n}}(\PP^2)$, and so, in the case of $\Hi^{n, n+1, n+2}(0)$, only for $\TT_{1^+}$ we know that the attracting sets are affine and that the positive parts of the tangent spaces at a fixed point give the dimension of the affine cells that contracts to it.
\end{definition}

\begin{example} 
\label{examplewall1}
A typical wall for $n$ happens for value like: $\frac{m}{k}$ with $1\leq m,k \leq n+1$. 

For example: let $n=3$, then the only wall is $W=2$. For any generic torus the fixed point sets  of  $\Hi^3(0)$ consists of the three isolated points given by the monomial ideals of length 3, i.e. $(x, y^3), (y, x^3)$ and $(x^2, xy, y^2)$. If  $W=2$ then the fixed points set consists of two components, one with all the points of the form $(\omega_1 y +\omega_2 x^2)+\mathfrak{m}^3$ with $[\omega_1, \omega_2] \in \PP^1$ and the other with the isolated point $(x, y^3)$. 

\begin{minipage}{0.6\textwidth}
\begin{center}
\ytableausetup{boxsize=1.3em, aligntableaux=bottom}
\begin{ytableau} 
\none[y^2]\\
\, &\none[xy]\\
\,&\, &\none[x^2]\\
\end{ytableau}
\begin{ytableau} 
\none[]&\none[]&\none[]&\none[]&\none[]\\
\none[]&\none[]&\none[\omega_1 y]&\none[+]&\none[\omega_2 x^2]&\none[]&\none[]\\
\none[]&\none[]&\none&\none[\longrightarrow]&\none[]&\none[]
\\\end{ytableau}
\begin{ytableau} 
\none[y]\\
\,&\, &\,  &\none[x^3]
\end{ytableau}
\end{center} 
\end{minipage} 
\begin{minipage}{0.39\textwidth}
\small{The one dimensional fixed components of the torus $\TT_{2}$ acting on $\Hi^3(0)$.}
\end{minipage}

\end{example}

\begin{example}
Even when we know that a torus action gives an affine cell decomposition, we cannot change freely the one parameter subgroup, hoping to still get an affine cell decomposition, or hoping to extract useful informations by looking at the positive part of the tangent spaces at the fixed points. We give an example. Consider $X$ the variety that consists of three $\PP^1$ that meet at the point $P$ so that the intersection is planar, i.e. $\dim T_P X = 2$. Suppose that $P$ is the point $[1:0]$ for each of the $\PP^1$ and that all other points of $X$ are smooth. Consider a $\TT^1_a$ action on $X$ that restricts to each $\PP^1$ as the usual rescaling action with weights $1$ and $a$, i.e.  $t\cdot [\omega_1:\omega_2] = [t \omega_1: t^a \omega_2]$. Then, if $a<1$, the $\TT^1_a$ action gives a decomposition of $X$ in affine cells parametrized by fixed points $P=P_1, P_2, P_3, P_4$ and given by  $A_i=\left\{x\in X \left \vert\,\, \lim_{t\to 0} t\cdot x = P_i \right\}\right.$, each of which has dimension given by the positive part of the tangent space at the fixed point. However if $a>1$ this is not longer true, as the only attracting set that is not a point is $A_1$ and the positive part of the tangent space at $P_1=P$ is only two dimensional thanks to the planar hypothesis. 
\end{example}

\begin{notation} 
Denote with $\Delta$ the positive quadrant $\NN\times \NN$ seen as a union of boxes. In the rest of the chapter we denote $\Gamma \vdash n$,  $(\Gamma_1, \Gamma_2) \vdash [n, n+1]$ (or the shorter version $\Gamma \vdash [n, n+1]$) and  $(\Gamma_1, \Gamma_2, \Gamma_3) \vdash [n, n+1, n+2]$ (or the shorter version $\Gamma \vdash [n, n+1, n+2]$) respectively Young diagram of size n, a couple of Young diagrams of size $n$ and $n+1$ that differ only in one box marked with a\, \ytableausetup{boxsize=0.9em, aligntableaux=bottom}\begin{ytableau} *(green)1\end{ytableau} , and a triple of Young diagrams of sizes $n,n+1$ and $n+2$ that differ, in order, in one box marked with a\,  \ytableausetup{boxsize=0.9em, aligntableaux=bottom}\begin{ytableau} *(green)1\end{ytableau} , and in one box marked with a\, \ytableausetup{boxsize=0.9em, aligntableaux=bottom}\begin{ytableau} *(cyan)2\end{ytableau} . We identify these with fixed points of, respectively,  $\Hi^{n}(\CC^2)$, $\Hi^{n,n+1}(\CC^2)$ or $\Hi^{n, n+1, n+2}(\CC^2)$. In this chapter all the objects we associated to the latter we can consider as associated to the diagrams themselves. For example we will often talk about the ideal $\Gamma$ to mean the ideal $I_{\Gamma}$. Also, for example, given $\Gamma\vdash n$ we will talk about its outer corners as the standard monomial generators of the associated monomial ideal $I_{\Gamma}$, and we will denote $B(\Gamma)$ the basis of the tangent space $T_{\Gamma} \Hi^{n}(\CC^2)$ constructed in Chapter 2, and so on. 
\end{notation}
\begin{definition}
\label{definitionposw}
Let $n\in \NN$ and $W$ be a wall for $n$ or $[n, n+1]$ or $[n, n+1, n+2]$. Then we define for $\Gamma\vdash n$ or $\Gamma\vdash [n,n+1]$ or $\Gamma\vdash [n, n+1, n+2]$ the following integers. 
\begin{align*}
\text{pos}_{W^+}\left(\Gamma \right)&:= \# \left\{ v\in B(\Gamma) \left\vert v \text{ is positive  wrt to the weights of } \TT_{W^+} \right. \right\} \\
\text{pos}_{W^-}\left( \Gamma \right)&:= \#\left\{ v\in B(\Gamma) \left\vert v \text{ is positive  wrt to the weights of } \TT_{W^-} \right.\right\} \\
\text{pos}_{W}\left(\Gamma \right)&:= \#\left\{ v\in B(\Gamma) \left\vert v \text{ is positive  wrt to the weights of } \TT_{W}\right.\right\}\\
\text{s}^+_{W}\left(\Gamma \right)&:= \#\left\{ v\in B(\Gamma) \left\vert v \text{ has zero weight wrt to} \TT_{W} \text{ and is positive wrt to } \TT_{W^+}\right.\right\}\\
\text{s}^-_{W}\left(\Gamma \right)&:= \#\left\{ v\in B(\Gamma) \left\vert v \text{  has zero weight wrt to } \TT_{W} \text{ and is positive wrt to } \TT_{W^-}\right.\right\}
\end{align*}
We also use the same operators to count the corresponding numbers of vectors of any subset of one of the $B(\Gamma)$.
\end{definition}

\begin{proposition}
\label{changingtorus}
Let $n\in \NN$ and let $W$ be a wall for $n$. Then for every $k \in \NN$ we have 
\[
\#\left\{ \Gamma \vdash n\left\vert \pos_{W^+}\left(\Gamma \right) = k  \right. \right\} = \# \left\{ \Gamma \vdash n\left\vert \pos_{W^-}\left(\Gamma \right) = k  \right. \right\} .
\]
In particular we have that the following two polynomials are the same: \[
\sum_{\Gamma \,\vdash\, n } q^{\pos_{\infty}\left(\Gamma\right)} = \sum_{\Gamma \,\vdash\, n } q^{\pos_{1^+}\left(\Gamma\right)}\,\,.
\]
\end{proposition} 
\begin{remark}
The statement is equivalent to the existence of a bijection of sets 
\[
\phi_W : \left\{ \Gamma \vdash n \right\} \to \left\{ \Gamma \vdash n \right\}
\]
such that $\pos_{W^+}(\Gamma) =\pos_{W^+}(\phi_W(\Gamma))$. We do not claim that there is a preferred such bijection, as Example \ref{orbitn10} shows.\\

\hspace{3em}As already said we could prove this fact simply by using smoothness of $\Hi^{n}(\CC^2)$. Smoothness, as we saw, implies that every generic one dimensional torus gives a decomposition of $\Hi^{n}(0)$ in affine cell whose dimensions are given by the positive parts of the tangent spaces at the attracting fixed points. \\

\hspace{3em}However we need to understand better what is going on in order to prove a similar statement for $\Hi^{n, n+1, n+2}(0)$. In particular we will introduce some combinatorial transformations on the set $\left\{ \Gamma \vdash n \right\}$ so to give a better description of $\phi_W$ and prove that a similar map, with similar properties exists for the set $\left\{ \Gamma \vdash [n, n+1, n+2] \right\} $. In \ref{proofofchangingtorus} we will give a more ad hoc proof of Proposition \ref{changingtorus}. Before introducing the definition we present an example. 
\end{remark}

\begin{example}
Consider the case $n=3$. There is only one wall $W=2$. We depict below the dimension of the positive parts (for the notation cfr. Definition \ref{definitionposw} ) at the fixed points with respect to the two different tori $\TT_{W^+} = \TT_{\infty}$ and $\TT_{W^-} = \TT_{1^+}$. 
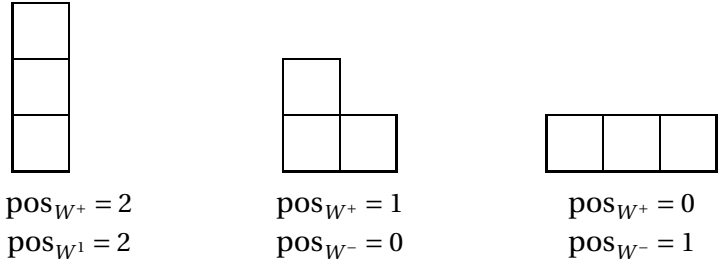
\begin{equation*}
\begin{matrix}
\ytableausetup{boxsize=1.9em, aligntableaux=bottom}
\begin{ytableau} 
\,\\
\, &\none[]\\
\, &\none[]
\end{ytableau}
\qquad\qquad
&
\ytableausetup{boxsize=1.9em, aligntableaux=bottom}
\begin{ytableau} 
\none[]\\
\, &\none[]\\
\, &\,
\end{ytableau}
\qquad\qquad
&
\ytableausetup{boxsize=1.9em, aligntableaux=bottom}
\begin{ytableau} 
\none[]\\
\none[] &\none[]\\
\,&\, &\,
\end{ytableau}\\
\begin{matrix}
\text{pos}_{W^+}=2\\
\text{pos}_{W^1}=2
\end{matrix}
\qquad\qquad
&
\begin{matrix}
\text{pos}_{W^+}=1\\
\text{pos}_{W^-}=0
\end{matrix}
\qquad\qquad
&
\begin{matrix}
\text{pos}_{W^+}=0\\
\text{pos}_{W^-}=1
\end{matrix}

\end{matrix}
\end{equation*}

The two polynomials 
\[
q^2 +q+1=\sum_{\Gamma \,\vdash\, n } q^{\pos_{\infty}\left(\Gamma\right)} = \sum_{\Gamma \,\vdash\, n } q^{\pos_{1^+}\left(\Gamma\right)} = q^2+1+q
\]
are clearly the same. We will see that in passing the wall $W$ the last two fixed points "exchange roles". In more details: observe that there are two tangent vectors of weight $W$ that are:
\[
f_{x^2, y} \in B\left(\ytableausetup{boxsize=1em, aligntableaux=bottom}
\begin{ytableau} 
\, &\none[]\\
\,&\, 
\end{ytableau}\right)
\qquad \text{ and } \qquad 
f_{y, x^2} \in B\left(\ytableausetup{boxsize=1em, aligntableaux=bottom}
\begin{ytableau} 
\none[] &\none[]\\
\,&\,&\,
\end{ytableau}\right).
\]
These are the two vectors that are, respectively, positive only for $\TT_{W^+}$ and positive only for $\TT_{W^-}$. They are also the tangent vectors to the $\TT_{W}$ fixed $\PP^1$ that connects the corresponding two fixed points: see Example \ref{examplewall1}. The plan is the following: we will define combinatorial transformations of Young diagrams associated to these two vectors that keep track of how the dimensions of the positive parts change in crossing the wall. For example the transformation associated to the vector $f_{x^2, y}$ is $T_{x^2, y}$:

\begin{center}

\ytableausetup{boxsize=1.3em, aligntableaux=bottom}
\begin{ytableau} 
\none[]\\
\, &\none[]\\
\,&\, &\none[]\\
\none[\pos_W= 0, (s^+_W, s^-_W) = (1,0)]
\end{ytableau}
\begin{ytableau} 
\none[]&\none[]&\none[]&\none[]&\none[]\\
\none[]&\none[]&\none[]&\none[T_{x^2, y}]&\none[]&\none[]\\
\none[]&\none[]&\none&\none[\longrightarrow]&\none[]&\none[]\\
\none[]&\none[]&\none[]&\none[]&\none[]
\end{ytableau}
\begin{ytableau} 
\none[]\\
\,&\, &\,  &\none[]\\
\none[]&\none[\pos_W= 0, (s^+_W, s^-_W) = (0,1)]
\end{ytableau}
\end{center}  
\end{example}

\begin{definition}[Sliding boxes]
\label{sliding}
Let $n\in \NN$, $W$ be a wall for $n$ and $\Gamma\vdash n$. For every $\ff \in B(\Gamma)$ that has zero weight with respect to $\TT_{W}$ we want to define a new partition $T_{\alpha, \beta}(\Gamma)$ of $n$ obtained from $\Gamma$ by sliding down (or up) some boxes in the opposite direction of the translation $\alpha \mapsto \beta$. 
\begin{example}[$W=\frac{4}{2}$]
$\,$
\vspace{1em}

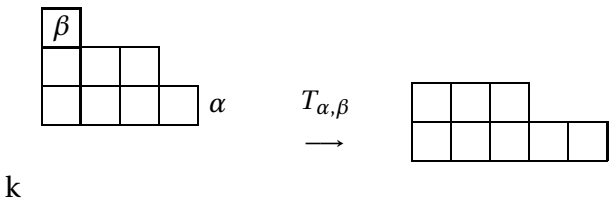
\begin{minipage}{0.24\textwidth}
\begin{center}
\ytableausetup{boxsize=1.3em, aligntableaux=bottom}
\begin{ytableau} 
\beta \\
\, &\, &\,  \\
\,&\, &\, &\, &\none[\alpha] \\
\end{ytableau}
\end{center}k
\end{minipage} 
\begin{minipage}{0.08\textwidth}
\begin{center}
\ytableausetup{boxsize=1.4em, aligntableaux=bottom}
\begin{ytableau} 
\none[]\\
\none[T_{\alpha, \beta}]\\
\none[\longrightarrow]
\\\end{ytableau}
\end{center} 
\end{minipage} 
\begin{minipage}{0.24\textwidth}
\begin{center}
\ytableausetup{boxsize=1.3em, aligntableaux=bottom}
\begin{ytableau} 
\none[]\\
\, &\, &\,  \\
\,&\, &\, &\, &\, \\
\end{ytableau}
\end{center} 
\end{minipage} 
\begin{minipage}{0.41\textwidth}
\small{Here only one of the boxes of $\Gamma$ is involved and the sliding is easily described. In other cases we need to slide in many steps. See Example \ref{examplesliding2}} below.
\end{minipage}
\end{example}
Recall that by definition $\ff \in \text{Hom}_R(I, \bigslant{R}{I})$ with $I = I_{\Gamma}$ and $\alpha$ is a minimal generator for $I$ and $\beta \in \Gamma$ is such that either $\beta \in P_{\alpha}$ or $\beta \in Q_{\alpha}$. The Definition is at \ref{definitionf}.  We consider $\ff$ as a set theoretic map of boxes of $\Delta$. 

\hspace{3em}Call $l$ the horizontal distance in number of boxes between $\beta$ and $\alpha$, and $m$ the vertical distance in number of boxes between $\alpha$ and $\beta$, where a negative distance means that we move either to the left or downwards, so that in Laurent monomials $\frac{y^m}{x^l} \alpha = \beta$ and in terms of boxes $\beta = \alpha + (-l, m)$. The choice is such that if $\beta \in P_{\alpha}$ then $l>0$ and $m>0$ whereas if $\beta \in Q_{\alpha}$ $l<0$ and $m<0$.   Define 
\[
B^1 = \left\{\gamma \in \Gamma \left\vert\,\, \exists \,\,\delta \in \Delta \setminus \Gamma \text{ with } \ff(\delta) = \gamma \right.  \right\} \quad \text{and } \quad B^1=\bigsqcup_{i} B^1_i
\]
where the $B^1_i$'s are the connected components of $B^1$, and two boxes are connected if they share a side or a corner. We start numbering from top to bottom if $\beta \in P_{\alpha}$ and from bottom to top if $\beta \in Q_{\alpha}$. Call $\alpha^1$ the generators of $I_{\Gamma}$ such that $\ff(\alpha^1 ) \in B^1_1$ in the lowest row of $B^1_1$ if $\beta \in P_{\alpha}$, resp. in the highest row of $B^1_1$ if $\beta \in Q_{\alpha}$. Call $\beta^1= \ff(\alpha^1)$.

\hspace{3em}We define $T_{\alpha, \beta}(\Gamma)$ by performing different steps. Suppose now, for the rest of the definition that $\beta \in P_{\alpha}$. The case $\beta \in Q_{\alpha}$ is completely analogous, in fact one can pass to the transpose Young diagram, perform $T_{\alpha, \beta}$ there, and re-transpose everything to get the result. 

At the step $1$ we slide all of $B^1_1$ by $(l, -m)$. Call what we obtain 
\[
T_{\alpha, \beta}^1(\Gamma) = \Gamma\setminus B^1_1 \sqcup \left\{\gamma+ (l, -m) \vert \gamma \in B^1_1 \right\}.
\]
If $T_{\alpha, \beta}^1(\Gamma)$ is Young diagram we stop and we define it $T_{\alpha, \beta}(\Gamma)$. Otherwise we look at all $\gamma \in B^1_1 +(l, -m)$ that are in a row strictly above $\beta^1+(l, -m)$ and such that $\gamma +(l, -m)$ is not already in $B^1_1+(l,-m)$ and slide all of these by adding again $(l,-m)$. We call what we obtain $T_{\alpha, \beta}^2(\Gamma)$, if it is a Young diagram we stop, otherwise we continue until all the elements of $B_1$ have been slid down the maximum, but always staying above $\alpha^1$. Call this step $T_{\alpha, \beta}^{s_1}(\Gamma)$. If $T_{\alpha, \beta}^{s_1}(\Gamma)$ is a Young diagram, we stop. Otherwise we restart from the generator immediately below of $\alpha^1$ call it $\alpha^2$ and call $\beta^2 = \ff(\alpha^2)$. We define: 
\[
B^2 = \left\{\gamma \in T_{\alpha, \beta}^1(\Gamma) \left\vert\,\,\deg_y(\gamma)<\deg_y (\beta_1)\text{ and } \exists \,\,\delta \in \Delta \setminus T_{\alpha, \beta}^1(\Gamma) \text{ with } \ff(\delta) = \gamma \right.  \right\} \text{, } \,\, B^2=\bigsqcup_{i } B^2_i
\]
 where $B^2_i$ are the connected components. We perform all the previous steps for $B^2_1$, stopping the first time we get a Young diagram. Again if we do not get a Young diagram when all the boxes above $\beta^2$ have been slid down the most possible, we define $\alpha^3$ and $\beta^3$ and $B^3$ as in the step $2$ and keep going until you do reach a Young diagram. Lemma \ref{slidingexists} makes sure that this procedure actually creates a Young diagram $T_{\alpha, \beta}(\Gamma)$.  
\begin{example}[$W=\frac{4}{2}$]
\label{examplesliding2}

A more complex example. Here we need to perform few steps before getting to a Young diagram. 

\vspace{1em}
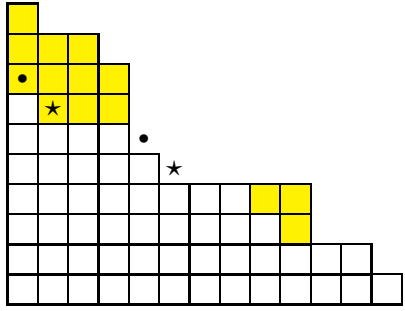
\begin{minipage}{0.4\textwidth}
\begin{center}
\ytableausetup{boxsize=1em, aligntableaux=bottom}
\begin{ytableau} 
*(yellow)\, \\
*(yellow) &*(yellow)\, &*(yellow)\, \\
*(yellow)\bullet &*(yellow) \,&*(yellow)\, &*(yellow)\, \\
\,&*(yellow) \star&*(yellow)\, &*(yellow)\, \\
\,& \,&\, &\,&\none[\bullet] \\
\,&\,& \,&\, &\,&\none[\star] \\
\,&\,& \,&\, &\, &\, &\, &\, &*(yellow)\,&*(yellow) \,\\
\,&\,& \,&\, &\, &\, &\, &\, &\,&*(yellow) \\
\,&\,& \,&\, &\, &\, &\, &\, &\,&\, &\,&\, \\
\,&\, &\, &\, &\, &\, &\,&\, &\,&\, &\,&\,&\,\\
\end{ytableau}

\small{Here $(\alpha, \beta)=(\bullet, \bullet)$. $B^1$ is in yellow, $B^1_1$ is the upper component of it. $(\alpha^1, \beta^1)=(\star, \star)$. }
\end{center}
\end{minipage} 
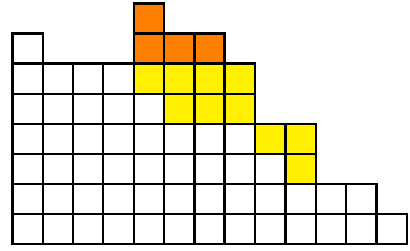
\begin{minipage}{0.08\textwidth}
\begin{center}
\ytableausetup{boxsize=1.4em, aligntableaux=bottom}
\begin{ytableau} 
\none[]\\
\none[T^1_{\bullet \bullet}]\\
\none[\longrightarrow]
\\\end{ytableau}
\end{center} 
\end{minipage} 
\begin{minipage}{0.4\textwidth}
\begin{center}
\ytableausetup{boxsize=1em, aligntableaux=bottom}
\begin{ytableau} 
\none[] \\
\none[] &\none[] &\none[] \\
\none[]&\none[]&\none[] &\none[] &*(orange)\,\\
\,&\none[] &\none[] &\none[] &*(orange) &*(orange)\, &*(orange)\,\\
\,& \,&\, &\,&*(yellow) &*(yellow) \,&*(yellow)\, &*(yellow)\, \\
\,&\,& \,&\, &\,&*(yellow) &*(yellow)\, &*(yellow)\, \\
\,&\,& \,&\, &\, &\, &\, &\, &*(yellow)\,&*(yellow) \,\\
\,&\,& \,&\, &\, &\, &\, &\, &\,&*(yellow) \\
\,&\,& \,&\, &\, &\, &\, &\, &\,&\, &\,&\, \\
\,&\, &\, &\, &\, &\, &\,&\, &\,&\, &\,&\,&\,\\
\end{ytableau}

\small{$T^1_{\bullet, \bullet}(\Gamma)$ is not a Y diagram. To do the next step we need to slide down the orange part.}
\end{center}
\end{minipage} 
\begin{minipage}{0.08\textwidth}
\begin{center}
\ytableausetup{boxsize=1.4em, aligntableaux=bottom}
\begin{ytableau} 
\none[]\\
\none[T^2_{\bullet \bullet}]\\
\none[\longrightarrow]
\\\end{ytableau}
\end{center} 
\end{minipage}

\vspace{1em}

\begin{minipage}{0.4\textwidth}
\begin{center}
\ytableausetup{boxsize=1em, aligntableaux=bottom}
\begin{ytableau} 
\,&\none[] &\none[] &\none[] \\
\,& \,&\, &\,&*(white) &*(white) \,&*(cyan)\, &*(cyan)\,&*(cyan)\, \\
\,&\,& \,&\, &\,&*(white) &*(cyan)\beta^2 &*(cyan)\, &*(cyan) &*(cyan)\, &*(cyan)\,\\
\,&\,& \,&\, &\, &\, &\, &\, &*(white)\,&*(white) \,\\
\,&\,& \,&\, &\, &\, &\, &\, &\,&*(white)&\none[\alpha^2] \\
\,&\,& \,&\, &\, &\, &\, &\, &\,&\, &\,&\, \\
\,&\, &\, &\, &\, &\, &\,&\, &\,&\, &\,&\,&\,\\
\end{ytableau}

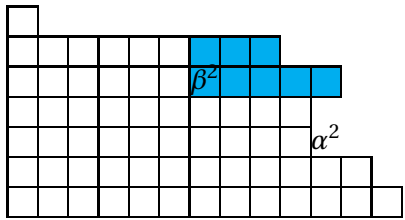
\small{This is $T^2_{\bullet \bullet}(\Gamma)=T^{s_1}_{\bullet \bullet}(\Gamma)$. Since it is still not a Y diagram we need to define $B^2$, in light blue, and slide $\beta^2$ to $\alpha^2$ .}
\end{center}
\end{minipage} 
\begin{minipage}{0.08\textwidth}
\begin{center}
\ytableausetup{boxsize=1.4em, aligntableaux=bottom}
\begin{ytableau} 
\none[]\\
\none[T^{s_1+1}_{\bullet \bullet}]\\
\none[\longrightarrow]
\\
\none[]\\
\none[]\\
\end{ytableau}
\end{center} 
\end{minipage} 
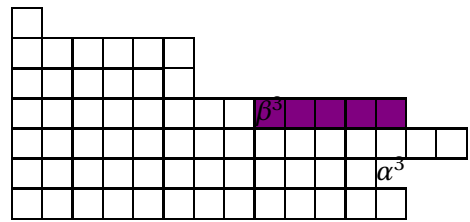
\begin{minipage}{0.4\textwidth}
\begin{center}
\ytableausetup{boxsize=1em, aligntableaux=bottom}
\begin{ytableau} 
\,&\none[] &\none[] &\none[] \\
\,& \,&\, &\,&*(white) &*(white) \,\\
\,&\,& \,&\, &\,&*(white) \\
\,&\,& \,&\, &\, &\, &\, &\, &*(violet)\beta^3 &*(violet) \, &*(violet)\, &*(violet)\,&*(violet)\,\\
\,&\,& \,&\, &\, &\, &\, &\, &*(white)&*(white)&*(white)&*(white)\, &*(white) &*(white)\, &*(white)\, \\
\,&\,& \,&\, &\, &\, &\, &\, &\,&\, &\,&\,&\none[\alpha^3] \\
\,&\, &\, &\, &\, &\, &\,&\, &\,&\, &\,&\,&\,\\
\end{ytableau}

\small{This is $T^4_{\bullet \bullet}(\Gamma)=T^{s_2}_{\bullet \bullet}(\Gamma)$. Since it is still not a Y diagram we need to define $B^3$, in violet, and  slide $\beta^3$ to $\alpha^3$ .}
\end{center}
\end{minipage} 
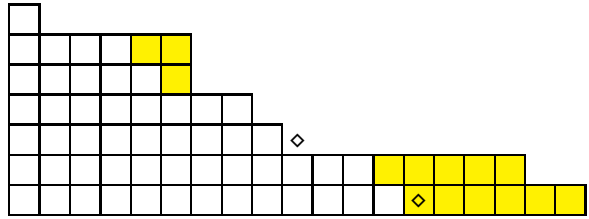
\begin{minipage}{0.08\textwidth}
\begin{center}
\ytableausetup{boxsize=1.4em, aligntableaux=bottom}
\begin{ytableau} 
\none[]\\
\none[\dots ]\\
\none[\longrightarrow]\\
\none[]\\
\none[]\\\end{ytableau}
\end{center} 
\end{minipage} 

\vspace{1em}

\begin{center}
\ytableausetup{boxsize=1em, aligntableaux=bottom}
\begin{ytableau} 
\,&\none[] &\none[] &\none[] \\
\,& \,&\, &\,&*(yellow) &*(yellow) \,\\
\,&\,& \,&\, &\,&*(yellow) \\
\,&\,& \,&\, &\, &\, &\, &\, \\
\,&\,& \,&\, &\, &\, &\, &\,&\,&\none[\diamond]  \\
\,&\,& \,&\, &\, &\, &\, &\, &\,&\, &\,&\,&*(yellow) &*(yellow) \, &*(yellow)\, &*(yellow)\,&*(yellow)\,\\
\,&\, &\, &\, &\, &\, &\,&\, &\,&\, &\,&\,&\,&*(yellow)\diamond &*(yellow)&*(yellow)&*(yellow)\, &*(yellow) &*(yellow)\,  \\
\end{ytableau}

\small{This is $T^6_{\bullet \bullet}(\Gamma)=T^{s_4}_{\bullet \bullet}(\Gamma)$ the last step. It is a Young diagram. We put in yellow $B^1$ for $(\alpha', \beta')= (\diamond, \diamond)$. If we were to perform $T_{\alpha', \beta'}$ we would get bet to the original Young diagram $\Gamma$. See Lemma \ref{inversesliding} below. \\
One can check the affirmations of Lemmas \ref{slidingandpos} and \ref{slidingandssymmetric} below in this example and appreciate their geometrical and free of indexes proof. }
\end{center}
\end{example}
\textbf{Definition \ref{sliding} }[Sliding boxes continued]
Suppose that $\Gamma\vdash n$ and $\beta\in P_{\alpha}\cup Q_{\alpha}$, as above, are fixed. Define $\Delta_{n}\subset \Delta$ the first $n\times n$ boxes in $\Delta$, i.e. where all the Young diagrams $\Gamma \vdash n$ live. It is convenient to see $T_{\alpha, \beta}$ as a bijection of $\Delta_{n+2}$ to itself. To this end we define $T_{\alpha, \beta}$ on $\Delta_{n+2} \setminus \Gamma$ referring to the same steps described in the first part of the definition: if $\delta \in \Delta_{n+2} \setminus \Gamma$ and at any point $t$ there is $\gamma \in T^t_{\alpha, \beta}(\Gamma)$ such that $T^{t+1}_{\alpha, \beta}(\gamma)=\delta$ and $t$ was not the last step in the definition of $T_{\alpha, \beta}$, then we pose $T^{t+1}_{\alpha, \beta}(\delta)=\gamma$, otherwise we leave $\delta$ invariant i.e. $T^{t+1}_{\alpha, \beta}(\delta)=\delta$. 
\end{definition}
\begin{notation}
Given $\Gamma \vdash n$ and $v \in B(\Gamma)$ if $v=\ff$ we will denote, according to convenience and context, the transformation defined with any of the following notations $T_{\alpha, \beta} = T=T_{\ff}=T_v$, and we will say that it a transformation \emph{of weight } $W$. We collect all such transformations in a set that we denote $\Theta_W$. Then we subdivide $\{\Gamma \vdash n\}$ in orbits for the action of $\Theta_W$. 
\begin{equation}
\label{slidingorbits}
\left\{ \Gamma\, \vdash\, n \right\} \quad =\quad \bigsqcup _{i=1}\,\,\,\mathcal{O}_i.
\end{equation}
\end{notation}

\begin{lemma}
\label{slidingexists}
With the notations as in Definition \ref{sliding}, $T^{u}_{\alpha, \beta}(\Gamma)$, the last step of the procedure, is a Young diagram, so that the definition is well posed. 
\end{lemma}
\begin{proof}
The fact that $T_{\alpha, \beta}$ is well defined is actually equivalent to the fact that $\ff$ is an $R=\CC[x, y]$ homomorphism $\ff: I_{\Gamma} \to \bigslant{R}{I_{\Gamma}}$, so it is equivalent to the fact that $\beta \in P_{\alpha} \cup Q_{\alpha}$. \\

\hspace{3em}We can suppose that $\beta \in P_{\alpha}$ as the other case is completely analogous.

\hspace{3em}Call $T=T_{\alpha, \beta}$, and let $m,l\in \NN$ be such that  $\alpha= (-l, m)+\beta$. At any step $k$ of the procedure described in Definition \ref{sliding} we call a box $(s,t)$ outside $T^k(\Gamma)$ an \emph{exterior corner} if $(s, t-1)$ and $(s-1, t) \in T^k(\Gamma)$. Of course at each step the total number of boxes is still $n$, and we have something that is connected. So we only need to prove that whenever $(i, j) \in T^{u}(\Gamma)$ then $(i-1, j)\text{ and } (i, j-1) \in T^{u}(\Gamma)$. Suppose now by contradiction that there is $(i,j)\in T^{u}(\Gamma)$ but, for example $(i-1, j)\notin T^{u}(\Gamma)$. Observe that this force $(i-l, j+m)\in T^{u-s}(\Gamma)$ for some $s>1$. If $(i-l-1, j+m)\notin T^{u-s}(\Gamma)$ for any $s>0$ we find the contradiction, since then we would have (a multiple of) $f_{\alpha, \beta}((i-1, j))= 0$ but (a multiple of) $f_{\alpha, \beta}((i, j)) \neq 0$ which is absurd since $\ff$ is an $R$ homomorphism. Then, indeed, for an $s'>0$, $(i-l-1, j+m)\in T^{u-s'}(\Gamma)$. Since by contradiction we supposed $(i-1, j) \notin T^u(\Gamma)$ we have that either $(i-l-1, j+m)$ is still there at the last step $T^u$ or it must slide to $(i-1, j)$ and then to $(i-l-1, j+m)$. But the latter scenario violates the hypothesis that $T^u$ was the last step because it means that there is still an attracting exterior corner to the right of $(i,j)$. Thus at the step $T^{u}$, the box $(i-l-1, j+m)$ is still there. Then lets look at $(i-2, j)$. It cannot be in $T^{u}(\Gamma)$ otherwise $(i-1,j)$ would have been an attracting exterior corner. Repeating the above argument we find that $(i-l-1, j+m)\in T^{u}(\Gamma)$. Then we can keep going, and we find that $(i-l-N, j+m)\in T^{u}(\Gamma)$ for every $N\in \NN$, that is absurd. 
\end{proof}
\begin{example}
$\,$
\begin{center}
\ytableausetup{boxsize=1em, aligntableaux=bottom}
\begin{ytableau} 
\,\\
\,&\,\\
\,&\bullet \\
\,&\, &\, &\none[\bullet] \\
\end{ytableau}\qquad\qquad
\ytableausetup{boxsize=1em, aligntableaux=bottom}
\begin{ytableau} 
\,\\
\,\\
\, &\none[]&\none[]&\,\\
\,&\, &\, &\,  \\
\end{ytableau}

\small{Suppose we want to mimic the procedure of Definition \ref{sliding} for the couple $(\bullet, \bullet)$ that does not represent an $f_{\bullet, \bullet} \in B(I_\Gamma)$. The result is clearly not a Young diagram, as explained in the proof.}
\end{center}
\end{example}

\begin{lemma}
\label{inversesliding}
Given $n\in \NN$, $W$ a wall for $n$ and $\Gamma \vdash n$, every sliding $T_{\alpha, \beta}$ with $\ff \in B(\Gamma)$ of weight $W$ has an inverse of the same form, meaning that there exists $f_{\alpha', \beta'} \in B(T_{\alpha, \beta}(\Gamma))$ such that 
\[
T_{\alpha', \beta'}\left(T_{\alpha, \beta}\left(\Gamma\right)\right) = \Gamma.
\]
If $\beta \in P_\alpha$, then $\beta' \in Q_{\alpha'}$ and vice-versa. The weight of the inverse transformation is still $W$. 
\end{lemma}
\begin{proof}
The proof is immediate from the construction of $T_{\alpha, \beta}\left(\Gamma\right)$: suppose $\alpha_u$ is, in Definition \ref{sliding}, the last step we need to complete to obtain a Young diagram, then it is sufficient to take $\alpha'=\beta^u = \ff(\alpha^u)$ and $\beta'= \alpha^u$.
\end{proof}
\begin{observation}
If $\ff \in B(\Gamma)$ is of weight $W$, then it contributes to $s^+_{W}(\Gamma)$ (resp. to $s^+_{W}(\Gamma)$ ) if and only if $\beta \in P_{\alpha}$ (resp.  $\beta \in Q_{\alpha}$). 
\end{observation}

\hspace{3em}We are now ready to prove the key properties of the transformations we defined on the set of Young diagrams. These are combinatorial properties related to the quantities defined in \ref{definitionposw} even though we will give an easy geometric proof for them. 
\begin{lemma}
\label{slidingandpos}
Let $n\in \NN$, $W$ be a wall for $n$ and $\Gamma \vdash n$. Let $T$ be a transformation of $\Gamma$ with weight $W$. Then the following two facts hold
\begin{itemize}
\item[(1)]  $\pos_W(\Gamma) = \pos_W\left(T(\Gamma)\right)$, 
\item[(2)] $s^+_W(\Gamma) + s^-_W\left (\Gamma\right ) =  s^+_W\left(T(\Gamma)\right) + s^-_W\left(T(\Gamma)\right). $
\end{itemize}
\end{lemma}
\begin{lemma}
\label{slidingandssymmetric}
Let $n\in \NN$, $W$ be a wall for $n$ and $\Gamma \vdash n$. Consider the division of $\{\Gamma\vdash n\}$ in orbits for transformations of weight $W$ as in (\ref{slidingorbits}). Then for every $i\in \NN$ and every $k\in \NN$ we have 
\[
\# \left\{\Gamma \in \mathcal{O}_i \left\vert s^+_{W}(\Gamma)=k \right.\right\} = \# \left\{\Gamma \in \mathcal{O}_i \left\vert s^-_{W}(\Gamma)=k \right.\right\}\, .
\]
\end{lemma}
To prove the previous two lemmas we need the following. 
\begin{lemma}
\label{orbitsandcomponents}
Let $n\in \NN$, $W$ be a wall for $n$ and $\Gamma \vdash n$. Consider the division of $\{\Gamma\vdash n\}$ in orbits for transformations of weight $W$ as in (\ref{slidingorbits}). Let 
\[\left(\Hi^{n}(\CC^2)\right)^{\TT_W} \quad =\quad \bigsqcup_{i} \,\,\, F_i\]
be the decomposition in connected components of the fixed points set of $\Hi^{n}(\CC^2)$ for the $T_W$ action. Then there is a bijective correspondence between the components $(F_i)_i$ and the orbits $(\mathcal{O}_i)_i$ given by: 
\[
{F}_i \mapsto \{\text{fixed points for the } \TT^+_W \text{ action  on } F_i\}. 
\]
\end{lemma}
\begin{proof}
Let $\Gamma \in \mathcal{O}_i$ and $T_{\alpha, \beta}$ be a transformation of $\Gamma$ of weight $W$. Without loss of generality we can assume that $\beta \in P_{\alpha}$, otherwise we consider the inverse picture. We denote $T_{\alpha, \beta}$ simply $T$. Observe that $\gamma\in \Delta_{n}$ then $\omega_1\gamma+\omega_2T (\gamma)$ with $[\omega_1:\omega_2] \in \PP^1$ is $\TT_W$ invariant by definition of $T$. Define 
\[
\Delta_{\Gamma, T} \,\,\,\,\,:=\quad  \left\{\gamma \in \Delta_{n+3}\left\vert \gamma\notin \Gamma\text{ or } \gamma\notin T(\Gamma) \right.\right\}.
\]
It is clear that all the generators of $I_{\Gamma}$ and of $I_{T(\Gamma)}$ are in $\Delta_{\Gamma, T}$. 
 Then consider the following $\PP_1$  of ideals embedded in $\Hi^{n}(0)$ 
\begin{equation}
\label{p1twinvariante}
\left(\omega_1\gamma +\omega_2T(\gamma)\left\vert \gamma\in \Delta_{\Gamma, T} \right.\right) \qquad [\omega_1:\omega_2] \in \PP^1.
\end{equation}

The fact that it is actually a family of ideals in $\Hi^n(0)$ is a consequence of the results of Iarrobino Theorem \ref{standardgenerators}, by exchanging in the proposition the role of $x$ and $y$ and using the fact that $W>1$. 

\hspace{3em}The fact that is a $T_W$ invariant family is clear since every generator of the ideal is. 

\hspace{3em}Finally observe that when $[\omega_1:\omega_2]=[0,1] \in \PP^1$ then we get $I_{\Gamma}$ and when $[\omega_1:\omega_2]$ is the point $[1,0] \in \PP^1$ we get $I_{T(\Gamma)}$.

\hspace{3em}This proves that whenever two points are in the same orbit $\Gamma, \Gamma'\in \mathcal{O}_i$, then they are in the same connected component $F_i$. Observe however that we proved more: we proved that locally around each fixed point $\Gamma \in \mathcal{O}_i$ the dimension of $F_i$ is exactly $s_W^+(\Gamma)+s^-_W(\Gamma)$ since the $\PP^1$'s described in (\ref{p1twinvariante}) are all different for different $T$'s, and the tangent space of  $\left(\Hi^n(\CC^2)\right)^{\TT_W}$ at $\Gamma$ has a basis formed by those vectors that contribute to $s_W^+(\Gamma)+s^-_W(\Gamma)$. Then, this concludes the proof, and it is worth a separate statement for future reference.   
\end{proof}

\begin{corollary}
\label{localdimesnion}
With the notation of the previous Lemma we have that, locally, around each point $\Gamma \in \mathcal{O}_i$ the dimension of $\left(\Hi^n(\CC^2)\right)^{\TT_W}$ is equal to $s_W^+(\Gamma)+s^-_W(\Gamma)$.
\end{corollary}

Now we are ready to give the proof of two main lemmas.  
\begin{proof}[\textbf{Proof of Lemma \ref{slidingandpos} and \ref{slidingandssymmetric}}]
The two fixed points $\Gamma$ and $T_{\alpha, \beta}(\Gamma)$ are in the same $\mathcal{O}$ orbit, then they are, thanks to Lemma \ref{orbitsandcomponents}, fixed points in the same component $F$. Since $\Hi^{n}(\CC^2)$ is smooth, its fixed points components are smooth, so $F$ is smooth. Then the integers $\pos_W(\Gamma)$ and $\pos_W\left(T_{\alpha, \beta}(\Gamma)\right)$ are the dimension of the positive normal bundle of $F$ in $\Hi^{n}(\CC^2)$, thus they are the same, and $s^+_{W}(\Gamma)+s^-_W({\Gamma})$  and $s^+_{W}(T_{\alpha, \beta}(\Gamma))+s^-_W(T_{\alpha, \beta}(\Gamma))$ are the dimension of the tangent space at, respectively, $\Gamma$ and $T_{\alpha, \beta}(\Gamma)$ in $F$, and since $F$ is smooth they, also, are the same. This proves Lemma \ref{slidingandpos}. \\

\hspace{3em}Consider $F$ as ambient variety. It is smooth, as said before, and projective, since contained in $\Hi^{n}(0)$. We have an action of $\TT^2$ on it, and of all its one dimensional subtori. Consider then the $\TT_{W^+}$ action on $F$: it has finitely many fixed points i.e. the elements of $\mathcal{O}$. Then we can apply Theorem \ref{Bialynicki-Birula} to obtain an affine cell decomposition, where each cell has dimension equal to the positive part of the tangent space at $\Gamma \in \mathcal{O}$ of $F$; this positive part is exactly $s_W^+(\Gamma)$. Recall also that $s^+_{W}(\Gamma)+s^-_W({\Gamma})$ is the dimension of $F$. Then we can apply Poincar\'{e} duality to have for all $k\in \NN$
\[
\# \left\{\Gamma \in \mathcal{O} \left\vert s^+_{W}(\Gamma)=k \right.\right\} = \# \left\{\Gamma \in \mathcal{O}_i \left\vert s^+_{W}(\Gamma)=s^+_{W}(\Gamma)+s^-_W({\Gamma})-k \right.\right\} 
\]
that, rearranging, is what we wanted to prove Lemma \ref{slidingandssymmetric}.
 \end{proof}
 
 \begin{proofspecial}[\textbf{ of Proposition \ref{changingtorus}}]
 \label{proofofchangingtorus}
 Now we can give a proof that is more convenient for us, of the fact that we can compute the positive part of the tangent spaces at the fixed points with respect to either the torus $\TT_{\infty}$ or the torus $\TT_{1^+}$ and obtain the same total polynomials 
 \[
 \sum_{\Gamma \vdash n} q^{\pos_{\infty}(\Gamma)} =  \sum_{\Gamma \vdash n} q^{\pos_{1^+}(\Gamma)}.
 \]
 In fact we consider the polynomial on the left and we start taking smaller $W$ to examine $\sum_{\Gamma \vdash n} q^{\pos_{W}(\Gamma)}$: nothing changes until we hit $W_1$ the first wall for $n$. Then passing on the other side only the vectors $\ff \in B(\Gamma)$ of weight $W_1$, for $\Gamma\vdash n$, are affected. Precisely those that contribute to $s^+_{W_1}(\Gamma)$ stop contributing, i.e. they contribute only on the left of the wall $W_1$ and those that contribute to $s^-_{W_1}(\Gamma)$ start contributing i.e. they contribute only on the right of the wall. Observe that 
 \begin{align*}
 \pos_{W^+} (\Gamma) &= \pos_{W}(\Gamma) + s^+_W(\Gamma)\qquad \text {and }\\
 \pos_{W^-} (\Gamma) &= \pos_{W}(\Gamma) + s^-_W(\Gamma).
 \end{align*}
 Then since we were able to group the Young diagrams of size $n$ in sets where the positive part that is not affected by passing the wall is fixed and the part that is effected by passing the wall is symmetric, we are sure that the results before and after the wall are the same. We repeat this for all the walls until we reach $W=1^+$. Observe that in this way we prove the existence of a map $\phi_W$ on the set $\{\Gamma\vdash n\}$ such that $pos_{W^+}(\Gamma) = \pos_{W^-}(\phi_W(\Gamma))$ even though we do not suggest that there is a preferred such map, as Example \ref{orbitn10} below shows.  
  \end{proofspecial}

\begin{example}
\label{orbitn10}
We look at a specific orbit $\mathcal{O}$ for $n=10$ and the wall $W=2$. We indicate with a couple $\bullet, \bullet$  (resp. $\star, \star$ ) a vector $\ff\in B(\Gamma)$ of weight $W$ that contributes to $s^+_W(\Gamma)$ (resp. $s^-_W(\Gamma)$). In this example the same box can be marked with both, or with two stars, meaning that two vectors are represented by maps that start or end there. We have the three possibilities $(l, m) = (2, 1), (4,2)$ and $(6,3)$ correponding to $W= \frac{2}{1}=\frac{4}{2}$ and $\frac{6}{3}$. For all $\Gamma$ depicted we have $\pos_W(\Gamma)=3$ and $s^+_W(\Gamma) + s^-_W\left (\Gamma\right )= 3$. In every line the vector $\left(s^+_W(\Gamma),s^-_W(\Gamma)\right)$ is constant and its value is indicated under the corresponding line. On the right we put the graph that represents the $T_{\bullet, \bullet}\,$'s, between the $\Gamma$'s, the $T_{\star, \star}$, not depicted, are the inverse arrows. 

\vspace{1em}

\begin{minipage}{0.7\textwidth}
\begin{center}
\ytableausetup{boxsize=1em, aligntableaux=bottom}
\begin{ytableau} 
\bullet\\
\,&\bullet &\none[\bullet]  \\
\,&\,& \bullet &\none[\bullet] \\
\,&\, &\, &\, &\none[\bullet] \\
\end{ytableau}

\small{$\left(s^+_W(\Gamma),s^-_W(\Gamma)\right) = (3, 0) $}
\end{center}

\begin{center}
\ytableausetup{boxsize=1em, aligntableaux=bottom}
\begin{ytableau} 
\,\\
\bullet &\bullet  \\
\,&\,&\none[\,\,\,\,\bullet, \star] \\
\,&\, &\, &\, &\star&\none[\bullet] \\
\end{ytableau}\qquad
\ytableausetup{boxsize=1em, aligntableaux=bottom}
\begin{ytableau} 
\,\\
\bullet &\none[\star]  \\
\,&\,& \bullet &\star \\
\,&\, &\, &\, &\none[\bullet] \\
\end{ytableau}\qquad
\ytableausetup{boxsize=1em, aligntableaux=bottom}
\begin{ytableau} 
\none[\star]\\
\bullet &\bullet &\star\\
\,&\,&\,  &\none[\bullet] \\
\,&\, &\, &\, &\none[\bullet] \\
\end{ytableau}

\small{$\left(s^+_W(\Gamma),s^-_W(\Gamma)\right) = (2, 1) $}
\end{center}

\begin{center}
\ytableausetup{boxsize=1em, aligntableaux=bottom}
\begin{ytableau} 
\bullet\\
\,&\none[\star] \\
\,&\,&\none[\star] \\
\,&\, &\, &\, &\star&\star&\none[\bullet] \\
\end{ytableau}\qquad
\ytableausetup{boxsize=1em, aligntableaux=bottom}
\begin{ytableau} 
\none[\star]\\
\bullet  \\
\,&\,& \none[\,\,\,\,\bullet, \star] \\
\,&\, &\, &\, &\star&\,&\star \\
\end{ytableau} \qquad
\ytableausetup{boxsize=1em, aligntableaux=bottom}
\begin{ytableau} 
\none[\star]\\
\,&\,&\star&\bullet&\,\\
\,&\, &\, &\, &\star &\none[\bullet] \\
\end{ytableau}

\small{$\left(s^+_W(\Gamma),s^-_W(\Gamma)\right) = (1,2) $}
\end{center}

\begin{flushright}
\ytableausetup{boxsize=1em, aligntableaux=bottom}
\begin{ytableau} 
\none[\,\,\,\,\star, \star]\\
\,&\,& \star& \none[\star] \\
\,&\, &\, &\, &\star&\star&\, \\
\end{ytableau}

\small{$\left(s^+_W(\Gamma),s^-_W(\Gamma)\right) = (0, 3) $}
\end{flushright}
\end{minipage}
\begin{minipage}{0.3\textwidth}
\[
\begin{tikzcd}
 \,& \bullet
  \arrow{dr}
  \arrow{d}
  \arrow{dl}  &\, \\
   \bullet  \arrow{d} \arrow{drr}& \bullet  \arrow{dl}  \arrow{dr} &\bullet  \arrow{dl}\arrow[bend left]{dd} \\
 \bullet  \arrow{r} & \bullet  \arrow{dr}  &\bullet  \arrow{d} \\
\,&\, &\bullet 
\end{tikzcd}\]
 \end{minipage}
\end{example}
\begin{observation}
\label{stillaffine}
Observe that, through this wall crossing procedure, we just reproved the statements that says that the attracting sets of $\Hi^n(0)$ for the torus $\TT_{\infty}$ are affine, by supposing that those for the torus $\TT_{1^+}$ are. In fact an attracting set is isomorphic to an affine space if and only if it is smooth, thanks to Bialynicki-Birula. Then our wall crossing procedure tells us that locally around each fixed point the dimension of the attracting set is always at least as big as the tangent space to the attracting set. In fact to each linearly independent vector $v\in B(\Gamma)$ of weight $W$ we associated another fixed point given by $T_v(\Gamma)$ and an invariant $\PP^1$ that connects the two fixed points and this proves the claim on the local dimension. See also \ref{localdimesnion}. 
\end{observation}

\begin{remark}
We divided the content of Lemma \ref{slidingandpos} and  \ref{slidingandssymmetric} in two different statements to underline this fact: while the first Lemma has a combinatorial proof, however less nice because full of indexes, we were not able to prove the second statement only combinatorially. \\

\hspace{3em}Moreover observe that the fact that $s^+_W(\Gamma)+s^-_W(\Gamma)$ is constant along $\Gamma \in \mathcal{O}_i$ is equivalent to the fact that $F_i$ is smooth, so that Lemma $\ref{slidingandssymmetric}$ can be seen as a geometric consequence of the results of Lemma \ref{slidingandpos} if this is proved combinatorially. This is what we will do for the case $[n, n+1, n+2]$. 
\end{remark}

\hspace{3em}Now we describe the little changes we need to implement to pass to the case $[n, n+1]$. They are mostly about the definition of the sliding transformation associated to elements of the basis of a tangent space: in fact we need to specified how we slide the marked box  \,\ytableausetup{boxsize=0.9em, aligntableaux=bottom}\begin{ytableau} *(green)1\end{ytableau} .

\begin{definition}{[Sliding for $[n, n+1]$]}
\label{sliding2}
Let $n\in \NN$, $W$ be a wall for $[n, n+1]$, and $\Gamma=(\Gamma_1, \Gamma_2)\vdash [n, n+1]$ a fixed point for the $\TT_{W^+}$ action on $\Hi^{n, n+1}(\CC^2)$. As usual call $\alpha_j$ the box in $\Gamma_2\setminus \Gamma_1$ marked with a\, \ytableausetup{boxsize=0.9em, aligntableaux=bottom}\begin{ytableau} *(green)1\end{ytableau} . \\

\hspace{3em}We will either define $T(\Gamma)$ by performing $T(\Gamma_1)$ and moving $\alpha_j$ as an any other element of $\Delta_{n+1}\setminus \Gamma_1$, or we will perform $T(\Gamma_2)$ and move $\alpha_j$ as any other element of $\Gamma_2$. Let us see the details. \\
 
\hspace{3em}Suppose $v \in B(\Gamma_1, \Gamma_2)$ is of weight $W$. Then either $v$ is of type $(\ff, \Su(\ff))$ with $\ff\in B(\Gamma_1)$ or v is of type $(0, h_{\alpha'_i, \alpha_j})$. In the latter case we define $T_{\alpha'_i, \alpha_j}(\Gamma)$ simply by sliding $\alpha_j \mapsto \alpha'_i$, which  is always possible since $W>1$.
\vspace{0.6em}

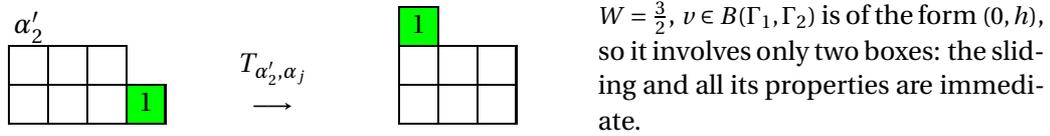
 \begin{minipage}{0.24\textwidth}
\begin{center}
\ytableausetup{boxsize=1.3em, aligntableaux=bottom}
\begin{ytableau} 
\none[\alpha'_2] \\
\, &\, &\,  \\
\,&\, &\, &*(green)1 \\
\end{ytableau}
\end{center}
\end{minipage} 
\begin{minipage}{0.08\textwidth}
\begin{center}
\ytableausetup{boxsize=1.4em, aligntableaux=bottom}
\begin{ytableau} 
\none[]\\
\none[T_{\alpha'_2, \alpha_j}]\\
\none[\longrightarrow]
\\\end{ytableau}
\end{center} 
\end{minipage} 
\begin{minipage}{0.24\textwidth}
\begin{center}
\ytableausetup{boxsize=1.3em, aligntableaux=bottom}
\begin{ytableau} 
*(green)1\\
\, &\, &\,  \\
\,&\, &\,  \\
\end{ytableau}
\end{center} 
\end{minipage} 
\begin{minipage}{0.4\textwidth}\small{
$W=\frac{3}{2}$, $v\in B(\Gamma_1, \Gamma_2)$ is of the form $(0, h)$, so it involves only two boxes: the sliding and all its properties are immediate. }
\end{minipage} 
\vspace{0.6em}

If $v= (\ff, \Su(\ff))$ we need to distinguish two cases: either $\alpha_j$ is not involved in the sliding, meaning that seen as generator of $\Gamma_1$ it does not attract any box of $\Gamma_1$ at any step of the procedure that defines $T_{\alpha, \beta}(\Gamma_1)$, or it does. In the first case we follow the rules of $T_{\alpha, \beta}(\Gamma_2)$: it must be $\alpha_j \neq \alpha$ and $\alpha$  is still a minimal generator of $\Gamma_2$ so  we can  consider $\alpha_j$ as any other box of $\Gamma_2$ and we perform the sliding $T_{\alpha, \beta}$ on $\Gamma_2$ and define
\[
T_{\alpha, \beta}(\Gamma_1, \Gamma_2) = T_{\alpha, \beta}(\Gamma_2)
\]
with the marked box \ytableausetup{boxsize=0.9em, aligntableaux=bottom}\begin{ytableau} *(green)1\end{ytableau} in $T_{\alpha, \beta}(\alpha_j)$ seen as any other element of $\Gamma_2$. 

\vspace{1em}
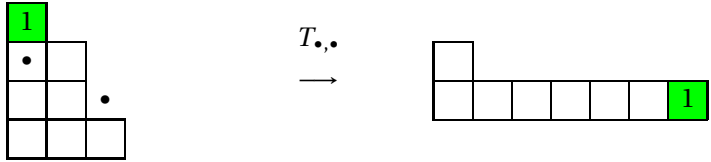
\begin{minipage}{0.34\textwidth}
\begin{center}
\ytableausetup{boxsize=1.3em, aligntableaux=bottom}
\begin{ytableau} 
*(green)1\\
\bullet &\, \\
\, &\, &\none[\bullet]\\
\,&\, &\,  \\
\end{ytableau}
\end{center}
\end{minipage} 
\begin{minipage}{0.1\textwidth}
\begin{center}
\ytableausetup{boxsize=1.4em, aligntableaux=bottom}
\begin{ytableau} 
\none[T_{\bullet, \bullet}]\\
\none[\longrightarrow]\\
\none[]
\\\end{ytableau}
\end{center} 
\end{minipage} 
\begin{minipage}{0.34\textwidth}
\begin{center}
\ytableausetup{boxsize=1.3em, aligntableaux=bottom}
\begin{ytableau} 
\,\\
\, &\, &\,  &\,&\, &\,&*(green)1  \\
\end{ytableau}
\end{center} \end{minipage} 

\begin{center}\small{
$W=\frac{2}{1}$, $v\in B(\Gamma_1, \Gamma_2)$ is of the form $(f_{\bullet, \bullet}, \Su(f_{\bullet, \bullet}))$, and $\alpha_j$ is not attracting any box in the sliding: we simply treat it as another box of $\Gamma_2$ and slide everything with the rule $T_{\bullet, \bullet}(\Gamma_2)$.}
\end{center}

\hspace{3em}In the second case, i.e. if $T_{\alpha, \beta}(\Gamma_1)$ involves sliding a box $\beta_j$ onto $\alpha_j$, we preform $T_{\alpha, \beta}(\Gamma_1)$ as if $\alpha_j$ was not there and then we put $\alpha_j$ where  $\beta_j$ was, i.e. we look at $\alpha_j$ as any other element of $\Delta_{n+1}$.

 \begin{minipage}{0.34\textwidth}
\begin{center}
\ytableausetup{boxsize=1.3em, aligntableaux=bottom}
\begin{ytableau} 
\bullet &\, \\
\, &\star &\none[\bullet]\\
\,&\, &\, &*(green)1 \\
\end{ytableau}
\end{center}
\end{minipage} 
\begin{minipage}{0.1\textwidth}
\begin{center}
\ytableausetup{boxsize=1.4em, aligntableaux=bottom}
\begin{ytableau} 
\none[T_{\bullet, \bullet}]\\
\none[\longrightarrow]\\
\none[]
\\\end{ytableau}
\end{center} 
\end{minipage} 
\begin{minipage}{0.34\textwidth}
\begin{center}
\ytableausetup{boxsize=1.3em, aligntableaux=bottom}
\begin{ytableau} 
\,&*(green)1\\
\, &\, &\,  &\,&\, &\,  \\
\end{ytableau}
\end{center} 
\end{minipage}   

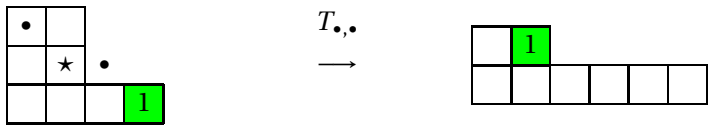
\begin{center}\small{
As before $\Gamma_1$ and $f_{\bullet, \bullet}$ are the same. However, here $\alpha_j$ is attracting the box marked with a star, so first we do the sliding $T_{\bullet, \bullet}(\Gamma_1)$ then we put $\alpha_j$ where there was the star.}
\end{center}  

\hspace{3em}For $\Gamma=(\Gamma_1, \Gamma_2)\vdash [n, n+1]$ and $v\in B(\Gamma)$, we will denote the result of one of the transformations defined with any of the following notations: 
\[
T_v(\Gamma)= T(\Gamma)= T(\Gamma_1, \Gamma_2)= (T(\Gamma_1), T(\Gamma_2))
\]
depending on the aspect of the sliding we want to underline, and we see $T$ as a bijection on $\Delta_{n+1}$ to itself as in Definition \ref{sliding}. 
\end{definition}

Having clarified the transformations we perform in this case, all the remaining steps are as before. The interpretation of the following Lemma \ref{slidingandpos2} in equations (\ref{slidingandposobs}) and (\ref{slidingandsobs}) is what we need to deal with the case $[n, n+1,n+2]$. 
\begin{observation}
Let $n\in \NN$ and $W$ be a wall for $[n, n+1]$. The transformations of \ref{sliding2} are well defined and every transformation is invertible. They partition the set of $(\Gamma_1, \Gamma_2)\vdash [n,n+1]$ in orbits
\begin{equation}
\label{slidingorbits2}
\left\{\,\,(\Gamma_1, \Gamma_2)\,\vdash\, [n,n+1]\,\,\right\} \quad =\quad \bigsqcup_{i=1} \,\, \mathcal{O}_i;
\end{equation}
where an orbit $\mathcal{O}_i $ is closed under the action of transformations $ T$  of weight $W$.
\end{observation}

\begin{lemma}
\label{slidingandpos2}
Let $n\in \NN$, $W$ be a wall for $[n, n+1]$ and $\Gamma \vdash [n, n+1]$. Let $T$ be a transformation of $\,\Gamma$ with weight $W$. Then the following two facts hold
\begin{itemize}
\item[(1)]  $\pos_W(\Gamma) = \pos_W\left(T(\Gamma)\right)$, 
\item[(2)] $s^+_W(\Gamma) + s^-_W\left (\Gamma\right ) =  s^+_W\left(T(\Gamma)\right) + s^-_W\left(T(\Gamma)\right). $
\end{itemize}
\end{lemma}
\begin{lemma}
\label{slidingandssymmetric2}
Let $n\in \NN$, $W$ be a wall for $[n, n+1]$ and $\Gamma \vdash [n, n+1]$. Consider the division of $\{\Gamma\vdash [n,n+1]\}$ in orbits for transformations of weight $W$ as in (\ref{slidingorbits2}). Then for every $i\geq 1$ and every $k\in \NN$ we have 
\[
\# \left\{\Gamma \in \mathcal{O}_i \left\vert s^+_{W}(\Gamma)=k \right.\right\} = \# \left\{\Gamma \in \mathcal{O}_i \left\vert s^-_{W}(\Gamma)=k \right.\right\} .
\]
\end{lemma}
To prove the previous two lemmas we need the following. 
\begin{lemma}
\label{orbitsandcomponents2}
Let $n\in \NN$, $W$ be a wall for $[n, n+1]$ and $\Gamma \vdash [n,n+1]$. Consider the division of $\{\Gamma\vdash [n, n+1]\}$ in orbits for transformations of weight $W$ as in (\ref{slidingorbits2}). Let 
\[
\left(\Hi^{n, n+1}(\CC^2)\right)^{\TT_W} \quad =\quad \bigsqcup_{i} \,\,\, F_i
\]
be the decomposition in connected components of the fixed points set of $\Hi^{n, n+1}(\CC^2)$ for the $T_W$ action. Then there is a bijective correspondence between  the components $(F_i)_i$ and the orbits $(\mathcal{O}_i)_i$ given by: 
\[
F_i \mapsto \{\text{fixed points for the } \TT^+_W \text{ action  on } F_i\}. 
\]
\end{lemma}
\begin{proof}
The proof of the complete bundle of results is completely analogous to that for $\Hi^n(\CC^2)$.\\

\hspace{3em}In the proof of Lemma \ref{orbitsandcomponents2} for each $v\in B(\Gamma_1, \Gamma_2)$ with weight $W$ we can, again, easily write down a $\TT_W$ invariant $\PP^1$ connecting $(\Gamma_1, \Gamma_2)$ and $T_v(\Gamma_1, \Gamma_2)$ with the same exact recipe. \\

\hspace{3em}We can still use that the $F_i$ are smooth, since the ambient space $\Hi^{n, n+1}(\CC^2)$ is. Then the quantity $\pos_W(\Gamma)$ is the dimension of the positive normal bundle, and thus the same for each $\Gamma \in F_i$ fixed point. The quantity  $s^+_W(\Gamma)+s^-_W(\Gamma)$ is also the same for every $\Gamma$ since it is the dimension of the tangent space of $F_i$ at $\Gamma \in F_i$. Moreover the number of such $\Gamma$ with given $s^+_W(\Gamma)$ is equal to the number of such $\Gamma$ with given $s^-_W(\Gamma)$ again thanks to Poincar\'{e} duality for the $F_i$ 's.   
\end{proof}

\hspace{3em}We want to spell out what points (1) and (2) of Lemma \ref{slidingandpos2} mean in terms of Definition \ref{definitionB(I1I2)} of $B(\Gamma_1, \Gamma_2)$ . 

\begin{remark} 
Use the notation of Definition \ref{sliding2}. Call $\alpha_i, i=0,\dots,s$ the generators of $\Gamma_1$, $\alpha_j$ the box in $\Gamma_2\setminus\Gamma_1$ and $\beta_i, i=0, \dots, t$ the generators of $T(\Gamma_1)$ and $\beta_h$ the box in $T(\Gamma_2)\setminus T(\Gamma_1)$. Since $B(\Gamma_1, \Gamma_2)$ is constructed by elements of the form $(0,h_{\alpha_i, \alpha_j})$ and $(\ff, \Su(\ff))$ for those $\alpha, \beta$ that are not in $\text{Obs}(\Gamma_1, \Gamma_2)$ and analogously for $T(\Gamma_1, \Gamma_2)$ we have by Lemma \ref{slidingandpos2} 
\begin{align*}
 \pos_W(\Gamma_1) - \pos_W(\text{Obs}(\Gamma_1, \Gamma_2))&+\pos_W(\{h_{\alpha_i, \alpha_j}\vert i=0, \dots, s\})= \nonumber \\=
 \pos_W(\Gamma_1, \Gamma_2) = &
 \pos_W(T(\Gamma_1, \Gamma_2)) = \\
= \pos_W(T(\Gamma_1)) - \pos_W(\text{Obs}(T(\Gamma_1, \Gamma_2)))&+\pos_W(\{h_{\beta_i, \beta_h}\vert i=0, \dots, t\}). \nonumber
\end{align*}
Since we know from Lemma \ref{slidingandpos} that $\pos_W(\Gamma_1)=\pos_W(T(\Gamma_1))$ we obtain that 
\begin{align}
\label{slidingandposobs}
 \pos_W(\text{Obs}(\Gamma_1, \Gamma_2))&-\pos_W(\{h_{\alpha_i, \alpha_j}\vert i=0, \dots, s\})\nonumber \\ = &\pos_W(\text{Obs}(T(\Gamma_1, \Gamma_2)))-\pos_W(\{h_{\beta_i, \beta_j}\vert i=0, \dots, t\}).
\end{align}
Analogously, for $s^+_W$ and $s^-_W $ we obtain that 
\begin{align}
\label{slidingandsobs}
 (s^+_W+s^-_W) &\left( \text{Obs}(\Gamma_1, \Gamma_2)\right)- (s^+_W+s^-_W)\left(\{h_{\alpha_i, \alpha_j}\vert i=0, \dots, s\}\right) \nonumber \\ 
 &\quad= \quad(s^+_W+s^-_W)\left( \text{Obs}(T(\Gamma_1, \Gamma_2))\right)- (s^+_W+s^-_W)\left(\{h_{\beta_i, \beta_j}\vert i=0, \dots, t\} \right).
\end{align}
\end{remark}
\begin{observation}
\label{alphaisteadofalphaprime}
Notice this important fact: we can write (\ref{slidingandposobs}) and (\ref{slidingandsobs}) in terms of the generators of $\Gamma_1$ and $T(\Gamma_1)$ and \emph{not} in term of the generators of $\Gamma_2$ and $T(\Gamma_2)$ because we are only interested in the positive part, and since $W>1$ there is not difference in the two: for example if $\alpha_j$ is of case $3)$, according to cases \ref{cases12}, then the list of $\alpha_i$ and that of $\alpha'_i$ differ only around $\alpha_j$ and in particular we have two new generators $x\alpha_j$ and $y\alpha_j$, but we have that $\left(0, h_{x\alpha_j, \alpha_j}\right)$ and $\left(0, h_{y\alpha_j, \alpha_j}\right)$ cannot be possibly positive for any $W$. Similarly for the other cases. This observation saves us a lot of work in distinguishing cases later.
\end{observation}

\hspace{3em}Now we are well placed to use equation (\ref{slidingandposobs}) and (\ref{slidingandsobs}) for the next step i.e. $[n, n+1, n+2]$. We start by spelling out the definitions of the transformation associated to a vector of weight $W$ in the tangent spaces of fixed points, even though the definitions are, in this case, exactly as in the previous case.  

\begin{definition}{Sliding for $[n, n+1, n+2]$}
\label{sliding3}

Let $n\in \NN$, $W$ be a wall for $[n, n+1, n+2]$, and $\Gamma=(\Gamma_1, \Gamma_2, \Gamma_3)\vdash [n, n+1, n+2]$ be a fixed point for the $\TT_{W^+}$ action on $\Hi^{n, n+1, n+2}(\CC^2)$. As usual call $\alpha_j$ the box in $\Gamma_2\setminus \Gamma_1$ marked with a \ytableausetup{boxsize=0.9em, aligntableaux=bottom}\begin{ytableau} *(green)1\end{ytableau} and $\alpha'_l$ the box in $\Gamma_3\setminus \Gamma_2$ marked with \ytableausetup{boxsize=0.9em, aligntableaux=bottom}\begin{ytableau} *(cyan)2\end{ytableau}. \\

\hspace{3em}Suppose $v \in B(\Gamma_1, \Gamma_2, \Gamma_3)$ is of weight $W$. We either have $v$  of type $(0,0,h_{\alpha''_i, \alpha'_l})$ or $v$ of type $\left(0, h_{\alpha'_i, \alpha_j}, \Su(h_{\alpha'_i, \alpha_j})\right)$ or, finally, $v$ of type  $\left(\ff, \Su(\ff), \Su(\Su(\ff))\right)$ with $\ff\in B(\Gamma_1)$. \\

\hspace{3em}In the first case we define $T_{\alpha''_i, \alpha_l}(\Gamma)$ simply by sliding $\alpha_l \mapsto \alpha''_i$, which  is always possible since $W>1$. This is exactly as if we were looking at the transformation of $\Gamma_3$ without caring for the marked boxes. 
\vspace{0.6em}

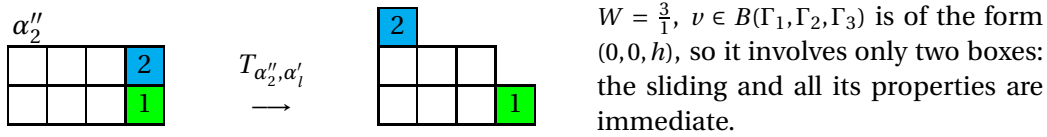
 \begin{minipage}{0.24\textwidth}
\begin{center}
\ytableausetup{boxsize=1.3em, aligntableaux=bottom}
\begin{ytableau} 
\none[\alpha''_2] \\
\, &\, &\, &*(cyan)2 \\
\,&\, &\, &*(green)1 \\
\end{ytableau}
\end{center}
\end{minipage} 
\begin{minipage}{0.08\textwidth}
\begin{center}
\ytableausetup{boxsize=1.4em, aligntableaux=bottom}
\begin{ytableau} 
\none[]\\
\none[T_{\alpha''_2, \alpha'_l}]\\
\none[\longrightarrow]
\\\end{ytableau}
\end{center} 
\end{minipage} 
\begin{minipage}{0.24\textwidth}
\begin{center}
\ytableausetup{boxsize=1.3em, aligntableaux=bottom}
\begin{ytableau} 
*(cyan)2\\
\, &\, &\,  \\
\,&\, &\,&*(green)1  \\
\end{ytableau}
\end{center} 
\end{minipage} 
\begin{minipage}{0.4\textwidth}\small{
$W=\frac{3}{1}$, $v\in B(\Gamma_1, \Gamma_2, \Gamma_3)$ is of the form $(0,0, h)$, so it involves only two boxes: the sliding and all its properties are immediate. }
\end{minipage} 
\vspace{0.6em}

\hspace{3em}If $v= \left(0, h_{\alpha'_i, \alpha_j}, \Su(h_{\alpha'_i, \alpha_j})\right)$ we need to distinguish two cases: If  $\alpha'_i = \alpha'_l$ we simply switch $\alpha_j$ and $\alpha'_l$ as it would have happen following the rule described in \ref{sliding2}: we are performing the transformation for $\Gamma_2$ and looking at $\alpha'_l$ as any other element of $\Delta_{n+2}\setminus{\Gamma_2}$. If $\alpha'_i \neq \alpha_l$ again we slide according to the transformation $T_{\alpha'_i,\alpha_j }$ of $\Gamma_3$ without caring for the marked boxes. Observe that we move more than one box only if it it happens that $\alpha'_l$ is on the same row or column of $\alpha_j$. In this case we might need to slide more than two boxes. 

\vspace{1em}
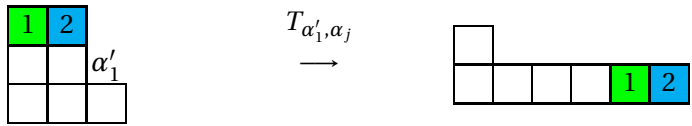
\begin{minipage}{0.34\textwidth}
\begin{center}
\ytableausetup{boxsize=1.3em, aligntableaux=bottom}
\begin{ytableau} 
*(green)1&*(cyan)2\\
\, &\, &\none[\alpha'_1]\\
\,&\, &\,  \\
\end{ytableau}
\end{center}
\end{minipage} 
\begin{minipage}{0.1\textwidth}
\begin{center}
\ytableausetup{boxsize=1.4em, aligntableaux=bottom}
\begin{ytableau} 
\none[T_{\alpha'_1, \alpha_j}]\\
\none[\longrightarrow]\\
\none[]
\\\end{ytableau}
\end{center} 
\end{minipage} 
\begin{minipage}{0.34\textwidth}
\begin{center}
\ytableausetup{boxsize=1.3em, aligntableaux=bottom}
\begin{ytableau} 
\,\\
\, &\, &\,  &\,&*(green)1 &*(cyan)2 \\
\end{ytableau}
\end{center} \end{minipage} 

\begin{center}\small{
$W=\frac{2}{1}$, $v\in B(\Gamma_1, \Gamma_2, \Gamma_3)$ is of the form $\left(0, h_{\alpha'_i, \alpha_j}, \Su(h_{\alpha'_i, \alpha_j})\right)$, we think of everything happening for $\Gamma_3$ without special meanings for \ytableausetup{boxsize=0.9em, aligntableaux=bottom}\begin{ytableau} *(green)1\end{ytableau} and \ytableausetup{boxsize=0.9em, aligntableaux=bottom}\begin{ytableau} *(cyan)2\end{ytableau}}\, .
\end{center}

\hspace{3em}Finally $v= \left(\ff, \Su(\ff), \Su(\Su(\ff))\right)$. Again the definition goes as for the case $[n, n+1]$: we either perform $T(\Gamma_1)$ or $T(\Gamma_2)$ or $T(\Gamma_3)$ and look at $\alpha_j$ as in $\Delta_{n+2}\setminus \Gamma_1$ or in $\Gamma_2$ and at $\alpha'_l$ as in $\Delta_{n+2}\setminus \Gamma_2$ or in $\Gamma_3$ depending on whether some box slide onto them or not.

\vspace{1em}
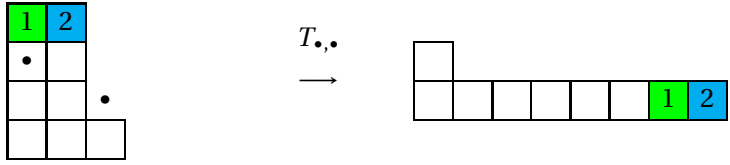
\begin{minipage}{0.34\textwidth}
\begin{center}
\ytableausetup{boxsize=1.3em, aligntableaux=bottom}
\begin{ytableau} 
*(green)1&*(cyan)2\\
\bullet &\, \\
\, &\, &\none[\bullet]\\
\,&\, &\,  \\
\end{ytableau}
\end{center}
\end{minipage} 
\begin{minipage}{0.1\textwidth}
\begin{center}
\ytableausetup{boxsize=1.4em, aligntableaux=bottom}
\begin{ytableau} 
\none[T_{\bullet, \bullet}]\\
\none[\longrightarrow]\\
\none[]
\\\end{ytableau}
\end{center} 
\end{minipage} 
\begin{minipage}{0.34\textwidth}
\begin{center}
\ytableausetup{boxsize=1.3em, aligntableaux=bottom}
\begin{ytableau} 
\,\\
\, &\, &\,  &\,&\, &\,&*(green)1&*(cyan)2  \\
\end{ytableau}
\end{center} \end{minipage} 

\begin{center}\small{
$W=\frac{2}{1}$, $v\in B(\Gamma_1, \Gamma_2, \Gamma_3)$ is of the form $\left(f_{\bullet, \bullet}, \Su(f_{\bullet, \bullet}), \Su(\Su(f_{\bullet, \bullet}))\right)$, and nor $\alpha_j$ nor $\alpha'_j$ are attracting any box in the sliding: we simply treat them as any another boxes and slide everything with the rule $T_{\bullet, \bullet}(\Gamma_3)$.}
\end{center}

 \begin{minipage}{0.34\textwidth}
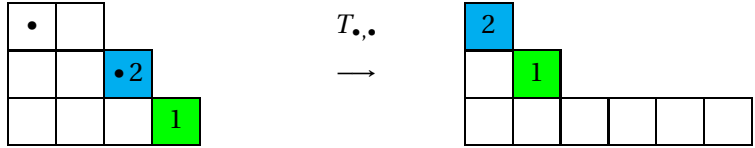
\begin{center}
\ytableausetup{boxsize=1.6em, aligntableaux=bottom}
\begin{ytableau} 
\bullet &\, \\
\, &\, &*(cyan)\bullet\, 2\\
\,&\, &\, &*(green)1 \\
\end{ytableau}
\end{center}
\end{minipage} 
\begin{minipage}{0.1\textwidth}
\begin{center}
\ytableausetup{boxsize=1.6em, aligntableaux=bottom}
\begin{ytableau} 
\none[T_{\bullet, \bullet}]\\
\none[\longrightarrow]\\
\none[]
\\\end{ytableau}
\end{center} 
\end{minipage} 
\begin{minipage}{0.34\textwidth}
\begin{center}
\ytableausetup{boxsize=1.6em, aligntableaux=bottom}
\begin{ytableau} 
*(cyan)2\\
\,&*(green)1\\
\, &\, &\,  &\,&\, &\,  \\
\end{ytableau}
\end{center} 
\end{minipage}   

\begin{center}\small{
Same situation as before meaning that $\Gamma_1$ and $f_{\bullet, \bullet}$ are the same. Here $\alpha_j$ and $\alpha'_l$ are each attracting a respective box: first we do the sliding $T_{\bullet, \bullet}(\Gamma_1)$ and then we put $\alpha_j$ and $\alpha'_l$ in the boxes that occupied their respective old positions.}
\end{center}  

\hspace{3em}For $\Gamma=(\Gamma_1, \Gamma_2, \Gamma_3)\vdash [n, n+1, n+2]$, and $v \in B(\Gamma)$,  we will denote the result of one of the transformations defined with any of the following notation: 
\[
 T(\Gamma)=T_v(\Gamma)= T(\Gamma_1, \Gamma_2, \Gamma_3)= (T(\Gamma_1), T(\Gamma_2), T(\Gamma_3))=\dots
\]
depending on the aspect of the sliding we want to underline. Again we see $T_v$ as a transformation of $\Delta_{n+2}$ as in definitions \ref{sliding} and \ref{sliding2}. 
\end{definition}

\hspace{3em}Now we are well placed to prove all the analogous results we proved in the previous cases. 

\begin{observation}
Let $n\in \NN$ and $W$ be a wall for $[n,n+1,n+2]$. The transformations of \ref{sliding3} are well defined and every transformation is invertible. They partition the set of $(\Gamma_1, \Gamma_2, \Gamma_3)\vdash [n,n+1, n+2]$ in orbits
\begin{equation}
\label{slidingorbits3}
\left\{\,\,(\Gamma_1, \Gamma_2, \Gamma_3)\,\vdash\, [n,n+1, n+2]\,\,\right\} \quad =\quad \bigsqcup_{i=1} \,\, \mathcal{O}_i;
\end{equation}
where $\mathcal{O}_i$ is closed under the action of transformations $T$ of weight  $W$. 
\end{observation}

\begin{lemma}
\label{slidingandpos3}
Let $n\in \NN$, $W$ be a wall for $[n, n+1, n+2]$ and $\Gamma \vdash [n, n+1, n+2]$. Let $T$ be a transformation of $\,\Gamma$ with weight $W$. Then the following two facts hold
\begin{itemize}
\item[(1)]  $\pos_W(\Gamma) = \pos_W\left(T(\Gamma)\right)$, 
\item[(2)] $s^+_W(\Gamma) + s^-_W\left (\Gamma\right ) =  s^+_W\left(T(\Gamma)\right) + s^-_W\left(T(\Gamma)\right). $
\end{itemize}
\end{lemma}
\begin{proof}

We use Observation \ref{alphaisteadofalphaprime} to cut the number of possible cases we need to distinguish between: in fact we want to use equations (\ref{slidingandposobs}) and (\ref{slidingandsobs}) separately on $\alpha_j$ and $\alpha'_l$. \\

\hspace{3em}Call $\alpha_i, i=0,\dots,s$ the generators of $\Gamma_1$, $\alpha_j$ the box in $\Gamma_2\setminus\Gamma_1$ and $\alpha'_l$ the box in $\Gamma_3\setminus\Gamma_2$. Call as well  $\beta_i, i=0, \dots, t$ the generators of $T(\Gamma_1)$,  $\beta_h$ the box in $T(\Gamma_2)\setminus T(\Gamma_1)$ and $\beta'_k$ the box in $T(\Gamma_3)\setminus T(\Gamma_3)$. 

\hspace{3em}Suppose now  that $\alpha'_l$ is not on the same row or column as $\alpha_j$.\\

\hspace{3em}Suppose also $j<l$, the other cases is treated in a completely similar way (in fact one can pass to the transpose to treat it.). Given the definition of $B(\Gamma_1, \Gamma_2, \Gamma_3)$ we can write that 
\begin{align*}
 \pos_W(\Gamma_1, \Gamma_2, \Gamma_3) &= \\
 &\,\,\pos_W\left(\Gamma_1 \right) -  \pos_W\left(\text{Obs}(\Gamma_1, \Gamma_2)\right) -  \pos_W\left(\text{Obs}(\Gamma_1, \Gamma_3) \right)\qquad+\\
  +\,\,&\pos_W\left( \{h_{\alpha_i, \alpha_j}\vert i=0, \dots, s\}\right) + \pos_W\left(\{h_{\alpha_i, \alpha'_l}\vert i=0, \dots, s\}\right) \qquad+\\
  + \,\,&\begin{cases} \pos_W\left( \alpha_{j-1} \mapsto \frac{\alpha_l}{y^{p_{j-1}}} \right) & \text{ if we are in case j 1a) or j3)}, \\
  0 & \text{ if if we are in case j 1b) or j2)}  \end{cases}\qquad+\\
  - \,\,& \pos_W\{\left(\alpha_j \mapsto \alpha_l \right)\}\qquad+ \\
  + \,\,&\begin{cases} \pos_W\left( y\alpha_{j} \mapsto {\alpha_l} \right) & \text{ if  we are in case j 1 a),}  \\
   \pos_W\left( x\alpha_{j} \mapsto {\alpha_l} \right) & \text{ if  we are in case j 1 b),}  \\
   \pos_W\left( x\alpha_{j} \mapsto {\alpha_l} \right) + \pos_W\left( y\alpha_{j} \mapsto {\alpha_l} \right) & \text{ if  we are in case j 2),}  \\
  0 & \text{ if we are in case j 3)}.
  \end{cases} 
\end{align*}
The cases we refer to are from Definition $\ref{cases12}$. 
 
\hspace{3em}The last three terms are there to compensate the choice of writing the other terms only with respect to the generators of $\Gamma_1$. Applying now Equation (\ref{slidingandposobs}) twice to $\alpha_j$ and to $\alpha_l$ and the fact that $\pos_W(\Gamma_1) = \pos(T(\Gamma_1))$, we see that we only need to deal with the last 3 terms, and prove that they contribute the same amount before and after the sliding $T$. Call this contribution $A(j , l)$. Since positiveness of $A(j, l)$ depends only on the \emph{relative} positions of $\alpha_j$ and $\alpha_l$ we can suppose that $\alpha_l$ is fixed by $T$ and that only $\alpha_j$ is affected by $T$. Then there are three cases: $\alpha'_l$ is higher than $\alpha_j$ and  $T(\alpha'_l)$ is higher than $T(\alpha_j)$, or $\alpha'_l$ is higher than $\alpha_j$ but  $T(\alpha_j)$ is higher than $T(\alpha'_l)$ and $\alpha_j$ is higher than $\alpha_l$ and  $T(\alpha_j)$ is higher than $T(\alpha'_l)$. We deal with the first case the others are extremely similar.

\hspace{3em}There are three possibilities: if $\alpha_j$ also is fixed by $T$ but its case is different after $T$; if $\alpha_j$ is slid by $T$ as an element of $\Gamma_2$; and finally if there is a $\beta_j \in \Gamma_1$ that slid to $\alpha_j$ with $T$. Since $A(j, l)$ now depends only on the case of $\alpha_j$ we can reduce the cases we need to analyze to 16: the cases of $\alpha_j$ times the cases $T(\alpha_j)$. However, every single case is immediate and ultimately they are all similar to one another. 

\hspace{3em}We do for example the case $j 2)$ and $ T(j) 1a)$. Here we have that the following vectors have weights that coincide: (in the first column we look at $\Delta_{n+2}$ before $T$, on the second column we look at $\Delta_{n+2}$ after $T$. We use the notation $\alpha -1$ to mean the generator immediately below $\alpha$.)
\begin{align}
\label{multicases}
\left(y\alpha_j \mapsto \alpha_l \right) &\longleftrightarrow \left(yT(\alpha_j ) \mapsto \alpha_l \right);\nonumber\\
\left(T(\alpha_j) \mapsto \alpha_l \right) &\longleftrightarrow \left(\alpha_j \mapsto \alpha_l \right);\\
\left(x\alpha_j\mapsto \alpha_l \right) &\longleftrightarrow \left(T(\alpha_j)-1 \mapsto \frac{\alpha_l}{y^{p_{T(\alpha_j)-1}}} \right). \nonumber
\end{align}
All other vectors remain the same before and after $T$. That proves that the weight of $A(j,l)$ remains the same after the transformation $T$. All other cases are completely similar. 

\begin{center}
\ytableausetup{boxsize=1.3em, aligntableaux=bottom}
\begin{ytableau} 
*(cyan)\alpha_l\\
\,\\
\,\\
\none[\vdots]\\
\,\\
\, &\, &\, &\, &\none[\bullet]\\
\, &\, &\, &\, &*(green)\alpha_j &\none[\diamond] \\
\, &\, &\, &\, &\, &\, &\, &\, &\, &\, \\
\, &\, &\, &\, &\, &\, &\, &\, &\, &\, &\none[\star]\\
\, &\, &\, &\, &\, &\, &\, &\, &\, &\, &\,\\
\, &\, &\, &\, &\, &\, &\, &\, &\, &\, &\,\\
\, &\, &\, &\, &\, &\, &\, &\, &\, &\, &\,\\
\, &\, &\, &\, &\, &\, &\, &\, &\, &\, &\,&\,\\
\, &\, &\, &\, &\, &\, &\, &\, &\, &\, &\,&\,&\none[\dots] \\
\end{ytableau}
\ytableausetup{boxsize=1.3em, aligntableaux=bottom}
\begin{ytableau} 
*(cyan)\alpha_l\\
\,\\
\,\\
\none[\vdots]\\
\,\\
\, &\, &\, &\, \\
\, &\, &\, &\,&\none[\star] \\
\, &\, &\, &\, &\, &\, &\, &\, &\, &\,&\none[\bullet] \\
\, &\, &\, &\, &\, &\, &\, &\, &\, &\,&*(green)\alpha_j  \\
\, &\, &\, &\, &\, &\, &\, &\, &\, &\, &\,\\
\, &\, &\, &\, &\, &\, &\, &\, &\, &\, &\,\\
\, &\, &\, &\, &\, &\, &\, &\, &\, &\, &\,&\none[\diamond]\\
\, &\, &\, &\, &\, &\, &\, &\, &\, &\, &\,&\,\\
\, &\, &\, &\, &\, &\, &\, &\, &\, &\, &\,&\,&\none[\dots] \\
\end{ytableau}

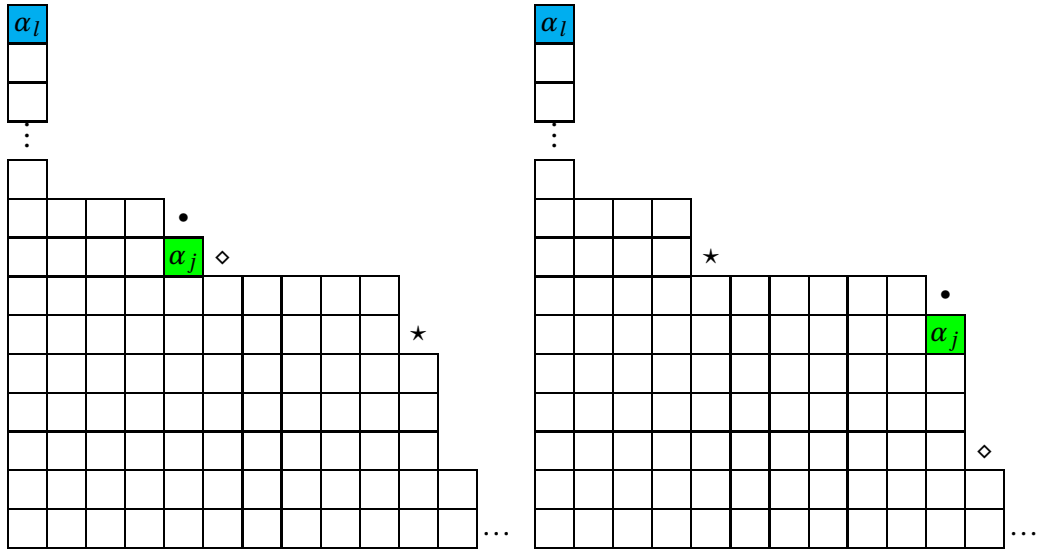
\small{An example of the case j 2) T(j) 1a). The symbols are the generators that correspond to one another in \ref{multicases}}.
\end{center}

\hspace{3em}The proof for $s^+_W$ and $s^-_W $ is completely analogous. In fact if one applies equation (\ref{slidingandsobs}) everything one needs to analyze is the the weight of the same term $A(j,l)$ as above. But then for every cases, there is a conservation of the total quantities of weights of the vectors that change before and after $T$. \end{proof}

\begin{lemma}
\label{orbitsandcomponents3}
Let $n\in \NN$, $W$ be a wall for $[n, n+1, n+2]$ and $\Gamma \vdash [n, n+1, n+2]$. Consider the division of $\{\Gamma\vdash [n, n+1, n+2]\}$ in orbits for transformations $T$ of weight $W$ as in (\ref{slidingorbits}). Let 
\[
\left(\Hi^{n, n+1, n+2}(\CC^2)\right)^{\TT_W} \quad =\quad \bigsqcup_{i} \,\,\, F_i
\]
be the decomposition in connected components of the fixed points set of $\Hi^{n, n+1, n+2}(\CC^2)$ for the $T_W$ action. Then there is a bijective correspondence between the components $(F_i)_i$ and the orbits $(\mathcal{O}_i)_i$ and it is given by: 
\[
F_i \mapsto \{\text{fixed points for the } \TT^+_W \text{ action  on } F_i\}. 
\]
Moreover the $F_i$ are smooth. 
\end{lemma}
\begin{proof}
The proof that there is a bijection between orbits of transformations $T$ of weight $W$ and fixed points components for $\TT_W$ goes exactly as before, since, as before, we can construct for each such $T$ a $\TT_W$ invariant $\PP^1$ connecting $\Gamma$ and $T(\Gamma)$. \\

\hspace{3em}This time however, we already know that $s^+_W(\Gamma) + s^-_W\left (\Gamma\right )$ is constant for all $\Gamma \in \mathcal{O}_i$ thanks to the combinatorial proof of Lemma \ref{slidingandpos3}. Then we have smoothness of $F_i$, even though the ambient space is not smooth.
\end{proof}

\begin{lemma}
\label{slidingandssymmetric3}
Let $n\in \NN$, $W$ be a wall for $[n, n+1, n+2]$ and $\Gamma \vdash [n, n+1, n+2]$. Consider the division of $\left\{\Gamma\vdash [n,n+1, n+2]\right\}$ in orbits for transformations of weight $W$ as in (\ref{slidingorbits2}). Then for every $i \geq 1$ and every $k\in \NN$ we have 
\[
\# \left\{\Gamma \in \mathcal{O}_i \left\vert s^+_{W}(\Gamma)=k \right.\right\} = \# \left\{\Gamma \in \mathcal{O}_i \left\vert s^-_{W}(\Gamma)=k \right.\right\} .
\]
\end{lemma}
\begin{proof}
Thanks to Lemma \ref{orbitsandcomponents3} we know that the $F_i$'s are smooth. Then we can apply Poincar\'{e} duality as in the proof of Lemma \ref{slidingandssymmetric} to conclude. 
\end{proof}

\begin{proposition}
\label{changingtorus3}
Let $n\in \NN$ and let $W$ be a wall for $[n, n+1, n+2]$. Then for every $k \in \NN$ we have 
\[
\#\left\{ \Gamma\, \vdash \,\,[n, n+1, n+2] \left\vert \pos_{W^+}\left(\Gamma \right) = k  \right. \right\} = \# \left\{ \Gamma\, \vdash \,\,[n, n+1, n+2]  \left\vert \pos_{W^-}\left(\Gamma \right) = k  \right. \right\} .
\]
In particular we have that the following two polynomials are the same: \begin{equation}
\label{changetorusequation}
\sum_{\Gamma \,\vdash\, [n, n+1, n+2]  } q^{\pos_{\infty}\left(\Gamma\right)} = \sum_{\Gamma \,\vdash\, [n, n+1, n+2]  } q^{\pos_{1^+}\left(\Gamma\right)}\,\,.
\end{equation}
\end{proposition} 

\begin{proof}
The proof is exactly the same as the proof of \ref{proofofchangingtorus}, using Lemmas \ref{slidingandpos3} and \ref{slidingandssymmetric3}. 
\end{proof}
\begin{observation}
Observe that we just proved that the attracting sets for the torus action $\TT_{\infty}$ are affine cells as well. This follows thanks to a reasoning completely similar to the one presented in Observation \ref{stillaffine}.
\end{observation}

Now we need to find a generating function for the polynomials on the left hand side of (\ref{changetorusequation}) when we consider all $n\in\NN$. This is the goal of the next section.
\section{$\Hi^{n, n+2}(\CC^2)$ }

We turn now our attention to $\Hi^{n, n+2}(\CC^2)$. By imposing a specific condition on the two ideals of length $n$ and $n+2$ of the flag, we define a smooth subspace $ \Hi^{n, n+2}(\CC^2)_{tr}$ of $\Hi^{n, n+2}(\CC^2)$ whose cell decomposition is studied by Nakajima and Yoshioka \cite[Chapter~5]{nakajima2008perverse}. It turns out that the generating function for the Poincar\'{e} polynomials of this smooth space resemble closely to the formula (\ref{generating123}). The similarity was noticed almost accidentally: the construction of the analogue subspace in the case of $\Hi^{n, n+1}(\CC^2)$ gives back the full $\Hi^{n, n+1}(\CC^2)$ and reproduces the results originally of Cheah. 

\hspace{3em}To prove our formula (\ref{generating123}) we thus simply need to match the combinatorics of the formula proved by Nakajima and Yoshioka. The question about some geometrical connections between the two families of spaces, however, remains mysterious . 

\hspace{3em} From the affine cell decomposition we found for $\Hi^{n, n+1, n+2}(0)$, we are then able to find an affine cell decomposition of $\Hi^{n, n+2}(0)$ and a generating function for its Poincar\'{e} polynomials (\ref{generating13}).

\begin{definition}
Let $(J, I)$ be a point of $\Hi^{n, n+2}(\CC^2)$. We say that $J, I$ are \emph{trivially related} if $\bigslant{J}{I} \cong \CC^2$ as trivial $\CC[x, y]$ modules. Said it otherwise 
\[
\bigslant{J}{I} \subseteq \text{ Ker } \begin{bmatrix} x \\ y  \end{bmatrix} : \bigslant{\CC[x,y]}{I} \to \CC^2 \otimes \bigslant{\CC[x,y]}{I}, 
\]
i.e. $xf \in I$ and $yf \in I$ for all $f \in J$. We then divide $\Hi^{n, n+2}(\CC^2)$ into two parts, those couples that are trivially related and those that are not
\[
\Hi^{n, n+2}(\CC^2) = \Hi^{n, n+2}(\CC^2)_{tr}\,\, \bigsqcup\,\,  \Hi^{n, n+2}(\CC^2)_{Ntr}.
\]
\end{definition}
We will see below more details of the following proposition, but we state it here to motivate the definition of this space. 
\begin{proposition}{\cite[Proposition~5.2, Corollary~5.4]{nakajima2008perverse}}
The space $\Hi^{n, n+2}(\CC^2)_{tr}$ is smooth for every $n$ in $\NN$. Moreover it has a cellular decomposition given by the torus action and the Poincar\'{e} polynomials have generating function: 
\begin{equation}
\label{generatingnakajima}
\sum_{n\geq 0} P_q\left(\Hi^{n, n+2}(0)_{tr}\right)t^n = \frac{1}{(1-tq)(1-t^2q^2)}\, \prod_{d=1}^{\infty}\, \frac{1}{1-q^{d-1}t^d}\,.
\end{equation}
\end{proposition}

We observe the following connection between $\Hi^{n,n+1, n+2}(0)$ and $\Hi^{n, n+2}(0)$. Note however that this is not the connection that explains the similarities between the generating function: for that we would need to relate $\Hi^{n,n+1, n+2}(0)$ and $\Hi^{n+1, n+3}_{tr}(0)$
\begin{lemma}
\label{fibration}
Consider the projection on first and third factor, i.e. the map that forgets the middle element of a three steps flag 
\[
p_{1,3}: \Hi^{n, n+1, n+2}(0) \to \Hi^{n, n+2}(0).
\]
Then 
\[p_{1,3}^{-1} \left((J, I)\right)\quad \cong\quad \begin{cases} \PP^1 &\text{ if }  \,\,(J,I) \in \Hi^{n, n+2}(0)_{tr} \\ 
\{\text{pt}\} &\text{ if } \,\, (J,I) \in \Hi^{n, n+2}(0)_{Ntr}.\end{cases}
\]
We call $\Hi^{n, n+1, n+2}(0)_{tr}$ those triples projecting to $\Hi^{n, n+2}(0)_{tr}$, and $\Hi^{n, n+1, n+2}(0)_{Ntr}$ those triples projecting to $\Hi^{n, n+2}(0)_{Ntr}$. The projection restricted to $\Hi^{n, n+2}(0)_{tr}$
\[
p_{1,3}: \Hi^{n, n+1, n+2}(0)_{tr} \to \Hi^{n, n+2}(0)_{tr}
\]
is a Zariski locally trivial bundle.
\end{lemma}
\begin{proof}
Let $(J,I) \in \Hi^{n, n+2}(0)$. Write $I=\langle f_1, f_2, \dots, f_N\rangle$ for a linear basis of $I$ seen as a vector space, i.e. $I \in$ $ \text{Gras}(N, \bigslant{\hat{R}}{\frak{m}^{n+2}})$, where $\hat{R}=\CC[[x,y]]$, $\frak{m}$ is the maximal ideal of $\hat{R}$ and $N+n+2= \frac{(n+2)(n+3)}{2}$.  Extend the basis of $I$ to one of $J$ as $J=\langle g_1, g_2, f_1, \dots, f_N \rangle $. We look at the Grassmannian of possible vector spaces $V$ of dimension $N+1$ such that 
\[
J\supset V \supset I.
\]
It is clear that this is isomorphic to a $\PP^1$. We need to prove that if $(J,I) \in \Hi^{n, n+2}(0)_{tr}$ then all such vector spaces are actually ideals, and if $(J,I) \in \Hi^{n, n+2}(0)_{Ntr}$ only one such vector space is an ideal. Remember that an ideal is a subvector space of $\hat{R}$ closed by multiplication for $x$ and $y$. 

\hspace{3em}Suppose first $(J,I) \in \Hi^{n, n+2}(0)_{tr}$. Then by definition we have $xg_i \in I$ and $yg_i \in I$ for $i=1,2$. Then if $V=\langle \omega_1 g_1 +\omega_2 g_2, f_1, \dots, f_N \rangle$ we have that $x(\omega_1 g_1 +\omega_2 g_2) \in I \subset V$ and $y(\omega_1 g_1 +\omega_2 g_2) \in I \subset V$, proving that all possible choices of $V$ give an ideal. 

\hspace{3em}Suppose now that $(J,I) \in \Hi^{n, n+2}(0)_{Ntr}$.Then by definition there must be a  choice of $[\omega_1 :\omega_2] \in \PP^1$ such that $x(\omega_1g_1 +\omega_2g_2) \notin I $ or $y(\omega_1g_1 +\omega_2g_2) \notin I $. Up to a linear chang of coordinates on $g_1$ and $g_2$ we can then suppose that $xg_1 \notin I$ while $yg_1, xg_2, yg_2 \in I$. Then we see that the only choice of $V$ that is an ideal is $V=\langle g_2, f_1, \dots, f_N \rangle$, as for all $\omega_2 \in \CC$ we have that  $\langle g_1+\omega_2 g_2, f_1, \dots, f_N \rangle$ is not closed by multiplication or it would be $N+2$ dimensional. \\

\hspace{3em} To prove the last claim observe that the map $p_{1,3}$ is simply the restriction to the space of ideals of the standard projection map defined at the level of flag varieties. Then it is a Zariski fibration since the map between flag varieties is. 
\end{proof}

\begin{remark}
We now want to prove that the Betti numbers of $\Hi^{n, n+1, n+2}(0)_{tr}$ coincide with those of $\Hi^{n-1, n, n+1}(0)$. Note the shift in the length of the ideals involved. The geometric relationship between the two spaces remains unclear.
\end{remark}

\begin{observation}[Fixed points $\Hi^{n, n+2}(\CC^2)_{tr}$]
The fixed points of $\Hi^{n, n+2}(\CC^2)_{tr}$ are in bijection with Young diagrams $Y$ of $n+2$ with $2$ \emph{removable} boxes marked. A removable box is a \emph{corner}, i.e. a box $(i, j)\in Y$ such that $(i+1, j) \notin Y$ and $(i, j+1) \notin Y$. We call such an object $(Y,S)$ where $Y$ is the Young diagram and $S$ is the set of marked boxes of cardinality 2.

\vspace{1em}
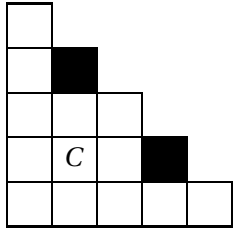
\begin{minipage}{0.45\textwidth}
\begin{center}
\ytableausetup{boxsize=1.5em, aligntableaux=bottom}
\begin{ytableau} 
\,  \\
\, &*(black)\, \\
\,&\,&\,\\
\,&C&\,&*(black)\,\\
\,&\,&\,&\,&\,
\end{ytableau}
\end{center}
\end{minipage}
\begin{minipage}{0.45\textwidth}
A Young diagram of size $n+2$ with two removable boxes marked in black. 
The box with the $C$ is the only box on the same column and on the same row of a marked box. 
\end{minipage}

\vspace{1em}

The fact that the two marked boxes cannot be in the same row or column is a consequence of the fact that we ask the monomial ideals $J, I$ to be trivially related. \\

We call \emph{relevant} all the boxes of $Y$ that are not marked and are not on the same row and the same column of a marked box. In the picture above the relevant boxes are the white, empty ones.  
\end{observation}
\begin{proposition}{\cite[Corollary~5.3]{nakajima2008perverse}} 
\label{propositionNaksmooth}
Let $\TT_{{\infty}}$ acts on $\Hi^{n, n+2}(\CC^2)_{tr}$.  Then the positive part of the tangent space of $\Hi^{n, n+2}(\CC^2)_{tr}$ at a fixed point $(Y,S)$ has dimension given by 
\begin{equation}
\label{posremovable}
\text{pos}_{\infty} \left((Y,S)\right) \quad = \quad n+1-\ell(Y).
\end{equation}
The Poincar\'{e} polynomial of $\Hi^{n, n+2}(\CC^2)_{tr}$ is 
\[
P_q\left(\Hi^{n, n+2}(\CC^2)_{tr}\right) = \sum_{(Y,S)} q^{\text{pos}_{\infty} \left((Y,S)\right)},
\]
where the sum is on all torus fixed points. 
\end{proposition}

Observe that, even though $Y$ is a partition of $n+2$, the formula compares the number of its columns with $n+1$ and not $n+2$.

\subsection*{A combinatorial correspondence}
We want to relate the parametrization of the fixed points of $\Hi^{n-1,n, n+1}(\CC^2)$ with that of $\Hi^{n, n+2}(\CC^2)_{tr}$. More precisely consider the two sets $A_n$ and $B_n$ consisting of 
\begin{itemize}
\item[$A_n$: ] Fixed points of $\Hi^{n-1,n, n+1}(\CC^2)$, $\Gamma=(\Gamma_1, \Gamma_2, \Gamma_3)$. We keep the notations of the previous chapters and we call the box $\Gamma_3\setminus \Gamma_2$ marked with \ytableausetup{boxsize=0.9em, aligntableaux=bottom}\begin{ytableau} *(cyan)2\end{ytableau} as $\alpha'_l$ and the box $\Gamma_2\setminus\Gamma_1$ marked with \ytableausetup{boxsize=0.9em, aligntableaux=bottom}\begin{ytableau} *(green)1\end{ytableau} as $\alpha_j$.
\item[$B_n$: ] A couple $(Y,S)$ of a Young diagram $Y$ of size $n+2$ and $S$ a subset of marked removable boxes of $Y$ of size $2$. 
\end{itemize}
The goal is to construct a $2$ to $1$ map $b_n: A_n\to B_n$ such that 
\begin{align}
\label{condition2to1}
\text{ if } \quad b_n^{-1}\left((Y,S)\right)= \left\{\Gamma, H \right\}& \quad \text{ and } \quad \text{pos}_{\infty}\left((Y,S)\right) = r \qquad \text{ then } 
\\
\text{pos}_{\infty}\left(\Gamma\right)& = r \text{ and } \text{pos}_{\infty}\left(H\right) = r+1 \nonumber
\end{align}
where the integers $\text{pos}_{\infty}\left((Y,S)\right)$ and $\text{pos}_{\infty}\left(\Gamma\right)$ are calculated combinatorially as in formulae (\ref{posremovable}) and (\ref{posskew}) respectively. \\

\hspace{3em}To construct the map $b_n: A_n\xrightarrow{2:1} B_n$ we actually start from an element of $B_n$ and create two elements of $A_n$ with a procedure that we show is both injective and surjective. \\

Let $\lambda \vdash n$ be a Young diagram of size $n$. We consider all possible $(Y,S) \in B_n$ with $Y\setminus S = \lambda$. We say that all such elements of $B_n$ are at \emph{level} $\lambda$. Call, as usual, $\alpha_0, \dots, \alpha_s$ the standard monomial generators of $I_{\lambda}$. Then we denote with $(\alpha_i, \alpha_k)$, with $i>k$, the element $(Y, S)$ of $B_n$ given by $Y=\lambda\cup \{\alpha_i\} \cup \{\alpha_k\}$ and $S=\{\alpha_i\} \cup \{\alpha_k\}$ i.e. with $\alpha_i, \alpha_k$ marked in black. Observe that we have $\frac{s(s+1)}{2}$ such choices. We order the elements of $B_n$ at level $\lambda$ in this way 
\[
(\alpha_i, \alpha_k) > (\alpha_r, \alpha_t) \text{ with } i>k, r>t \qquad \text{ if and only if }\qquad \begin{cases}i>r &\text{or,}\\
i=r, \, k>t.
\end{cases}
\]
It is clear that varying $\lambda\vdash n$ we obtain all of the elements of $B_n$.

\vspace{1em}
\begin{minipage}{0.45\textwidth}
\begin{center}
\ytableausetup{boxsize=1.3em, aligntableaux=bottom}
\begin{ytableau} 
*(lightgray)\alpha_3\\
\, &*(lightgray)\alpha_2 \\
\, &\, \\
\,&\,&*(lightgray)\alpha_1\\
\,&\,&\,&*(lightgray)\alpha_0
\end{ytableau}

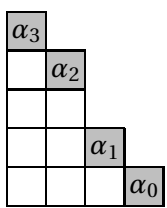
\small{The Young diagram $\lambda$ is in white and in gray there are the $\alpha_i$s between which we need to chose two boxes to create an element of $B_n$ of level $\lambda$.  }
\end{center}
\end{minipage}
\begin{minipage}{0.45\textwidth}
\begin{center}
\ytableausetup{boxsize=1.3em, aligntableaux=bottom}
\begin{ytableau} 
*(black)\\
\, \\
\, &\, \\
\,&\,\\
\,&\,&\,&*(black)
\end{ytableau}
\ytableausetup{boxsize=1.3em, aligntableaux=bottom}
\begin{ytableau} 
\none[\,]\\
\none[\,]\\
\none[>]\\
\none[\,]\\
\none[\,]\\
\end{ytableau}
\ytableausetup{boxsize=1.3em, aligntableaux=bottom}
\begin{ytableau} 
\none[\,]\\
\, &*(black) \\
\, &\, \\
\,&\,&*(black)\\
\,&\,&\,
\end{ytableau}

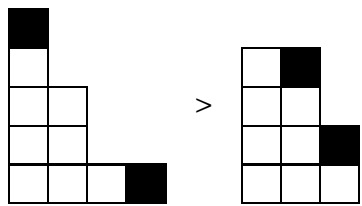
\small{The element of $B_n$ represented as  $(\alpha_3, \alpha_0)$ is bigger then the one represented as $(\alpha_2, \alpha_1)$ }
\end{center} 
\end{minipage}

\vspace{1em}

Given $\lambda\vdash n$ we denote $\beta_0, \dots, \beta_{s-1}$ the \emph{corners} of $\lambda$. We start numbering from the right. 

\vspace{1em}
\begin{minipage}{0.35\textwidth}
\begin{center}
\ytableausetup{boxsize=1.3em, aligntableaux=bottom}
\begin{ytableau} 
\beta_2 \\
\, &\beta_1 \\
\,&\,\\
\,&\,&\beta_0
\end{ytableau}
\end{center}
\end{minipage}
\begin{minipage}{0.55\textwidth}
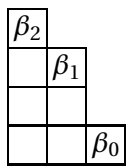
The corners of $\lambda\vdash n$. Recall that a corner is box in a Young diagram $\lambda$ such that there are no other boxes of $\lambda$ on its row to the right nor on its column above. 
\end{minipage}
\vspace{1 em}

We will denote $(\alpha_i, \beta_r)$ the element $\Gamma=(\Gamma_1, \Gamma_2, \Gamma_3)$ of $A_n$ given by $\Gamma_1=\lambda$, $\Gamma_2=\lambda\cup \{\beta_r\}$ and $\Gamma_3= \lambda\cup \{\beta_r\}\cup\{\alpha_i\}$. This is the same as marking with\, \ytableausetup{boxsize=0.9em, aligntableaux=bottom}\begin{ytableau} *(green)1\end{ytableau} the box $\beta_r$ and with\, \ytableausetup{boxsize=0.9em, aligntableaux=bottom}\begin{ytableau} *(cyan)2\end{ytableau} the box $\alpha_r$. We say that all elements in $A_n$ that arise in this way are at \emph{level} $\lambda$. It is clear that varying $\lambda\vdash n$ we obtain all of the elements of $A_n$.\\

\begin{definition}
Let $\lambda\vdash n$. We utilize all the notations as specified above. Given a couple of generators of $I_{\lambda}\,\,$, $(\alpha_{s-i}, \alpha_{s-k})$, with $0\leq i<k \leq s$,  we associate to it a couple of  corners of $\lambda$, precisely  $(\beta_{k-i-1}, \beta_{s-i-1})$ . Then to the element $(\alpha_i, \alpha_k) \in B_n$ at level $\lambda$ we associate the two elements of $A_n$ at level $\lambda$ given by:
\begin{equation}
\label{definitionbn}
(\alpha_i, \alpha_k) \in B_n \xrightarrow{b_n^{-1}}  \left\{ (\alpha_{s-i}, \beta_{k-i-1}), (\alpha_{s-k}, \beta_{s-i-1})\right\} \subset A_n.
\end{equation}
\end{definition}

\begin{observation}Maybe the above map is more clear if we interpret it in terms of the order we have introduced on $B_n$. More precisely, if 
\begin{align*}
(\alpha_s, \alpha_{s-1})>(\alpha_{s}, \alpha_{s-2})>(\alpha_{s}, \alpha_{s-3})>\dots &> (\alpha_s, \alpha_{0})>\\
>(\alpha_{s-1}, \alpha_{s-2})>(\alpha_{s-1}, \alpha_{s-3})>\dots&> (\alpha_{s-1}, \alpha_{0})>\\
\dots\\
 >(\alpha_2, \alpha_{1})&> (\alpha_2, \alpha_{0})>\\
 &> (\alpha_1, \alpha_{0})
\end{align*}
 is the ordered list of all elements of $B_n$ of level $\lambda$, then we associate to them, respectively,  
 \begin{align*}
 \left\{(\alpha_s, \beta_0), (\alpha_{s-1}, \beta_{s-1})\right\},  \left\{(\alpha_s, \beta_1), (\alpha_{s-2}, \beta_{s-1})\right\}, \dots, & \left\{(\alpha_s, \beta_{s-1}), (\alpha_{0}, \beta_{s-1})\right\},\\
  \left\{(\alpha_{s-1}, \beta_0), (\alpha_{s-2}, \beta_{s-2})\right\}, \dots,  &\left\{(\alpha_{s-1}, \beta_{s-2}), (\alpha_{0}, \beta_{s-2})\right\}, \\
\dots \\
  \left\{(\alpha_2,\beta_{0}), (\alpha_{1}, \beta_{1})\right\},  & \left\{(\alpha_2,\beta_1)( \alpha_{0}, \beta_1)\right\},\\
 & \left\{(\alpha_1,\beta_{0})( \alpha_{0}, \beta_0)\right\}. 
 \end{align*}

\end{observation}
\begin{example}
Below we depict an example. On the left we have the first element of $B_n$ at level $\lambda$ and on the right the two elements of $A_n$ of level $\lambda$ in its preimage along $b_n$. Observe that $\text{pos}_{\infty}\left((\alpha_3, \alpha_2) \right) = \text{pos}_{\infty}\left((\alpha_3, \beta_0) \right)-1=\text{pos}_{\infty}\left((\alpha_2, \beta_2) \right)$. 

\vspace{0.3em}

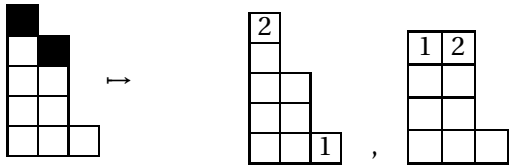
\begin{minipage}{0.2\textwidth}
\begin{center}
\ytableausetup{boxsize=1 em, aligntableaux=bottom}
\begin{ytableau} 
*(black)\\
\, &*(black) \\
\, &\, \\
\,&\,\\
\,&\,&\,
\end{ytableau}
\begin{ytableau} 
\none[\,]\\
\none[\,]\\
\none[\mapsto]\\
\none[\,]\\
\none[\,]\\
\end{ytableau}
\end{center}
\end{minipage}
\begin{minipage}{0.34\textwidth}
\begin{center}
\ytableausetup{boxsize=1 em, aligntableaux=bottom}
\begin{ytableau} 
2\\
\, \\
\, &\, \\
\,&\,\\
\,&\,&1
\end{ytableau}
\ytableausetup{boxsize=1 em, aligntableaux=bottom}
\begin{ytableau} 
\none[\,]\\
\none[\,]\\
\none[\,]\\
\none[\,]\\
\none[,]\\
\end{ytableau}
\ytableausetup{boxsize=1.1em, aligntableaux=bottom}
\begin{ytableau} 
\none[\,]\\
1 & 2 \\
\, &\, \\
\,&\, \\
\,&\,&\,
\end{ytableau}
\end{center}
\end{minipage}

\vspace{0.8em}

The example below, instead, shows $b_n^{-1}\left((\alpha_2, \alpha_0) \right) = \left\{(\alpha_2, \beta_1), (\alpha_0, \beta_1)\right\}$. Observe that $\text{pos}_{\infty}\left((\alpha_2, \alpha_0) \right)$ $ = \text{pos}_{\infty}\left((\alpha_0, \beta_1) \right)=\text{pos}_{\infty}\left((\alpha_2, \beta_1) \right)-1 $.

\begin{minipage}{0.2\textwidth}
\begin{center}
\ytableausetup{boxsize=1 em, aligntableaux=bottom}
\begin{ytableau} 
\, &*(black) \\
\, &\, \\
\,&\,\\
\,&\,&\,&*(black)
\end{ytableau}
\begin{ytableau} 
\none[\,]\\
\none[\,]\\
\none[\mapsto]\\
\none[\,]\\
\none[\,]\\
\end{ytableau}
\end{center}
\end{minipage}
\begin{minipage}{0.34\textwidth}
\begin{center}
\ytableausetup{boxsize=1 em, aligntableaux=bottom}
\begin{ytableau} 
\,&2 \\
\, &1 \\
\,&\,\\
\,&\,&\,
\end{ytableau}
\ytableausetup{boxsize=1 em, aligntableaux=bottom}
\begin{ytableau} 
\none[\,]\\
\none[\,]\\
\none[\,]\\
\none[\,]\\
\none[,]\\
\end{ytableau}
\ytableausetup{boxsize=1.1em, aligntableaux=bottom}
\begin{ytableau} 
\none[\,]\\
\, \\
\, &1 \\
\,&\, \\
\,&\,&\,& 2
\end{ytableau}
\end{center}
\end{minipage}
\end{example}

\vspace{0.7em}
\begin{lemma}
\label{proofpropertybn}
The map $b_n : A_n \to B_n$  of Definition \ref{definitionbn} is two to one and satisfies the condition (\ref{condition2to1}). 
\end{lemma}
\begin{proof}
It is clear that $b_n$ is well defined and two to one at level $\lambda$ for each $\lambda \vdash n$.  It is also clear that it is defined on all of $A_n$ and surjective. In fact all elements of $B_n$ and $A_n$ arise in the way we described and are involved in the construction. To prove that is well defined at different levels $\lambda$ and $\mu$ it is sufficient to observe that if there exist $\beta_1$ corner of $\lambda$,  $\beta_2$ corner of $\mu$, $\alpha_1$ generator of $I_{\lambda}$ and $\alpha_2$ generator of $I_{\mu}$ such that $(\alpha_1,\beta_1 )=(\alpha_2, \beta_2)$, then necessarily $\alpha_1=\alpha_2$ and so $\lambda = \mu$ implying that we are actually at the same $\lambda$ level. \\

Then we only need to prove that condition (\ref{condition2to1}) is satisfied. To do so observe that:
\[
\text{pos}_{\infty}\left((\alpha_i,\alpha_k)\right) = \begin{cases}n+1-\ell(\lambda) &\text{ for all } i>k>0,\\
 n-\ell(\lambda) &\text{ for all } i>k=0 \end{cases}
\]
whereas 
\[
\text{pos}_{\infty}\left((\alpha_i,\beta_r)\right) = \begin{cases}n+1-\ell(\lambda) &\text{ for all } r \geq i-1>-1, \\
n+1-\ell(\lambda)+1 &\text{ for all }  i>0, i-1>r, \\
 n-\ell(\lambda) &\text{ for all } i=0. \end{cases}
\]
Then looking at the definition of $b_n^{-1}(\alpha_{s-i}, \alpha_{s-k})$ in (\ref{definitionbn}) it is clear that : either $s-k=0$ so that $k-i-1\geq s-i-1$ and then 
\[\text{pos}_{\infty}\left((\alpha_{s-i}, \beta_{k-i-1})\right)=\text{pos}_{\infty}\left((\alpha_{s-i}, \alpha_{0})\right)+1 = \text{pos}_{\infty}\left((\alpha_{s-k}, \beta_{s-i-1})\right)+1;
\]
or $s-k>0$ and then $s-i>k-i-1$ so that, again, 
\[
\text{pos}_{\infty}\left((\alpha_{s-i}, \beta_{k-i-1})\right)=\text{pos}_{\infty}\left((\alpha_{s-i}, \alpha_{0})\right)+1 = \text{pos}_{\infty}\left((\alpha_{s-k}, \beta_{s-i-1})\right)+1.
\]
This completes the proof. 
\end{proof}

\begin{proposition}
The Poincar\'{e} polynomials of $\Hi^{n, n+1, n+2}(0)$ satisfy 
\begin{equation}
\label{generatingtheorem}
\sum_{n\geq 0 } P_q\, \left(\Hi^{n, n+1, n+2}(0) \right) z^n = \frac{q+1}{(1-zq)(1-z^2q^2)}\,\, \prod_{m\geq 1} \frac{1}{1-z^mq^{m-1}}.
\end{equation}
\end{proposition}
\begin{proof}
We proved in Proposition \ref{affinepavingfor123} that $\Hi^{n, n+1, n+2}(0)$ has an affine paving, given by the attracting sets for the torus $\TT_{1^+}$, and that the dimensions of these affine cells are given by the positive parts of the tangent spaces. In the previous section, Proposition  \ref{changingtorus3}, we proved that the Poincar\'{e} polynomial for every $\Hi^{n, n+1, n+2}(0)$ can be computed by using the positive parts of the tangent spaces at fixed points with respect to the torus $\TT_{\infty}$ and given by formula (\ref{posskew}). Then, the  existence of map $b_n$ with property (\ref{condition2to1}), as proved in Lemma \ref{proofpropertybn}, shows that 
\[
P_q(\Hi^{n, n+1, n+2}(0))= (q+1)P_q(\Hi^{n+1, n+3}(0)_{tr})
\]
thanks to Proposition \ref{propositionNaksmooth}. Finally we use equation (\ref{generatingnakajima}) to conclude. 
\end{proof}

Now we see how we get for free a cell decomposition of $\Hi^{n, n+2}(0)$, a combinatorial formula for its Poincar\'{e} polynomial and a generating function for all the  Poincar\'{e} polynomials of $\Hi^{n, n+2}(0)$ for different $n\in \NN$. 

\begin{lemma}
\label{cellrespected}
Let $\Gamma= (\Gamma_1, \Gamma_2, \Gamma_3)\vdash [n, n+1, n+2]$ be a fixed point of $\Hi^{n, n+1, n+2}(0)_{Ntr}$. Call $\AAA_{\Gamma} \subset \Hi^{n, n+1, n+2}(0)$ the attracting affine cell with respect to the torus action $\TT_{1^+}$. If $I=(I_1, I_2, I_3) \in \AAA_{\Gamma}$, then $I\in \Hi^{n, n+1, n+2}(0)_{Ntr}$ . 
\end{lemma}
\begin{proof}
By hypothesis $\Gamma= (\Gamma_1, \Gamma_2, \Gamma_3) \in \Hi^{n, n+1, n+2}(0)_{Ntr}$. This means that either $x\alpha'_l$ or $y\alpha'_l$ are not in $I_{\Gamma_3}$, suppose $x\alpha'_l \notin I_{\Gamma_3}$. Since $\lim_{t\to0} t\cdot I_1 = I_{\Gamma_1}$ there exists $f \in I_1$ such that $\lim_{t\to 0} t\cdot f = \alpha'_l$. Since $\lim_{t\to 0} t\cdot xf = x\alpha'_{l}$, if $xf \in I_3$ we would have an absurd. Thus $I=(I_1, I_2, I_3)$ is not trivially related. 
\end{proof}

\begin{observation}
Notice that a similar statement for $\Gamma= (\Gamma_1, \Gamma_2, \Gamma_3)$ fixed point of $\Hi^{n, n+1, n+2}(0)_{tr}$ is false. Consider for example the fixed point given by the nested monomial ideals $(x, y^2) \supset (x^2, xy, y^2) \supset (x^2, xy, y^3)$. The attracting cell is parametrized as

\begin{minipage}{0.6\textwidth}
\[\AAA_{\Gamma} \cong \AAA^{3} \cong \left\{
\begin{matrix}
(x+\omega y , y^2) &\supset & (x^2, xy, y^2) \\
& \supset & (x^2+\alpha y^2, xy+\beta y^2, y^3)
\end{matrix}
\right\}\]
\end{minipage}
\begin{minipage}{0.35\textwidth}
\begin{center}
\ytableausetup{boxsize=1.1em, aligntableaux=bottom}
\begin{ytableau} 
\none[] \\
\none[] \\
\none[\Gamma =] \\
\none[] 
\end{ytableau}
\ytableausetup{boxsize=1.1em, aligntableaux=bottom}
\begin{ytableau} 
*(cyan) 2\\
\,\\
\,& *(green)1
\end{ytableau}
\end{center}
\end{minipage}

\vspace{0.8em}
One can see that all of the points in $\AAA_{\Gamma}$ with $\alpha \neq -\omega^2$ or $\beta\neq \omega$  are in $\Hi^{n, n+1, n+2}(0)_{Ntr}$. 
\end{observation}

\begin{lemma}
\label{cellprojected}
Consider the $\TT_{1^+}$ action on $R=\CC[x,y]$. Let $ (\Gamma_1, \Gamma_3)$ with $\Gamma_1 \vdash n$  and $\Gamma_3\vdash n+2$ be a fixed point of $\Hi^{n, n+2}(0)_{tr}$. Call $\{\alpha_j, \alpha_l \} = \Gamma_3\setminus \Gamma_1$ with $\alpha_j = \lim_{t\to 0 } t\cdot(\alpha_j +\alpha_l)$. We define 
\[
\Gamma := (\Gamma_1, \Gamma_2,\Gamma_3) \text{ with } \Gamma_2:= \Gamma_1\cup \{\alpha_j\}
\] 
and denote $\AAA_{\Gamma} \subset \Hi^{n, n+1, n+2}(0)$ the attracting affine cell of $\,\Gamma$. Then if ${A}_{\Gamma_1,\Gamma_3}\subset \Hi^{n, n+2}(0)_{tr}$ is the attracting set of the fixed point $\Gamma_1,\Gamma_3$ we have that:
\[
p_{1,3}: \AAA_{\Gamma}  \cong {A}_{\Gamma_1,\Gamma_3}
\] 
where $p_{1,3}:\Hi^{n, n+1, n+2}(0) \to \Hi^{n, n+2}(0) $ is the projection on the first and third factor.
\end{lemma}
\begin{proof}
By Lemma \ref{fibration} we have that 
\[
p_{1,3}^{-1}(I_{\Gamma_1}, I_{\Gamma_3})= \left\{I_{\Gamma_1}\supset I_{\Gamma_1} + \left(\omega_1 \alpha_j +\omega_2\alpha_l\right)\supset I_{\Gamma_3}\,\,\left\vert\,\, [\omega_1:\omega_2] \in \PP^1\,\right.\right\} \subset \Hi^{n, n+1, n+2}(0).
\] 
Since $p_{1,3}: \AAA_{\Gamma}  \to {A}_{\Gamma_1,\Gamma_3}$ is clearly surjective we only need to prove that it is injective. Let $I=(I_1, I_2, I_3), J=(J_1, J_2, J_3)\in \AAA_{\Gamma}$, with $p_{1,3}(I)= p_{1,3}(J)$. Then $I=J$. In fact let $f_j, f_l \in I_1=I_3$ be such that $f_j, f_l \notin I_3=J_3$ and $\lim_{t\to 0}t\cdot f_i = \alpha_i$ for $i=j,l$. Then $\omega_1 f_j +\omega_2 f_l \in I_2$ if and only if $[\omega_1:\omega_2]=[0:1]$: otherwise the hypothesis $\alpha_j = \lim_{t\to 0 }t\cdot ( \alpha_j +\alpha_l)$ implies $\alpha_j \in I_{\Gamma_2}$ that is an absurd by definition. The same is true for $J_2$ proving $J_2=I_2$ and thus $I=J$. 
\end{proof}

\begin{definition} We write $\Gamma=(\Gamma_1, \Gamma_3)\vdash [n, n+2]$ for a couple of nested Young diagrams of size $n$ and $n+2$ respectively.  We define a map 
\[s_n : \left\{\Gamma\, \vdash \,[n, n+2]\right\} \to \left\{\Gamma\, \vdash \,[n,n+1, n+2]\right\}\]
as $s_n(\Gamma_1, \Gamma_3) = (\Gamma_1, \Gamma_2, \Gamma_3)$ where 
\[
\Gamma_2 := \Gamma_1 \cup \left\{\alpha \in \Gamma_3\setminus \Gamma_1\,\, \left\vert\,\, \alpha \text{ has minimal degree and minimal $y$ degree in }\Gamma_3\setminus \Gamma_1\,\right.\right\}.
\]
\end{definition}

\begin{proposition}
Let $n\in \NN$. The space $\Hi^{n, n+2}(0)$ has an affine cell decomposition indexed by $\Gamma \vdash [n, n+2]$. Its Poincar\'{e} polynomial is given by: 
\begin{equation}
\label{poicarennpiudue}
P_q \,\left(\Hi^{n, n+2}(0) \right) = \sum_{\Gamma \vdash [n, n+2]} \,\, q^{\pos_{1^+} \left(s_n(\Gamma)\right)}
\end{equation}
where the quantity $\pos_{1^+} \left(s_n(\Gamma)\right)$ is specified in Formula \ref{posskewmale} of Proposition \ref{mainproposition}. The Poincar\'{e} polynomials satisfy:
\begin{equation}
\label{generating13}
\sum_{n\geq 0 } P_q\, \left(\Hi^{n, n+2}(0) \right) z^n = \frac{1+q-qz}{(1-zq)(1-z^2q^2)}\,\, \prod_{m\geq 1} \frac{1}{1-z^mq^{m-1}}.
\end{equation}

\end{proposition}
\begin{proof}
Consider the action of $\TT_{1^+}$ on $\Hi^{n, n+2}(0)$: Lemma \ref{cellrespected} and Lemma \ref{cellprojected} show that the attracting sets are affine cells of the affine cell decomposition of \newline$\Hi^{n,n+1, n+2}(0)$. The dimensions of these cells were calculated in Proposition \ref{mainproposition}. Then Proposition \ref{Fulton} proves the first equality for the Poincar\'{e} polynomial. \\

\hspace{3em} To prove the identity (\ref{generating13}) for the generating function, thank to Lemma \ref{fibration}, it is enough to use the generating functions (\ref{generatingnakajima}) and (\ref{generatingtheorem}).
 
\end{proof}

\nocite{*}
\bibliographystyle{alpha}
\bibliography{thesis}
\end{document}